\documentclass[11pt]{amsart}

\usepackage{amsmath}
\usepackage{amsfonts}
\usepackage{amssymb}
\usepackage{amscd}
\usepackage{amsthm}
\usepackage{framed}
\usepackage{fullpage}
\usepackage{graphicx}
\usepackage{latexsym}
\usepackage[numbers]{natbib}    
\usepackage{multirow}
\usepackage{tikz}
\usetikzlibrary{positioning}
\usetikzlibrary{arrows}

\usepackage{hyperref}
\hypersetup{colorlinks,linkcolor={red},citecolor={blue},urlcolor={blue}}

\newtheorem{theorem}{Theorem}[section]
\newtheorem{lemma}[theorem]{Lemma}
\newtheorem{proposition}[theorem]{Proposition}
\newtheorem{corollary}[theorem]{Corollary}

\theoremstyle{definition}
\newtheorem{definition}[theorem]{Definition}

\newtheorem{remark}[theorem]{Remark}

\newtheorem{problem}[theorem]{Problem}

\newtheorem{Notation}[theorem]{Notation}

\numberwithin{equation}{section}

\usepackage[inline]{enumitem}   
\makeatletter
\newcommand{\inlineitem}[1][]{%
\ifnum\enit@type=\tw@
    {\descriptionlabel{#1}}
  \hspace{\labelsep}%
\else
  \ifnum\enit@type=\z@
       \refstepcounter{\@listctr}\fi
    \quad\@itemlabel\hspace{\labelsep}%
\fi}
\makeatother

\newcommand{\CC}{\mathbb C}
\newcommand{\HH}{\mathbb H}

\newcommand{\NN}{\mathbb N}
\newcommand{\cD}{\mathcal D}
\newcommand{\cA}{\mathcal A}
\newcommand{\cH}{\mathcal H}

\newcommand{\PP}{\mathbb P}
\newcommand{\QQ}{\mathbb Q}
\newcommand{\RR}{\mathbb R}
\newcommand{\ZZ}{\mathbb Z}

\newcommand{\SL}{\mathop{\mathrm {SL}}\nolimits}
\newcommand{\SO}{\mathop{\mathrm {SO}}\nolimits}
\newcommand{\Sp}{\mathop{\mathrm {Sp}}\nolimits}
\newcommand{\Orth}{\mathop{\null\mathrm {O}}\nolimits}

\newcommand{\im}{\mathop{\mathrm {Im}}\nolimits}
\newcommand{\rank}{\mathop{\mathrm {rk}}\nolimits}
\newcommand{\latt}[1]{{\langle{#1}\rangle}}
\newcommand{\ord}{\mathop{\mathrm {ord}}\nolimits}
\newcommand{\II}{\mathop{\mathrm {II}}\nolimits}
\def\Grit{\mathbf{G}}
\def\Borch{\mathbf{B}}

\def\dim{\operatorname{dim}}

\def\det{\operatorname{det}}
\def\ch{\operatorname{ch}}
\def\sch{\operatorname{sch}}
\def\w{\operatorname{w}}
\newcommand{\bQ}{\mathbf{Q}_\mathfrak{g}}
\newcommand{\bP}{\mathbf{P}_\mathfrak{g}}

\newenvironment{psmallmatrix}
  {\left(\begin{smallmatrix}}
{\end{smallmatrix}\right)}

\begin{document}

\title[Singular automorphic products and BKM algebras]{Singular automorphic products and BKM algebras}

\author{Kaiwen Sun}

\address{School of Mathematical Sciences, University of Science and Technology of China, Hefei 230026, Anhui, China}

\email{kwsun@ustc.edu.cn}

\author{Haowu Wang}

\address{School of Mathematics and Statistics, Wuhan University, Wuhan 430072, Hubei, China}

\email{haowu.wangmath@whu.edu.cn}

\author{Brandon Williams}

\address{Institute for Mathematics, Heidelberg University, 69120 Heidelberg, Germany}

\email{bwilliams@mathi.uni-heidelberg.de}

\subjclass[2020]{11F27, 11F50, 11F55, 17B22, 17B65, 17B69}

\date{\today}

\keywords{affine Lie superalgebra, automorphic product of singular weight, Borcherds--Kac--Moody superalgebra, denominator identity, Jacobi form, modular invariant, vertex algebra, Weyl--Kac character formula}

\begin{abstract}
After proving the moonshine conjecture, Borcherds asked in 1995 whether there exist only finitely many holomorphic automorphic products of singular weight. We show that, up to conjugation, there are exactly 27 such products on lattices of the form $2U\oplus L$. The classification proceeds by identifying each product at a standard $1$-dimensional cusp with the superdenominator of a BKM superalgebra that carries a graded module structure over an affine Lie superalgebra $\hat{\mathfrak{g}}$. The auxiliary algebras $\hat{\mathfrak{g}}$ are shown to satisfy strong conditions that reduce the possibilities to an initial list of 1372 candidates. An elimination argument leaves exactly 85 algebras $\hat{\mathfrak{g}}$ that correspond to the 27 products. This classification is closely connected with certain distinguished families of vertex operator superalgebras. Our results provide a systematic framework for extending affine Lie superalgebras to hyperbolic ones with modularity, generalizing the work of Feingold and Frenkel (1983). 
\end{abstract}

\maketitle

\begin{small}
\tableofcontents
\end{small}

\section{Introduction}
In this paper, we study infinite-dimensional Lie superalgebras that exhibit modularity and apply them to classify certain hyperbolic analogs of the celebrated Macdonald identities \cite{Mac72}. The subject traces back to the 1970s, when Kac and Moody first recognized the Macdonald identities as the denominator formulas of affine Lie algebras and superalgebras. This observation is a fundamental connection between Lie algebras and modular forms and motivated further study of modularity in algebraic structures, leading to a series of landmark advances in the following years. 

In 1979, Conway and Norton proposed the monstrous moonshine conjecture \cite{CN79} based on observations of McKay and Thompson that related the representation theory of the monster group to the Fourier coefficients of modular functions. In 1984, Kac and Peterson \cite{KP84} proved that the characters of irreducible integrable highest weight modules over affine Kac--Moody algebras are modular invariant. Analogous modular phenomena were observed shortly thereafter for the minimal series of the Virasoro algebra  \cite{RC83} and for rational conformal field theories \cite{Car86}. In 1986, Borcherds introduced vertex algebras \cite{Bor86}, and two years later he generalized Kac--Moody algebras to what are now known as Borcherds--Kac--Moody (or simply BKM) algebras \cite{Bor88}. Also in 1988, Frenkel, Lepowsky, and Meurman \cite{FLM88} constructed the moonshine module as a vertex operator algebra whose automorphism group is precisely the monster group. In 1992, Borcherds \cite{Bor92} realized a BKM algebra as the BRST cohomology of this vertex operator algebra and, by analyzing the induced action of the monster group on the resulting denominator identity, completed the proof of the moonshine conjecture. Subsequently, he developed the singular theta lift \cite{Bor95, Bor98}, producing automorphic products, that is, modular forms on orthogonal groups $\Orth(l,2)$ with infinite product expansions, which serve as candidate denominator functions of BKM algebras. At the same time, Zhu \cite{Zhu96} established the modular invariance of characters of suitable vertex operator algebras, and Dong--Li--Mason \cite{DLM00} later extended that theorem to broader settings.

Borcherds' proof of the moonshine conjecture reveals a profound relationship between vertex algebras, BKM algebras, and modular forms for $\Orth(2,2)$. A natural next step is to explore this relationship in the more general $\Orth(l,2)$ case. 

However, unlike affine Lie algebras and vertex algebras, BKM algebras lack known algebraic criteria for modularity. In search of such a criterion, Borcherds \cite[Problem 16.2]{Bor98} suggested finding all hyperbolic BKM algebras whose denominator functions are modular invariant. Such algebras would constitute natural generalizations of affine Lie algebras, and their associated denominator functions would provide hyperbolic analogs of the Macdonald identities. Among these, those whose denominator functions are holomorphic automorphic products of singular weight are especially interesting: a few known examples \cite{Bor90, Bor92, Sch00, HS03, CKS07, HPPV21, Mol21} indicate that such algebras admit natural vertex-algebraic constructions, making their symmetry groups explicitly describable. In 1995, Borcherds posed the following problem:

\begin{problem}[Problem 3 in Section 17 of \cite{Bor95}]\label{problem:Bor95}
Are there a finite or infinite number of automorphic forms of singular weight that can be written as modular products? 
\end{problem}

In this paper, we obtain a complete classification of singular automorphic products on lattices of type $2U\oplus L$ without imposing any additional conditions. We find that there are exactly $27$ such functions up to conjugation. Each of them gives rise to a BKM (super)algebra at a standard $1$-dimensional cusp, which in turn carries a graded module structure over an affine Lie (super)algebra $\hat{\mathfrak{g}}$. We show that precisely $85$ such algebras $\hat{\mathfrak{g}}$ arise, corresponding to several distinguished families of vertex algebras: holomorphic vertex operator algebras of central charge $24$, holomorphic vertex operator superalgebras of central charge $12$, and exceptional modular invariants of affine Lie (super)algebras. These results connect affine Lie algebras, BKM algebras, vertex algebras, and their associated modular forms. In the remainder of the introduction, we describe the setup and state the main theorems. 

\subsection{Affine Lie superalgebras and Jacobi forms}
Let $\mathfrak{g}$ be a finite-dimensional semi-simple \textit{Kac--Moody Lie superalgebra}; in our convention, such a $\mathfrak{g}$ is required to have no imaginary simple roots (cf. \cite[\S 2.1]{Ray06}). Its irreducible components are simple Lie algebras or orthosymplectic Lie superalgebras of type $\mathfrak{osp}_{1|2r}$. This class admits elegant characterizations. For example, a fundamental result \cite{frenkel1992vertex, DongLepowsky1993, li1996local, GORELIK2007, AL22, CREUTZIG2022108678} states that the affine vertex operator (super)algebra associated with a finite-dimensional simple Lie (super)algebra is rational and $C_2$-cofinite if and only if the underlying algebra is a Lie algebra or Lie superalgebra of type $\mathfrak{osp}_{1|2r}$ and the level is a positive integer. In this paper, we give a new characterization of this class in terms of Jacobi forms, which is essential to our classification of singular automorphic products.

Assume for simplicity that $\mathfrak{g}$ is simple. 
Let $\Delta^+_{\mathfrak{g},\bar{0}}$ and $\Delta^+_{\mathfrak{g},\bar{1}}$ denote the sets of even and odd positive roots, respectively. Throughout this paper, we normalize the bilinear form of $\mathfrak{g}$ so that long roots have square norm $2$. The product side of the (super)denominator identity of the associated affine Lie (super)algebra $\hat{\mathfrak{g}}$ is given by the holomorphic function
\begin{equation}\label{eq:superdenominator-affine}
\vartheta_{\mathfrak{g}}(\tau,\mathfrak{z})=\eta(\tau)^{\rank(\mathfrak{g})} \cdot \prod_{\alpha\in \Delta_{\mathfrak{g},\bar{0}}^+}\frac{\vartheta(\tau, \latt{\alpha,\mathfrak{z}})}{\eta(\tau)}\cdot \prod_{\beta\in \Delta_{\mathfrak{g},\bar{1}}^+}\frac{\eta(\tau)}{\vartheta(\tau, \latt{\beta,\mathfrak{z}})},    
\end{equation}
where $\eta$ and $\vartheta$ are the Dedekind eta function and the odd Jacobi theta function, respectively:
\begin{align*}
\eta(\tau)&=q^{\frac{1}{24}}\prod_{n=1}^\infty(1-q^n), \quad \tau\in \HH, \; q=e^{2\pi i\tau},\\
\vartheta(\tau,z)&=-q^{\frac{1}{8}}\zeta^{-\frac{1}{2}}\prod_{n=1}^\infty(1-q^{n-1}\zeta)(1-q^n\zeta^{-1})(1-q^n), \quad z\in \CC, \; \zeta= e^{2\pi iz}.
\end{align*}
The functions $\vartheta_{\mathfrak{g}}$ are prototypical examples of lattice-index Jacobi forms (see, e.g., \cite{Gri88, GSZ19}), which generalize the classical Jacobi forms introduced by Eichler and Zagier \cite{EZ85}.  

Let $L$ be an even positive-definite lattice. A holomorphic function $\varphi : \HH \times (L \otimes \CC) \rightarrow \CC$ is called a holomorphic (or weakly holomorphic) \textit{Jacobi form} of weight $k$ and index $L$ if it satisfies
\begin{align*}
\varphi \left( \frac{a\tau +b}{c\tau + d},\frac{\mathfrak{z}}{c\tau + d} 
\right)& = (c\tau + d)^k 
\exp{\left(i \pi \frac{c(\mathfrak{z},\mathfrak{z})}{c 
\tau + d}\right)} \varphi ( \tau, \mathfrak{z} ), \quad A=\begin{psmallmatrix}
a & b \\
c & d
\end{psmallmatrix} \in \SL_2(\ZZ),\\
\varphi (\tau, \mathfrak{z}+ x \tau + y)&= 
\exp{\big(-i \pi \big( (x,x)\tau +2(x,\mathfrak{z})\big)\big)} 
\varphi (\tau, \mathfrak{z} ), \quad x,y\in L,
\end{align*}
and is holomorphic (or meromorphic) at infinity. Let $\rank(L)$ denote the rank of $L$. The weight of any non-constant holomorphic Jacobi form of index $L$ is bounded below by $\frac{1}{2}\rank(L)$; this minimal possible weight is referred to as the \textit{singular weight}. 

With these definitions, $\vartheta_{\mathfrak{g}}$ is a holomorphic Jacobi form of singular weight and index $P_\mathfrak{g}^\vee(h_\mathfrak{g}^\vee)$ with a character, where $P_\mathfrak{g}^\vee$ is the dual of the root lattice and $h_\mathfrak{g}^\vee$ is the dual Coxeter number. 

Following Gritsenko--Skoruppa--Zagier \cite{GSZ19}, meromorphic Jacobi forms with product expansions of the form \eqref{eq:superdenominator-affine} are called \textit{theta quotients}. We prove the following result, which generalizes \cite{Wan25a}. 
\begin{theorem}\label{MTH:theta-quotients}
If a theta quotient is a holomorphic Jacobi form of singular weight, then it coincides with the superdenominator function of the affine Lie superalgebra that extends a finite-dimensional semi-simple Kac–Moody Lie superalgebra. 
\end{theorem} 

This yields a one-to-one correspondence between finite-dimensional semi-simple Kac--Moody Lie superalgebras and a distinguished class of Jacobi forms. The present paper aims to extend this correspondence to BKM superalgebras and singular automorphic products. 

\subsection{Automorphic products and BKM superalgebras} Let $l\geq 3$ be an integer. A \textit{modular form} of integral weight $k$ and character $\chi$ for an arithmetic subgroup $\Gamma<\Orth(l,2)$ is a holomorphic function on the affine cone over the associated type IV symmetric domain satisfying 
\begin{align*}
F(t\mathcal{Z})&=t^{-k} F(\mathcal{Z}), \quad \text{for any} \; t\in \CC^\times,\\
F(g\mathcal{Z})&=\chi(g) F(\mathcal{Z}),  \quad \text{for any} \; g \in \Gamma. 
\end{align*}
This definition extends naturally to half-integral weights. The weight of any non-constant $F$ satisfies the lower bound $k\geq l/2-1$; this minimal possible value is called the \textit{singular weight}.

Let $M$ be an even lattice of signature $(l,2)$. Borcherds' singular theta lift takes as input a vector-valued modular form of weight $1-l/2$ with integral principal part for the Weil representation of $\mathrm{Mp}_2(\ZZ)$ associated with the discriminant form $M'/M$, and produces a meromorphic modular form for a suitable subgroup of $\Orth(M)$. This modular form admits an infinite product expansion at each $0$-dimensional cusp and its divisor is a $\ZZ$-linear combination of hyperplanes. Such forms are called \textit{automorphic products}, or Borcherds products. 

BKM superalgebras generalize Kac--Moody algebras and are defined by Chevalley--Serre generators and relations encoded in a generalized Cartan matrix. Unlike Kac--Moody algebras, they allow diagonal Cartan entries to be non-positive, so simple roots may be imaginary rather than necessarily real. BKM superalgebras possess character and supercharacter formulas for irreducible integrable highest weight modules, which give rise to denominator and superdenominator identities. 

Computing root multiplicities of a BKM superalgebra is generally difficult. Following Borcherds' approach, we focus on those with hyperbolic root lattices whose (super)denominator functions are automorphic products. The root (super)multiplicities are then encoded in the Fourier expansion of the automorphic product's input. Two principal classes arise. 

In the first class, the (super)denominator function is a cuspidal automorphic product, and the BKM (super)algebra has only finitely many real simple roots. However, imaginary simple roots are hard to describe, and such algebras have no known natural constructions beyond generator-relation presentations. This class has applications to hyperbolic reflection groups of finite covolume and it was systematically studied by Gritsenko and Nikulin \cite{GN96a, GN96b, GN98a, GN98, GN02, GN18}. 

In the second class, the (super)denominator function is a holomorphic automorphic product of singular weight, and the BKM (super)algebra typically has infinitely many real simple roots. Since the Fourier coefficients of a singular-weight modular form are supported only on isotropic vectors, the imaginary simple roots are easily determined. Furthermore, as noted earlier, such algebras may admit natural constructions with known symmetries inherited from vertex algebras. This case has also been extensively explored around Problem \ref{problem:Bor95}. We summarize the known results below.  

Scheithauer \cite{Sch06, Sch17} achieved the first breakthrough, classifying singular automorphic products on all prime-level lattices. Scheithauer \cite{Sch06} and Dittmann \cite{Dit19} later classified such products on lattices of type $U\oplus U(N)\oplus L$ with squarefree level $N$. More recently, Driscoll-Spittler, Scheithauer, and Wilhelm \cite{DSW23} classified anti-symmetric singular automorphic products on even-rank regular lattices of type $2U\oplus L$. However, these results rely on additional assumptions, including reflectivity, simple zeros, and non-negative principal parts, and they do not cover lattices of odd rank. Separately, singular automorphic products on simple lattices with non-negative principal parts have also been classified \cite{DHS15, OS19}; we recall that there are only finitely many simple lattices and they were completely classified in \cite{BEF16}. The present paper contributes further results in this direction. 

\subsection{Classification strategy and main results} 
In this part, we present a complete solution to a suitable interpretation of Problem \ref{problem:Bor95}. The classification is carried out in three main steps. Our first step is to establish the reflectivity of singular automorphic products. 

Let $F$ be a singular automorphic product on an even lattice of signature $(l,2)$ with $l\geq 3$. The last two named authors \cite{WW23} have shown that $F$ has only simple zeros and that, whenever $F$ vanishes on the hyperplane $\lambda^\perp$, one has $F(\sigma_\lambda(\mathcal{Z}))=-F(\mathcal{Z})$, where $\lambda\in M\otimes\QQ$ satisfies $(\lambda,\lambda)>0$, and $\sigma_\lambda$ is the reflection fixing $\lambda^\perp$, defined by 
$$
\sigma_\lambda : \; M\otimes\RR \to M\otimes\RR, \quad v \mapsto v-\frac{2(\lambda,v)}{(\lambda,\lambda)}\lambda. 
$$
Consequently, $F$ is a reflective modular form on a new underlying lattice $\mathbf{M}$, which in turn implies that such forms do not exist when $l>26$. \textit{Reflectivity} here means that $\sigma_\lambda$ preserves $\mathbf{M}$ whenever $F$ vanishes along $\lambda^\perp$. The lattice $\mathbf{M}$ is typically constructed by rescaling a suitable sublattice of $M\otimes \QQ$, and is therefore not easily controlled. 

We refine this result in the case $M=2U\oplus L$, where $L$ is an even positive-definite lattice with dual lattice $L'$, and $U\cong \begin{psmallmatrix}
0 & -1 \\ -1 & 0    
\end{psmallmatrix}$ is an even unimodular lattice of signature $(1,1)$. In this setting, vector-valued modular forms for the Weil representation are equivalent to Jacobi forms \cite{Gri94}, so the input to the Borcherds theta lift can be realized as a weakly holomorphic Jacobi form of weight $0$ and index $L$ with integral singular Fourier coefficients, i.e., $f(n,\ell)\in\ZZ$ whenever $2n<(\ell,\ell)$.   

\begin{theorem}\label{MTH:reflectivity}
Let $F$ be a singular automorphic product on $2U\oplus L$, and let $\phi$ denote its Jacobi form input with Fourier expansion
\begin{equation}\label{eq:Intro-Fourier}
\phi(\tau,\mathfrak{z})=\sum_{\substack{n\in\ZZ,\, \ell\in L'\\ n\gg -\infty}} f(n,\ell) q^n \zeta^{\ell}, \quad q=e^{2\pi i \tau}, \; \zeta^\ell = e^{2\pi i(\ell,\mathfrak{z})}. 
\end{equation}
Then there exists an even overlattice $\hat{L}$ of $L$ such that $F$ becomes a reflective automorphic product of singular weight on $2U\oplus \hat{L}$. Moreover, $\hat{L}$ is unique when $f(-1,0)=1$, but not in general. Furthermore, $\hat{L}$ can be chosen canonically, by specifying that its dual $\hat{L}'$ is generated over $\ZZ$ by all $\ell \in L'$ with $f(n,\ell)\neq 0$ for some $n\in\ZZ$. 
\end{theorem}  

Note that $f(-1,0)\in\{0,1\}$. The product $F$ is called \textit{anti-symmetric} when $f(-1,0)=1$, and \textit{symmetric} when $f(-1,0)=0$. Moreover, the singular weight condition gives $f(0,0)=\rank(L)$, while the simple zero property implies $f(n,\ell)\in\{-1,0,1\}$ whenever $2n<(\ell,\ell)$. 

\vspace{2mm}

Our second step is to attach a BKM superalgebra to each singular automorphic product. Fix coordinates on $U\oplus L'$ as follows: write $v\in U\oplus L'$ as $(n,\ell,m)$ with $n,m\in \ZZ$ and $\ell\in L'$, so that $(v,v)=(\ell,\ell)-2nm$. In these coordinates, the Weyl vector of $F$ in Theorem \ref{MTH:reflectivity} takes the form $\rho=(-A, \vec{B}, -C)$, with
\begin{equation}
A=\frac{1}{24}\sum_{\ell\in L'}f(0,\ell), \quad \vec{B}=\frac{1}{2}\sum_{\ell>0}f(0,\ell)\ell, \quad C=\frac{1}{\rank(L)}\sum_{\ell>0}f(0,\ell)(\ell,\ell).    
\end{equation}
Observe that $2\vec{B}\in L'$, and $\rho$ is nonzero and isotropic, i.e., $(\rho,\rho)=0$. 

\begin{theorem}\label{MTH:relation to BKM}
Let $F$ be a singular automorphic product on $2U\oplus L$, with Jacobi form input $\phi$ having Fourier expansion \eqref{eq:Intro-Fourier} and Weyl vector $\rho$. Assume that $L$ is canonical in the sense of Theorem \ref{MTH:reflectivity}. Then there exists a BKM superalgebra $\mathcal{G}$ subject to the following conditions (1)-(5). 
\begin{enumerate}
\item The superdenominator function of $\mathcal{G}$ agrees with the product expansion of $F$ at the standard $0$-dimensional cusp of type $U$.   
\item The Weyl vector of $\mathcal{G}$ is $\rho$.
\item The root lattice of $\mathcal{G}$ is $U\oplus L'$. 
\item The real roots of $\mathcal{G}$ are precisely those vectors $\alpha=(n,\ell,m)\in U\oplus L'$ satisfying $2nm<(\ell,\ell)$ and $f(nm,\ell)\neq 0$. A real root $\alpha$ is even if $f(nm,\ell)>0$, and odd if $f(nm,\ell)<0$. Moreover, a real root $\alpha$ is simple if and only if $(\rho,\alpha)=\frac{1}{2}(\alpha,\alpha)$. 
\item Let $t$ be the minimal positive integer with $t\rho\in U\oplus L'$. The imaginary simple roots can only be $-nt\rho$ for positive integers $n$, with supermultiplicity given by $f(n^2t^2AC, nt\vec{B})$. 
\end{enumerate}
\end{theorem}

If $\mathcal{G}$ has nontrivial odd parts, its superdenominator depends solely on root supermultiplicities rather than ordinary multiplicities, so $\mathcal{G}$ is generally non-unique. 

The root lattices of BKM superalgebras motivate the canonicality condition in Theorem \ref{MTH:reflectivity}. 

Theorem \ref{MTH:relation to BKM} yields constraints on singular automorphic products and their underlying lattices. Recall that the level of $L$ is the smallest positive integer $N$ with $N(x,x)\in 2\ZZ$ for all $x\in L'$. An even lattice $M$ of level $N$ is called \textit{regular} if $D^{N/p}$ contains a nontrivial isotropic element for each prime $p\mid N$, where $D:=M'/M$ is the discriminant form and $D^{c}:=\{cx : \, x\in D \}$.

\begin{corollary}\label{Cor:relation to regular}
We retain the assumptions of Theorem \ref{MTH:relation to BKM}. Let $N$ denote the least positive integer with $N\cdot\rho \in U\oplus L$. Then we have the following: 
\begin{enumerate}
\item The level of $L$ is $N$.
\item If $\vec{B}\in L'$, then $2U\oplus L$ is a regular lattice.
\item If $\vec{B}\in L'$ and $C\in\frac{1}{2}\ZZ$, then the input of $F$ has non-negative principal parts, or equivalently, the BKM superalgebra $\mathcal{G}$ has no odd real roots. 
\item If $F$ is anti-symmetric and $C\in\ZZ$, then $\vec{B}\in L'$ and $\rho\in U\oplus L'$. 
\item If $\rho\in U\oplus L'$, then $U\oplus L\cong U(N)\oplus K$ for some even positive-definite lattice $K\cong \rho^\perp / \rho$.  
\end{enumerate}      
\end{corollary}

In the special case that $\rho \in U\oplus L$, we obtain $N=1$ and that $L$ is unimodular, so $2U\oplus L \cong \II_{26,2}$ and $F$ is the denominator function of the fake monster algebra \cite{Bor90}. 
 
Corollary \ref{Cor:relation to regular} reveals a natural link between singular automorphic products and regular lattices. The concept of regularity was introduced by Driscoll-Spittler, Scheithauer, and Wilhelm \cite{DSW23} in the course of their classification work. This condition enables them to find a nonzero $\CC$-linear combination of input components of an anti-symmetric reflective automorphic product on even-rank lattices with small pole orders at each cusp, so that the Riemann–Roch theorem gives an effective upper bound on the level of the lattice. Borcherds' obstruction principle \cite{Bor99} then gives the striking result that exactly $11$ anti-symmetric reflective singular automorphic products with non-negative principal parts occur on even-rank regular lattices of type $2U\oplus L$. 

Despite this progress, the symmetric case remains less tractable, as the analogous upper bound is not sufficiently strong for classification. Also, the method of \cite{DSW23} relies heavily on explicit Weil representation computations that are difficult to extend to odd-rank lattices, and the presence of negative principal parts makes the use of the obstruction principle considerably more difficult. Furthermore, even in the anti-symmetric case, the regularity condition cannot easily be removed. 
Hence a complete classification of singular automorphic products on all lattices of type $2U\oplus L$ has remained elusive. 

\vspace{2mm}

We overcome these difficulties in the third step, where we identify each singular automorphic product at a $1$-dimensional cusp of type $2U$ with a finite-dimensional Lie superalgebra structure. The following theorem extends our previous results \cite{SWW23}, in which only the case of $\mathfrak{g}$ with trivial odd parts was addressed and the level statement in Part (3) was not established.

\begin{theorem}\label{MTH:Lie structure}
We retain the assumptions of Theorem \ref{MTH:relation to BKM}. We define the set
$$
\mathcal{R}:=\{ \ell \in L'\backslash\{0\} : \; f(0,\ell)\neq 0 \}. 
$$

If $\mathcal{R}$ is empty, then $L$ is the Leech lattice, $F$ is the denominator of the fake monster algebra, and the Jacobi form input $\phi$ is the full character of the Leech lattice vertex operator algebra.

\vspace{2mm}

If $\mathcal{R}$ is non-empty, then it is identical to the set of roots of a finite-dimensional semi-simple Kac--Moody Lie superalgebra $\mathfrak{g}$ with $\rank(\mathfrak{g})=\rank(L)$. Decomposing $\mathfrak{g}$ into its irreducible summands gives $\mathfrak{g}=\oplus_{j=1}^s \mathfrak{g}_{j,k_j}$, where $k_j$ denotes the positive integral level of the simple component $\mathfrak{g}_j$. Let $h_j^\vee$ be the dual Coxeter number of $\mathfrak{g}_j$, and define the superdimension of $\mathfrak{g}$ by $\mathrm{sdim}\,\mathfrak{g}:=\sum_{j=1}^s \mathrm{sdim}\, \mathfrak{g}_j$, where $\mathrm{sdim}\, \mathfrak{g}_j$ denotes the difference between the dimensions of the even and odd parts of $\mathfrak{g}_j$. Recall that $f(-1,0)\in\{0,1\}$. Then we have the following identity: 
\begin{equation}\label{eq:Intro-Schellekens type}
\frac{\mathrm{sdim}\,\mathfrak{g}}{24} - f(-1,0) = \frac{h_j^\vee}{k_j}=C, \quad \text{for any $1\leq j \leq s$}.     
\end{equation}
In the symmetric case, where $f(-1,0)=0$, we additionally have $k_j\geq 2$ for all $1\leq j \leq s$. 

\vspace{2mm}

Let $c_{\mathfrak{g}}$ denote the central charge of the affine vertex operator superalgebra $V_\mathfrak{g}$ associated with $\mathfrak{g}$ at the specified levels. Then $c_{\mathfrak{g}}=24$ in the anti-symmetric case, while in the symmetric case
\begin{equation}\label{eq:Intro-central charge}
c_{\mathfrak{g}}=24C/(C+1),    
\end{equation}    
which yields $c_\mathfrak{g}=12$ in the particular case of $C=1$. 
\vspace{2mm}

In addition, the leading Fourier--Jacobi coefficient of $F$ at the $1$-dimensional cusp of type $2U$ agrees with the superdenominator function $\vartheta_{\mathfrak{g}}$ of the affine Lie superalgebra $\hat{\mathfrak{g}}$ that extends $\mathfrak{g}$. 

\vspace{2mm} 

Moreover, the underlying canonical lattice $L$ satisfies the following conditions. 
\begin{enumerate}
\item The rescaled lattice $L(C)$ is integral. 
\item Let $P_j^\vee$ be the coweight lattice of $\mathfrak{g}_j$, identified with the dual of the root lattice, and let $Q_j^\vee$ be the coroot lattice. 
Then we have the bounds $\bQ<L<\bP$, where 
$$
\bQ= \bigoplus_{j=1}^s Q_j^\vee(k_j) \quad \text{and} \quad \bP= \bigoplus_{j=1}^s P_j^\vee(k_j). 
$$
More precisely, when $\mathfrak{g}_j$ is of type $E_8$, $F_4$, $G_2$, or $\mathfrak{osp}_{1|2r}$, we have $P_j^\vee=Q_j^\vee$. Let $J$ denote the set of such indices and set $J':=\{1,\ldots,s\}\backslash J$. Then 
\begin{equation}\label{eq:bounds-direct-sum}
L=K\oplus\bigoplus_{j\in J} P_j^\vee(k_j), \quad \text{where} \quad \bigoplus_{j\in J'}Q_j^\vee(k_j) < K < \bigoplus_{j\in J'}P_j^\vee(k_j).     \end{equation}
\item Let $\rho_j$ denote the Weyl vector of $\mathfrak{g}_j$. The Weyl vector $\rho=(-A, \vec{B}, -C)$ of $F$ is given by
$$
A=\frac{\mathrm{sdim}\,\mathfrak{g}}{24}, \quad \vec{B}=\sum_{j=1}^s \frac{\rho_j}{k_j}, \quad C=\frac{h_j^\vee}{k_j} \quad \text{for any $1\leq j\leq s$}. 
$$
Moreover, $\vec{B}$ and $\rho$ are precisely the normalized Weyl vectors of $\mathfrak{g}$ and $\hat{\mathfrak{g}}$, respectively. Let $N_\mathfrak{g}$ be the smallest positive integer with $N_\mathfrak{g}\cdot \vec{B} \in \bQ$. Then the level of $L$ is $N_\mathfrak{g}$ or $N_\mathfrak{g}/2$. 
\end{enumerate}
\end{theorem} 

Various special cases of Equation \eqref{eq:Intro-Schellekens type} have been discovered in different vertex algebra contexts. 

When $\mathfrak{g}$ has trivial odd parts and $f(-1,0)=1$, this equation was established by Schellekens \cite{Sch93} in 1993 for holomorphic vertex operator algebras of central charge $24$. He found precisely $221$ solutions and excluded $152$ of them; the remaining $69$ solutions correspond to the semi-simple Lie algebra structures of such VOAs (see also \cite{ESS20}). 

When $\mathfrak{g}$ has trivial odd parts and $f(-1,0)=0$ and $C=1$, Harrison, Paquette, Persson, and Volpato \cite{HPPV21} obtained the equation in 2021 for the holomorphic vertex operator superalgebra $F_{24}$ of central charge $12$, composed of 24 chiral fermions \cite{CDR18}. They showed that the equation admits $8$ solutions, which are in bijection with the $\mathcal{N}=1$ superconformal structures of $F_{24}$. 

When $\mathfrak{g}$ has nontrivial odd parts and $f(-1,0)=1$, Equation \eqref{eq:Intro-Schellekens type} was derived by van Ekeren and Rodr\'iguez Morales \cite{ER23} in 2023 for holomorphic $\ZZ$-graded vertex operator superalgebras of central charge $24$ with nontrivial odd parts, producing $1106$ solutions (revised from their original count). However, none of these solutions could be excluded in their paper, and the existence of the corresponding VOSAs remains open. 

It is not known whether Equation \eqref{eq:Intro-Schellekens type} can be derived for vertex algebras in the remaining cases. Surprisingly, our theorem establishes the equation in a unified framework of BKM superalgebras. 

In our earlier work \cite{SWW23}, we eliminated the same extraneous solutions as Schellekens for $\mathfrak{g}$ with trivial odd parts and $f(-1,0)=1$, and we removed $5$ of the $17$ solutions with trivial odd parts and $f(-1,0)=0$. We now complete the picture by treating the case of nontrivial odd parts. In the anti-symmetric case, we rule out all $1106$ solutions, thereby suggesting the non-existence of the VOSAs in question. In the symmetric case, the equation has $28$ solutions, of which 
$24$ are eliminated. The elimination method is described in the next subsection. Altogether, we eliminate $1287$ of the $1372$ solutions of Equation \eqref{eq:Intro-Schellekens type}, leaving $85$ solutions. We summarize these results in the following theorem. For notational simplicity, we use $A_{1,k}^*$ for the root system of $\mathfrak{osp}_{1|2r}$ at level $k$ when $r=1$, and $C_{r,k}^*$ when $r\geq 2$.

\begin{theorem}\label{MTH:classification}
There are $85$ possibilities for the semi-simple Lie superalgebra $\mathfrak{g}$ in Theorem \ref{MTH:Lie structure}. They fall into four distinct categories: 
\begin{enumerate}
\item the $69$ cases appearing in Schellekens' list of semi-simple $V_1$ structures of holomorphic vertex operator algebras of central charge $24$;
\item the $8$ cases corresponding to the $\mathcal{N}=1$ superconformal structures of the holomorphic vertex operator superalgebra $F_{24}$ of central charge $12$, composed of $24$ fermions;
\item the $4$ cases $A_{1,16}$, $A_{1,8}^2$, $A_{1,4}^4$, and $A_{2,9}$, which possess exceptional modular invariants derived from nontrivial fusion algebra automorphisms;
\item the $4$ cases $A_{1,36}^*$, $C_{2,10}^*$, $A_{1,9}^*\oplus A_{1,12}$, and $A_{1,3}^*\oplus A_{1,4}\oplus A_{2,6}$, which have nontrivial odd parts and admit exceptional modular invariants for the corresponding affine Lie superalgebras. 
\end{enumerate}

For each such $\mathfrak{g}$, there exists a singular automorphic product $\Psi_\mathfrak{g}$ on a canonical lattice $2U\oplus L_\mathfrak{g}$, realizing $\mathfrak{g}$ as its Lie superalgebra structure in the sense of Theorem \ref{MTH:Lie structure}. Additionally, the Jacobi form input $\phi_\mathfrak{g}$ can be written as a $\ZZ$-linear combination of the supercharacters of the affine vertex operator superalgebra $V_\mathfrak{g}$ generated by the affine Lie superalgebra $\hat{\mathfrak{g}}$. 
\end{theorem}

Considering the ranks of these $\mathfrak{g}$ yields the following non-existence result. 

\begin{corollary}\label{Cor:non-existence}
Let $M$ be an even lattice of signature $(l,2)$ with $l\geq 3$. If $M$ has a singular automorphic product and $M\cong 2U\oplus L$, then $l\in \{3,4,6,8,10,12,14,18,26\}$.     
\end{corollary}
Note that the associated values of $l/2$, except for $6$, are exactly the weights of known lacunary holomorphic eta quotients. We conjecture that the assumption $2U\oplus L$ can be completely removed.

\vspace{2mm}

The above groupings are motivated by the central charge of $V_\mathfrak{g}$ (cf. \eqref{eq:Intro-central charge}). The constructions of $\Psi_\mathfrak{g}$ and their connection with vertex algebras are summarized below. The first three classes were considered with in our earlier work \cite{SWW23}; the remaining class will be considered here.

In the first class, $L_\mathfrak{g}$ is the orbit lattice from H\"{o}hn's construction \cite{Hoh17} of the holomorphic VOA of central charge $24$ with $V_1=\mathfrak{g}$. According to Lam \cite{Lam19}, this lattice  encodes the group-like fusion of a certain Leech lattice orbifold VOA $V_{\Lambda_g}^{\hat{g}}$. The singular product $\Psi_{\mathfrak{g}}$ is the Borcherds theta lift of the full character of the holomorphic VOA. Although the reflectivity of 
$\Psi_{\mathfrak{g}}$ and the existence of $L_\mathfrak{g}$ were originally established independently in \cite{SWW23} and \cite{DSW23} by rather difficult arguments, they now follow as consequences of Theorem \ref{MTH:reflectivity}. 

In the second class, the upper bound $\bP$ is integral and its maximal even sublattice $\bP^{\mathrm{ev}}$ gives $L_\mathfrak{g}$. The Jacobi form input of $\Psi_\mathfrak{g}$ is expressed in terms of the characters of the holomorphic VOSA $F_{24}$, with respect to the sectors and the $\mathcal{N}=1$ structure, as
\begin{equation}\label{eq:Intro-F24}
\phi_{\mathfrak{g}}=(\chi_{\mathrm{NS}} - \chi_{\widetilde{\mathrm{NS}}} - \chi_{\mathrm{R}})/2.    
\end{equation}
By \cite[Theorem 1.3]{WW21}, the additive lift of the denominator $\vartheta_\mathfrak{g}$ of $\hat{\mathfrak{g}}$ is automatically a singular automorphic product on $2U\oplus \bP^{\mathrm{ev}}$, which gives precisely $\Psi_\mathfrak{g}$. The reflectivity of $\Psi_\mathfrak{g}$ follows directly from Theorem \ref{MTH:reflectivity}, although it was originally proved in \cite{DW21} by explicit computation. 

In the third class, $\bP$ is even, and $L_\mathfrak{g}=\bP$. These cases are tied to exceptional modular invariants of affine Lie algebras. We illustrate the construction of $\Psi_\mathfrak{g}$ with an example. For $\mathfrak{g}=A_{2,9}$, the Jacobi form input is written in terms of affine characters as 
$$
\phi_{A_{2,9}}=\chi^{A_{2,9}}_{11,\frac14}+\chi^{A_{2,9}}_{17,\frac{9}{4}}+\chi^{A_{2,9}}_{71,\frac{9}{4}}-\chi^{A_{2,9}}_{33,\frac{5}{4}}.
$$
Moreover, the difference between the simple current modular invariant and the exceptional modular invariant from \cite{Moore:1988ss} is given by $|\phi_{A_{2,9}}|^2$. 

In the final class, $L_\mathfrak{g}=\bP$. Unlike the previous cases, the principal part of the input contains negative Fourier coefficients. The singular product $\Psi_{\mathfrak{g}}$ for $\mathfrak{g}=A^*_{1,36}$ was constructed by Gritsenko and Nikulin \cite{GN98} in 1998. To our knowledge, the other three singular products in this class were not previously known. Since a theory of modular invariants for affine Lie superalgebras is not available in the literature, this class motivates us to develop one for the type $\mathfrak{osp}_{1|2r}$. For $\mathfrak{g}=A_{1,36}^*$, the Jacobi form input is expressed in terms of the supercharacters of  $\widehat{\mathfrak{osp}}_{1|2}$ at level $36$ as
\begin{align*}
\phi_{A_{1,36}^*} =&\,\frac{\vartheta(\tau,10z)\vartheta(\tau,z)}{\vartheta(\tau,5z)\vartheta(\tau,2z)}= \widetilde{\chi}^{\mathfrak{osp}_{1|2,36}}_2 - \widetilde{\chi}^{\mathfrak{osp}_{1|2,36}}_{12} - \widetilde{\chi}^{\mathfrak{osp}_{1|2,36}}_{17} + \widetilde{\chi}^{\mathfrak{osp}_{1|2,36}}_{27} + \widetilde{\chi}^{\mathfrak{osp}_{1|2,36}}_{32}   \\
=&\,(\zeta^{2}+\zeta^{-2})-(\zeta+\zeta^{-1}) + 1 +O(q),
\end{align*}
which reduces to the $n=2$ case of the Macdonald-type identity established in Theorem \ref{th:Macdonald-osp}. For $\mathfrak{g}=C_{2,10}^*$, the input is expressed in terms of the supercharacters of $\widehat{\mathfrak{osp}}_{1|4}$ at level $10$ as 
$$
\phi_{C_{2,10}^*} = \widetilde{\chi}^{\mathfrak{osp}_{1|4,10}}_{2,0} - \widetilde{\chi}^{\mathfrak{osp}_{1|4,10}}_{4,2} - \widetilde{\chi}^{\mathfrak{osp}_{1|4,10}}_{1,4}+\widetilde{\chi}^{\mathfrak{osp}_{1|4,10}}_{8,1}+\widetilde{\chi}^{\mathfrak{osp}_{1|4,10}}_{0,8}+\widetilde{\chi}^{\mathfrak{osp}_{1|4,10}}_{5,5}.
$$
This expression is used to construct an exceptional modular invariant. The other cases are similar; see Section \ref{sec:singular-exceptional} for more details. 

\vspace{2mm}

We now explain the relationship between the BKM superalgebra of Theorem \ref{MTH:relation to BKM} and the affine Lie superalgebra appearing in Theorem \ref{MTH:Lie structure}. Let $\mathcal{G}_\mathfrak{g}$ denote the BKM superalgebra attached to $\Psi_\mathfrak{g}$. Its superdenominator is the product expansion of $\Psi_\mathfrak{g}$:
$$
e^\rho \prod_{\alpha>0}\left(1-e^{-\alpha}\right)^{f(nm,\ell)}, \quad \alpha=(n,\ell,m)\in U\oplus L_\mathfrak{g}',\quad \rho=(-A,\vec{B}, -C).
$$
The coefficient $f(nm,\ell)$ determines the supermultiplicity of a candidate root $\alpha$ of $\mathcal{G}_\mathfrak{g}$:
$$
s\text{-}\mathrm{mult}(\alpha):=\mathrm{mult}_{\bar{0}}(\alpha)-\mathrm{mult}_{\bar{1}}(\alpha)=f(nm,\ell).
$$
Recall that the root lattice of $\mathcal{G}_\mathfrak{g}$ is $U\oplus L_\mathfrak{g}'$. The set 
$$
\mathcal{R}_+:=\{ (0,\ell,0): \; 0<\ell\in L_\mathfrak{g}', \; f(0,\ell)=\pm 1 \}
$$
is exactly the set of positive roots of $\mathfrak{g}$. The positive roots of the affine Lie superalgebra $\hat{\mathfrak{g}}$ are then
$$
\widehat{\mathcal{R}}_+:=\{ (n,0,0): n\in \ZZ_{>0} \}\cup \mathcal{R}_+ \cup \{ (n,\ell,0):\; n\in\ZZ_{>0},\; 0\neq \ell\in L_\mathfrak{g}', \; f(0,\ell)=\pm 1 \}.
$$
The simple roots of $\hat{\mathfrak{g}}$ are also simple for $\mathcal{G}_\mathfrak{g}$, so $\hat{\mathfrak{g}}$ embeds into $\mathcal{G}_\mathfrak{g}$ as a Lie sub-superalgebra. Indeed,
$$
\mathcal{G}_\mathfrak{g}[0]:=\bigoplus_{\alpha\in \widehat{\mathcal{R}}_+} \mathcal{G}_{\alpha} \oplus \mathcal{H} \oplus \bigoplus_{\alpha\in \widehat{\mathcal{R}}_+} \mathcal{G}_{-\alpha}\cong \hat{\mathfrak{g}}, 
$$
where $\mathcal{G}_\alpha$ denotes the root space and $\mathcal{H}$ is the common Cartan subalgebra of $\hat{\mathfrak{g}}$ and $\mathcal{G}_\mathfrak{g}$. 
For $m\in\ZZ\backslash\{0\}$, let $\mathcal{G}_\mathfrak{g}[m]$ be the direct sum of root spaces for roots of the form $(*,*,m)$. Since BKM superalgebras are graded by their root lattices, we have the following theorem.

\begin{theorem}\label{MTH:Hyperbolization}
For each $\mathfrak{g}$ in Theorem \ref{MTH:Lie structure}, the decomposition $\mathcal{G}_\mathfrak{g}=\bigoplus_{m\in\ZZ} \mathcal{G}_\mathfrak{g}[m]$ endows $\mathcal{G}_\mathfrak{g}$ with the structure of a $\ZZ$-graded module over $\hat{\mathfrak{g}}$. We refer to $\mathcal{G}_\mathfrak{g}$ as a \textit{hyperbolization} of $\hat{\mathfrak{g}}$.    
\end{theorem}

In particular, the $\hat{\mathfrak{g}}$-module  $\mathcal{G}_\mathfrak{g}[-1]:=\bigoplus_{n,\ell}\mathcal{G}_{(n,\ell,-1)}$ has a supercharacter equal to $\phi_\mathfrak{g}$, which hints at a possible vertex algebra structure. 

The above argument is independent of the construction of $\Psi_{\mathfrak{g}}$, which leads to the question of uniqueness of singular automorphic products and hyperbolizations.


\vspace{2mm}

The uniqueness of the pair $(L_\mathfrak{g},\Psi_\mathfrak{g})$ for each $\mathfrak{g}$ in Theorem \ref{MTH:classification} is comparable to the well-known uniqueness of holomorphic VOAs of central charge $24$ with a fixed semi-simple $V_1$ structure. The latter was settled independently in \cite{BLS22} and \cite{SM21} using distinct systematic methods. We now address the former uniqueness problem, which yields the full solution to Problem \ref{problem:Bor95} in the context of $2U\oplus L$.  

\begin{theorem}\label{MTH:Uniqueness}
For each $\mathfrak{g}$ in Theorem \ref{MTH:classification}, the canonical lattice $L_\mathfrak{g}$ is unique up to isomorphism, and the singular automorphic product $\Psi_{\mathfrak{g}}$ is unique up to the action of $\Orth(L_\mathfrak{g})$. The corresponding lattices $M=2U\oplus L_\mathfrak{g}$ fall into $27$ equivalence classes, and for each of them, there is a unique singular automorphic product modulo $\Orth^+(M)$. Hence there are exactly $27$ holomorphic automorphic products of singular weight on lattices of the form $2U\oplus L$ up to conjugation. The relevant data are summarized in Tables \ref{table:class-1-8}-\ref{tab:class-4}. 
\end{theorem}

In the anti-symmetric case, the $69$ Lie algebras $\mathfrak{g}$ give rise to $11$ genera. This correspondence was first recognized by H\"{o}hn \cite{Hoh17} and verified by Lam \cite{Lam19} for holomorphic VOAs of central charge $24$ with semi-simple $V_1$ structures, and later independently confirmed by \cite{SWW23} and \cite{DSW23} for singular automorphic products. In this case, $\mathcal{G}_\mathfrak{g}$ is a BKM algebra and $\Psi_\mathfrak{g}$ is its denominator function. 

In the symmetric case, each of the $16$ Lie (super)algebras $\mathfrak{g}$ corresponds to a unique genus containing a single lattice. Here, $\mathcal{G}_\mathfrak{g}$ is a BKM superalgebra with nontrivial odd parts, and $\Psi_\mathfrak{g}$ is its superdenominator function. For $\mathfrak{g}$ from the second and third classes of Theorem \ref{MTH:classification}, $\mathcal{G}_\mathfrak{g}$ has no odd real roots, but such roots are present in the last class.

\vspace{2mm} 

The $11$ genera in the anti-symmetric case were first identified in \cite{DSW23}. Our contribution is to remove all of their assumptions (regularity, even rank, non-negative principal parts, reflectivity, and simple zeros) and to extend the classification to the symmetric case, where we find exactly $16$ singular automorphic products, $4$ of which have negative Fourier coefficients in the principal parts of their inputs. Furthermore, the hyperbolization framework is a natural way to extend affine Lie superalgebras to hyperbolic ones with modularity. This originated with Feingold and Frenkel \cite{FF83} in 1983, who studied the hyperbolic extension of affine $A_1$ and related it to genus two Siegel modular forms. Our results classify all such nice extensions and thereby generalize their work. With this interpretation, our main theorems assert that exactly $85$ affine Lie superalgebras have a hyperbolization and that exactly $27$ classes of BKM superalgebras arise as hyperbolizations of affine Lie superalgebras (see Section \ref{sec:hyperbolization}). 

\vspace{2mm}

Finally, we will discuss the relationship between hyperbolizations $\mathcal{G}_{\mathfrak{g}}$ of $\hat{\mathfrak{g}}$ and twists of the fake monster algebra. Let $g$ be a conjugacy class of the Conway group $\mathrm{Co}_0$. Borcherds \cite{Bor92} defined a BKM superalgebra $\mathbb{G}_g$ as the $g$-twist of the fake monster algebra, with superdenominator given by the twisted denominator $\Phi_g$, and conjectured that $\Phi_g$ is always a singular automorphic product. This conjecture was proved by Scheithauer \cite{Sch04, Sch06, Sch08, Sch09, Sch15, Sch17} and Carnahan \cite{Car12} in many cases, and the last two named authors \cite{WW25} recently gave a complete proof. More precisely, $\Phi_g$ is a singular automorphic product on $U(N_g)\oplus U\oplus \Lambda^g$, where $N_g$ is the level of $g$ and $\Lambda^g$ is a fixed-point sublattice of the Leech lattice. These twists provide natural constructions for the hyperbolizations of $\hat{\mathfrak{g}}$ from the first two classes of Theorem \ref{MTH:classification} but they do not cover the last two classes. 

The imaginary simple roots of $\mathcal{G}_\mathfrak{g}$ are encoded in the Fourier expansion of $\Psi_\mathfrak{g}$ at the cusp determined by $U$:
\begin{equation}\label{eq:intro-sum}
\Psi_\mathfrak{g}(Z) = \sum_{\sigma \in W_\mathfrak{g}} \varepsilon'(\sigma) \cdot \sigma\big( \eta_g((\rho, Z)) \big),   
\end{equation}
where $W_\mathfrak{g}$ is the Weyl group; $\varepsilon'$ is a sign function, equal to $\det$ if $\mathcal{G}_\mathfrak{g}$ does not have odd real roots; and $\eta_g(\tau)=\prod_{k}\eta(k\tau)^{b_k}$ is the eta quotient attached to a cycle shape $g=\prod_k k^{b_k}$. There are $27$ such cycle shapes, listed in Tables \ref{table:class-1-8}-\ref{tab:class-4}.

Among these $g$, $19$ occur in the first two classes of Theorem \ref{MTH:classification} and correspond to conjugacy classes of $\mathrm{Co}_0$; the associated lattices are regular. $16$ of these have order equal to their level; in these cases, $\rho\in U\oplus L_{\mathfrak{g}}'$, $\Psi_\mathfrak{g}=\Phi_g$ is modular for the full orthogonal group; and $\mathcal{G}_\mathfrak{g}\cong\mathbb{G}_g$. 
The other $3$ cycle shapes appear in the anti-symmetric case, where $\rho\not\in U\oplus L_{\mathfrak{g}}'$ ($\vec{B}\in L_\mathfrak{g}'$ but $C\in \frac{1}{2}+\ZZ$); $\Psi_\mathfrak{g}=\Phi_g$ is not modular for the full orthogonal group; and $\mathcal{G}_\mathfrak{g}\not\cong\mathbb{G}_g$: the only difference is that $-\rho$ is an imaginary simple root of $\mathbb{G}_g$ with zero supermultiplicity, but it is not a root of $\mathcal{G}_\mathfrak{g}$. 

For each $\mathfrak{g}$ in the second class, we can construct a BKM superalgebra $\mathcal{G}_\mathfrak{g}$ whose imaginary simple roots have suitable ordinary multiplicities so that its denominator function is a singular automorphic product $\widetilde{\Psi}_\mathfrak{g}$ on $U(2)\oplus U\oplus L_\mathfrak{g}$ with Jacobi form pair input $(\tilde{\phi}_\mathfrak{g},\phi_\mathfrak{g})$ related to characters of $F_{24}$ (see Theorem \ref{MTH:relation to BKM} and \eqref{eq:Intro-F24}). The associated cycle shapes $\tilde{g}$ are listed in Table \ref{tab:class-2}. Note that $\widetilde{\Psi}_\mathfrak{g}$ can be non-reflective on $U(2)\oplus U\oplus L_\mathfrak{g}$. Here, $\widetilde{\phi}_\mathfrak{g}$ is a Jacobi form of weight $0$ on $\Gamma_0(2)$: 
$$
\widetilde{\phi}_\mathfrak{g} = (\chi_{\mathrm{NS}} - \chi_{\widetilde{\mathrm{NS}}} + \chi_{\mathrm{R}})/2. 
$$

The remaining $8$ cycle shapes $g$, which occur in the last two classes, need not arise from conjugacy classes of $\mathrm{Co}_0$. Here, $C\not\in\ZZ$; $\vec{B}\in L_\mathfrak{g}'$ iff $\mathfrak{g}=A_{2,9}$; $2U\oplus L_\mathfrak{g}$ is regular iff $\mathfrak{g}=A_{2,9}$; $\Psi_\mathfrak{g}$ is modular for $\Orth^+(2U\oplus L_\mathfrak{g})$. Unlike the second class, the modularity of the denominator of $\mathcal{G}_\mathfrak{g}$ is unclear. 

\vspace{2mm}

We also note the following: The monster algebra \cite{Bor92} is a hyperbolization of the trivial Lie algebra, associated with the monster VOA, with $L=0$. Similarly, the fake monster algebra \cite{Bor90} is a hyperbolization of the $24$-dimensional abelian Lie algebra, associated with the Leech lattice VOA, with $L$ the Leech lattice. These two cases are complementary to our classification.

\subsection{Sketch of the proof}
We start with the proof of Theorem \ref{MTH:theta-quotients}. Let $\phi$ be a holomorphic theta quotient of singular weight and index $L$, and set $\hat{\phi}(Z):=\phi(\tau,\mathfrak{z})\cdot e^{2\pi i \omega}$, viewed as a modular form for an infinite-index subgroup of $\Orth^+(2U\oplus L)$. Applying the Laplace operator to $\hat{\phi}$ yields $\sigma(\hat{\phi})=-\hat{\phi}$ for reflections $\sigma$ associated with zeros of $\phi$. The result then follows from the classical classification of root systems by comparing the zero divisors of $\sigma(\hat{\phi})$ and $\hat{\phi}$ for appropriate reflections. 

\vspace{2mm}

To prove Theorem \ref{MTH:reflectivity}, we first enlarge the lattice by adding isotropic vectors arising from the zero divisors of the singular automorphic product. We then combine the result of \cite{WW23} with the Eichler criterion to establish reflectivity. Finally, the existence of the canonical lattice, as well as the uniqueness of the overlattice in the anti-symmetric case, follows from properties of discriminant forms and the Weil representation. 

\vspace{2mm}

We construct the BKM superalgebra $\mathcal{G}$ of Theorem \ref{MTH:relation to BKM} by the Chevalley--Serre method. A priori, the real simple roots are derived from the Weyl group $W$ of the singular automorphic product $F$. By reflectivity and the simple zero property, the (real) roots $\alpha$ of $W$ are determined by the zeros of $F$, and the simple ones are characterized by the equality $(\rho,\alpha)=(\alpha,\alpha)/2$. The imaginary simple roots can be read off of the additive expansion \eqref{eq:intro-sum} of $F$. Here, we also need to show that the coefficients $f(n^2t^2AC,nt\vec{B})$ are integral; this is a consequence of Theorem \ref{th:integrality}. The canonicity of the underlying lattice then ensures the correct root lattice. 

\vspace{2mm}

Corollary \ref{Cor:relation to regular} follows from the properties of the BKM superalgebra in Theorem \ref{MTH:relation to BKM}. The level statement uses the fact that $U\oplus L'$ is generated by $t\rho$ and real simple roots $\alpha$, along with $(\rho,\alpha)=(\alpha,\alpha)/2$. For regularity, we check that $\frac{N}{p}\vec{B}$ is nontrivial and isotropic in $L'/L$ for each prime $p\mid N$. For Part (4), we show that the singular automorphic product is reflective on $2U\oplus L_0$, with $L_0:=\{x\in L:\, (x,\vec{B})\in\ZZ \}$, and then invoke the uniqueness part of Theorem \ref{MTH:reflectivity}. 

\vspace{2mm}

We prove Theorem \ref{MTH:Lie structure} in two steps. First, we observe that the leading Fourier--Jacobi coefficient of a singular automorphic product is a holomorphic theta quotient of singular weight. The main result then follows from Theorem \ref{MTH:theta-quotients}, combined with reflectivity and the theory of Jacobi forms. The level statement depends on the canonicity of $L$ and is derived from Corollary \ref{Cor:relation to regular}. 

\vspace{2mm}

We prove Theorem \ref{MTH:classification} by first solving Equation \eqref{eq:Intro-Schellekens type}. When $\mathfrak{g}$ has trivial odd part, we find $238$ solutions, with $157$ already eliminated in \cite{SWW23}. In the case of nontrivial odd parts, there are $1134$ solutions, and all of thm have $C\not\in\ZZ$ and $\vec{B}\not\in L'$, so Parts (2)-(5) of Corollary \ref{Cor:relation to regular} are inapplicable. To eliminate all $1106$ solutions in the anti-symmetric case, we first reduce the candidate list. The non-existence result of \cite{WW23} yields the bound $\rank(\mathfrak{g})\leq 18$. By \eqref{eq:bounds-direct-sum}, each irreducible component $\mathfrak{osp}_{1|2l}$ at level $k$ induces a decomposition $L=\ZZ^l(2k)\oplus K$, and therefore the pullbacks of the singular automorphic product $F$ give reflective automorphic products on lattices of the form $2U\oplus \ZZ(2k_1)\oplus \ZZ(2k_2)$. Classifying these reflective products reduces the number of solutions to $156$. We then exclude these remaining solutions, along with the $24$ extraneous solutions among the $28$ symmetric solutions, by the following arguments:

For each extraneous $\mathfrak{g}$, Parts (1)-(3) of Theorem \ref{MTH:Lie structure} imply that $L(C)$ is integral, $L$ lies between $\bQ<L<\bP$, and the level of $L$ is either $N_\mathfrak{g}$ or $N_\mathfrak{g}/2$ (with $N_\mathfrak{g}$ calculated by Lemma \ref{lem:values of delta}). Some cases already violate these constraints and can be excluded immediately. Otherwise, we choose an even overlattice $L_1$ of $L$ such that $2U\oplus L_1 \cong 2U\oplus L_2\oplus K$, where $2U\oplus K$ does not have reflective automorphic products of a certain type. Lemmas \ref{lem:overlattice} and \ref{lem:sym-overlattice}, combined with the pullback trick, guarantee the existence of such products on $2U\oplus K$, leading to a contradiction. We illustrate the argument with two examples. 
\begin{enumerate}
\item For the symmetric case $\mathfrak{g}=C_{2,2}^*A_{4,4}$, we have $C=5/4$, $N_\mathfrak{g}=8$, $\bQ=2A_1(2)\oplus A_4(4)$, and $\bP=2A_1(2)\oplus A_4'(4)$. This gives $L=2A_1(2)\oplus A_4(4)$. However, the condition $N_\mathfrak{g}=8$ forces the level of $L$ to be $4$ or $8$, while the actual level of $2A_1(2)\oplus A_4(4)$ is $40$, a contradiction.  
\item For the anti-symmetric case $\mathfrak{g}=A_{2,2}^4 C_{4,3}^*$, we have $C=3/2$, $N_\mathfrak{g}=12$, and $L=4A_1(3)\oplus K$ with $4A_2(2)<K<4A_2'(2)$. The integrality of both $K$ and $K(3/2)$ ensures that $K(1/2)$ is integral. From $4A_2<K(1/2)<4A_2'$, we deduce $K(1/2)<E_8$, which gives 
$$
L<4A_1(3)\oplus E_8(2)<A_1(3)\oplus A_1\oplus A_2(2)\oplus E_8.
$$
We rule this case out by proving that $2U\oplus A_1\oplus A_2(2)\oplus E_8$ does not have reflective automorphic products of any weight. 
\end{enumerate}

We apply the elimination arguments above case by case to exclude $1130$ extraneous solutions with nontrivial odd parts. This leaves only $4$ symmetric solutions.  To complete the proof of Theorem \ref{MTH:classification}, we construct the corresponding singular automorphic products as additive lifts (with characters) of the superdenominators $\vartheta_\mathfrak{g}$ of  $\hat{\mathfrak{g}}$, and identify their Jacobi form inputs with the supercharacters of the associated affine VOSAs. The proof shows that Part (3) of Theorem \ref{MTH:Lie structure} plays a role analogous to the level restriction for regular lattices in \cite{DSW23}. 

\vspace{2mm}

We finally outline the proof of Theorem \ref{MTH:Uniqueness}. In the anti-symmetric case, there are $69$ Lie algebras $\mathfrak{g}$. For the $57$ cases with $C\in\ZZ$, Part (4) of Corollary \ref{Cor:relation to regular} yields $\vec{B}\in L'$. In the remaining cases, where $C\in1/2+\ZZ$, we adapt the above elimination argument to prove that $\vec{B}\in L'$ still holds. By Parts (2)-(3) of Corollary \ref{Cor:relation to regular}, $2U\oplus L$ is regular, and the associated singular automorphic products have non-negative principal parts. Since all such $\mathfrak{g}$ are of even rank, the theorem then follows from the classification result of Driscoll-Spittler, Scheithauer, and Wilhelm \cite{DSW23}. 

Now we consider the symmetric case, which consists of $16$ Lie superalgebras $\mathfrak{g}$. We enumerate all lattices $L$ that meet the constraints of Parts (1)-(3) of Theorem \ref{MTH:Lie structure}. In the $8$ cases with $C=1$ corresponding to $F_{24}$, we check that $\vec{B}\in L'$ and use Part (5) of Corollary \ref{Cor:relation to regular} to filter the candidates. After excluding those without singular automorphic products, we establish the uniqueness of the underlying lattices and the associated singular products. The symmetric case is more delicate because a singular automorphic product on $2U\oplus L$ may still be reflective on sublattices, but any such additional lattices are ruled out by canonicity. The same argument can be adapted to give an alternative proof of the anti-symmetric case that does not rely on \cite[Theorems 3.6 and 5.4]{DSW23}. We illustrate both arguments with two examples.

\begin{enumerate}
\item For the symmetric case $\mathfrak{g}=A_{2,3}^3$, we have $C=1$, $N_\mathfrak{g}=3$, $\bQ=3A_2(3)$, and $\bP=3A_2$. The level of $L$ is $3$, so the bounds $\bQ<L<\bP$ force $L$ to be $E_6'(3)$ or $3A_2$. Although $2U\oplus E_6'(3)$ has singular automorphic products, $E_6'(3)$ is not canonical. It follows that $L=3A_2$.  
\item For the anti-symmetric case $\mathfrak{g}=A_{8,2}F_{4,2}$, we have $C=9/2$, $N_\mathfrak{g}=4$, $\bQ=A_8(2)\oplus D_4(2)$, and $\bP=A_8'(2)\oplus D_4(2)$. The level of $L$ is $4$. The integrality of both $L$ and $L(9/2)$ implies that $L(1/2)$ is integral. Since $L(1/2)$ has level $2$ and lies between
$$
A_8\oplus D_4 < L(1/2) < A_8'\oplus D_4,
$$
we conclude $L(1/2)=E_8\oplus D_4$, hence $L=E_8(2)\oplus D_4(2)$.
\end{enumerate}

\subsection{Outline of the paper}
The remainder of the paper is organized as follows. 
Section \ref{sec:preliminaries} collects the necessary background on orthogonal modular forms, automorphic products, Jacobi forms, BKM superalgebras, and Lie superalgebras of type $\mathfrak{osp}_{1|2r}$. 

In Section \ref{sec:theta quotients}, we introduce theta quotients and prove Theorem \ref{MTH:theta-quotients}. 

Section \ref{sec:reflectivity} contains the proof of Theorem \ref{MTH:reflectivity}.

Theorem \ref{MTH:relation to BKM} and Corollary \ref{Cor:relation to regular} are proved in Section \ref{sec:BKM from singular}. 

Section \ref{sec:Lie-algebra-structure} contains the proof of Theorem \ref{MTH:Lie structure}. 

The construction of singular automorphic products is given in Section \ref{sec:additive-lift}.

In Section \ref{sec:A1+A1}, we classify reflective automorphic products on $2U\oplus A_1(t)\oplus A_1(s)$. 

The first part of Theorem \ref{MTH:classification} is proved in Sections \ref{sec:sym-classification} and \ref{sec:anti-sym-classification}. 

Section \ref{sec:uniqueness} contains a proof of Theorem \ref{MTH:Uniqueness}.

Section \ref{sec:hyperbolization} defines hyperbolizations of affine Lie superalgebras and proves Theorem \ref{MTH:Hyperbolization}.

In Section \ref{sec:denominator-modularity}, we study the modularity of the denominators of BKM superalgebras arising from singular automorphic products of symmetric type. 

Section \ref{sec:modular-invariants} develops the theory of modular invariants for affine Lie superalgebras of type $\mathfrak{osp}_{1|2}$.

In Section \ref{sec:singular-exceptional}, we express the Jacobi form inputs in terms of affine VOSA supercharacters for the four singular products involving negative principal parts, and construct the corresponding exceptional modular invariants. This concludes the proof of the last part of Theorem \ref{MTH:classification}.

In Appendix \ref{sec:appendix}, we show that all Fourier coefficients of the Jacobi form inputs of automorphic products on $2U\oplus L$ are integral. This establishes the integrality of the candidate supermultiplicities $f(n^2t^2AC,nt\vec{B})$ used in the proof of Theorem \ref{MTH:relation to BKM}.

\newpage

\begin{footnotesize}

\begin{table}[h!]
\begin{minipage}[t]{0.45\textwidth}
\[
\renewcommand\arraystretch{1.1}
\begin{array}{c|c|l|c}
\hline
g & \text{genus}  &  {\mathfrak{g}} &  C  \\ \hline
\multirow{24}{*}{$1^{24}$} & \multirow{24}{*}{$\II_{26,2}$}
     & D_{24,1}        & 46            \\ \cline{3-4}
   & & D_{16,1}E_{8,1}      & 30       \\ \cline{3-4}
      &   & E_{8,1}^3           & 30        \\ \cline{3-4}
   &   & A_{24,1}            &25        \\ \cline{3-4}
   &    & D_{12,1}^2           &22       \\ \cline{3-4}
   &    & A_{17,1}E_{7,1}       &18     \\ \cline{3-4}
   &    & D_{10,1}E_{7,1}^2     & 18      \\ \cline{3-4}
   &   & A_{15,1}D_{9,1}    & 16          \\ \cline{3-4}
 &   & D_{8,1}^3                & 14   \\ \cline{3-4}
   &   & A_{12,1}^2         &13          \\ \cline{3-4}
   &   & A_{11,1}D_{7,1}E_{6,1} &12     \\ \cline{3-4}
  &   & E_{6,1}^4               & 12     \\ \cline{3-4}
   &   & A_{9,1}^2D_{6,1}       &10     \\ \cline{3-4}
   &   & A_{8,1}^3                 &9   \\ \cline{3-4}
   &   & A_{7,1}^2D_{5,1}^2          &8 \\ \cline{3-4}
   &   & A_{6,1}^4         &7           \\ \cline{3-4}
   &   & A_{5,1}^4D_{4,1}    &6         \\ \cline{3-4}
   &   & D_{4,1}^6             &6      \\ \cline{3-4}
   &   & A_{4,1}^6         &5          \\ \cline{3-4}
   &   & A_{3,1}^8        &4            \\ \cline{3-4}
   &    & A_{2,1}^{12}    &3             \\ \cline{3-4}
  &    & A_{1,1}^{24}       &2          \\ \cline{3-4}
   &     & \mathbb{C}^{24}    &0                \\ \hline
\multirow{6}{*}{$1^82^8$} & \multirow{6}{*}{$\II_{18, 2}(2_{\II}^{+10})$}
     & E_{8,2}B_{8,1}   &    15       \\ \cline{3-4}
 &  & C_{10,1}B_{6,1}    &     11    \\ \cline{3-4}
  & & C_{8,1}F_{4,1}^2     &     9  \\ \cline{3-4}
&   & E_{7,2}B_{5,1}F_{4,1}   &   9  \\ \cline{3-4}
 &  & D_{9,2}A_{7,1}        &  8 \\ \cline{3-4}
 &  & D_{8,2}B_{4,1}^2      &   7   
\end{array}
\]
\end{minipage}
\begin{minipage}[t]{0.50\textwidth}
\[
\renewcommand\arraystretch{1.1}
\begin{array}{c|c|l|c}
g & \text{genus}  &  {\mathfrak{g}} &  C  \\ \hline
\multirow{11}{*}{$1^82^8$} & \multirow{11}{*}{$\II_{18, 2}(2_{\II}^{+10})$} 
   & C_{6,1}^2B_{4,1}     &  7   \\ \cline{3-4}
 & & E_{6,2}C_{5,1}A_{5,1}   & 6 \\ \cline{3-4}
 &   & A_{9,2}A_{4,1}B_{3,1}      & 5  \\ \cline{3-4}
&  & D_{6,2}C_{4,1}B_{3,1}^2    &  5 \\ \cline{3-4}
&   & C_{4,1}^4                  &  5 \\ \cline{3-4}
&   & A_{7,2}C_{3,1}^2A_{3,1}    &  4 \\ \cline{3-4}
&   & D_{5,2}^2A_{3,1}^2         &  4 \\ \cline{3-4}
&   & A_{5,2}^2B_{2,1}A_{2,1}^2  &  3\\ \cline{3-4}
&    & D_{4,2}^2B_{2,1}^4         &  3 \\ \cline{3-4}
&    & A_{3,2}^4A_{1,1}^4         &  2 \\ \cline{3-4}
&  & A_{1,2}^{16}               & 1 \\ \hline
\multirow{6}{*}{$1^63^6$} & \multirow{6}{*}{$\II_{14,2}(3^{-8})$}
  &   E_{7,3}A_{5,1}             & 6\\ \cline{3-4}
&  & D_{7,3}A_{3,1}G_{2,1}      &   4\\ \cline{3-4}
&   & E_{6,3}G_{2,1}^3           & 4  \\ \cline{3-4}
&   & A_{8,3}A_{2,1}^2           &  3 \\ \cline{3-4}
&   & A_{5,3}D_{4,3}A_{1,1}^3    &  2 \\ \cline{3-4}
&   & A_{2,3}^{6}                &  1 \\ \hline
\multirow{5}{*}{$1^42^24^4$} & \multirow{5}{*}{$\II_{12,2}(2_{2}^{+2}4_{\II}^{+6})$}
  &   C_{7,2}A_{3,1}             & 4 \\ \cline{3-4}
&   & E_{6,4}B_{2,1}A_{2,1}      &  3 \\ \cline{3-4}
&   & A_{7,4}A_{1,1}^3           &  2 \\ \cline{3-4}
&   & D_{5,4}C_{3,2}A_{1,1}^2    & 2  \\ \cline{3-4}
&   & A_{3,4}^3A_{1,2}           & 1 \\ \hline
\multirow{2}{*}{$1^45^4$} & \multirow{2}{*}{$\II_{10,2}(5^{+6})$}
    & D_{6,5}A_{1,1}^2           &  2 \\ \cline{3-4}
&   & A_{4,5}^2                  &  1 \\ \hline 
\multirow{2}{*}{$1^22^23^26^2$} & \multirow{2}{*}{$\II_{10,2}(2_{\II}^{+6}3^{-6})$}
  &   C_{5,3}G_{2,2}A_{1,1}      & 2 \\ \cline{3-4}
&   & A_{5,6}B_{2,3}A_{1,2}      & 1 \\ \hline 
\multirow{1}{*}{$1^37^3$} & \multirow{1}{*}{$\II_{8,2}(7^{-5})$}
  &     A_{6,7}                   &  1 \\ \hline 
\multirow{1}{*}{$1^22^14^18^2$} & \multirow{1}{*}{$\II_{8,2}(2_{5}^{+1}4_{1}^{+1}8_{\II}^{+4})$}
  &    D_{5,8}A_{1,2}             &  1 \\ \hline
\end{array}
\]
\end{minipage}
\medskip
\caption{The first class: the $8$ genera with $C\in\ZZ$: $t=1$, $\vec{B}\in L_\mathfrak{g}'$} 
\label{table:class-1-8}
\end{table}

\begin{table}[h!]
\[
\renewcommand\arraystretch{1.12}
\begin{array}{c|c|c|l|c}
\hline
g & \text{genus} &    L_{\mathfrak{g}} &  \mathfrak{g}  & C  \\ \hline
\multirow{9}{*}{$2^{12}$} & \multirow{9}{*}{$\II_{14,2}(2_{\II}^{-10}4_{\II}^{-2})$} 
  & \multirow{6}{*}{ $D_{12}(2)$} &  B_{12,2}                   &  23/2 \\  \cline{4-5}
& &     & B_{6,2}^2                  &  11/2 \\ \cline{4-5}
& &     & B_{4,2}^3                  &  7/2 \\ \cline{4-5}
& &     & B_{3,2}^4                  &  5/2  \\ \cline{4-5}
& &     & B_{2,2}^6                  & 3/2   \\ \cline{4-5}
& &      & A_{1,4}^{12}               &  1/2  \\ \cline{3-5}
& &   \multirow{3}{*}{ $E_{8}(2)\oplus D_4(2)$} &  A_{8,2}F_{4,2}         & 9/2  \\ \cline{4-5}
& &     & C_{4,2}A_{4,2}^2           &  5/2 \\ \cline{4-5}
& &     & D_{4,4}A_{2,2}^4           &  3/2 \\ \hline
\multirow{2}{*}{$2^36^3$} & \multirow{2}{*}{$\II_{8,2}(2_{\II}^{+4}4_{\II}^{-2}3^{+5})$}
   &   \multirow{2}{*}{ $D_4(6)\oplus A_2(2)$} &  F_{4,6}A_{2,2}             &  3/2 \\ \cline{4-5}
&    &   &  D_{4,12}A_{2,6}            &  1/2 \\ \hline 
\multirow{1}{*}{$2^210^2$} & \multirow{1}{*}{$\II_{6,2}(2_{\II}^{-2}4_{\II}^{-2}5^{+4})$}
    & D_4(10)  & C_{4,10}                   &   1/2 \\ \hline 
\end{array}
\]
\medskip
\caption{The first class: the $3$ genera with $C\in\frac{1}{2}+\ZZ$: $t=2$, $\vec{B}\in L_\mathfrak{g}'$}
\label{table:class-1-3} 
\end{table}

\end{footnotesize}

\newpage 

\begin{footnotesize}

\begin{table}[h!]
\renewcommand\arraystretch{1.8}
\[
\begin{array}{c|c|c|c|c|c|c} 
\hline
\mathfrak{g} & \text{genus} & g & N & C & c_\mathfrak{g} & \tilde{g} \\ \hline
A_{4,5} & \II_{6,2}(5^{+3}) & 1^{-1}5^5 & 5 &  1 & 12 & 1^3 2^{-2} 5^1 10^2 \\ \hline
A_{1,2}B_{3,5} &  \II_{6,2}(2_{\II}^{+2}5^{+3})& 1^{-2}2^35^210^1 & 10 & 1 & 12 & 1^2 2^{-3} 4^2 5^{-2} 10^7 20^{-2} \\ \hline
A_{1,2}C_{3,4} & \II_{6,2}(2_3^{-1}4_1^{+1}8_{\II}^{-2}) & 1^{-2}2^34^18^2 & 8 & 1  & 12 & 1^2 2^{-3} 4^3 8^2 \\ \hline
B_{2,3}G_{2,4} & \II_{6,2}(2_6^{+2}4_{\II}^{-2}3^{-3})& 1^{-2}2^23^24^112^1 & 12 & 1 & 12 & 1^2 2^{-4} 3^{-2} 4^3 6^6 12^{-1} \\ \hline
A_{2,3}^3 & \II_{8,2}(3^{-3}) & 1^{-3}3^9 & 3 & 1 & 12 & 1^5 2^{-4} 3^1 6^4 \\ \hline                      
A_{1,2}^3A_{3,4} & \II_{8,2}(2_6^{+2}4_{\II}^{+2}) & 1^{-4}2^64^4 & 4 & 1 & 12 & 1^4 2^{-6} 4^8 \\ \hline  
A_{1,2}^2 A_{2,3} B_{2,3} & \II_{8,2}(2_{\II}^{+2}3^{-3}) & 1^{-4}2^53^46^1 & 6 & 1 & 12 & 1^4 2^{-7} 3^{-4} 4^4 6^{13} 12^{-4} \\ \hline  
A_{1,2}^8 & \II_{10,2}(2_{\II}^{+2}) & 1^{-8}2^{16} & 2 & 1 & 12 & 1^8 2^{-8} 4^8\\ \hline
\end{array}
\] 
\smallskip
\caption{The second class: $8$ genera, $t=1$, $\vec{B}\in L_\mathfrak{g}'$}
\label{tab:class-2}
\end{table}

\begin{table}[h!]
\renewcommand\arraystretch{1.8}
\[
\begin{array}{c|c|c|c|c|c|c|c} 
\hline
\mathfrak{g} & \text{genus} & L_\mathfrak{g} & g & C & c_\mathfrak{g} & N & t \\ \hline
A_{1,16} & \II_{3,2}(8_1^{+1}) & A_1(4) & 8^{-1}16^2 & \frac{1}{8} & \frac{8}{3} & 16 & 8 \\ \hline
A_{2,9}  & \II_{4,2}(3^{-1} 9^{-1}) & A_2(3) & 3^{-1} 9^3 & \frac{1}{3} & 6 & 9 & 3 \\ \hline
A_{1,8}^2  & \II_{4,2}(4_2^{+2}) & 2A_1(2) & 4^{-2}8^4 & \frac{1}{4} & \frac{24}{5} & 8 & 4 \\ \hline
A_{1,4}^4  & \II_{6,2}(2_4^{+4}) & 4A_1 & 2^{-4}4^8 & \frac{1}{2} & 8 & 4 & 2 \\ \hline
\end{array}
\] 
\smallskip
\caption{The third class: $4$ genera, $\vec{B}\in L_\mathfrak{g}'$ iff $\mathfrak{g}= A_{2,9}$}
\label{tab:class-3}
\end{table}

\begin{table}[h!]
\renewcommand\arraystretch{1.8}
\[
\begin{array}{c|c|c|c|c|c|c|c} 
\hline
\mathfrak{g} & \text{genus} & L_\mathfrak{g} & g & C & c_\mathfrak{g} & N & t \\ \hline
A_{1,36}^* & \II_{3,2}(8_1^{+1} 9^{-1}) & A_1(36) & 24^{-1} 48^1 72^2 144^{-1} & \frac{1}{24} & \frac{24}{25} & 144 & 24 \\ \hline
A_{1,9}^*A_{1,12}  & \II_{4,2}(2_4^{-2} 3^{-1} 9^{-1}) & A_1(9)\oplus A_1(3) & 6^{-2} 12^3 18^2 36^{-1} & \frac{1}{6} & \frac{24}{7} & 36 & 6 \\ \hline
C_{2,10}^*  & \II_{4,2}(4_2^{+2} 5^{+2}) & 2A_1(10) & 4^{-1} 8^1 20^3 40^{-1} & \frac{1}{4} & \frac{24}{5} & 40 & 4 \\ \hline
A_{1,3}^*A_{1,4}A_{2,6}  & \II_{6,2}(2_4^{+4} 3^{-2}) & A_1(3)\oplus A_1\oplus A_2(2) & 2^{-3} 4^3 6^5 12^{-1} & \frac{1}{2} & 8 & 12 & 2 \\ \hline
\end{array}
\] 
\smallskip
\caption{The fourth class: $4$ genera, $\vec{B}\not\in L_\mathfrak{g}'$}
\label{tab:class-4}
\end{table}

\end{footnotesize}

\noindent 
\textbf{Notation}: $\rho=(-A,\vec{B},-C)$ is the Weyl vector; $t$ is the smallest positive integer with $t\cdot\rho\in U\oplus L_\mathfrak{g}'$; $c_\mathfrak{g}$ is the central charge of the affine VOSA associated with $\mathfrak{g}$; $N$ is the level of $L_\mathfrak{g}$; $g$ is the cycle shape corresponding to the superdenominator of the BKM superalgebra $\mathcal{G}_\mathfrak{g}$; $\tilde{g}$ is the cycle shape corresponding to the denominator of $\mathcal{G}_\mathfrak{g}$ in the second class; $A_{1,k}^*$ and $C_{r,k}^*$ denote $\mathfrak{osp}_{1|2r}$ at level $k$ for $r=1$ and $r\geq 2$, respectively. 

\newpage

\section{Preliminaries}\label{sec:preliminaries}
In this section, we review orthogonal modular forms, automorphic products, Jacobi forms, reflective modular forms, BKM superalgebras, and Lie superalgebras of type $\mathfrak{osp}_{1|2r}$. 

\subsection{Modular forms for orthogonal groups}\label{subsec:orthogonal modular forms}
Let $\ZZ$ and $\NN$ denote the sets of integers and non-negative integers, respectively. 
Let $M$ be an even lattice of signature $(l,2)$ with $l\geq 3$. The symmetric domain of type IV attached to $M$ is
$$
\cD(M):=\cA(M) / \mathbb{C}^{\times}=\{[\mathcal{Z}] \in  \PP(M\otimes \CC):  \mathcal{Z} \in \cA(M) \},
$$ 
where $\cA(M)$ is one of the two connected components of $$
\{\mathcal{Z} \in  M\otimes \CC:  (\mathcal{Z}, \mathcal{Z})=0, (\mathcal{Z},\bar{\mathcal{Z}}) < 0\}.
$$
Let $\Orth^+(M)$ denote the subgroup of $\Orth(M\otimes\RR)$ that preserves $M$ and $\cA(M)$.  Let $\Gamma$ be a finite-index subgroup of $\Orth^+(M)$; the most important such $\Gamma$ will be the \textit{discriminant kernel}
$$
\widetilde{\Orth}^+(M)=\{ g \in \Orth^+(M):\; g(v) - v \in M, \; \text{for all $v\in M'$} \},
$$
where $M'$ is the dual lattice of $M$. 

\begin{definition}
Let $k\in \ZZ$ and let $\chi: \Gamma\to \CC^\times$ be a character. A holomorphic function $F: \cA(M)\to \CC$ is called a modular form of weight $k$ and character $\chi$ on $\Gamma$ if it satisfies
\begin{align*}
F(t\mathcal{Z})&=t^{-k}F(\mathcal{Z}), \quad\  \text{for all $t \in \CC^\times$},\\
F(g\mathcal{Z})&=\chi(g)F(\mathcal{Z}), \quad \text{ for all $g\in \Gamma$}.
\end{align*}
\end{definition}
The definition may be extended to half-integral weights by passing to a double cover of $\cA(M)$. If $F$ is nonzero, then either $k=0$ (and $F$ is constant) or $k\geq l/2-1$. The minimal positive weight $l/2-1$ is called the \textit{singular weight}. 

Orthogonal modular forms have Fourier expansions. Let $c$ be a primitive isotropic vector of $M$ and choose $c'\in M'$ such that $(c,c')=1$. Then 
$$
M_{c,c'}:=\{ x \in M :\; (x,c)=(x,c')=0 \}
$$ 
is an even lattice of signature $(l-1,1)$. Around the $0$-dimensional cusp $c$, the symmetric domain $\cD(M)$ can be realized as a tube domain $\HH_{c,c'}$, that is, a connected component of 
$$
\{ Z = X + iY :\; X, Y \in M_{c,c'}\otimes\RR, \; (Y,Y)<0  \},
$$
which embeds into $\cA(M)$ via the map
$$
\phi_{c,c'}:\; \HH_{c,c'} \to \cA(M), \quad Z \mapsto c'+Z-\frac{(Z,Z)+(c',c')}{2}c.
$$
This induces an action of $\Orth^+(M)$ on $\HH_{c,c'}$ and determines an automorphy factor, which also allows us to define modular forms of half-integral weights. A modular form of trivial character on $\widetilde{\SO}^+(M)$ can be expanded on $\HH_{c,c'}$ into the Fourier series
$$
F(Z)=\sum_{\lambda \in M_{c,c'}', \, (\lambda, \lambda)\leq 0} c(\lambda) e^{2\pi i(\lambda, Z)}.
$$
Modular forms $F$ of general level $\Gamma$ have similar expansions with $M_{c,c'}$ replaced by a finite-index sublattice. When $F$ has singular weight, its Fourier coefficients $c(\lambda)$ are zero whenever $(\lambda,\lambda)\neq 0$.

Orthogonal modular forms of singular weight can also be characterized by the Laplace operator. Let $e_1,...,e_{l}$ be any $\RR$-basis of the lattice $M_{c,c'}$ and define the Gram matrix $\mathcal{S} = (s_{ij})_{i, j=1}^{l}$ with the inverse $\mathcal{S}^{-1} = (s^{ij})_{i, j}$, where $s_{ij} := (e_i, e_j)$.
The \textit{holomorphic Laplace operator} is defined as
$$
\mathbf{\Delta} = \mathbf{\Delta}_{c, c'} := \frac{1}{2} \sum_{i,j=1}^{l} s^{ij} \frac{\partial^2}{\partial e_i \partial e_j},
$$
which acts on functions defined on $\mathbb{H}_{c, c'}$. This operator is independent of the basis $e_i$.

For any $\lambda \in M_{c,c'}\otimes\QQ$, we have 
$$
\mathbf{\Delta} e^{2\pi i (\lambda, Z)} = -2\pi^2 (\lambda,\lambda) e^{2\pi i (\lambda, Z)}.
$$ 
Therefore, a non-constant function $F$ that is represented around $c$ by a Fourier series $$F(Z) = \sum_{\lambda\in M_{c,c'}\otimes\QQ} c(\lambda) e^{2\pi i (\lambda, Z)}$$ satisfies $\mathbf{\Delta} F = 0$ if and only if it is \emph{singular at $c$}, that is, $c(\lambda) = 0$ for every $\lambda$ of nonzero norm.

For any meromorphic function $F : \mathbb{H}_{c, c'} \rightarrow \mathbb{C}$, matrix elements $g \in \mathrm{O}^+(M \otimes \mathbb{R})$, and positive integer $m $, we have (cf. \cite[Lemma 2.4]{Wil21})
\begin{equation}\label{eq:singular-laplace}
\big(\mathbf{\Delta}^m F\big) \Big|_{l/2 + m} g = \mathbf{\Delta}^m \Big( F \Big|_{l/2 - m} g \Big).    
\end{equation} 
A modular form $F$ has singular weight if and only if $\mathbf{\Delta} F = 0$ at an arbitrary $0$-dimensional cusp $c$.

\subsection{Borcherds' theta lift and automorphic products}\label{subsec:Borcherds}
An even lattice $M$ of signature $(b^+,b^-)$ induces a discriminant form $D_M:=(M'/M, \mathfrak{q})$ with the quadratic form
$$
\mathfrak{q} : \, M'/M \rightarrow \mathbb{Q}/\mathbb{Z}, \; \mathfrak{q}(x + M) = (x,x)/2 + \mathbb{Z}.
$$
Let $\mathrm{Mp}_2(\mathbb{Z})$ be the metaplectic group consisting of pairs $A = (A, \phi_A)$, where $A = \begin{psmallmatrix} a & b \\ c & d \end{psmallmatrix} \in \mathrm{SL}_2(\mathbb{Z})$ and $\phi_A$ is a holomorphic square root of the function $\tau \mapsto c \tau + d$ on the complex upper half plane $\mathbb{H}$, with standard generators 
$$
T = \big(\begin{psmallmatrix} 1 & 1 \\ 0 & 1 \end{psmallmatrix}, 1\big) \quad  \text{and} \quad S = \big(\begin{psmallmatrix} 0 & -1 \\ 1 & 0 \end{psmallmatrix}, \sqrt{\tau}\big).
$$
The unitary representation of $\mathrm{Mp}_2(\mathbb{Z})$ on the group ring 
$\mathbb{C}[D_M] = \mathrm{span}(e_x: \, x \in D_M)$
defined by 
$$
\rho_M(T) e_x = \mathbf{e}(-\mathfrak{q}(x)) e_x \quad \text{and} \quad \rho_M(S) e_x = \frac{\mathbf{e}( \mathrm{sign}(M) / 8)}{\sqrt{|D_M|}} \sum_{y \in D_M} \mathbf{e}(( x,y)) e_y
$$
is called the \emph{Weil representation},
where $\mathrm{sign}(M)=b^+ - b^- \bmod 8$, and $\mathbf{e}(t)=e^{2\pi it}$ for $t\in \CC$. The dual representation of $\rho_M$ is the complex conjugate of $\rho_M$. 

Let $k \in \frac{1}{2}\mathbb{Z}$. A holomorphic function $f : \mathbb{H} \rightarrow \mathbb{C}[D_M]$ is called a \emph{weakly holomorphic modular form} of weight $k$ for the Weil representation $\rho_M$ if $f$ satisfies
$$
\phi_A(\tau)^{-2k} f(A \cdot \tau) = \rho_M(A) f(\tau), \quad \text{for all $A \in \mathrm{Mp}_2(\mathbb{Z})$,}
$$ 
and if $f$ is expanded into a Fourier series of the form 
\begin{equation}\label{eq:principal part}
f(\tau) = \sum_{x \in D_M} \sum_{\substack{n \in \mathbb{Z} - \mathfrak{q}(x)\\ n \gg -\infty}} c_x(n) q^n e_x,
\end{equation}
where $n\gg -\infty$ means that $n$ is bounded from below. This definition can be extended to any complex finite-dimensional representation of $\mathrm{Mp}_2(\ZZ)$. 

The finite sum $c_x(n)q^n e_x$ with $n<0$ is called the \textit{principal part} of $f$. 
If $f$ is holomorphic at infinity, that is, its principal part is zero, then it is called a \textit{holomorphic modular form}. A nonzero form $f$ exists only if $k+\mathrm{sign}(M)/2 \in \ZZ$. If $\mathrm{sign}(M)$ is even, then $\rho_M$ factors through a representation of $\SL_2(\ZZ)$. Denote by $M_k^!(\rho_M)$ and $M_k(\rho_M)$ the spaces of weakly holomorphic and holomorphic modular forms of weight $k$ for $\rho_M$, respectively. 

There is a natural homomorphism $\Orth(M) \to \Orth(D_M)$ with kernel $\widetilde{\Orth}(M)$. The group $\Orth(D_M)$ acts on any modular form $f = \sum c_x(n) q^n e_x$ by
$$
\sigma(f):=\sum c_x(n)q^n e_{\sigma(x)}, \quad \sigma \in \Orth(D_M).
$$

Assume from now on that $M$ has signature $(l,2)$ with $l\geq 3$.
Let $f\in M^!_{1-l/2}(\rho_M)$ with integral principal part \eqref{eq:principal part}. The Borcherds singular theta lift \cite{Bor95, Bor98} yields a meromorphic modular form $\Borch(f)$ of weight $c_0(0)/2$ and a character (or multiplier system) on
$$
\Orth^+(M,f):=\{ \sigma \in \Orth^+(M): \sigma(f)=f\} \supset \widetilde{\Orth}^+(M)
$$
which has the following properties. 
\begin{enumerate}
\item All zeros and poles of $\Borch(f)$ lie on rational quadratic divisors 
$$
\lambda^\perp:=\{ [\mathcal{Z}]\in \cD(M): \; (\lambda,\mathcal{Z})=0 \},
$$
for primitive positive-norm $\lambda \in M'$. The multiplicity of $\lambda^\perp$ in the divisor is  
$$
\sum_{d=1}^\infty c_{d\lambda}(-d^2\lambda^2/2),
$$
depending only on the principal-part Fourier coefficients of $f$. 
\item At every $0$-dimensional cusp $c$, $\Borch(f)$ has an infinite product expansion on the associated tube domain $\HH_{c,c'}$, with exponents given by the Fourier coefficients of $f$. 
\end{enumerate}

Such modular forms are called \textit{automorphic products} or \textit{Borcherds products}. Throughout this paper, automorphic products are identified up to nonzero constant scalars, and we will simply write “singular automorphic products” to mean holomorphic automorphic products of singular weight.

\subsection{Jacobi forms of lattice index}\label{subsec:Jacobi}
Let $L$ be an integral positive-definite lattice of rank $\rank(L)$, with bilinear form $(-,-)$ and dual lattice 
$$
L'=\{ x\in L\otimes \QQ : \; (x,y)\in \ZZ, \; \text{for all $y\in L$}  \}.
$$
The shadow of $L$ is defined by
$$
L^\bullet = \{ x \in L\otimes\QQ : (x,y) - (y,y)/2 \in \ZZ \quad \text{for all $y\in L$} \},
$$
which equals $L'$ if $L$ is even. 
For any $a\in\ZZ\backslash\{0\}$, write $L(a)$ for the lattice $L$ with bilinear form scaled by $a$. Let $v_\eta$ be the multiplier system of the Dedekind eta function
$$
\eta(\tau) = q^{\frac{1}{24}}\prod_{n=1}^\infty(1-q^n), \quad \tau \in \HH, \; q=e^{2\pi i \tau} 
$$
viewed as a modular form of weight $1/2$ on $\SL_2(\ZZ)$. The multiplier $v_\eta$ has order $24$ and it generates all characters of $\mathrm{Mp}_2(\ZZ)$. Following \cite{Gri88, GSZ19}, we define Jacobi forms of lattice index, which generalize the classical Jacobi forms due to Eichler and Zagier \cite{EZ85}. 

\begin{definition}
Let $k\in\frac{1}{2}\ZZ$ and $D\in\ZZ/24\ZZ$. A holomorphic function $\varphi(\tau,\mathfrak{z}) : \HH \times (L\otimes \CC) \rightarrow \CC$ is called a \textit{holomorphic Jacobi form} of weight $k$, character $v_\eta^D$, and index $L$, if it satisfies 
\begin{align*}
\varphi \left( \frac{a\tau +b}{c\tau + d},\frac{\mathfrak{z}}{c\tau + d} 
\right)& = v_\eta(A)^D(c\tau + d)^k 
\exp{\left(\pi i\frac{c(\mathfrak{z},\mathfrak{z})}{c\tau + d}\right)} \varphi ( \tau, \mathfrak{z} ),  \\
\varphi (\tau, \mathfrak{z}+ x \tau + y)&=(-1)^{(x,x)+(y,y)}\exp\left( -\pi i\big( (x,x)\tau + 2(x,\mathfrak{z}) \big) \right) 
\varphi (\tau, \mathfrak{z} ),       
\end{align*}
for all $A=\begin{psmallmatrix} a & b \\ c & d \end{psmallmatrix} \in \SL_2(\ZZ)$ 
and $x,y \in L$, and if its Fourier expansion has the form
\begin{equation}\label{eq:Fourier}
\varphi ( \tau, \mathfrak{z} )= \sum_{\substack{n\in \frac{D}{24}+\ZZ, \; \ell \in L^\bullet \\ 2n\geq (\ell,\ell) }}f(n,
\ell)q^n\zeta^\ell, \quad \zeta^\ell=e^{2\pi i (\ell,\mathfrak{z})}.
\end{equation}  
If the condition $2n\geq (\ell,\ell)$ above is relaxed to $n\geq 0$ or $n\gg -\infty$, then $\varphi$ is called a {\it weak} or {\it weakly holomorphic} Jacobi form, respectively. Jacobi forms on congruence subgroups are defined similarly, with the conditions on the Fourier expansion required to hold at every cusp.
\end{definition}

Denote the spaces of weakly holomorphic, weak, and holomorphic Jacobi forms of weight $k$, character $v_\eta^D$, and index $L$ by 
$$
J^{!}_{k,L}(v_\eta^D) \supset J^{\w}_{k,L}(v_\eta^D) \supset J_{k,L}(v_\eta^D).
$$
In the case of trivial character, we write $J^{!}_{k,L}\supset J^{\w}_{k,L} \supset  J_{k,L}$. Classical Jacobi forms as in \cite{EZ85} are recovered as the special case of the rank-one lattice $\ZZ(a)$ with bilinear form $ax^2$ ($x\in\ZZ$). 

The odd Jacobi theta function is the prototypical example of a Jacobi form: 
\begin{equation}\label{eq:Jacobi-theta}
\vartheta(\tau,z):=\sum_{n\in\ZZ}\left(\frac{-4}{n}\right)q^{\frac{n^2}{8}}\zeta^{\frac{n}{2}} =-q^{\frac{1}{8}}\zeta^{-\frac{1}{2}}\prod_{n=1}^\infty(1-q^{n-1}\zeta)(1-q^n\zeta^{-1})(1-q^n).   
\end{equation}
It is a holomorphic Jacobi form of weight $\frac{1}{2}$, character $v_\eta^3$, and index $\ZZ$ (cf. \cite{GN98, GSZ19}), with simple zeros precisely in the lattice points $z\in\ZZ\tau+\ZZ$. 

\vspace{2mm}

Jacobi forms occur in the Fourier–Jacobi expansions of orthogonal modular forms at $1$-dimensional cusps. Let $U=\ZZ e + \ZZ f$ be a hyperbolic plane, with $(e,e)=(f,f)=0$ and $(e,f)=-1$. Let $U_1=\ZZ e_1 +\ZZ f_1$ be a second hyperbolic plane and let $L$ be an even positive-definite lattice. Define $M=U_1\oplus U\oplus L$. The tube domain at the $0$-dimensional cusp $e_1$ is 
\begin{equation}\label{eq:tube domain}
\HH(L):=\{ Z= \tau f + \mathfrak{z} +\omega e : \tau, \omega \in \HH,\; \mathfrak{z}\in L\otimes\CC, \; 2\im(\tau)\im(\omega) >(\im(\mathfrak{z}),\im(\mathfrak{z})) \}.    
\end{equation}
For $\alpha=ne+\ell+mf\in U\oplus L'$, we have
$$
-(\alpha, Z)=n\tau - (\ell, \mathfrak{z}) + m\omega.
$$
This convention ensures that the product expansions of automorphic products match the superdenominators of BKM superalgebras. 

\begin{Notation}\label{notation2}
Write $v=x_1e_1+xe+\ell+yf+y_1f_1\in U_1\oplus U\oplus L'$ with coordinates $(x_1,x,\ell,y,y_1)$ and $\alpha=xe+\ell+yf\in U\oplus L'$ with coordinates $(x,\ell,y)$. Then $v^2=(\ell,\ell)-2(xy+x_1y_1)$. 
\end{Notation}

When $F$ is a modular form of weight $k$ and trivial character on $\widetilde{\SO}^+(M)$, its Fourier–Jacobi expansion on $\HH(L)$ takes the form
$$
F(Z)=\sum_{\substack{m,n\in\NN,\, \ell \in L' \\ 2nm\geq (\ell,\ell)}} f(n,\ell,m)q^n\zeta^{-\ell} \xi^m =: \sum_{m=0}^\infty \phi_m(\tau,\mathfrak{z})\xi^m, \quad \zeta^{\ell}=e^{2\pi i (\ell,\mathfrak{z})},\; \xi=e^{2\pi i\omega}. 
$$ 
The Jacobi group is the semi-direct product of $\SL_2(\mathbb{Z})$ with the integral Heisenberg group of $L$. It embeds as an infinite-index subgroup of $\widetilde{\SO}^+(M)$ preserving the isotropic plane spanned by $e$ and $e_1$. Therefore, each $\phi_m$ is a holomorphic Jacobi form of weight $k$, trivial character, and index $L(m)$, while $\phi_0$ is independent of $\mathfrak{z}$ and is a holomorphic modular form of weight $k$ on $\SL_2(\ZZ)$. 

There is a well-known equivalence between modular forms for the Weil representation and Jacobi forms. For an even positive-definite lattice $L$ of rank $\rank(L)$ and a class $\gamma\in L'/L$, define
$$
\Theta_{L,\gamma}(\tau,\mathfrak{z}) = \sum_{\ell \in L+\gamma} e^{\pi i(\ell,\ell)\tau + 2\pi i(\ell, \mathfrak{z})}. 
$$
The tuple $(\Theta_{L,\gamma})_{\gamma\in L'/L}$ transforms as a vector-valued Jacobi form of weight $\frac{1}{2}\rank(L)$ for the dual representation $\overline{\rho}_L=\rho_{L(-1)}$. This yields 
\begin{equation}
M_{k-\frac{1}{2}\rank(L)}^!(\rho_L) \stackrel{\sim}{\longrightarrow} J_{k,L}^!, \quad 
\sum_{\gamma} f_\gamma e_\gamma \longmapsto \sum_{\gamma} f_\gamma \cdot \Theta_{L,\gamma}. 
\end{equation}
The isomorphism extends to forms with characters and it preserves holomorphy. In particular, holomorphic Jacobi forms of weight $\frac{1}{2}\rank(L)$ and index $L$ are $\CC$-linear combinations of these $\Theta_{L,\gamma}$. This weight is the \textit{singular weight} for Jacobi forms. 

Using this isomorphism, automorphic products can be viewed as lifts of Jacobi forms. Let $\phi \in J_{0,L}^!$ with Fourier expansion
$$
\phi(\tau,\mathfrak{z})=\sum_{n\in \ZZ}\sum_{\ell\in L'}f(n,\ell)q^n\zeta^\ell.
$$
The terms with $2n-(\ell,\ell)<0$ are the \textit{singular Fourier coefficients}, corresponding to the principal part of the preimage of $\phi$. 

\begin{remark}\label{rem:singular}
For $\phi \in J_{0,L}^!$, $f(n,\ell)=f(n,-\ell)$, and quasi-periodicity implies $f(n_1,\ell_1)=f(n_2,\ell_2)$ whenever $2n_1-(\ell_1,\ell_1)=2n_2-(\ell_2,\ell_2)$ and $\ell_1-\ell_2\in L$. Therefore, every singular coefficient appears as $f(n,\ell)q^n\zeta^\ell$ with $n\leq \hat{\delta}_L$ and $2n<(\ell,\ell)$, where $\hat{\delta}_L$ is the largest integer less than $\delta_L/2$, and
\begin{equation}
\delta_L := \max\big\{ \min\{(y,y): y\in L + x \} : x \in L' \big\}.
\end{equation}
\end{remark}

The following theorem gives the Jacobi form expression of Borcherds' singular theta lift.  

\begin{theorem}[{\cite[Theorem 4.2]{Gri18}}]\label{th:product}
Assume $L$ is even and $\phi\in J_{0,L}^!$ has integral singular Fourier coefficients. Then the Borcherds lift of $\phi$ yields a meromorphic modular form of weight $f(0,0)/2$ and some character on $\widetilde{\Orth}^+(M)$, with infinite product expansion on an open subset of $\HH(L)$:
$$
\Borch(\phi)(Z)=q^A \zeta^{\vec{B}} \xi^C\prod_{\substack{n,m\in\ZZ,\; \ell
\in L'\\ (n,\ell,m)>0}}\Big(1-q^n \zeta^{-\ell} \xi^m\Big)^{f(nm,\ell)}, 
$$ 
where $(n,\ell,m)>0$ means that $m>0$, or $m=0$ 
and $n>0$, or $m=n=0$ and $\ell>0$. The Weyl vector of $\Borch(\phi)$ is $\rho= (-A,\vec{B}, -C)$ given by
$$
A:=\frac{1}{24}\sum_{\ell\in L'}f(0,\ell),\quad
\vec{B}:=\frac{1}{2}\sum_{\ell>0} f(0,\ell)\ell, \quad C:=\frac{1}{\rank(L)}\sum_{\ell>0}f(0,\ell)(\ell,\ell).
$$
The Fourier--Jacobi expansion is 
\begin{equation}\label{eq:FJ-Borcherds}
\Borch(\phi)(Z)=
\Big(\Theta_{f(0,\ast)}
(\tau,\mathfrak{z})\cdot \xi^C \Big) \cdot
\exp \left(-\sum_{m=1}^\infty \Big(\phi |_0 T_{-}^{(1)}(m)\Big) (\tau, \mathfrak{z}) \cdot \xi^m \right),
\end{equation}
with leading coefficient 
\begin{equation*}
\Theta_{f(0,\ast)}(\tau,\mathfrak{z})
=\eta(\tau)^{f(0,0)}\prod_{\ell >0}
\biggl(\frac{\vartheta(\tau,(\ell,\mathfrak{z}))}{\eta(\tau)} 
\biggr)^{f(0,\ell)},
\end{equation*}
a meromorphic Jacobi form of weight $f(0,0)/2$, character $v_\eta^{24A}$, and index $L(C)$. Moreover, 
$$
\Borch(\phi)(\omega,\mathfrak{z},\tau)=(-1)^{\mathfrak{D}}\cdot\Borch(\phi)(\tau,\mathfrak{z},\omega), \quad \text{where} \quad \mathfrak{D}:=\sum_{n<0} \sigma_0(-n)\cdot f(n,0),
$$
where $\sigma_0(-n)$ denotes the number of positive divisors of $-n$. The product $\Borch(\phi)$ is called symmetric if  $\mathfrak{D}$ is even, and anti-symmetric if $\mathfrak{D}$ is odd. 
\end{theorem} 

In the next lemma, we record some properties of the components $A$ and $C$ of the Weyl vector. We refer to \cite[Theorems 10.5 and 11.2]{Bor98}, \cite[Proposition 2.6]{Gri18}, and \cite[Corollary 4.5]{Wan24} for proofs.   

\begin{lemma}\label{Lem:q^0-term}
Let $L$ be an even positive-definite lattice of rank $\rank(L)$ and $\phi\in J_{0,L}^!$. Then we have
\begin{equation}\label{eq:q^0-term}
C:=\frac{1}{\rank(L)} \sum_{\ell>0}f(0,\ell)(\ell,\ell) = \frac{1}{24}\sum_{\ell\in L'}f(0,\ell)-\sum_{n<0}\sum_{\ell\in L'}f(n,\ell)\sigma_1(-n),
\end{equation}
where $\sigma_1(-n)$ denotes the sum of positive divisors of $-n$. Also, 
\begin{equation}\label{eq:vectorsystem}
\sum_{\ell\in L'}f(0,\ell)(\ell,\mathfrak{z})^2=2C(\mathfrak{z},\mathfrak{z}),\quad \text{for any} \; \mathfrak{z}\in L\otimes \CC.
\end{equation}
If $f(0,\ell)\in\ZZ$ for all $\ell\in L'$, then $C(x,y)\in\ZZ$ for all $x,y\in L$. 
\end{lemma}

The operator $T_{-}^{(1)}(m)$ in \eqref{eq:FJ-Borcherds} is a special case of the following.

\begin{proposition}[Proposition 3.1 in \cite{CG13}]\label{prop:Hecke}
Let $k\in\ZZ$, $D\mid24$ even, and let $L$ be integral and positive-definite. Let $\varphi \in J_{k,L}^!(v_{\eta}^D)$, and set $Q=24/D$. If $Q$ is odd, assume that $L$ is even. Then for $m$ coprime to $Q$, 
$$
\Big(\varphi \lvert_{k}T_{-}^{(Q)}(m)\Big)(\tau,\mathfrak{z}):=m^{-1}\sum_{\substack{ad=m,\,a>0\\ 0\leq b <d}}a^k\cdot v_{\eta}(\sigma_a)^D\cdot\varphi\left(\frac{a\tau+bQ}{d},a\mathfrak{z}\right)\in J_{k,\, L(m)}^!(v_{\eta}^{Dx}),
$$
where $x,y\in \ZZ$ satisfy $mx+Qy=1$, 
and
$$\sigma_a:=\left(\begin{array}{cc}
dx+Qdxy & -Qy \\ 
Qy & a
\end{array}  \right) \in \SL_2(\ZZ).$$ 
The Fourier coefficients $f_m(n,\ell)$ of the image and $f(n',\ell')$ of the preimage are related by 
$$
f_m(n,\ell)=\sum_{a \mid (n,\ell,m)}a^{k-1} \upsilon_{\eta}^D(\sigma_a)f\left( \frac{nm}{a^2},\frac{\ell}{a}\right),
$$
where $a\! \mid\! (n,\ell,m)$ means that $a\! \mid\! nQ$, $a^{-1}\ell\in L^\bullet$, and $a\! \mid\! m$.
\end{proposition}

We also need additive lifts with characters to construct orthogonal modular forms.

\begin{theorem}[Theorem 3.2 in \cite{CG13}]\label{th:additive}
Let $G_k$ be the normalized weight-$k$ Eisenstein series on $\SL_2(\mathbb{Z})$ with Fourier coefficient $1$ at $q^1$. Under the assumptions of Proposition \ref{prop:Hecke}, if $\varphi\in J_{k,L}(v_\eta^D)$, then
$$ 
\Grit(\varphi)(Z)=f(0,0)\cdot G_k(\tau)+\sum_{m\in 1+Q\NN}\Big(\varphi \lvert_{k}T_{-}^{(Q)}(m)\Big)(\tau,\mathfrak{z})\cdot \xi^{m/Q}
$$
is a holomorphic modular form of weight $k$ and some character on $\widetilde{\Orth}^+(U_1\oplus U\oplus L(Q))$. It is always invariant under the involution $(\omega,\mathfrak{z},\tau)\mapsto (\tau,\mathfrak{z},\omega)$. 
\end{theorem}

\subsection{Reflective modular forms}\label{subsec:reflective}
Let $M$ be an even lattice of signature $(l,2)$ with $l\geq 3$. For $\lambda \in M\otimes\QQ$ with $(\lambda,\lambda)>0$, the reflection fixing $\lambda^\perp$ is the element of $\Orth^+(M\otimes\QQ)$ defined by 
$$
\sigma_\lambda(x)=x-\frac{2(x,\lambda)}{(\lambda,\lambda)}\lambda, \quad x\in M\otimes\QQ.
$$
The divisor $\lambda^\perp$ is called \textit{reflective} if $\sigma_\lambda$ preserves $M$, i.e., $\sigma_\lambda \in \Orth^+(M)$. A non-constant holomorphic modular form for $\Gamma<\Orth^+(M)$ is called \textit{reflective} if its divisor is a linear combination of reflective divisors. By Bruinier's converse theorem \cite{Bru02,Bru14}, if $M$ is of the form $U_1(n)\oplus U\oplus L$, then every reflective form for $\widetilde{\Orth}^+(M)$ is an automorphic product.  Reflective modular forms were introduced in 1998 by Borcherds \cite{Bor95, Bor98} and Gritsenko–Nikulin \cite{GN98}, and their classification has been an active area of research over the past thirty years (see \cite{GN98, GN02, Bar03, Sch06, Sch17, Ma17, Ma18, Dit19, Wan24, Wan23b, Wan25b}). They are interesting due to their connections with infinite-dimensional Lie algebras \cite{Bor98, GN98a, GN02, Sch06} and their applications to hyperbolic reflection groups \cite{Bor00, GN18}, algebraic geometry \cite{Bor96, BKP98, GHS07, GH14, Ma18, Gri18}, and free algebras of modular forms \cite{Wan21a, WW21}. 

For primitive $\lambda\in M'$, $\lambda^\perp$ is reflective if and only if $(\lambda, \lambda)=2/d$ and $d\lambda \in M$ for some integer $d>0$. The order of $\lambda$ in $M'/M$ is either $d$ or $d/2$ (with $d/2$ even in the latter case). This restricts the principal part of the input to a reflective product:

\begin{lemma}[Lemma 2.1 in \cite{Wan25b}]\label{lem:principal-part-reflective}
The principal part of the input of a reflective automorphic product on $M=U\oplus K$ has the form
\begin{equation}\label{eq:principal-part}
\sum_{t=1}^\infty \sum_{\substack{x\in K'/K\\ \ord(x)=t}} c_{x}(-1/t)q^{-1/t}e_x + \sum_{t=1}^\infty \sum_{\substack{y\in K'/K\\ \ord(y)=2t}} c_y(-1/4t)q^{-1/4t}e_y,    
\end{equation}
where $c_x(-1/t)\geq 0$, $c_{2y}(-1/t) + c_y(-1/4t) \geq 0$, and $\ord(\gamma)$ denotes the order of $\gamma$ in $K'/K$. 
\end{lemma}

Let $M=U_1\oplus U\oplus L$ and $\phi\in J_{0,L}^!$ such that $\Borch(\phi)$ is reflective. Then the Fourier expansion of $\phi$ begins
\begin{equation}\label{eq:reflective-Jacobi}
\phi = f(-1,0)q^{-1} + \sum_{t=1}^\infty \sum_{\substack{r\in L'\\ r^2=2/t \\ \ord(r)=t}} f(0,r) \zeta^r + \sum_{t=1}^\infty \sum_{\substack{s\in L'\\ s^2=1/(2t) \\ \ord(s)=2t}} f(0,s) \zeta^s + f(0,0) + O(q).
\end{equation}
Here, $f(0,0)/2$ is the weight of $\Borch(\phi)$; $f(-1,0)$ and $f(0,r)$ are non-negative integers; and $f(0,s)$ is integral with $f(0,s)+f(0,2s)\geq 0$ (here $2s$ is a vector of type $r$). In addition, $f(0,r)=f(-1,0)$ if $(r,r)=2$. The reflective product $\Borch(\phi)$ is symmetric if $f(-1,0)$ is even and anti-symmetric if $f(-1,0)$ is odd (cf. the last sentence of Theorem \ref{th:product}). 

We now collect some results on reflective modular forms that will be used later. 

\begin{lemma}[Lemma 3.1 in \cite{Wan25b}]\label{lem:sym-bound}
Let $M=U\oplus K$ be an even lattice of signature $(l,2)$ with $l\geq 3$. If $M$ has a reflective automorphic product whose input has zero $q^{-1}e_0$ coefficient, then $l\leq 13$.  
\end{lemma}

\begin{lemma}[Lemma 3.3 in \cite{Wan25b}]\label{lem:overlattice}
If $U\oplus K$ has a reflective automorphic product whose input has nonzero $q^{-1}e_0$ coefficient, then the same holds on $U\oplus K_1$ for any even overlattice $K_1$ of $K$. 
\end{lemma}

\begin{lemma}[Lemma 12.2 in \cite{SWW23}]\label{lem:sym-overlattice}
Let $M=U\oplus K\oplus P$. Fix $\lambda \in U$ with $\lambda^2=2$ and $\mu \in U\oplus K'$ with $\mu^2>0$. If $M$ has a reflective automorphic product vanishing along $\mu^\perp$ but not $\lambda^\perp$, then the same holds for $U\oplus K\oplus P_1$ for any even overlattice $P_1$ of $P$. 
\end{lemma}

The following theorem shows that every singular automorphic product is automatically reflective:

\begin{theorem}[Theorems 1.2 and 1.4 in \cite{WW23}]\label{th:singular-product-reflective}
Let $F$ be a singular automorphic product on any lattice $M$. If $F$ vanishes along $\lambda^\perp$ for some $\lambda \in M\otimes\QQ$,  then its multiplicity is one and 
$$
F\big(\sigma_\lambda(\mathcal{Z})\big) = - F(\mathcal{Z}).
$$
Moreover, there exists an even lattice $\mathbf{M}\subset M\otimes\QQ$ such that $F$ is a reflective modular form for a finite-index subgroup $\Gamma<\Orth^+(\mathbf{M})$.
\end{theorem}

The group $\Gamma$ contains all the reflections $\sigma_\lambda$ for which $F$ vanishes along $\lambda^\perp$. By Margulis' normal subgroup theorem, the subgroup generated by such reflections has finite index in $\Gamma$.  

Reflectivity implies the non-existence of singular automorphic products on large-rank lattices. 

\begin{theorem}[Theorem 1.5 in \cite{WW23}]\label{th:upper-bound-singular}
Let $M$ be an even lattice of signature $(l,2)$ with $l\geq 3$. If $M$ has a singular automorphic product, then $l\leq 20$ or $l=26$. For $l=26$, the product is unique and equals the denominator function of the fake monster algebra.  
\end{theorem}

For a given lattice $M$, one typically needs to determine all reflective automorphic products and, in some cases, prove that no such product exists. To this end, we use Borcherds' obstruction principle \cite{Bor99}, which states that a formal finite sum
$$
\sum_{n < 0} \sum_{x \in D_M} c_x(n) q^n e_x
$$ 
occurs as the principal part of a weakly holomorphic modular form of weight $\kappa$ for the Weil representation $\rho_M$ if and only if 
$$
\sum_{n < 0} \sum_{x \in D_M} c_x(n) a_x(-n) = 0
$$ 
for every cusp form of weight $2-\kappa$ for the dual representation $\bar{\rho}_M=\rho_{M(-1)}$ with Fourier expansion
$$
\sum_{n > 0} \sum_{x \in D_M} a_x(n) q^n e_x.
$$

We construct bases of cusp forms for $\rho_{M(-1)}$ using the algorithm of \cite{Wil18}, which reduces the problem of computing reflective automorphic products to the enumeration of lattice points in a polyhedral cone defined by finitely many inequalities (coming from non-negative order along reflective divisors) and finitely many linear equations (coming from the restrictions imposed by reflective divisors and the obstructing cusp forms). The solution can be conveniently expressed using the notion of a Hilbert basis, i.e., a minimal set of reflective automorphic products $F_1,...,F_r$ such that every reflective product can be written as
$$
F = F_1^{n_1} \cdot ... \cdot F_r^{n_r}
$$ 
with non-negative integral exponents $n_1,...,n_r$. Such a Hilbert basis can be computed using the software Normaliz \cite{Normaliz}. In particular, one can prove that a given lattice admits no reflective automorphic products by verifying that the polyhedral cone contains no nonzero integral points.

\subsection{Borcherds--Kac--Moody superalgebras}\label{subsec:BKM}
Borcherds--Kac--Moody (BKM) superalgebras were introduced by Borcherds \cite{Bor88} in 1988 as a generalization of Kac--Moody algebras. These objects form a class of Lie superalgebras that are typically infinite-dimensional. We will always take $\CC$ as the base field. 

A Lie superalgebra is a $\mathbb{Z}_2$-graded algebra $G=G_{\bar{0}}\oplus G_{\bar{1}}$ with a Lie bracket that satisfies
$$
[a,b]=-(-1)^{\mathrm{deg}(a)\mathrm{deg}(b)}[b,a] \quad \text{and} \quad [a,[b,c]]=[[a,b],c]+(-1)^{\mathrm{deg}(a)\mathrm{deg}(b)}[b,[a,c]],
$$
where for any homogeneous element $x\in G_{\bar{n}}$ ($n=0,1$), we write $\mathrm{deg}(x):=\bar{n}$. The subspaces $G_{\bar{0}}$ and $G_{\bar{1}}$ are called the even and odd parts of $G$, respectively. Such algebras are important in theoretical physics, where they describe the mathematics of supersymmetry. 

Similarly to finite-dimensional simple Lie algebras, BKM superalgebras can be defined in terms of Chevalley–Serre generators and relations encoded in a generalized Cartan matrix. The restrictions on the matrix are weaker than in the Kac–Moody case. In particular, simple roots are allowed to have non-positive norm (i.e., they may be imaginary roots). BKM superalgebras also admit character and supercharacter formulas for irreducible integrable highest weight modules, which in turn yield denominator and superdenominator identities. We review the theory following \cite{Ray06}.

Let $I$ be a finite or countably infinite indexing set for the simple roots, and let $S\subset I$ index the odd simple roots. Let $\mathcal{H}_{\RR}$ be a real vector space with a non-degenerate symmetric bilinear form $(-,-)$, and containing elements $h_i$ for $i\in I$ satisfying
\begin{enumerate}
\item $(h_i,h_j)\leq 0$ if $i\neq j$;
\item if $(h_i,h_i)>0$, then $\frac{2(h_i,h_j)}{(h_i,h_i)}\in\ZZ$ for all $j\in I$;
\item if $(h_i,h_i)>0$ and $i\in S$, then $\frac{(h_i,h_j)}{(h_i,h_i)}\in\ZZ$ for all $j\in I$.
\end{enumerate}
Let $\mathcal{H}=\mathcal{H}_\RR\otimes_\RR\CC$, and let $A$ denote the symmetric real-valued matrix with entries $a_{ij}:=(h_i,h_j)$. Note that the elements $h_i$ need not be linearly independent or even distinct. 

The BKM superalgebra $\mathcal{G}$ associated with $\mathcal{H}$ and $A$ is the Lie superalgebra generated by the abelian Lie algebra $\mathcal{H}$ and by elements $e_i$, $f_i$ ($i\in I$), subject to the following relations: 
\begin{enumerate}
\item $[e_i,f_j]=\delta_{ij}h_i$;
\item $[h,e_i]=(h,h_i)e_i$ and $[h,f_i]=-(h,h_i)f_i$;
\item $\deg(e_i)=0=\deg(f_i)$ if $i\not\in S$, $\deg(e_i)=1=\deg(f_i)$ if $i\in S$;
\item if $a_{ii}>0$ and $i\neq j$, then
$$
\big( \mathrm{ad}(e_i) \big)^{1-\frac{2a_{ij}}{a_{ii}}}e_j=0=\big( \mathrm{ad}(f_i) \big)^{1-\frac{2a_{ij}}{a_{ii}}}f_j;
$$
\item $[e_i,e_j]=0=[f_i,f_j]$ whenever $a_{ij}=0$. 
\end{enumerate}

The matrix $A$ is the \textit{generalized Cartan matrix}, and $\mathcal{H}$ is a \textit{generalized Cartan subalgebra}. If $a_{ii}>0$ for all $i\in I$, then $\mathcal{G}$ is called a \textit{Kac--Moody Lie superalgebra}. This definition (from \cite[\S 2.1]{Ray06}) may differ from the usual one in the literature on Lie superalgebras, but it is convenient for our purposes. If $S$ is empty, then $\mathcal{G}$ is a Lie algebra and is called a \textit{BKM algebra}. A Kac--Moody Lie superalgebra with $S=\emptyset$ is called a \textit{Kac--Moody algebra}. The even part $\mathcal{G}_{\bar{0}}$ of a BKM superalgebra $\mathcal{G}$ is itself a BKM algebra. There is a criterion for determining whether a given Lie superalgebra is a BKM superalgebra (cf. \cite[Theorem 2.5.4 and Corollary 2.5.11]{Ray06}). 

All finite-dimensional semi-simple Lie algebras and affine Lie algebras are Kac--Moody algebras, and hence BKM algebras. However, most finite-dimensional simple Lie superalgebras are not BKM superalgebras (see \cite[Corollary 2.1.23]{Ray06}). For example, $\mathfrak{sl}_{m|n}$ with $m>n$ is a BKM superalgebra if and only if $n\leq 2$. The only finite-dimensional simple Kac--Moody Lie superalgebras with nontrivial odd parts are the orthosymplectic Lie superalgebras of type $\mathfrak{osp}_{1|2r}$. These superalgebras play a major role in this paper and we will review them in the next subsection. 

The \textit{formal root lattice} $\mathcal{Q}$ is the free abelian group generated by $\alpha_i$ ($i\in I$) with the real bilinear form defined by $(\alpha_i, \alpha_j)=a_{ij}$. The elements $\alpha_i$ are called \textit{simple roots}. Roots, positive roots, and root spaces are defined as usual. A root is called \textit{real} if it has positive norm  and \textit{imaginary} otherwise. In general, $\mathcal{Q}$ is not a true lattice in the usual sense, as its rank may be infinite. 

There is a natural homomorphism of abelian groups from $\mathcal{Q}$ to $\mathcal{H}$, sending $\alpha_i$ to $h_i$. This map is not generally injective since the $h_i$ are not necessarily linearly independent; it is possible that several imaginary simple roots have the same image $h$ in $\mathcal{H}$. By abuse of terminology, the images of roots under this map are also called roots. The root space of a root $\alpha\in \mathcal{Q}$ is either even or odd. However, the root space of the corresponding element of $\mathcal{H}$ generally has both even and odd parts. Depending on the context, roots will be taken to lie in $\mathcal{Q}$, in $\mathcal{H}$, or in its dual $\mathcal{H}^*$.

Let $\mathcal{G}=\mathcal{G}_{\bar{0}}\oplus \mathcal{G}_{\bar{1}}$ be a BKM superalgebra with even part $\mathcal{G}_{\bar{0}}$ and odd part $\mathcal{G}_{\bar{1}}$. If $\mathcal{G}_{\bar{1}}$ is trivial, then $\mathcal{G}$ is a BKM algebra. Let $\mathcal{H}$ be a generalized Cartan subalgebra of $\mathcal{G}$, which is unique up to conjugation by inner automorphisms. Let $\Delta^+$ denote the set of positive roots of $\mathcal{G}$, viewed as elements of $\mathcal{H}$. Then $\mathcal{G}$ admits the root space decomposition
$$
\mathcal{G}=\bigoplus_{\alpha \in \Delta^+} \mathcal{G}_\alpha \oplus \mathcal{H} \oplus \bigoplus_{\alpha \in \Delta^+} \mathcal{G}_{-\alpha},
$$
where the root space $\mathcal{G}_\alpha$ for $\alpha \in \mathcal{H}$ is 
$$
\mathcal{G}_\alpha:=\{ x\in \mathcal{G} : \; [h, x]=(h,\alpha)x,\; \text{for all}\; h\in \mathcal{H} \}.
$$
For $\alpha \in \Delta^+$, we define
$$
\mathrm{mult}_{\bar{0}}(\alpha) := \dim (\mathcal{G}_\alpha\cap \mathcal{G}_{\bar{0}}),  \quad \mathrm{mult}_{\bar{1}}(\alpha) := \dim (\mathcal{G}_\alpha\cap \mathcal{G}_{\bar{1}}), \quad
\mathrm{mult}(\alpha) := \dim (\mathcal{G}_\alpha),
$$
and define the supermultiplicity as the integer 
$$
\text{$s$-$\mathrm{mult}(\alpha)$}:=\mathrm{mult}_{\bar{0}}(\alpha) - \mathrm{mult}_{\bar{1}}(\alpha). 
$$
The sets of positive even and odd roots are defined by
$$
\Delta^+_{\bar{0}}=\{ \alpha\in \Delta^+: \mathrm{mult}_{\bar{0}}(\alpha) >0 \} \quad \text{and} \quad \Delta^+_{\bar{1}}=\{ \alpha\in \Delta^+: \mathrm{mult}_{\bar{1}}(\alpha) >0 \}.
$$

If $\alpha_i$ is a simple root in $\mathcal{Q}$, then the root spaces of $\alpha_i$ and $-\alpha_i$ are each one-dimensional, generated by $e_i$ and $f_i$, respectively. In particular, $\alpha_i$ is either even or odd. These facts also hold for real simple roots regarded as elements of $\mathcal{H}$ or $\mathcal{H}^*$. Every real root in $\mathcal{H}$ has multiplicity one, and is therefore either even or odd. However, an imaginary root in $\mathcal{H}$ may be even and odd simultaneously. If $\alpha$ is an odd real simple root, then $2\alpha$ is an even real root. For any even real root $\alpha$, the scalar multiple $k\alpha$ is an even root if and only if $k=\pm 1$. 

The \textit{Weyl group} $\mathcal{W}$ is generated by reflections $\sigma_\alpha$ for real simple roots $\alpha$, where 
$$
\sigma_\alpha(x):=x-\frac{2(x,\alpha)}{(\alpha,\alpha)}\alpha, \quad x\in \mathcal{H}. 
$$
For any positive real root $\alpha$, there exists $\sigma \in \mathcal{W}$ such that either $\sigma(\alpha)$ or $\frac{1}{2}\sigma(\alpha)$ is a real simple root. The \textit{fundamental Weyl chamber} is the subset 
$$
\mathcal{C}:= \{ h\in \mathcal{H}_\RR : \; (h,h_i)> 0, \; \text{for all $i\in I$ with $a_{ii}>0$} \}.
$$
The Weyl chambers of $\mathcal{G}$ are the conjugates of $\mathcal{C}$ under the action of $\mathcal{W}$. 

A vector $\rho\in \mathcal{H}$ satisfying
$$
(\rho,\alpha_i)=(\alpha_i,\alpha_i)/2, \quad \text{for all $i\in I$ with $a_{ii}>0$}
$$
is called a \textit{Weyl vector} of $\mathcal{G}$. Such vectors may not exist; when they do, they are generally not unique and may even be zero. For example, the fake monster Lie superalgebra \cite{Sch00} has no real roots, so any vector in $\mathcal{H}$ can serve as a Weyl vector. In addition, every imaginary simple root in its generalized Cartan subalgebra is both even and odd, with multiplicity $8$ in each parity. 

We now review the representation theory of BKM superalgebras, focusing on the case where $\mathcal{G}$ is infinite-dimensional or has no imaginary simple roots. Assume $\mathcal{G}$ has a Weyl vector $\rho$. The irreducible integrable highest weight modules are indexed by the set 
\begin{equation}\label{eq:dominant weights}
\begin{split}
\mathcal{P}^+:=&\Big\{ \Lambda \in \mathcal{H} :  \; (\Lambda,\alpha_i)\geq 0 \; \text{if $i\in I$}; \; \frac{2(\Lambda,\alpha_i)}{(\alpha_i,\alpha_i)}\in\NN \; \text{if $i\in I\backslash S$ and $a_{ii}>0$};\\
& \qquad \qquad \frac{(\Lambda,\alpha_i)}{(\alpha_i,\alpha_i)}\in\NN \; \text{if $i\in S$ and $a_{ii}>0$}\Big\}.    
\end{split}   
\end{equation}
Let $\mathcal{L}(\Lambda)$ be an irreducible integrable highest weight module indexed by a weight $\Lambda\in \mathcal{P}^+$. Then $\mathcal{H}$ acts semisimply on $\mathcal{L}(\Lambda)$, giving the decomposition 
$$
\mathcal{L}(\Lambda)=\bigoplus_{\lambda\in \mathcal{H}} \mathcal{L}(\Lambda)_\lambda,
$$
where $\mathcal{L}(\Lambda)_\lambda$ is the eigenspace with eigenvalue $\lambda$, defined by
$$
\mathcal{L}(\Lambda)_\lambda:=\big\{ v\in \mathcal{L}(\Lambda):\; hv=(\lambda,h)v, \; \text{for any $h\in\mathcal{H}$} \big\}.
$$
We note that $\dim \mathcal{L}(\Lambda)_\lambda<\infty$, $\dim \mathcal{L}(\Lambda)_\Lambda=1$, and $\dim \mathcal{L}(\Lambda)_\lambda = \dim \mathcal{L}(\Lambda)_{\sigma(\lambda)}$ for any $\sigma\in\mathcal{W}$. Furthermore, the weight space $\mathcal{L}(\Lambda)_\lambda$ decomposes into even and odd parts:
$$
\mathcal{L}(\Lambda)_\lambda = \mathcal{L}(\Lambda)_{\bar{0},\lambda}\oplus \mathcal{L}(\Lambda)_{\bar{1},\lambda}.
$$
The character and supercharacter of $\mathcal{L}(\Lambda)$ are defined respectively as the formal generating series
\begin{align*}
\ch \mathcal{L}(\Lambda)&:=\sum_{\lambda\in\mathcal{H}} \dim\mathcal{L}(\Lambda)_\lambda \cdot e^\lambda,\\
\sch\mathcal{L}(\Lambda)&:=\sum_{\lambda\in\mathcal{H}} \big(\dim\mathcal{L}(\Lambda)_{\bar{0},\lambda}- \dim\mathcal{L}(\Lambda)_{\bar{1},\lambda}\big) \cdot e^\lambda,
\end{align*}
where $e^\lambda$ are formal exponents. These series are given by the Borcherds–Kac–Moody character and supercharacter formulas (see below).

Expanding a root as $\alpha=\sum_{i\in I} k_i \alpha_i$, we define its height and even height, respectively, by
$$
\mathrm{ht}(\alpha) = \sum_{i\in I} k_i \quad \text{and} \quad \mathrm{ht}_0(\alpha) = \sum_{i\in I \backslash S } k_i.
$$
For any $\Lambda\in\mathcal{P}^+$, we define two formal sums
$$
\mathcal{T}_{\Lambda}:=e^{\Lambda+\rho} \sum_{\mu}(-1)^{\mathrm{ht}(\mu)} e^{-\mu} \quad \text{and} \quad \mathcal{T}'_{\Lambda}:=e^{\Lambda+\rho} \sum_{\mu}(-1)^{\mathrm{ht}_0(\mu)} e^{-\mu}.
$$
The sums are taken over all sums of mutually orthogonal \textit{imaginary} simple roots perpendicular to $\Lambda$, with the convention that each imaginary simple root occurs at most once, unless it is an odd root of norm zero. Here, imaginary simple roots are regarded as elements of the formal root lattice $\mathcal{Q}$. For simplicity, we use the same notation for the images of $\mathcal{T}_\Lambda$ and $\mathcal{T}'_\Lambda$ under the natural map $\mathcal{Q}\to\mathcal{H}$. 

Any $\sigma\in \mathcal{W}$ can be expressed as a product of reflections associated with real simple roots. Let $l(\sigma)$ and $l'(\sigma)$ denote the number of real simple roots and even real simple roots appearing in such an expression, respectively. We define 
\begin{equation}
\varepsilon(\sigma):=(-1)^{l(\sigma)} \quad \text{and} \quad \varepsilon'(\sigma):=(-1)^{l'(\sigma)}.    
\end{equation}
Both $\varepsilon$ and $\varepsilon'$ are multiplicative on $\mathcal{W}$, and $\varepsilon=\varepsilon'$ if $\mathcal{G}$ has no odd real simple roots. Define
\begin{equation*}
\mathcal{R}:= \frac{\prod_{\alpha \in \Delta_{\bar{0}}^+}(1-e^{-\alpha})^{\mathrm{mult}_{\bar{0}}(\alpha)}}{\prod_{\alpha \in \Delta_{\bar{1}}^+}(1+e^{-\alpha})^{\mathrm{mult}_{\bar{1}}(\alpha)}} \quad \text{and} \quad \mathcal{R}':=\frac{\prod_{\alpha \in \Delta_{\bar{0}}^+}(1-e^{-\alpha})^{\mathrm{mult}_{\bar{0}}(\alpha)}}{\prod_{\alpha \in \Delta_{\bar{1}}^+}(1-e^{-\alpha})^{\mathrm{mult}_{\bar{1}}(\alpha)}}.
\end{equation*}
Then the Borcherds--Kac--Moody character and supercharacter formulas are
\begin{align}
\ch\mathcal{L}(\Lambda) &= \mathcal{R}^{-1}\cdot e^{-\rho}\sum_{\sigma\in\mathcal{W}} \varepsilon(\sigma)\sigma(\mathcal{T}_\Lambda),\\ \sch\mathcal{L}(\Lambda) &= \mathcal{R}'^{-1}\cdot e^{-\rho}\sum_{\sigma\in\mathcal{W}} \varepsilon'(\sigma)\sigma(\mathcal{T}'_\Lambda), 
\end{align}
where $\sigma$ acts on $e^{\lambda}$ by $e^{\sigma(\lambda)}$. The second formula corrects \cite[Theorem 2.6.34]{Ray06} when odd real simple roots are present. These formulas are independent of the choice of the Weyl vector.

The $\mathcal{G}$-module $\mathcal{L}(0)$ is trivial, so it has dimension $1$, and hence $\ch\mathcal{L}(0)=\sch\mathcal{L}(0)=1$. This yields the Borcherds--Kac--Weyl denominator identity
\begin{equation}\label{eq:BKM-d}
    e^\rho \cdot \frac{\prod_{\alpha \in \Delta_{\bar{0}}^+}(1-e^{-\alpha})^{\mathrm{mult}_{\bar{0}}(\alpha)}}{\prod_{\alpha \in \Delta_{\bar{1}}^+}(1+e^{-\alpha})^{\mathrm{mult}_{\bar{1}}(\alpha)}} = \sum_{\sigma\in \mathcal{W}} \varepsilon(\sigma) \sigma(\mathcal{T}_0),
\end{equation}
and the superdenominator identity 
\begin{equation}\label{eq:BKM-sd}
    e^\rho \cdot \frac{\prod_{\alpha \in \Delta_{\bar{0}}^+}(1-e^{-\alpha})^{\mathrm{mult}_{\bar{0}}(\alpha)}}{\prod_{\alpha \in \Delta_{\bar{1}}^+}(1-e^{-\alpha})^{\mathrm{mult}_{\bar{1}}(\alpha)}} = \sum_{\sigma\in \mathcal{W}} \varepsilon'(\sigma) \sigma(\mathcal{T}'_0). 
\end{equation}

When $\mathcal{G}$ has no imaginary simple roots, we have $\mathcal{T}_{\Lambda}=\mathcal{T}'_{\Lambda}=e^{\Lambda+\rho}$. The above formulas essentially characterize the BKM superalgebra $\mathcal{G}$. For certain $\mathcal{G}$, they define modular forms of various types; in such cases, it is necessary to view roots as elements of $\mathcal{H}$ when the rank of $\mathcal{Q}$ is infinite. For example, the denominators of affine Lie algebras define holomorphic Jacobi forms of singular weight and lattice index; the characters of affine Lie algebras at a fixed level define a vector-valued weakly holomorphic Jacobi form of weight $0$ for a certain representation of $\SL_2(\ZZ)$; and the denominator of the fake monster algebra defines a holomorphic modular form for the orthogonal group $\Orth^+(\II_{26,2})$. 

\subsection{Affine Lie superalgebras of type \texorpdfstring{$\mathfrak{osp}_{1|2r}$}{}}\label{subsec:osp}
Unlike affine Lie algebras (see, e.g., \cite[\S 2.2]{SWW23} for a brief overview), the affine Lie superalgebras of type $\mathfrak{osp}_{1|2r}$ are less well known. We review them following \cite{Kac77, Kac78, Kac90}. In particular, we recall the Weyl–Kac formulas for computing supercharacters and identify their denominator and superdenominator functions as theta quotients, which are precisely holomorphic Jacobi forms of singular weight. 

\subsubsection{The simple Lie superalgebra \texorpdfstring{$\mathfrak{osp}_{1|2r}$}{}}
Finite-dimensional simple Lie superalgebras over $\CC$ were classified by Kac \cite{Kac77} in 1977. In particular, Kac proved that the orthosymplectic Lie superalgebras $\mathfrak{osp}_{1|2r}$ are the only finite-dimensional simple basic Lie superalgebras with nontrivial odd parts and no isotropic roots, which admit a nondegenerate, consistent, supersymmetric bilinear form that can be normalized so that the induced form on the real span of the roots is positive definite. This family is also distinguished as the unique class of finite-dimensional simple Kac–Moody Lie superalgebras with nontrivial odd parts (see Section \ref{subsec:BKM}). Hence $\mathfrak{osp}_{1|2r}$ behaves in many respects like a simple finite-dimensional Lie algebra (see also \cite{CGL24}). 

For $r\geq 1$, the Lie superalgebra $\mathfrak{osp}_{1|2r}$ is simple of rank $r$ and dimension $2r^2+3r$. Its even part is the simple Lie algebra $\mathfrak{sp}_{2r}$ with root system of type $C_r$, and its odd part has dimension $2r$. We denote its Cartan subalgebra and root system by $\mathfrak{h}_{\mathfrak{osp}}$ and $\Delta_{\mathfrak{osp}}$, respectively. The invariant symmetric bilinear form $\latt{-,-}$ on $\mathfrak{h}_{\mathfrak{osp}}^*$ is normalized so that long roots have square norm two. We fix the set of positive odd roots as 
$$
\Delta_{\mathfrak{osp},\bar{1}}^+:=\{ \epsilon_j : \; 1\leq j \leq r \}\subset \mathfrak{h}_{\mathfrak{osp}}^*,
$$
where $\latt{\epsilon_i,\epsilon_j}=\frac{1}{2}\delta_{ij}$ for $1\leq i,\, j \leq r$. Then the set of positive even roots is
$$
\Delta_{\mathfrak{osp},\bar{0}}^+:=\{ 2\epsilon_j : \; 1\leq j \leq r \}\cup \{ \epsilon_i\pm \epsilon_j : \; 1\leq i < j \leq r \}\subset \mathfrak{h}_{\mathfrak{osp}}^*,
$$
and the simple roots are
$$
\alpha_1:=\epsilon_1-\epsilon_2,\; \alpha_2:=\epsilon_2-\epsilon_3,\; \ldots , \;\alpha_{r-1}:=\epsilon_{r-1}-\epsilon_r,\; 
\alpha_r:=\epsilon_r.
$$
Note that $\alpha_r$ is odd, while the other simple roots are even. The highest root is the long root 
$$
\theta:=2\alpha_1+2\alpha_2+\cdots+2\alpha_r=2\epsilon_1.
$$
We identify $\mathfrak{h}_{\mathfrak{osp}}$ with $\mathfrak{h}^*_{\mathfrak{osp}}$ and define the coroot of a root $\alpha\in\mathfrak{h}^*_{\mathfrak{osp}} $ by 
$$
\alpha^\vee:=2\alpha/\latt{\alpha, \alpha}.
$$
The fundamental weights $w_i\in \mathfrak{h}_{\mathfrak{osp}}^*$ are defined by 
$$
\latt{w_i, \alpha_j^\vee} = \delta_{ij}, \quad 1\leq i,\, j \leq r.
$$
Explicitly, they are given by 
$$
w_1=\epsilon_1, \; w_2=\epsilon_1+\epsilon_2,\;\ldots,\; w_{r-1}=\epsilon_1+\cdots+\epsilon_{r-1},\; w_r=(\epsilon_1+\cdots+\epsilon_r)/2.
$$
The Weyl vector $\rho_\mathfrak{osp}$ is defined by
$$
\rho_\mathfrak{osp}=\sum_{j=1}^r w_j=\frac{1}{2}\Big(\sum_{\alpha\in \Delta_{\mathfrak{osp},\bar{0}}^+} \alpha \; - \; \sum_{\beta\in \Delta_{\mathfrak{osp},\bar{1}}^+} \beta\Big) =\sum_{j=1}^r \big( r+1/2 -j \big) \epsilon_j,
$$
so that $\latt{\rho_\mathfrak{osp},\alpha_i}=\latt{\alpha_i,\alpha_i}/2$ for $1\leq i\leq r$. 
The root lattice $Q_{\mathfrak{osp}}$ is $\ZZ^{r}(1/2)$, while the coroot lattice $Q^\vee_{\mathfrak{osp}}$ is $\ZZ^r(2)$. The weight lattice $P_{\mathfrak{osp}}$ (dual to the coroot lattice) is therefore isomorphic to $\ZZ^r(1/2)$, while the coweight lattice $P^\vee_{\mathfrak{osp}}$ (dual to the root lattice) is isomorphic to $\ZZ^r(2)$. 

The reflection associated with a root $\alpha$ is defined as usual by
$$
\sigma_{\alpha}(x)= x - \latt{x,\alpha^\vee}\alpha, \quad x\in Q_\mathfrak{osp}\otimes \RR.
$$
The Weyl group $W_\mathfrak{osp}$ is the semi-direct product of the symmetric group $\mathcal{S}_r$ (acting by permutations of  $\epsilon_i$) by $(\ZZ/2\ZZ)^r$ (acting by sign changes $\epsilon_i\mapsto \pm \epsilon_i$). Its order is $2^r\cdot r!$. 

The coroot of the highest root $\theta$ admits the linear combination
$$
\theta = \theta^\vee = \alpha_1^\vee + \cdots +\alpha_{r-1}^\vee+\alpha_r^\vee/2
$$
with comarks $1$, \ldots, $1$, $\frac{1}{2}$. The sum of comarks plus one is the dual Coxeter number
$$
h_\mathfrak{osp}^\vee= r+\frac{1}{2} = \frac{1}{r}\Big(\sum_{\alpha\in \Delta_{\mathfrak{osp},\bar{0}}^+} \latt{\alpha, \alpha} \; - \; \sum_{\beta\in \Delta_{\mathfrak{osp},\bar{1}}^+} \latt{\beta, \beta}\Big),
$$
which satisfies 
\begin{equation}\label{eq:osp-eutatic star}
\sum_{\alpha\in \Delta_{\mathfrak{osp},\bar{0}}^+} \latt{\alpha, \mathfrak{z}}^2 \; - \; \sum_{\beta\in \Delta_{\mathfrak{osp},\bar{1}}^+} \latt{\beta, \mathfrak{z}}^2 = h_\mathfrak{osp}^\vee\latt{\mathfrak{z},\mathfrak{z}}, \quad \mathfrak{z}\in Q_\mathfrak{osp}\otimes\CC.    
\end{equation}
The superdimension $\mathrm{sdim}(\mathfrak{osp})$ of $\mathfrak{osp}_{1|2r}$ is defined as the difference between the dimensions of the even and odd parts. We have 
$$
\mathrm{sdim}(\mathfrak{osp}) = 2r^2-r,
$$
and an analog of the Freudenthal–de Vries strange formula:
$$
h_\mathfrak{osp}^\vee \cdot \mathrm{sdim}(\mathfrak{osp}) =12 \latt{\rho_\mathfrak{osp},\rho_\mathfrak{osp}}.
$$

\subsubsection{The affine Lie superalgebra \texorpdfstring{$\widehat{\mathfrak{osp}}_{1|2r}$}{}}
Affine Lie superalgebras are contragredient $\ZZ$-graded Lie superalgebras of infinite dimension and of polynomial growth. They were classified by Kac \cite{Kac68} and van de Leur \cite{van89}. Affine Lie algebras are Kac--Moody algebras, while affine Lie superalgebras are not necessarily BKM superalgebras. 

The (untwisted) affine Lie superalgebra $\widehat{\mathfrak{osp}}_{1|2r}$ is an extension of $\mathfrak{osp}_{1|2r}$ defined by
\begin{align*}
&\widehat{\mathfrak{osp}}_{1|2r}=\mathfrak{osp}_{1|2r}[t,t^{-1}] \oplus \CC K \oplus \CC d, \\
&[xt^m,yt^n]=[x,y]t^{m+n}+m\delta_{m,-n}(x,y)K, \\
&[K,xt^m]=[K,d]=0, \quad [d,xt^m]=mxt^m, 
\end{align*}
where $K$ is central, $d$ is a derivation, and $x,y\in \mathfrak{osp}_{1|2r}$. This is a Kac--Moody Lie superalgebra with affine Cartan subalgebra $\hat{\mathfrak{h}}_{\mathfrak{osp}}=\mathfrak{h}_{\mathfrak{osp}}\oplus \CC K \oplus \CC d$. We write $\hat{\mathfrak{h}}_{\mathfrak{osp}}^*=\mathfrak{h}_{\mathfrak{osp}}^*\oplus \CC \hat{w}_0 \oplus \CC \delta$, with
\begin{align*}
&\hat{w}_0(\mathfrak{h}\oplus \CC d)=0, \qquad\hat{w}_0(K)=1,\\
&\delta(\mathfrak{h}\oplus \CC K)=0, \qquad\ \delta(d)=1.
\end{align*}
Embedding $\mathfrak{h}^*_{\mathfrak{osp}}$ into $\hat{\mathfrak{h}}^*_{\mathfrak{osp}}$, we define 
$$
\alpha_0:=\delta-\theta \quad \text{and} \quad \alpha_0^\vee:=K-\theta^\vee.
$$
Then $\{\alpha_0, \alpha_1,\cdots, \alpha_r\}$ is a set of simple roots of $\widehat{\mathfrak{osp}}_{1|2r}$, and $\{\alpha_0^\vee, \alpha_1^\vee,\cdots, \alpha_r^\vee\}$ is the corresponding set of coroots. We further set
$$
\hat{w}_i := w_i + \hat{w}_0, \; 1\leq i\leq r-1, \quad \text{and} \quad \hat{w}_r:=w_r+\hat{w}_0/2.
$$
Then
$\{ \hat{w}_0, \hat{w}_1,\cdots ,\hat{w}_r \}$ is a set of fundamental weights of $\widehat{\mathfrak{osp}}_{1|2r}$. The Weyl vector is taken to be
$$
\hat{\rho}_\mathfrak{osp}= \rho_\mathfrak{osp} + \big(r+\frac{1}{2}\big)\hat{w}_0 - \frac{2r^2-r}{24} \delta,
$$
so that $(\hat{\rho}_{\mathfrak{osp}},\alpha_i)=(\alpha_i,\alpha_i)/2$ for $0\leq i\leq r$, and $(\hat{\rho}_\mathfrak{osp},\hat{\rho}_\mathfrak{osp})=0$; the coefficients $r+\frac{1}{2}$ and $2r^2-r$ are the dual Coxeter number and the superdimension of $\mathfrak{osp}_{1|2r}$, respectively. 

\subsubsection{Representation theory of \texorpdfstring{$\widehat{\mathfrak{osp}}_{1|2r}$}{}}
As a BKM superalgebra, the irreducible integrable highest weight representations of $\widehat{\mathfrak{osp}}_{1|2r}$ are indexed by dominant integral weights in the set $\mathcal{P}^+$ defined in  \eqref{eq:dominant weights}. A dominant integral weight has the form
$$
\Lambda = \sum_{j=0}^r x_j \hat{w}_j + c\delta,
$$
where $x_j$ are non-negative integers for $0\leq j \leq r$, and $x_r$ is even so that the level of $\Lambda$, that is, $\Lambda(K)=x_0+x_1+\cdots+x_{r-1}+x_r/2$ is integral. 

Let $L_{\widehat{\mathfrak{osp}}_{1|2r}}(\Lambda)$ denote the irreducible highest weight module of $\widehat{\mathfrak{osp}}_{1|2r}$ labeled by a dominant integral weight $\Lambda$. 
Its supercharacter is given by 
\begin{equation}\label{eq:formal-osp-sch}
\chi_\Lambda^{\widehat{\mathfrak{osp}}} = \frac{\sum_{\hat{\sigma}\in \widehat{W}_{\mathfrak{osp}}} \varepsilon'(\hat{\sigma}) e^{\hat{\sigma}(\Lambda+\hat{\rho}_{\mathfrak{osp}})}}{e^{\hat{\rho}_{\mathfrak{osp}}} \prod_{\alpha \in \widehat{\Delta}_{\mathfrak{osp}}^+}(1-e^{-\alpha})^{s\text{-}\mathrm{mult}(\alpha)}}.
\end{equation}
In particular, when $\Lambda=0$,  $\chi_\Lambda^{\widehat{\mathfrak{osp}}}=1$, leading to the \textit{superdenominator identity}:
\begin{equation}\label{eq:formal-ops-sd}
e^{\hat{\rho}_{\mathfrak{osp}}} \prod_{\alpha \in \widehat{\Delta}_{\mathfrak{osp}}^+}(1-e^{-\alpha})^{s\text{-}\mathrm{mult}(\alpha)} = \sum_{\hat{\sigma}\in \widehat{W}_{\mathfrak{osp}}} \varepsilon'(\hat{\sigma})e^{\hat{\sigma}(\hat{\rho}_{\mathfrak{osp}})}.    
\end{equation}

Let us explain the above notation. The superalgebra $\widehat{\mathfrak{osp}}_{1|2r}$ has real roots and imaginary roots of norm $0$. The sets of even and odd positive real roots are respectively
\begin{align*}
\widehat{\Delta}_{\mathfrak{osp},\bar{0}}^{+,\mathrm{re}} &= \big\{  n\delta \pm \alpha : \, n \in \ZZ_{>0}, \; \alpha \in \Delta_{\mathfrak{osp},\bar{0}}^+ \big\} \cup \Delta_{\mathfrak{osp},\bar{0}}^+,   \\
\widehat{\Delta}_{\mathfrak{osp},\bar{1}}^{+,\mathrm{re}} &= \big\{ n\delta\pm \alpha :\, n \in \ZZ_{>0}, \; \alpha \in \Delta_{\mathfrak{osp},\bar{1}}^+ \big\} \cup \Delta_{\mathfrak{osp},\bar{1}}^+.
\end{align*}
All imaginary roots are even, and the set of positive imaginary roots is
$$
\widehat{\Delta}_{\mathfrak{osp}}^{+,\mathrm{im}} = \big\{ n\delta : n \in \ZZ_{>0} \big\}.
$$
The set $\widehat{\Delta}_{\mathfrak{osp}}^+$ of positive roots is the union of the three sets above. Every real root has multiplicity one, while every imaginary root has multiplicity $r$. Hence, the supermultiplicity $s\text{-}\mathrm{mult}(\alpha)$ is $1$, $-1$, or $r$ according to whether $\alpha$ is an even real, odd real, or imaginary root. 

The Weyl group $\widehat{W}_{\mathfrak{osp}}$ is a semi-direct product of $W_\mathfrak{osp}$ by a translation group $T_{\mathfrak{osp}}$. 
Every $\hat{\sigma}\in \widehat{W}_{\mathfrak{osp}}$ can be written as a product of reflections associated with simple roots. The sign $\varepsilon'(\hat{\sigma})$ is $1$ or $-1$ according to whether the number of even simple roots in such an expression is even or odd.  For $\beta$ in the coroot lattice $Q_{\mathfrak{osp}}^\vee$ of $\mathfrak{osp}_{1|2r}$, the associated translation is defined by 
$$
t_\beta(x)=x+\latt{x,K}\beta-\Big( \latt{x,\beta}+\frac{1}{2}\latt{\beta,\beta}\latt{x,K} \Big)\delta, \quad x\in \hat{\mathfrak{h}}_{\mathfrak{osp}}^*. 
$$
For any $\beta, \beta_1, \beta_2\in Q_{\mathfrak{osp}}^\vee$ and $\sigma\in W_{\mathfrak{osp}}$, we have 
$$
t_{\beta_1}t_{\beta_2}=t_{\beta_1+\beta_2} \quad \text{and} \quad t_{\sigma(\beta)} \cdot \sigma = \sigma \cdot t_{\beta}. 
$$
Hence every element of $\widehat{W}_{\mathfrak{osp}}$ can be uniquely written as 
$$
\hat{\sigma}=t_{2m_1\epsilon_1+2m_2\epsilon_2+\cdots+2m_r\epsilon_r}\cdot \sigma, \quad \sigma\in W_{\mathfrak{osp}},
$$
where $m_1$, \ldots, $m_r\in\ZZ$. For $1\leq i,j\leq  r$ and $m\in\ZZ$, we have 
\begin{align*}
t_{2m\epsilon_j}(\hat{w}_0)&=\hat{w}_0+2m\epsilon_j-m^2\delta,\\
t_{2m\epsilon_j}(\epsilon_i)&=\epsilon_i-\delta_{ij} m\delta,\\
t_{2m\epsilon_j}(\delta)&=\delta. 
\end{align*}
Note that $\varepsilon'(t_{2\epsilon_j})=-1$, since $t_{2\epsilon_j}=\sigma_{\epsilon_j}\cdot \sigma_{2\epsilon_j+\delta}$. 
For $\sigma\in W_{\mathfrak{osp}}$, if $\sigma$ is a permutation, then $\varepsilon'(\sigma)=\mathrm{sgn}(\sigma)$; if $\sigma$ is the sign change  $\epsilon_j\mapsto -\epsilon_j$, then $\varepsilon'(\sigma)=1$. The above information suffices to determine the numerator of \eqref{eq:formal-osp-sch}; see Sections \ref{sec:modular-invariants} and \ref{sec:singular-exceptional} for some explicit formulas. 

We now interpret the formal series \eqref{eq:formal-osp-sch}-\eqref{eq:formal-ops-sd} as ordinary analytic functions. 
Let $(\tau, \mathfrak{z})\in \HH \times \mathfrak{h}$, with $\mathfrak{z}=\sum_{j=1}^r z_j\cdot 2\epsilon_j$, $z_j\in\CC$. Substituting $e^{\alpha}$ by $e^{2\pi i \alpha(\mathfrak{z}-\tau d)}$, $\chi_\Lambda^{\widehat{\mathfrak{osp}}}(\tau, \mathfrak{z})$ becomes holomorphic on $\HH\times \CC^r$. The product side of the superdenominator identity yields the theta quotient
\begin{equation}\label{eq:superdenominator-osp}
\vartheta_{\mathfrak{osp}_{1|2r}}(\tau,\mathfrak{z}):=\eta(\tau)^r \prod_{1\leq i < j \leq r} \frac{\vartheta(\tau,z_i+z_j)\vartheta(\tau,z_i-z_j)}{\eta(\tau)^2} \prod_{j=1}^r \frac{\vartheta(\tau,2z_j)}{\vartheta(\tau,z_j)}.  \end{equation}

These theta quotients generalize theta blocks that have been treated systematically by Gritsenko--Skoruppa--Zagier\cite{GSZ19}; see also the next section. 

The modularity of $\vartheta$ \eqref{eq:Jacobi-theta} and the sum side of the superdenominator identity \eqref{eq:formal-ops-sd} yield: 

\begin{theorem}\label{th:osp-sd-modularity}
The superdenominator function \eqref{eq:superdenominator-osp} of the affine Lie superalgebra $\widehat{\mathfrak{osp}}_{1|2r}$ is a holomorphic Jacobi form of singular weight $r/2$, character $v_\eta^{2r^2-r}$, and index $\ZZ^r(2r+1)$. 
\end{theorem}

For $r=1$, this specializes to the Watson quintuple product:
\begin{equation}\label{eq:Jacobi-theta-3/2}
\begin{split}
\vartheta_{3/2}(\tau,z):=&\,\eta(\tau)\frac{\vartheta(\tau,2z)}{\vartheta(\tau,z)}=\sum_{n\in\ZZ}  \left( \frac{12}{n} \right) q^{\frac{n^2}{24}} \zeta^{\frac{n}{2}}\\ =&\, q^{\frac{1}{24}}(\zeta^{\frac{1}{2}}+\zeta^{-\frac{1}{2}})\prod_{n=1}^\infty (1+q^n\zeta)(1+q^n\zeta^{-1})(1-q^{2n-1}\zeta^2)(1-q^{2n-1}\zeta^{-2})(1-q^n).       
\end{split}    
\end{equation}

Similarly, we establish the modularity of the denominator of $\widehat{\mathfrak{osp}}_{1|2r}$. 

\begin{theorem}
The denominator function of the affine Lie superalgebra $\widehat{\mathfrak{osp}}_{1|2r}$ is given by  
\begin{equation}\label{eq:denominator-osp}
\vartheta^*_{\mathfrak{osp}_{1|2r}}(\tau,\mathfrak{z}):=\frac{\eta(2\tau)^r}{\eta(\tau)^r} \prod_{1\leq i < j \leq r} \frac{\vartheta(\tau,z_i+z_j)\vartheta(\tau,z_i-z_j)}{\eta(\tau)^2} \prod_{j=1}^r \frac{\vartheta(\tau,2z_j)\vartheta(\tau,z_j)}{\vartheta(2\tau,2z_j)}. \end{equation}
It is a holomorphic Jacobi form of singular weight and index $\ZZ^r(2r+1)$ on $\Gamma_0(2)$ with a character. 
\end{theorem}

\begin{proof}
The theta quotient is derived from the product side of the denominator identity \eqref{eq:BKM-d}. The sum side shows that $\vartheta^*_{\mathfrak{osp}_{1|2r}}$ is holomorphic at $\infty$. The holomorphy at $0$ follows from \eqref{eq:singular-laplace} after viewing $\vartheta^*_{\mathfrak{osp}_{1|2r}}(\tau,\mathfrak{z})\cdot e^{2\pi i\omega}$ as a degenerate orthogonal modular form as in Section \ref{subsec:Jacobi}.  \end{proof}

We now review the modularity of supercharacters. 
For $k\in\ZZ_{>0}$, the module $L_{\widehat{\mathfrak{osp}}_{1|2r}}(k\hat{w}_0)$ has a canonical structure as a simple, rational, and $C_2$-cofinite vertex operator superalgebra \cite{AL22,CREUTZIG2022108678}, called the \textit{affine VOSA} at level $k$. It has central charge 
\begin{equation}
c_{\mathfrak{osp}_{1|2r}}=\frac{k\cdot \mathrm{sdim}(\mathfrak{osp}_{1|2r})}{k + h_{\mathfrak{osp}_{1|2r}}^\vee}=\frac{2k(2r^2-r)}{2(k+r)+1}.
\end{equation}
The irreducible modules of this VOSA are exactly the irreducible integrable highest weight modules $L_{\widehat{\mathfrak{osp}}_{1|2r}}(k,\lambda)$, corresponding to the dominant integral weights $k\hat{w}_0+\lambda$ of level $k$, where 
$$
\lambda=k_1w_1+\cdots+k_{r-1}w_{r-1}+2k_rw_r,
$$
with $k_j\in\NN$ for $1\leq j\leq r$ and $\sum_{j=1}^rk_j\leq k$. The conformal weight of $L_{\widehat{\mathfrak{osp}}_{1|2r}}(k,\lambda)$ is 
\begin{equation}
h_\lambda=\frac{\latt{\lambda, \lambda + 2\rho_\mathfrak{osp}}}{2(k+h^\vee_\mathfrak{osp})}.
\end{equation}
The full supercharacter of $L_{\widehat{\mathfrak{osp}}_{1|2r}}(k,\lambda)$ is precisely $\chi_{k\hat{w}_0+\lambda}^{\widehat{\mathfrak{osp}}}(\tau,\mathfrak{z})$ as defined above. For convenience, we often write
\begin{equation}
\chi_{k\hat{w}_0+\lambda}^{\widehat{\mathfrak{osp}}}=\chi_{k_1,\ldots,k_r}^{\mathfrak{osp}_{1|2r, \, k}}. 
\end{equation}
From modularity theory \cite{Zhu96,DZ05,DZ10, Van13}, we conclude that all modified level-$k$ supercharacters 
$$
\widetilde{\chi}_{k\hat{w}_0+\lambda}^{\widehat{\mathfrak{osp}}}(\tau,\mathfrak{z}):=q^{h_\lambda-\frac{c_{\mathfrak{osp}}}{24}}\cdot \chi_{k\hat{w}_0+\lambda}^{\widehat{\mathfrak{osp}}}(\tau,\mathfrak{z})
$$
become a vector-valued weakly holomorphic Jacobi form of weight $0$ and index $\ZZ^r(2k)$ for a certain unitary representation $\varrho_{\mathfrak{osp}}$ of $\SL_2(\ZZ)$ that is invariant under the action of $W_\mathfrak{osp}$ on $\mathfrak{z}$. Let $T$ and $S$ denote the standard generators of $\SL_2(\ZZ)$. Analogously to the transformation laws for affine Lie algebras, from \cite{Kac78, KP84, Kac90} we derive 
\begin{align*}
\varrho_{\mathfrak{osp}}(T)_{\lambda\mu}&=\mathbf{e}\Big(h_\lambda - \frac{c_{\mathfrak{osp}}}{24}\Big)\delta_{\lambda\mu},\\
\varrho_{\mathfrak{osp}}(S)_{\lambda\mu}&=i^{\lvert\widetilde{\Delta}_{\mathfrak{osp}}^+\rvert} \cdot\mathrm{disc}(Q_{\mathfrak{osp}}^\vee)^{-\frac{1}{2}}\cdot(k+h_{\mathfrak{osp}}^\vee)^{-\frac{r}{2}}\cdot\sum_{\sigma\in W_{\mathfrak{osp}}} \varepsilon'(\sigma) \cdot \mathbf{e}\Big(- \frac{\latt{\sigma(\lambda+\rho_{\mathfrak{osp}}), \, \mu+\rho_{\mathfrak{osp}}}}{k+h_{\mathfrak{osp}}^\vee}\Big),  \\
&=i^{r(r-1)}\cdot \big(2k+2r+1\big)^{-r/2} \cdot \sum_{\sigma\in W_{\mathfrak{osp}}} \varepsilon'(\sigma) \cdot \mathbf{e}\Big(-\frac{\latt{\sigma(\lambda+\rho_{\mathfrak{osp}}), \, \mu+\rho_{\mathfrak{osp}}}}{k+r+1/2}\Big),
\end{align*}
where $k\hat{w_0}+\lambda$ and $k\hat{w}_0+\mu$ are dominant integral weights of level $k$ as above; $\delta_{\lambda\mu}$ is $1$ if $\lambda=\mu$ and $0$ otherwise; $\mathbf{e}(z):=e^{2\pi i z}$ for $z\in \CC$; and $\widetilde{\Delta}_{\mathfrak{osp}}^+:=\Delta_{\mathfrak{osp},\bar{0}}^+\setminus \{ 2\alpha : \; \alpha \in \Delta_{\mathfrak{osp},\bar{1}}^+ \}$. In particular, the trivial case  $k=0$ yields the identity
$$
\sum_{\sigma\in W_{\mathfrak{osp}}} \varepsilon'(\sigma) \cdot e\Big(-\frac{\latt{\sigma(\rho_{\mathfrak{osp}}), \, \rho_{\mathfrak{osp}}}}{r+1/2}\Big) = i^{r(r-1)}\cdot (2r+1)^{r/2}. 
$$

\section{Singular theta quotients and affine Lie superalgebras}\label{sec:theta quotients}
We have seen that the superdenominator function of $\widehat{\mathfrak{osp}}_{1|2r}$, as a theta quotient, is a holomorphic Jacobi form of singular weight. We now show that these are the only theta quotients with that property, thereby proving Theorem \ref{MTH:theta-quotients}. 

Let $L$ be an integral positive-definite lattice with bilinear form $(-,-)$ and dual lattice $L'$. 

\begin{definition}\label{def:eutactic star}
Let $\mathbf{s}$ be a family of nonzero elements $\alpha_i, \beta_j\in L'$, $i\in I$, $j\in J$,  where $I$ and $J$ are finite index sets, and $\alpha_i\neq \pm \beta_j$ for $(i,j)\in I\times J$. We call $\mathbf{s}$ a \textit{generalized eutactic star} on $L$ if 
\begin{equation*}
\sum_{i\in I} (\alpha_i, x)^2 - \sum_{j\in J} (\beta_j, x)^2 = (x,x), \quad \text{for all $x\in L$}.
\end{equation*} 
If $J$ is empty, then $\mathbf{s}$ is called a \textit{eutactic star}, as in \cite{GSZ19}.    
\end{definition}

Note that $\mathbf{s}$ is not necessarily a set, as we allow the possibility that $\alpha_i=\alpha_{i'}$ or $\beta_j=\beta_{j'}$ for $i\neq i'$ or $j\neq j'$.  Every eutactic star $\mathbf{s}$ induces an isometric embedding 
$$
\iota_\mathbf{s} : L \to \ZZ^N, \quad x \mapsto \big((\alpha_i,x)\big)_{i\in I}. 
$$
In contrast, any isometric embedding from $L$ to $\ZZ^N$ arises in this way. 

A generalized eutactic star $\mathbf{s}$ on $L$ induces meromorphic Jacobi forms of index $L$ through theta quotients. Let $l\in\ZZ$. Following Gritsenko--Skoruppa--Zagier \cite{GSZ19}, we define the theta quotient 
\begin{equation}\label{eq:theta-quotient}
\vartheta_{\mathbf{s},\,l}(\tau, \mathfrak{z}) := \eta(\tau)^l\frac{\prod_{i\in I}\vartheta(\tau, (\alpha_i, \mathfrak{z}))}{\prod_{j\in J}\vartheta(\tau, (\beta_j, \mathfrak{z}))},
\end{equation}
which is a meromorphic Jacobi form of weight $(l+|I|-|J|)/2$, index $L$, and character $v_\eta^{l+3(|I|-|J|)}$. When $J$ is empty, $\vartheta_{\mathbf{s},\,l}(\tau, \mathfrak{z})$ is called a \emph{theta block}. We identify $\mathbf{s}$ up to independent sign changes, since such changes yield the same theta quotient up to scalar $\pm 1$. 

In general, $\vartheta_{\mathbf{s},\,l}$ has poles on $\HH\times(L\otimes\CC)$. Even when it is holomorphic on $\HH\times(L\otimes\CC)$, it need not be holomorphic at infinity; that is, the condition $2n\geq (\ell,\ell)$ in the Fourier expansion \eqref{eq:Fourier} may fail to hold. By the singular weight bound, if $\vartheta_{\mathbf{s},\,l}$ is a holomorphic Jacobi form, then $l+|I|-|J|\geq \rank(L)$. 

\begin{definition}\label{def:extremal eutactic star}
A generalized eutactic star $\mathbf{s}$ on $L$ is called \emph{extremal} if the associated function 
\begin{equation}\label{eq:singular theta quotient}
\vartheta_\mathbf{s}^*(\tau, \mathfrak{z}) :=\vartheta_{\mathbf{s},\, \rank(L)+|J|-|I|}(\tau, \mathfrak{z})    
\end{equation}
is a holomorphic Jacobi form of singular weight $\frac{1}{2}\rank(L)$ and index $L$ on $\SL_2(\ZZ)$.    
\end{definition}

A generalization of \cite[Proposition 5.2]{GSZ19} shows that $\mathbf{s}$ is extremal if and only if 
$$
\min\Big\{\sum_{i\in I} B\big( (\alpha_i,x) \big) - \sum_{j\in J} B\big( (\beta_j, x) \big) : x\in L\otimes\RR \Big\} \geq \frac{|I|-|J|-\rank(L)}{24},
$$
where $B(x)=\frac{1}{2}( y - \frac{1}{2})^2$ with $y-x\in \ZZ$ and $0\leq y < 1$.

As shown in \cite{Wan25a}, extremal eutactic stars correspond essentially to root systems of finite-dimensional simple Lie algebras $\mathfrak{g}$ or their direct sums. In this case, the associated theta block is the denominator function of the affine Lie algebra $\hat{\mathfrak{g}}$:
\begin{equation}
\vartheta_{\mathfrak{g}}(\tau,\mathfrak{z}):=\eta(\tau)^{\rank(\mathfrak{g})}\prod_{\alpha\in \Delta_{\mathfrak{g}}^+}\frac{\vartheta(\tau,(\alpha,\mathfrak{z}))}{\eta(\tau)}.  
\end{equation}
This is a holomorphic Jacobi form of singular weight $\frac{1}{2}\rank(\mathfrak{g})$, character $v_\eta^{\dim\mathfrak{g}}$, and index $P_{\mathfrak{g}}^\vee(h_\mathfrak{g}^\vee)$ on $\SL_2(\ZZ)$. Here, $\Delta_{\mathfrak{g}}^+$ is the set of positive roots, $h_\mathfrak{g}^\vee$ is the dual Coxeter number, and $P_{\mathfrak{g}}^\vee$ is the coweight lattice (dual to the root lattice); see \cite[\S 2.2]{SWW23} for details. 

We now classify extremal generalized eutactic stars. By \eqref{eq:osp-eutatic star} and Theorem \ref{th:osp-sd-modularity}, the set of positive roots of $\mathfrak{osp}_{1|2r}$ defines an extremal generalized eutactic star (with even roots contributing the $\alpha_i$ and odd roots the $\beta_j$), and the superdenominator function \eqref{eq:superdenominator-osp} of $\widehat{\mathfrak{osp}}_{1|2r}$ recovers the associated theta quotient. 

\begin{lemma}\label{lem:singular Jacobi}
Let $L$ be an even positive-definite lattice, and let $\phi(\tau,\mathfrak{z})$ be a nonzero holomorphic Jacobi form of singular weight and index $L$. Fix $\ell\in L'\backslash\{0\}$ and $t\in\ZZ$, and set
$$
\hat{\phi}(Z):=\phi(\tau,\mathfrak{z})\cdot e^{2\pi i \omega}, \quad Z=(\omega,\mathfrak{z},\tau)\in \HH(L),
$$
where $\HH(L)$ is the tube domain \eqref{eq:tube domain}. 
If $\hat{\phi}$ vanishes along the hyperplane
$$
(t,\ell,0)^\perp:= \{ (\omega,\mathfrak{z},\tau)\in \HH(L):\; (\ell,\mathfrak{z})=t\tau \},
$$
then $(t,\ell,0)^\perp$ has multiplicity one in the divisor of $\hat{\phi}$, and 
$$
\hat{\phi}\big(\sigma_{(t,\ell,0)}(Z)\big) = - \hat{\phi}(Z), \quad \text{for any $Z\in \HH(L)$},
$$
where $\sigma_{(t,\ell,0)}$ is a reflection defined by
$$
\sigma_{(t,\ell,0)}(\omega,\mathfrak{z},\tau) = \left( \omega + \frac{2t(t\tau-(\ell,\mathfrak{z}))}{(\ell,\ell)},\, \mathfrak{z}+\frac{2(t\tau-(\ell,\mathfrak{z}))}{(\ell,\ell)}\ell, \, \tau \right).
$$
\end{lemma}

\begin{proof}
Let $U\cong \begin{psmallmatrix}
0 & -1 \\ -1 & 0    
\end{psmallmatrix}$ and $M=U\oplus L$. Write $\lambda\in M'=U\oplus L'$ as $(n,v,m)$ for $n, m\in \ZZ$ and $v\in L'$, with $(\lambda,\lambda)=(v,v)-2nm$. 
The Fourier expansion of $\hat{\phi}(Z)$ is supported only on the norm-zero vectors of $M\otimes\QQ$. Applying the Laplace operator (see Section \ref{subsec:orthogonal modular forms}) to $\hat{\phi}$, we find
$$
\mathbf{\Delta}(\hat{\phi})=0, \quad \text{where} \quad \mathbf{\Delta} = - \frac{\partial}{\partial \tau} \frac{\partial}{\partial \omega} + \frac{1}{2} \left(\frac{\partial}{\partial\mathfrak{z}},\frac{\partial}{\partial\mathfrak{z}}\right). 
$$
Let $K$ denote the orthogonal complement of $(t,\ell,0)$ in $M$.
We write $Z=z\cdot (t,\ell,0)+Z'$ for $z\in \CC$ and $Z' \in K\otimes\CC$, and expand $\hat{\phi}(Z)$ into the Taylor series around $z=0$ as
$$
\hat{\phi}(Z) = F_d(Z')z^d+O(z^{d+1}), \quad F_d(Z')\not\equiv 0.
$$
Here $d\geq 1$, because $\phi(\tau,\mathfrak{z})=0$ whenever $(\ell,\mathfrak{z})=t\tau$ (i.e., $\hat{\phi}(Z)=0$ if $z=0$).
We now consider $\mathbf{\Delta}$ to be $\frac{1}{2(\ell,\ell)}\frac{\partial^2}{\partial z^2}+\frac{1}{2}\big( \frac{\partial}{\partial Z'},\frac{\partial}{\partial Z'} \big)$. 
From $\mathbf{\Delta}(\hat{\phi})=0$ we deduce $d=1$. 
This leads to the Taylor expansion
$$
\hat{\phi}(Z) = F_1(Z') z + O(z^2).
$$
We apply $\mathbf{\Delta}$ to the function 
$$
\tilde{\phi}(Z):=\hat{\phi}(\sigma_{(t,\ell,0)}(Z))+ \hat{\phi}(Z) = \hat{\phi}(-z,Z') + \hat{\phi}(z,Z') = O(z^2),
$$
then from \eqref{eq:singular-laplace} we derive $\mathbf{\Delta}(\tilde{\phi})=0$, which forces $\hat{\phi}(\sigma_{(t,\ell,0)}(Z))=-\hat{\phi}(Z)$.    
\end{proof}

\begin{theorem}\label{th:singular theta quotient}
Let $\mathbf{s}$ be an extremal generalized eutactic star. Then $\mathbf{s}$ is identifiable with the positive root system of a semi-simple finite-dimensional Lie superalgebra $\mathfrak{g}=\oplus_{t=1}^m\mathfrak{g}_t$, whose simple factors $\mathfrak{g}_t$ are either simple Lie algebras or simple Lie superalgebras of type $\mathfrak{osp}_{1|2r}$. More explicitly, the $\alpha_i$ are the positive even roots and the $\beta_j$ are the positive odd roots. Moreover, the theta quotient \eqref{eq:singular theta quotient} associated with $\mathbf{s}$ is the superdenominator function of the affine Lie superalgebra $\hat{\mathfrak{g}}$. 
\end{theorem}

\begin{proof}
Without loss of generality, we assume that $L$ is even; otherwise, we replace $L$ with the even lattice $\{ 2x:\, x\in L\}$ and identify $\mathbf{s}$ in this lattice. Let $\mathbf{s}_0$ and $\mathbf{s}_{1}$ denote the subfamilies of $\alpha_i$ and $\beta_j$, respectively. Note that $\mathbf{s}_0$ and $\mathbf{s}_{1}$ are disjoint. Recall that $\vartheta(\tau,z)=0$ if and only if $z\in \ZZ\tau + \ZZ$ and the zeros are simple. Therefore, for any nonzero $\gamma \in L'$, the zeros of $\vartheta(\tau,(\gamma,\mathfrak{z}))$ are simple and characterized by $(\gamma,\mathfrak{z})\in \ZZ\tau + \ZZ$. Since $\vartheta_{\mathbf{s}}^*$ is holomorphic on $\HH\times(L\otimes\CC)$, for any $j\in J$, there exists $i\in I$ such that $\alpha_i=\pm\lambda\beta_j$ for an integer $\lambda>1$. For any vector $\beta'$ in $\mathbf{s}_{1}$, there exists $\beta\in L'$ such that every rational scalar $\lambda' \beta'$ in $\mathbf{s}$ can be expressed as an integral scalar $m\beta$. It follows that $\vartheta_{\mathbf{s}}^*$ can be expressed as
\begin{equation}
\vartheta_{\mathbf{s}}^*(\tau,\mathfrak{z})=\eta(\tau)^{l}\cdot \left( \prod_{\alpha} \frac{\vartheta(\tau,(\alpha,\mathfrak{z}))}{\eta(\tau)}\right)\cdot \left( \prod_{\beta} \frac{\prod_{i\in I_{\beta}}\vartheta(\tau,(s_i\beta,\mathfrak{z}))}{\prod_{j\in J_{\beta}}\vartheta(\tau,(t_j\beta,\mathfrak{z}))}\right),
\end{equation}
where we adopt the following notation: 
\begin{enumerate}
\item $\alpha$ runs over a certain subfamily $\tilde{\mathbf{s}}_0$ of $\mathbf{s}_0$;
\item $\beta$ runs over a certain subset $\tilde{\mathbf{s}}_1$ of $L'$; 
\item  for any distinct $\beta,\beta'\in\tilde{\mathbf{s}}_1$, $\beta\neq \lambda\beta'$ for any $\lambda\in\QQ$; 
\item for any $\alpha\in\tilde{\mathbf{s}}_0$ and $\beta\in\tilde{\mathbf{s}}_1$, $\alpha\neq \lambda\beta$ for any $\lambda\in\QQ$;
\item $I_\beta$ and $J_\beta$ are finite index sets;
\item $s_i$ and $t_j$ are positive integers depending on $\beta$; 
\item $s_i\neq t_j$ for any $(i,j)\in I_\beta\times J_\beta$; 
\item for any $j\in J_\beta$, there exists $i\in I_\beta$ such that $t_j|s_i$.
\end{enumerate}
By Lemma \ref{lem:singular Jacobi}, $\vartheta_{\mathbf{s}}^*$ vanishes along $(\alpha,\mathfrak{z})=0$ with multiplicity one. Therefore, the multiplicity of $\alpha$ in the subfamily $\tilde{\mathbf{s}}_0$ is one, and $\lambda\alpha\not\in\tilde{\mathbf{s}}_0$ for any rational $\lambda\neq 1$. In addition, $\vartheta_{\mathbf{s}}^*$ vanishes along $(\beta,\mathfrak{z})=0$ with multiplicity $|I_\beta|-|J_\beta|$, so Lemma \ref{lem:singular Jacobi} yields $|J_\beta|\leq |I_\beta|\leq |J_\beta|+1$. We set 
$$
s=\max\{s_i:i\in I_\beta\}.
$$
Then $t_j<s$ for any $j\in J_\beta$. Similarly, $\vartheta_{\mathbf{s}}^*$ vanishes along $(s\beta,\mathfrak{z})=\tau$ with multiplicity one. Therefore, there is a unique $i\in I_\beta$ with $s=s_i$. Furthermore, Lemma \ref{lem:singular Jacobi} shows that $\hat{\vartheta}_{\mathbf{s}}^*$ is anti-invariant under reflection $\sigma_{(1,s\beta,0)}$, that is,
$$
\hat{\vartheta}_{\mathbf{s}}^*(\omega,\mathfrak{z},\tau)=-\hat{\vartheta}_{\mathbf{s}}^*\Big(\omega + \frac{2(\tau-s(\beta,\mathfrak{z}))}{s^2(\beta,\beta)},\, \mathfrak{z}+\frac{2(\tau-s(\beta,\mathfrak{z}))}{s(\beta,\beta)}\beta,\, \tau \Big).
$$
Therefore, the two functions
\begin{align*}
&\prod_{\alpha}\vartheta\big(\tau,(\alpha,\mathfrak{z})\big)\cdot\prod_{\beta'\neq \beta}\prod_{i\in I_{\beta'}}\vartheta\big(\tau,s'_i(\beta',\mathfrak{z})\big)\prod_{j\in J_{\beta'}}\vartheta\left(\tau,t'_j\big(\sigma_\beta(\beta'),\mathfrak{z}\big)+\frac{2t'_j(\beta,\beta')}{s(\beta,\beta)}\tau\right)\\
&\quad \cdot \prod_{i\in I_{\beta}}\vartheta\big(\tau,s_i(\beta,\mathfrak{z})\big)\prod_{j\in J_{\beta}}\vartheta\left(\tau,t_j(\beta,\mathfrak{z})-\frac{2t_j}{s}\tau\right) \end{align*}
and 
\begin{align*}
&\prod_{\alpha}\vartheta\left(\tau,\big(\sigma_\beta(\alpha),\mathfrak{z}\big)+\frac{2(\alpha,\beta)}{s(\beta,\beta)}\tau\right)\cdot\prod_{\beta'\neq \beta}\prod_{i\in I_{\beta'}}\vartheta\left(\tau,s'_i\big(\sigma_\beta(\beta'),\mathfrak{z}\big)+\frac{2s'_i(\beta',\beta)}{s(\beta,\beta)}\tau\right)\prod_{j\in J_{\beta'}}\vartheta\big(\tau,t'_j(\beta',\mathfrak{z})\big)\\
&\quad \cdot \prod_{i\in I_{\beta}}\vartheta\left(\tau,s_i(\beta,\mathfrak{z})-\frac{2s_i}{s}\tau\right)\prod_{j\in J_{\beta}}\vartheta\big(\tau,t_j(\beta,\mathfrak{z})\big)    
\end{align*}
have the same zero divisors (counting multiplicities). Recall that the zero divisor of $\vartheta\big(\tau,(\gamma,\mathfrak{z})+a\tau\big)$ is given by $(\gamma,\mathfrak{z})\in \ZZ\tau + \ZZ - a\tau$, where $\gamma\in L\otimes\QQ$ and $a\in\QQ$. Note that $\sigma_\beta(\alpha)$ and $\sigma_\beta(\beta')$ are never rational scalars of $\beta$. By comparing zero divisors, we find that $2s_i/s$ and $2t_j/s$ are always integral. Then $t_j<s$ implies that $s$ is even and $t_j=s/2$ for $j\in J_\beta$, and thus $s_i=s$ for $i\in I_\beta$. Therefore, $|I_\beta|=|J_\beta|=1$. Without loss of generality, we assume $\beta\in \mathbf{s}_1$ and $s=2$. Then we have $2\beta\in \mathbf{s}_0$.  Furthermore, by comparing divisors, up to sign, the set $\tilde{\mathbf{s}}_0$ of those $\alpha$ and the set $\tilde{\mathbf{s}}_1$ of those $\beta'$ are all closed under the action of $\sigma_\beta$, and 
$$
\frac{(\alpha,\beta)}{(\beta,\beta)}\in \ZZ \quad \text{and} \quad \frac{(\beta,\beta')}{(\beta,\beta)}\in \ZZ. 
$$

Let $\alpha\in \tilde{\mathbf{s}}_0$. Then $\vartheta_{\mathbf{s}}^*$ vanishes along $(\alpha,\mathfrak{z})=\tau$ with multiplicity one, and thus $\hat{\vartheta}_{\mathbf{s}}^*$ is anti-invariant under reflection $\sigma_{(1,\alpha,0)}$. Similarly, by comparing the divisors, we conclude that (up to sign) both sets $\tilde{\mathbf{s}}_0$ and $\tilde{\mathbf{s}}_1$ are closed under the action of $\sigma_\alpha$, and 
$$
\frac{2(\alpha,\alpha')}{(\alpha,\alpha)}\in \ZZ \quad \text{and} \quad \frac{2(\alpha,\beta)}{(\alpha,\alpha)}\in \ZZ 
$$
for any $\alpha'\in \tilde{\mathbf{s}}_0$ and $\beta\in \tilde{\mathbf{s}}_1$. Therefore, both $\tilde{\mathbf{s}}_0\cup\{\beta\in \tilde{\mathbf{s}}_1\}$ and $\mathbf{s}_0=\tilde{\mathbf{s}}_0\cup \{ 2\beta : \beta\in \tilde{\mathbf{s}}_1 \}$ define (rescaled) reduced root systems. The desired theorem then follows from the classical classification of irreducible root systems. 
\end{proof}

\begin{remark}\label{rem:KMLie-super}
Recall from Section \ref{subsec:BKM} that BKM superalgebras without imaginary simple roots are called Kac--Moody Lie superalgebras. The finite-dimensional simple ones are precisely simple Lie algebras and Lie superalgebras of type $\mathfrak{osp}_{1|2r}$. By Theorem \ref{th:singular theta quotient}, extremal generalized eutactic stars are in bijection with finite-dimensional semi-simple Kac–Moody Lie superalgebras. Equivalently, theta quotients that define holomorphic Jacobi forms of singular weight correspond bijectively to the superdenominator functions of affine Lie superalgebras extending finite-dimensional semi-simple Kac–Moody Lie superalgebras. 
\end{remark}

\section{Reflectivity of singular automorphic products}\label{sec:reflectivity}
In this section, we demonstrate that each singular automorphic product on $2U\oplus L$ can be viewed as a reflective automorphic product on $2U\oplus \hat{L}$ for a suitable even overlattice $\hat{L}$ of $L$. We further address the uniqueness of $\hat{L}$ and determine its canonical choice, thus proving Theorem \ref{MTH:reflectivity}. 

The first step is to establish the existence of $\hat{L}$. 

\begin{lemma}\label{lem:reflective-extension}
Let $F$ be a singular automorphic product on $M=U_1\oplus U\oplus L$, and let $\phi$ denote its Jacobi form input. Then there exists an even overlattice $\tilde{L}$ of $L$ such that $F$ becomes a reflective automorphic product on $U_1\oplus U\oplus \tilde{L}$. In particular, $\phi$ becomes a Jacobi form of index $\tilde{L}$. 
\end{lemma}

\begin{proof}
We first recall the Eichler criterion following \cite[\S 3]{GHS09}, using the coordinates fixed in Notation \ref{notation2}. Let $U_1=\ZZ e_1+\ZZ f_1$, $U=\ZZ e+\ZZ f$, and $L_1:=U\oplus L$. For $a,c\in M\otimes\QQ$ with $(c,c)=0$ and $(a,c)=0$, the associated Eichler transvection is an element of $\Orth(M\otimes\QQ)$ defined by
$$
t(c,a):\; v \mapsto v-(a,v)c+(c,v)a-\frac{1}{2}(a,a)(c,v)c. 
$$ 
Let $E_U(L_1)$ denote the group generated by $t(e_1,x)$ and $t(f_1,x)$ for all $x\in L_1$. Note that $E_U(L_1)$ is a subgroup of $\widetilde{\mathrm{SO}}^+(M)$. Let $\alpha$ and $\beta$ be two primitive vectors in $M'$. The Eichler criterion states that if $(\alpha,\alpha)=(\beta,\beta)$ and $\alpha-\beta\in M$ then there exists $g\in E_U(L_1)$ such that $g(\alpha)=\beta$. 

Let $\gamma\in M'$ be a primitive positive-norm vector. Assume $F$ vanishes along $\gamma^\perp$. We find an integer $n$ and a vector $\ell \in L'$ such that $(\ell,\ell)-2n=(\gamma,\gamma)$ and $\gamma-\ell\in M$. According to the Eichler criterion, there exists $g\in E_U(L_1)$ such that $g(\gamma)=ne+\ell+f$. The multiplicity of $\gamma^\perp$ in the zero divisor of $F$ is given by $\sum_{k=1}^\infty f(k^2n,k\ell)$, a sum of certain Fourier coefficients of $\phi$. We know from \cite[Corollary 3.3]{WW23} that $(\gamma,\gamma)=2/d$ for a positive integer $d$. 

Let $\mathfrak{R}$ denote the subset of all vectors of type  $\alpha=ne+\ell+f$ with $n\in\ZZ$, $\ell\in L'$, and $2n<(\ell,\ell)$, for which $F$ vanishes along $\alpha^\perp$. For any $\alpha\in \mathfrak{R}$, we define $d_\alpha:=2/(\alpha,\alpha)$, which is a positive integer by the discussion above. When $d_\alpha \alpha \in L_1$, that is, $d_\alpha \ell \in L$, the reflection $\sigma_\alpha$ fixes the lattice $M$, and thus the divisor $\alpha^\perp$ is reflective. If $d_\alpha \alpha \in L_1$ for any $\alpha\in \mathfrak{R}$, then $F$ is reflective on $M=U_1\oplus U\oplus L$ and we can take $\tilde{L}=L$. 

We now assume that there exists $\alpha=ne+f+\ell \in \mathfrak{R}$ with $d_\alpha \alpha \not\in L_1$. Then $K:=L+\ZZ \cdot d_\alpha \ell$ is an even positive-definite lattice and $\mathrm{disc}(K)<\mathrm{disc}(L)$. Let $K_1:=U\oplus K$. By \cite[(6)-(11)]{GHS09}, $E_U(K_1)$ is generated by $E_U(L_1)$, $t(e_1,d_\alpha \alpha)$, and $t(f_1,d_\alpha \alpha)$. We find that
$$
t(e_1,d_\alpha \alpha) = \sigma_{\alpha} \sigma_{\alpha+e_1} \quad \text{and}  \quad t(f_1,d_\alpha \alpha) = \sigma_{\alpha} \sigma_{\alpha+f_1}.
$$
By Theorem \ref{th:singular-product-reflective}, $F$ is anti-invariant with respect to reflections $\sigma_{\alpha}$, $\sigma_{\alpha+e_1}$, and $\sigma_{\alpha+f_1}$. Thus, $F$ is invariant under the actions of $t(e_1,d_\alpha \alpha)$ and $t(f_1,d_\alpha \alpha)$. Since $F$ is modular for $E_U(L_1)<\widetilde{\Orth}^+(M)$, it is modular for the larger group $E_U(K_1)$. By the Eichler criterion, for any $g\in \widetilde{\Orth}^+(U_1\oplus K_1)$, both $F$ and $F|g$ have the same zero divisors, because $g(x)-x\in U_1\oplus K_1$ for any $x\in U_1\oplus K_1'$. According to Koecher's principle, $F|g=\lambda F$ for $\lambda \in \CC^\times$. Thus, $F$ is modular for $\widetilde{\Orth}^+(U_1\oplus K_1)$, and therefore gives rise to an automorphic product on $U_1\oplus U\oplus K$. We then iterate the above process, replacing $L$ with $K$, and obtain the desired lattice $\tilde{L}$ in finitely many steps, since the discriminant of the lattice strictly decreases at each stage.
\end{proof}

We then determine the maximal underlying lattice of an automorphic product. 

\begin{lemma}\label{lem:canonical lattice}
Let $U\oplus K$ be an even lattice of signature $(l,2)$ with $l\geq 3$, and let 
$$
f=\sum_{x\in K'/K}\sum_{\substack{n\in \ZZ-x^2/2\\ n\gg-\infty}} c_{x}(n) q^n e_x = \sum_{x\in K'/K} f_x(\tau) e_x
$$
be a weakly holomorphic modular form of weight $k\in \frac{1}{2}\ZZ$ for the Weil representation of $\mathrm{Mp}_2(\ZZ)$ attached to $K'/K$. 
Suppose $f_0 \neq 0$. Define a set
$$
\mathcal{K}:=\{ v \in K' : \, f_v\neq 0 \} \supset K
$$
and a lattice
$$
\widetilde{K}:=\{ y \in K' : \, (v,y)\in\ZZ, \; \text{for any $v\in \mathcal{K}$}  \} \supset K. 
$$
Then $\widetilde{K}$ is an even lattice, its dual is generated over $\ZZ$ by $\mathcal{K}$, and $f$ becomes a weakly holomorphic modular form for the Weil representation attached to $\widetilde{K}'/\widetilde{K}$. 
\end{lemma}

\begin{proof}
Let $\rho_{K}$ denote the Weil representation of $\mathrm{Mp}_2(\ZZ)$ attached to $K'/K$. Recall that
$$
\rho_K(S)f = c\cdot \sum_{y\in K'/K} \sum_{x\in K'/K} f_x \cdot e^{-2\pi i (x,y)} \cdot e_y,
$$
where $c$ is a nonzero constant. Let $y\in \widetilde{K}$. If $f_x\neq 0$, then $x\in \mathcal{K}$ and $(x,y)\in\ZZ$. Therefore,
\begin{equation}\label{eq:S-action}
 \tau^{-k} f_y(-1/\tau)=c\sum_{x\in K'/K} f_x(\tau)=\tau^{-k} f_0(-1/\tau).  
\end{equation}
It follows that $f_y=f_0\neq 0$ and thus $y\in \mathcal{K}$. 
In particular, the Fourier series of $f_y$ is supported on integral exponents, and therefore $(y,y) \in 2\ZZ$. It follows that $\widetilde{K}$ is an even lattice and its dual is generated by $\mathcal{K}$ over $\ZZ$. For any $y_1,y_2\in K'$ with $y_1-y_2\in \widetilde{K}$, we have 
$$
e^{-2\pi i (x,y_1)} = e^{-2\pi i (x,y_2)} \quad \text{for any $x\in K'$ such that $f_x\neq 0$},
$$
which forces $ \tau^{-k} f_{y_1}(-1/\tau)=\tau^{-k} f_{y_2}(-1/\tau)$, and thus $f_{y_1}=f_{y_2}$. Consequently, $f$ gives rise to a weakly holomorphic modular form for the Weil representation attached to $\widetilde{K}'/\widetilde{K}$. 
\end{proof}

The following result implies the uniqueness of the overlattice $\tilde{L}$ from Lemma \ref{lem:reflective-extension} whenever the singular automorphic product is anti-symmetric. In the symmetric case, however, $\tilde{L}$ is generally not unique: for instance, $2U\oplus \ZZ(8)$ has a reflective automorphic product of singular weight that is also reflective on $2U\oplus \ZZ(32)$. 

\begin{lemma}\label{lem:canonical-lattice-unique}
Let $F$ be a reflective automorphic product on $U\oplus K$ whose input has Fourier expansion $f(\tau)=\sum_{x\in K'/K} f_x(\tau) e_x$ and satisfies $f_0=cq^{-1}+O(1)$ with $c\neq 0$. Then $K'$ is generated over $\ZZ$ by $x\in K'$ with $f_x\neq 0$. Moreover, $f_x\neq 0$ for any $x\in K'/K$. 
\end{lemma}

\begin{proof} 
As in the proof of Lemma \ref{lem:canonical lattice}, for any $y\in \widetilde{K}$, we have $f_y=f_0=cq^{-1}+O(1)$. Since $F$ is reflective on $U\oplus K$, it follows that $y\in K$ and thus $\widetilde{K}=K$. The last claim follows from the facts that $\tau^{-k}f_x(-1/\tau)$ is a $\CC$-linear combination of components of $f$ and that $f_y=O(q^{-1/d})$ with integral $d>1$ for $y\neq 0$ in $K'/K$.  
\end{proof}

We finally describe the canonical choice of $\tilde{L}$ in Lemma \ref{lem:reflective-extension}. 

\begin{proposition}\label{prop:reflective-canonical-extension}
Let $F$ be a singular automorphic product on $2U\oplus L$ with Jacobi form input $\phi\in J_{0,L}^!$ and vector-valued form input
$f=\sum_{x\in L'/L} f_x(\tau) e_x$. Define 
$$
\mathcal{L}:=\{ \ell \in L' : \, f_\ell \neq 0  \} \quad \text{and} \quad \hat{L}:=\{ v \in L' : (v, \ell)\in \ZZ, \; \text{for any $\ell \in \mathcal{L}$} \}.
$$
Then $\hat{L}$ is an even overlattice of $L$, its dual $\hat{L}'$ is generated by $\mathcal{L}$ over $\ZZ$, and  $\phi$ gives a Jacobi form of index $\hat{L}$. Moreover, $F$ becomes a reflective automorphic product on $2U\oplus \hat{L}$.   \end{proposition}

\begin{proof}
The first part follows from Lemma \ref{lem:canonical lattice}. 
Suppose that $F$ is not reflective on $2U\oplus \hat{L}$. By Lemma \ref{lem:reflective-extension}, there exists an even overlattice $K$ of $\hat{L}$ such that $K\neq \hat{L}$ and $\phi$ is a Jacobi form of index $K$. The dual lattice $\hat{L}'$ is generated by all $x\in L'$ with $f_x\neq 0$, so it is contained in $K'$, which contradicts $K'\subsetneq \hat{L}'$.  
\end{proof}

\begin{remark}\label{rk:why-canonical}
As we will see in the next section, the above singular automorphic product gives the superdenominator of a BKM superalgebra $\mathcal{G}$ with root lattice $U\oplus \hat{L}'$. Thus, $\hat{L}$ represents the \textit{canonical} choice among the even overlattices $\tilde{L}$ of $L$ for which $F$ is reflective on $2U\oplus \tilde{L}$. 
\end{remark}

\section{BKM superalgebras from singular automorphic products} \label{sec:BKM from singular}
This section is dedicated to proving Theorem \ref{MTH:relation to BKM} and Corollary \ref{Cor:relation to regular}. We establish that each singular automorphic product on $2U\oplus L$ naturally gives rise to a BKM superalgebra, and we subsequently derive several conditions satisfied by the singular product and its canonical lattice. 

\subsection{Fourier expansions of singular automorphic products}\label{subsec:Fourier-expansion}
We now compute the Fourier expansion of a singular automorphic product at a $0$-dimensional cusp, which encodes the imaginary simple roots of the associated BKM superalgebra. Although this computation is standard (cf. \cite[Example 13.7]{Bor98}, \cite[Section 9]{Sch04}, and \cite[Section 5]{DHS15}), we include the details, as it is essential for our main results and this is the first time we treat the case where the input may contain negative Fourier coefficients in its principal part. 

Let $F$ be a singular automorphic product on $M=U_1\oplus U\oplus L$ with Jacobi form input 
$$
\phi(\tau,\mathfrak{z}) = \sum_{n\in\ZZ} \sum_{\ell \in L'} f(n,\ell) q^n \zeta^\ell.
$$
Assume that $L$ is canonical (see Remark \ref{rk:why-canonical}). We fix $U_1=\ZZ e_1+\ZZ f_1$ and $K=U\oplus L$, and use the coordinates in Notation \ref{notation2}. The input $\phi$ may have negative singular Fourier coefficients, i.e.
$$
f(n,\ell)q^n\zeta^\ell \quad \text{with} \quad (\ell,\ell)-2n>0 \quad \text{and} \quad f(n,\ell)<0.
$$
Since $F$ is reflective and has only simple zeros (cf. Theorem \ref{th:singular-product-reflective}), by Lemma \ref{lem:principal-part-reflective}, $f(n,\ell)=-1$, $f(4n,2\ell)=1$, and there exists a positive integer $d$ such that $(\ell,\ell)/2-n=1/(4d)$ and $2d\ell \in L$. 

For $\alpha=(n,\ell,1)$ with $2n<(\ell,\ell)$ and $f(n,\ell)<0$, the reflection $\sigma_\alpha$ preserves the lattice $M$, but $F$ does not vanish along $\alpha^\perp$. However, we have 
$$
F(\sigma_\alpha(\mathcal{Z}))=F(\mathcal{Z}).
$$
This claim is confirmed as follows. Assume that $F$ vanishes along $\lambda^\perp$ with primitive $\lambda\in K'=U\oplus L'$. Since $F$ vanishes along $(2\alpha+e_1)^\perp$, Theorem \ref{th:singular-product-reflective} shows that it is anti-invariant under the associated reflection. Therefore, $F$ also vanishes along the hyperplane perpendicular to the primitive vector
$$
\sigma_{2\alpha+e_1}(\lambda) = \sigma_\alpha(\lambda) - \frac{(\alpha,\lambda)}{(\alpha,\alpha)}e_1 \in U_1\oplus K'.
$$
Notice that $\sigma_\alpha(\lambda)\in K'$ is primitive. 
By the Eichler criterion, $F$ vanishes along $\sigma_\alpha(\lambda)^\perp$. Therefore, $\sigma_\alpha$ preserves the zeros  of $F$, which forces $F(\sigma_\alpha(\mathcal{Z}))=F(\mathcal{Z})$.   
\vspace{2mm}

We turn to the Fourier expansion of $F$ at $U_1$. Let $W$ be the Weyl group of $F$, that is, the subgroup of $\Orth(K)$ generated by the reflections associated with $(n,\ell,m)\in U\oplus L'$ with $2nm<\ell^2$ and $f(nm,\ell)=1$. Fix a Weyl chamber $\mathcal{C}$ of $F$ and denote its closure by $\overline{\mathcal{C}}$. Let $\rho$ be the corresponding Weyl vector. Then $\rho\in \overline{\mathcal{C}}$. As a function on the associated tube domain, $F$ satisfies
\begin{equation}\label{eq:invariance of F}
F(\sigma(Z)) = \varepsilon'(\sigma) F(Z), \quad \text{for all $\sigma\in W$}. 
\end{equation}
Regarding $\sigma$ as a product of reflections, let $l'(\sigma)$ denote the number of reflections whose associated vectors are not of type $(2n,2\ell,2m)$ with $f(nm,\ell)=-1$. Then $\varepsilon'(\sigma)=(-1)^{l'(\sigma)}$. 

Since $W$ acts simply transitively on the Weyl chambers of the negative cone of $K\otimes\RR$, the Fourier expansion of $F$ takes the form 
$$
F(Z)=\sum_{\sigma\in W} \varepsilon'(\sigma) \sum_{\substack{\lambda \in K', \, \lambda + \rho \in \overline{\mathcal{C}} \\ \lambda=0 \; \text{or} \; (\lambda, \mathcal{C})<0 }} c(\lambda) \mathbf{e}\Big(\big(\sigma(\lambda+\rho), Z\big)\Big),
$$
where $c(0)=1$ and $\mathbf{e}(t)=e^{2\pi it}$. Since $F$ has singular weight, $c(\lambda)=0$ whenever $(\lambda+\rho, \lambda+\rho)\neq 0$.
Note that $(\lambda, \mathcal{C})<0$ and $\rho \in \overline{\mathcal{C}}$ together imply $(\lambda, \rho) \leq 0$. On the other hand, if $(\lambda+\rho, \lambda+\rho)= 0$, then from $(\lambda, \lambda+\rho)\leq 0$ we derive
$$
2(\lambda, \rho) = (\lambda+\rho, \lambda+\rho) - (\lambda, \lambda) - (\rho,\rho) = -\lambda^2
$$
and 
$$
(\lambda, \rho) = (\lambda, \lambda+\rho) - (\lambda, \lambda) \leq -\lambda^2.
$$
Thus, $-\lambda^2/2\leq -\lambda^2$, that is, $\lambda^2\leq 0$, so $(\lambda, \rho)\geq 0$. Therefore, $(\lambda,\rho)=0$. Write $\lambda$ and $\rho$ in coordinates as $(x_1,x_2,...,x_l,x_0) \in \RR^{l,1}$ and $(y_1,y_2, ...,y_l,y_0)\in \RR^{l,1}$, respectively, such that 
\begin{align*}
x_1^2+x_2^2+...+x_l^2-x_0^2&=(\lambda, \lambda)\leq 0,\\
y_1^2+y_2^2+...+y_l^2-y_0^2&=(\rho, \rho)=0,\\
x_1y_1+x_2y_2+...+x_ly_l-x_0y_0&=(\lambda,\rho)=0.
\end{align*}
By the Cauchy--Schwarz inequality, we have 
$$
x_0^2y_0^2=\Big(\sum_{j=1}^l x_jy_j\Big)^2 \leq \sum_{j=1}^l x_j^2 \sum_{j=1}^l y_j^2 \leq x_0^2y_0^2, 
$$
from which it follows that $\lambda=c\rho \in K'$ for $c\in \QQ$. Note that $(v, \mathcal{C})<0$ for $v\in \mathcal{C}$. From $(\lambda, \mathcal{C})<0$ and $\rho\in \overline{\mathcal{C}}$, we deduce $c\geq 0$.

Write $\lambda=c\rho=\sum \lambda_j$ with $\lambda_j\in K'$ and $(\lambda_j, \mathcal{C})<0$.  Suppose that, in the product expansion of $F$ at $U_1$, these $\lambda_j$ contribute to $c(\lambda)$. Then $(\lambda_j, \rho)\leq 0$ and $\sum (\lambda_j, \rho)=(c\rho, \rho)=0$ yield $(\lambda_j, \rho)=0$. 

If $(\lambda_j,\lambda_j)\leq 0$, then a similar argument as before shows that $\lambda_j$ is a positive rational multiple of $\rho$. Otherwise, suppose $(\lambda_j, \lambda_j)>0$ for some $j$. Then the reflection $\sigma_{\lambda_j}$ lies in the Weyl group $W$.  Since $\sigma_{\lambda_j}$ maps $\mathcal{C}$ to a new Weyl chamber, it maps $\rho$ to the corresponding new Weyl vector, i.e., $\sigma_{\lambda_j}(\rho)\neq \rho$ (cf. \eqref{eq:relation-simple roots}). It follows that $(\lambda_j, \rho)\neq 0$, leading to a contradiction. Therefore, every $\lambda_j$ is a positive rational multiple of $\rho$ and $\lambda_j\in K'$. 

Let $t$ be a positive rational number such that $t\rho$ is a primitive vector of $K'$. Then every $\lambda_j$ is a positive integral multiple of $t\rho$. Recall that $\rho=(-A,\vec{B},-C)$ with 
$$
A=\frac{1}{24}\sum_{\ell\in L'}f(0,\ell), \quad \vec{B}=\frac{1}{2}\sum_{\ell>0}f(0,\ell)\ell, \quad C=\frac{1}{\rank(L)}\sum_{\ell>0}f(0,\ell)(\ell,\ell). 
$$
We then obtain the following expression:
\begin{equation}\label{eq:Singular-Fourier}
\begin{split}
F(Z)&=\sum_{\sigma\in W} \varepsilon'(\sigma) \; \mathbf{e}\big((\sigma(\rho), Z)\big)\prod_{n=1}^\infty \Big(1-\mathbf{e}\big((\sigma(nt\rho), Z)\big)\Big)^{f(n^2t^2AC,\, nt\vec{B})}\\ 
&=\sum_{\sigma\in W} \varepsilon'(\sigma)\, \sigma\left( \mathbf{e}\big((\rho, Z)\big)\prod_{n=1}^\infty \Big(1-\mathbf{e}\big(nt(\rho, Z)\big)\Big)^{f(n^2t^2AC,\, nt\vec{B})} \right)\\
&=: \sum_{\sigma\in W} \varepsilon'(\sigma)\, \sigma\Big( \Delta_F\big( (\rho, Z) \big) \Big).
\end{split}   
\end{equation}
By Remark \ref{rem:singular}, the value of $f(n^2t^2AC,\, nt\vec{B})$ depends only on the coset of $nt\vec{B}$ in $L'/L$, and 
$$
f(n^2t^2AC, nt\vec{B})=f(n^2t^2AC, -nt\vec{B})=f(m^2t^2AC, mt\vec{B})
$$
if $(n\pm m)t\vec{B} \in L$. The function $\Delta_F$ turns out to be an eta quotient; see Section \ref{sec:denominator-modularity}.

\subsection{Proof of Theorem \ref{MTH:relation to BKM}}\label{subsec:relation to BKM}
Let $F$ be a singular automorphic product on $2U\oplus L$ with $L$ canonical. 
As in the previous subsection, we consider the Fourier expansion of $F$ at the $0$-dimensional cusp of type $U$. Let $\mathcal{C}$ be a Weyl chamber related to the product expansion of $F$ at $U$. Recall that the corresponding Weyl vector is $\rho=(-A, \vec{B}, -C)$. We set
$$
\mathcal{S}:=\{ (n,\ell,m)\in U\oplus L' : \, f(nm,\ell)\neq 0, \, 2nm<(\ell,\ell) \}.
$$
The Weyl group $W$ of $F$ is generated by reflections $\sigma_\alpha$ for $\alpha\in \mathcal{S}$. The reflectivity of $F$ on $2U\oplus L$, together with \eqref{eq:invariance of F}, implies that the set of roots of the hyperbolic reflection group $W$ is precisely $\mathcal{S}$. All these roots are real. By an argument similar to Property (5) of \cite[Section 12]{Bor95} (a typo should be corrected), a (real) root $\alpha=(n,\ell,m)\in \mathcal{S}$ is simple if and only if 
\begin{equation}\label{eq:relation-simple roots}
(\rho,\alpha)=\frac{(\alpha,\alpha)}{2}\sum_{d=1}^\infty d\cdot f(d^2nm,d\ell) = \frac{(\alpha,\alpha)}{2},
\end{equation}
where the last equality is justified as follows: if $f(nm,\ell)=1$, then $f(d^2nm,d\ell)=0$ for any $d>1$; if instead $f(nm,\ell)=-1$, then $f(4nm,2\ell)=1$ and $f(d^2nm,d\ell)=0$ for any $d>2$. In addition, we regard a (real) simple root $\alpha=(n,\ell,m)\in \mathcal{S}$ as even if $f(nm,\ell)=1$, and odd if $f(nm,\ell)=-1$. 

As in the previous subsection, let $t$ be a positive rational number such that $t\rho$ is a primitive vector of $U\oplus L'$. For each simple root $\alpha=(n,\ell,m)\in \mathcal{S}$, the reflectivity yields that $(\alpha,\alpha)=2/d$ and $d\alpha\in U\oplus L$ for a certain positive integer $d$, leading to 
$$
(t\rho,d\alpha)=td\cdot(\alpha,\alpha)/2=t\in \ZZ.
$$
Therefore, $t$ is also the smallest positive integer with $t\rho \in U\oplus L'$. 

For any $n\in\ZZ_{>0}$, we consider $-nt\rho$ as an imaginary simple root (which can be even and odd) of supermultiplicity $f(n^2t^2AC,nt\vec{B})$. The integrality of $f(n^2t^2AC,nt\vec{B})$ follows from Theorem \ref{th:integrality}. 

From these root data, we construct a BKM superalgebra $\mathcal{G}$ using Chevalley--Serre generators and relations; it is easy to check Conditions (1)-(3) in Section \ref{subsec:BKM} by reflectivity. By \eqref{eq:BKM-sd}, the sum side of the superdenominator identity of $\mathcal{G}$ depends only on real simple roots and imaginary simple roots (counting supermultiplicities, not ordinary multiplicities). Consequently, this sum side agrees with \eqref{eq:Singular-Fourier}, which is precisely the Fourier expansion of $F$ at $U$. From the product expansion of $F$, we see that $(n,\ell,m)\in U\oplus L'$ is a root of $\mathcal{G}$ if $f(nm,\ell)\neq 0$. Thus, the canonicality of $L$ implies that $U\oplus L'$ is the root lattice of $\mathcal{G}$. Clearly, $\rho$ gives the Weyl vector of $\mathcal{G}$. Therefore, the BKM superalgebra $\mathcal{G}$ constructed above satisfies all the properties formulated in Theorem \ref{MTH:relation to BKM}. 

\subsection{Proof of Corollary \ref{Cor:relation to regular}}
As before, let $F$ be a singular automorphic product on $2U\oplus L$ with $L$ canonical, and let $\mathcal{G}$ denote the associated BKM superalgebra. The Weyl vector of both $F$ and $\mathcal{G}$ is $\rho=(-A,\vec{B},-C)$, where $A=C+f(-1,0)$, $f(-1,0)\in\{0,1\}$, and $2\vec{B}\in L'$. 

We first prove Part (1). Since $U\oplus L'$ is the root lattice of $\mathcal{G}$, it is generated over $\ZZ$ by all simple roots of $\mathcal{G}$, including all imaginary simple roots $-nt\rho$ for positive integral $n$ and all real simple roots characterized among all real roots $\alpha$ by the equality $(\rho,\alpha)=(\alpha,\alpha)/2$. Let $N$ be the smallest positive integer with $N\rho\in U\oplus L$, and let $\tilde{N}$ denote the level of $L$. Note that $(\rho,\rho)=0$. For any real simple roots $\alpha$ and $\beta$, we have 
$$
N\cdot(\alpha,\alpha)/2=(N\rho, \alpha)\in \ZZ \quad \text{and} \quad \frac{(\alpha,\beta)}{(\rho,\alpha)}=\frac{2(\alpha,\beta)}{(\alpha,\alpha)}\in \ZZ. 
$$
It follows that $N(x,x)\in 2\ZZ$ for all $x\in U\oplus L'$, and therefore $\tilde{N}\mid N$. On the other hand, by the definition of level, we obtain
$$
(\tilde{N}\rho, \alpha) = \tilde{N}(\alpha,\alpha)/2 \in \ZZ,
$$
which yields $(\tilde{N}\rho, U\oplus L')\subset \ZZ$ and thus $\tilde{N}\rho \in U\oplus L$. It follows that $N\mid \tilde{N}$ and hence $N=\tilde{N}$. Note that $N\rho$ is primitive in $U\oplus L$, since for every real simple root $\alpha$ there is a positive integer $d_\alpha$ such that $(\rho, \alpha)=(\alpha,\alpha)/2=1/d_\alpha$ (notice that $N$ is the least common multiple of these $d_\alpha$), and thus a certain $\ZZ$-linear combination $x$ of real simple roots satisfies $(\rho,x)=1/N$. 

\vspace{2mm}

We then prove Part (2). Recall that an even lattice $M$ of level $N$ is called \textit{regular} if $D^{N/p}$ contains a nontrivial isotropic element for each prime $p\mid N$, where $D:=M'/M$ is the discriminant form and $D^{c}:=\{cx : \, x\in D \}$. Let $p$ be any prime divisor of $N$. To prove the regularity of $2U\oplus L$, we show that $\frac{N}{p}\vec{B}$ is a nontrivial isotropic element of $L'/L$. Observe
$$
(\vec{B},\vec{B})=2AC, \quad  A=C+1 \; \text{or} \; A=C.
$$
By assumption, $\vec{B}\in L'$. Since the level of $L$ is $N$, we have $ACN\in\ZZ$. If $p$ is a prime divisor of the denominator of $C$, then $p^2\mid N$. It follows that the order of $\vec{B}$ in $L'/L$ is $N$ and
$$
\Big(\frac{N}{p}\vec{B},\, \frac{N}{p}\vec{B}\Big)=\frac{N}{p^2}\cdot 2ACN \in 2\ZZ.
$$
Therefore, $\frac{N}{p}\vec{B}$ is a nontrivial isotropic element of $L'/L$, and thus $2U\oplus L$ is regular.  

\vspace{2mm}

We now prove Part (3). Suppose that $\mathcal{G}$ has odd real simple roots, and let $\alpha$ be such a root. By Lemma \ref{lem:principal-part-reflective},  there exists an even positive integer $d$ such that $(\alpha,\alpha)=2/d$ and $\frac{d}{2}\alpha \in U\oplus L$. The fact $(\frac{d}{2}\alpha,\frac{d}{2}\alpha)=\frac{d}{2}\in 2\ZZ$ leads to $4\mid d$. Note that $(\rho,\alpha)=(\alpha,\alpha)/2=1/d$. By assumption, $\vec{B}\in L'$ and $C\in\frac{1}{2}\ZZ$. We then have $(\frac{d}{2}\alpha, \vec{B})\in\ZZ$. However, 
$$
\Big(\frac{d}{2}\alpha,\, \vec{B}\Big) = \Big(\frac{d}{2}\alpha,\, \rho\Big)-\Big(\frac{d}{2}\alpha,\, \rho-\vec{B}\Big) \in \frac{1}{2}+\ZZ,
$$
leading to a contradiction. Therefore, $\mathcal{G}$ has no odd real roots. 

\vspace{2mm}

We now turn to Part (4). In the anti-symmetric case, $A=C+1$. Define
$$
L_0:=\{x\in L: \; (x,\vec{B})\in\ZZ \}.
$$
Then $L_0'=L'+\ZZ\cdot\vec{B}$. Recall that the reflection $\sigma_\alpha$ preserves $U\oplus L$ for any real simple root $\alpha\in\mathcal{S}$. We now claim that $\sigma_\alpha$ also preserves $U\oplus L_0$. To prove the claim, it suffices to verify 
$$
\big(\sigma_\alpha(y), \vec{B}\big)=\big(y,\sigma_\alpha(\vec{B})\big)\in \ZZ, \quad \text{for any $y\in U\oplus L_0$.}
$$ 
We do it as follows. Write $\alpha=(a,v,b)\in U\oplus L'$ with $(\alpha,\alpha)=-2ab+(v,v)=2/d$. Recall 
$$
(\rho,\alpha)=(\alpha,\alpha)/2=1/d, \quad \text{i.e.,} \quad (a+b)C+b+(v,\vec{B})=1/d.
$$
We then have $\sigma_\alpha(\vec{B})=\vec{B}-d(v,\vec{B})\alpha$ and thus
\begin{align*}
(y,\sigma_\alpha(\vec{B}))=(y,\vec{B}) -d(v,\vec{B})(\alpha,y)
= (y,\vec{B}) +\big( bd-1+(a+b)dC \big)\cdot (\alpha,y). 
\end{align*}
Since $\alpha\in U\oplus L' \subset U\oplus L_0'$, we have $(\alpha,y)\in \ZZ$. Hence, $(y,\sigma_\alpha(\vec{B}))\in \ZZ$ whenever $C\in\ZZ$. This proves our claim. We conclude from the claim and the Eichler criterion that $F$ is also reflective on $2U\oplus L_0$. Then Lemma \ref{lem:canonical-lattice-unique} leads to $L'=L_0'$ and $L=L_0$. Therefore, $\vec{B}\in L'$ and $\rho \in U\oplus L'$.

\vspace{2mm}

We finally consider Part (5) in more general settings. Define $L_1:=U\oplus L$. Let $N$ denote the level of $L$, which is also the smallest positive integer such that $N\rho$ is a primitive vector of $L_1$. Denote by $N_1$ the positive integer such that $(N\rho, L_1)=N_1\ZZ$. Then $t:=N/N_1\in\ZZ$, which is precisely the integer that appears in Part (5) of Theorem \ref{MTH:relation to BKM}. Note that $t\rho$ is a primitive vector of $L_1'=U\oplus L'$. We consider the index-$t$ sublattice of $L_1$ defined by
$$
K_1:=\{ x\in L_1:\, (\rho,x)\in\ZZ \}.
$$
Obviously, $K_1'=L_1'+\ZZ\rho$. Since $N\rho\in L_1$ is primitive, there exists $c\in L_1'$ with $(N\rho, c)=1$. Then 
$$
\tilde{c}:=N\big(c-N(c,c)\rho/2\big) \in L_1 \quad \text{and} \quad (\tilde{c},\tilde{c})=0.
$$
It follows that $K_1=U(N)\oplus K$, where
$$
K:= \{ y \in K_1: \, (y, N\rho)=0,\, (y,\tilde{c})=0  \} \quad \text{and} \quad U(N)=\ZZ\cdot N\rho + \ZZ \cdot \tilde{c}.
$$
When $t=1$, i.e., $\rho\in U\oplus L'$ (equivalently, $C\in\ZZ$ and $\vec{B}\in L'$), we have $U\oplus L=U(N)\oplus K$.  

\begin{remark}\label{rem:U(N)-BKM}
The Weyl group $W$ of $\mathcal{G}$ preserves $K_1=U(N)\oplus K$, since $K_1'$ is generated over $\ZZ$ by the real simple roots $\alpha$ and the Weyl vector $\rho$, and each $\sigma_\alpha$ preserves the set of real roots and maps $\rho$ to $\rho-\alpha$. In fact, $K_1'=U(1/N)\oplus K'$ gives rise to the root lattice of a BKM superalgebra $\tilde{\mathcal{G}}$ that has the same real simple roots and the same imaginary simple root supermultiplicities as $\mathcal{G}$, but whose imaginary simple roots include $-\rho$ with supermultiplicity $f(AC,\vec{B})$. 

When $t\neq 1$, $f(AC,\vec{B})=0$, and $-\rho$ is both even and odd with identical multiplicity as a root of $\tilde{\mathcal{G}}$, but not of $\mathcal{G}$. Therefore, $\tilde{\mathcal{G}}$ and $\mathcal{G}$ are non-isomorphic but they do have identical superdenominator functions, namely, the expansion of the singular automorphic product $F$ at the $U$-type cusp. If $F$ is anti-symmetric, then Lemma \ref{lem:canonical-lattice-unique} implies that $F$ is not reflective on $U\oplus U(N)\oplus K$. 

Additionally, no real roots occur within the $K'$-component, since $(\rho, \alpha)>0$ for every positive real root $\alpha$, while $\rho\in U(1/N)$ is orthogonal to $K'$. In particular, in the anti-symmetric case, $K$ contains no $2$-roots, and is usually a sublattice of the Leech lattice. 
\end{remark}

\begin{remark}
In the symmetric case, one can find singular automorphic products on $2U\oplus L$, where $L$ is canonical, $\vec{B}\not\in L'$, and $2U\oplus L$ is not regular. A concrete example is given by the singular automorphic product on $2U\oplus \ZZ(8)$, with $\rho=(-1/8,1/16,-1/8)$, $N=16$, $t=8$, and $K=\ZZ(2)$.   
\end{remark}

\begin{remark}\label{rem:canonical-check}
Let $F$ be a singular automorphic product on $2U\oplus L$ with $L$ canonical. According to Theorem \ref{MTH:relation to BKM}, the discriminant group $L'/L$ is generated by the class of $t\vec{B}$ and the classes $x\in L'/L$ for which the input of $F$ has a nonzero Fourier coefficient of type $q^{n}e_x$ for some $n<0$.    
\end{remark}

\section{Finite-dimensional Lie superalgebras from singular automorphic products}\label{sec:Lie-algebra-structure}
In this section, we give a proof of Theorem \ref{MTH:Lie structure} and describe the solutions of Equation \eqref{eq:Intro-Schellekens type}. Let $F$ be a singular automorphic product on $2U\oplus L$ with canonical $L$, and let $\mathfrak{g}$ be a finite-dimensional semi-simple Kac--Moody Lie superalgebra with affine extension $\hat{\mathfrak{g}}$ (see Remark \ref{rem:KMLie-super}). Theorem \ref{MTH:Lie structure} assigns to each $F$ a specific $\mathfrak{g}$ subject to certain restrictions, ensuring that the leading Fourier–Jacobi coefficient of $F$ at the $1$-dimensional cusp of type $2U$ agrees with the superdenominator $\vartheta_\mathfrak{g}$ of $\hat{\mathfrak{g}}$. These restrictions give rise to a finite set of possible $\mathfrak{g}$, and consequently a finite list of candidate $\hat{\mathfrak{g}}$, each of which admits an extension to a BKM superalgebra $\mathcal{G}_\mathfrak{g}$ whose superdenominator coincides with the expansion of $F$ at the $0$-dimensional cusp of type $U$. 

Let $\phi$ denote the Jacobi form input of $F$. When the $q^0$-term of $\phi$ has only non-negative Fourier coefficients, the corresponding $\mathfrak{g}$ is a finite-dimensional semi-simple Lie algebra, and $\mathcal{G}_\mathfrak{g}$ is often a BKM algebra or a BKM superalgebra without odd real roots. This case has been treated in our previous work \cite{SWW23}. We now consider the general case where the $q^0$-term of $\phi$ can contain negative Fourier coefficients so that the corresponding $\mathfrak{g}$ can have an irreducible component of type $\mathfrak{osp}_{1|2r}$ and $\mathcal{G}$ turns out to be a BKM superalgebra with odd real roots. Recall that $F$ is called symmetric (resp. anti-symmetric) if $F$ is invariant (resp. anti-invariant) under reflection $(\omega,\mathfrak{z},\tau) \mapsto (\tau,\mathfrak{z},\omega)$. Note that Theorem \ref{MTH:Lie structure}, with the exception of the last claim of Part (3), remains valid when $L$ is not necessarily canonical, provided that $F$ is reflective on $2U\oplus L$. 

\begin{proof}[Proof of Theorem \ref{MTH:Lie structure}]
Let $\phi\in J_{0,L}^!$ be the Jacobi form input of $F$. Since $F$ is reflective, we conclude from \eqref{eq:reflective-Jacobi} of Section \ref{subsec:reflective} that the Fourier expansion of $\phi$ begins with
$$
\phi(\tau,\mathfrak{z}) = f(-1,0)q^{-1} + \sum_{\ell \in L'}f(0,\ell)\zeta^\ell + O(q). 
$$
Since $F$ is of singular weight, the constant coefficient $f(0,0)$ is given by $\rank(L)$. 
By Theorem \ref{th:singular-product-reflective}, $F$ has only simple zeros. In particular, $f(-1,0)$ is equal to $0$ or $1$. The product $F$ is symmetric if $f(-1,0)=0$, and anti-symmetric if $f(-1,0)=1$.  
Note that $f(0,\ell)=f(0,-\ell)$ and $f(0,\ell)\in\ZZ$. From Lemma \ref{Lem:q^0-term}, we derive 
\begin{align}
\label{eq:C-1}C:=\frac{1}{\rank(L)}\sum_{\ell>0}f(0,\ell)(\ell,\ell) &=\frac{1}{24}\sum_{\ell\in L'} f(0,\ell) - f(-1,0),\\
\label{eq:C-2}\sum_{\ell>0}f(0,\ell)(\ell,\mathfrak{z})^2&=C(\mathfrak{z},\mathfrak{z}).
\end{align}

Suppose $C=0$. By \eqref{eq:reflective-Jacobi} and the associated remarks, we find $f(0,\ell)=0$ for any $\ell\in L'\backslash\{0\}$. Thus, \eqref{eq:C-1} restricts to $f(-1,0)=\rank(L)/24$, so $f(-1,0)=1$ and $\rank(L)=24$. We further know from \cite[Proposition 3.2]{Ma17} that $2U\oplus L\cong \II_{26,2}$, which forces $L$ to be the Leech lattice and $F$ to be the denominator function of the fake monster algebra (see \cite[Theorem 3.5]{WW25}). 

We now turn to the case $C\neq 0$. Then \eqref{eq:reflective-Jacobi} leads to $C>0$. Lemma \ref{Lem:q^0-term} shows that $L(C)$ is an integral positive-definite lattice. By Theorem \ref{th:product}, the leading Fourier--Jacobi coefficient of $F$ at $2U$ is given by the theta quotient
$$
\Theta_{f(0,*)}(\tau,\mathfrak{z})=\eta(\tau)^{f(0,0)} \prod_{\ell>0}\left( \frac{\vartheta\big(\tau,(\ell,\mathfrak{z})\big)}{\eta(\tau)}\right)^{f(0,\ell)}.
$$
Since $F$ has singular weight, its leading coefficient $\Theta_{f(0,*)}$ is a holomorphic Jacobi form of singular weight and index $L(C)$ on $\SL_2(\ZZ)$ with a character. By Theorem \ref{th:singular theta quotient}, as an extremal generalized eutactic star, the set
$$
\mathcal{R}=\{ \ell\in L' : \, \ell\neq 0,\; f(0,\ell)\neq 0 \}
$$
corresponds to a finite-dimensional semi-simple Kac--Moody Lie superalgebra $\mathfrak{g}$. Consequently, if $\ell\in \mathcal{R}$, then $f(0,\ell)$ is $1$ or $-1$. A nonzero vector $\ell$ is an even root of $\mathfrak{g}$ if $f(0,\ell)=1$, and an odd root if $f(0,\ell)=-1$. In addition, if $\ell\in\mathcal{R}$ and $f(0,\ell)=-1$, then $f(0,2\ell)=1$. 

For any $\ell\in\mathcal{R}$ with $f(0,\ell)=1$, $F$ vanishes along $(0,1,\ell,0,0)^\perp$ that is a reflective divisor, so there exists a positive integer $d$ such that $(\ell,\ell)=2/d$ and $d\ell\in L$. For any $\ell'\in\mathcal{R}$ with $f(0,\ell')=-1$, $F$ vanishes along $(0,1,2\ell',0,0)^\perp$, which is also a reflective divisor, so there exists a positive integer $d'$ such that $(2\ell',2\ell')=2/d'$ and $d'\cdot 2\ell'\in L$. It follows that $(\ell',\ell')=\frac{1}{2d'}$ and $2d'\cdot\ell'\in L$. 

We denote by $\mathfrak{g}_j$ ($1\leq j\leq s$) the finite-dimensional simple Kac--Moody Lie superalgebras that occur as the irreducible components of $\mathfrak{g}$, where, as usual, we normalize the bilinear form $\latt{-,-}$ of $\mathfrak{g}_j$ so that long roots have square norm two. We take positive integers $k_j$ for $1\leq j\leq s$ to rescale the root norms of $\mathfrak{g}_j$ so that the root set of $\mathfrak{g}$ is identical to $\mathcal{R}$. More precisely, if $\ell\in\mathcal{R}$ corresponds to a certain long root of $\mathfrak{g}_j$ and $(\ell,\ell)=2/d$, then we take $k_j=d$ and rescale the bilinear form from $\latt{-,-}$ to $\latt{-,-}/k_j$.  In this way, we write $\mathfrak{g}=\bigoplus_{j=1}^s\mathfrak{g}_{j,k_j}$. 

Recall that $\rank(L)=\sum_{j=1}^s \rank(\mathfrak{g}_j)$. Under the above identification, from \eqref{eq:C-1} we derive
$$
C+f(-1,0)=\frac{1}{24}\sum_{j=1}^s \left( |\Delta^j_{\bar{0}}|-|\Delta^j_{\bar{1}}| +\rank(\mathfrak{g}_j) \right) = \frac{1}{24}\sum_{j=1}^s \mathrm{sdim}\,\mathfrak{g}_j = \frac{\mathrm{sdim}(\mathfrak{g})}{24},
$$
where $\Delta_{\bar{0}}^j$ and $\Delta_{\bar{1}}^j$ denote the sets of even roots and odd roots of $\mathfrak{g}_j$, respectively. For any variable $\mathfrak{z}_j$ in the Cartan subalgebra of $\mathfrak{g}_j$, from \eqref{eq:C-2} we deduce
$$
\sum_{\alpha\in \Delta_{\bar{0}}^j} \latt{\alpha,\mathfrak{z}_j}^2 - \sum_{\beta\in \Delta_{\bar{1}}^j} \latt{\beta,\mathfrak{z}_j}^2 = 2Ck_j\latt{\mathfrak{z}_j, \mathfrak{z}_j},
$$
leading to $C=h_j^\vee/k_j$ for any $1\leq j\leq s$. The quasi-periodicity of $\phi$ yields $f(0,\ell)=f(-1,0)$ for any $\ell\in L$ with $(\ell,\ell)=2$. Therefore, if $f(-1,0)=0$, then $(\ell,\ell)<2$ for any $\ell\in\mathcal{R}$. It follows that each level $k_j$ is at least $2$ in the symmetric case.  We have thus proved \eqref{eq:Intro-Schellekens type} and the related properties.  

From the central charge formula
\begin{align*}
    c_\mathfrak{g} = \sum_{j=1}^s \frac{k_j \cdot \mathrm{sdim}\, \mathfrak{g}_{j}}{ k_j+h_j^\vee } = \sum_{j=1}^s \frac{\mathrm{sdim}\, \mathfrak{g}_{j}}{1 + C} = \frac{\mathrm{sdim}\, \mathfrak{g}}{1+C}
\end{align*}
and the equality 
$$
\mathrm{sdim}\, \mathfrak{g} = 24\big(C+f(-1,0)\big),
$$
we derive $c_\mathfrak{g}=24$ in the anti-symmetric case, and $c_\mathfrak{g}=24C/(C+1)$ in the symmetric case. 

From the identification between $\mathcal{R}$ and the root set of $\mathfrak{g}=\oplus_{j=1}^s\mathfrak{g}_{j,k_j}$, we conclude that the leading Fourier--Jacobi coefficient $\Theta_{f(0,*)}$ of $F$ coincides with the superdenominator function of $\hat{\mathfrak{g}}$. 

Let $Q_j$ be the root lattice of $\mathfrak{g}_j$. Each root of $\mathfrak{g}_j$ lies in $L'$, so the rescaled lattice $Q_j(1/k_j)$ is contained in $L'$, and thus $L$ is a sublattice of $\bigoplus_{j=1}^s P_j^\vee(k_j)$, where $P_j^\vee$ is the dual of $Q_j$. For each long root $\alpha\in \mathfrak{g}_j$, we have $(\alpha,\alpha)=2/k_j$ and $k_j\alpha\in L$. Recall that the coroot lattice $Q_j^\vee$ is generated by long roots. Thus, $\bigoplus_{j=1}^s Q_j^\vee(k_j)$ is a sublattice of $L$. When $\mathfrak{g}_j$ is of type $E_8$, $F_4$, $G_2$, or $\mathfrak{osp}_{1|2r}$, it is easy to verify $P_j^\vee = Q_j^\vee$. We then establish the bounds of $L$. 

We now prove Part (3) of Theorem \ref{MTH:Lie structure}. This part relies on the canonicity of $L$. Let $\rho_j$ be the Weyl vector of $\mathfrak{g}_j$. It is immediate to check 
\begin{equation}
\rho = (-A,\,\vec{B},\,-C) = \left( -\frac{\mathrm{sdim}\, \mathfrak{g}}{24},\, \sum_{j=1}^s \frac{\rho_j}{k_j},\, -\frac{h_j^\vee}{k_j} \right), \quad \text{for any $1\leq j\leq s$.} 
\end{equation}
By Lemma \ref{lem:values of delta} below, whenever $N_\mathfrak{g}\cdot\vec{B}\in \bQ$, we have $N_\mathfrak{g}\cdot C\in\ZZ$, and thus $N_\mathfrak{g}\cdot \rho \in U\oplus \bQ \subset U\oplus L$. Consequently, Part (1) of Corollary \ref{Cor:relation to regular} implies that the level $N$ of $L$ divides $N_\mathfrak{g}$. Observe that every $\mathfrak{g}_j$ has at least one short root $\beta$, which appears to be a real simple root of the associated BKM superalgebra $\mathcal{G}$. In addition, when we embed $\beta$ in $L'$, its square norm is given by $2/(\tilde{\delta}_j \cdot k_j)$, where $\tilde{\delta}_j=1$ if $\mathfrak{g}_j$ is of type $ADE$; $\tilde{\delta}_j=2$ if $\mathfrak{g}_j$ is of type $BC$ or $F_4$; $\tilde{\delta}_j=3$ if $\mathfrak{g}_j$ is of type $G_2$; $\tilde{\delta}_j=4$ if $\mathfrak{g}_j$ is of type $\mathfrak{osp}_{1|2r}$. Let $\tilde{N}_\mathfrak{g}$ be the least common multiple of these $\tilde{\delta}_j\cdot k_j$. We derive $\tilde{N}_\mathfrak{g}\mid N$ from
$$
(\vec{B},\beta)=(\rho,\beta)=(\beta,\beta)/2 \quad  \text{and} \quad (N\rho,\beta)\in \ZZ.
$$
Observe $\delta_j/\tilde{\delta}_j\in\{1,2\}$ for any $j$, so $N_\mathfrak{g}/\tilde{N}_\mathfrak{g}\in \{1,2\}$. Therefore, the level $N$ of $L$ is $N_\mathfrak{g}$ or $N_\mathfrak{g}/2$. 
\end{proof}

\begin{lemma}\label{lem:values of delta}
We retain the assumptions of Theorem \ref{MTH:Lie structure}. Let $\delta_j$ denote the smallest positive integer such that $\delta_j\cdot \rho_j$ lies in the coroot lattice $Q_j^\vee$. Then $N_\mathfrak{g}$ equals the least common multiple of the integers $\delta_j\cdot k_j$ for $1\leq j\leq s$. 
The values of $\delta_j$ are formulated as follows. 
\begin{enumerate}
\item For $A_l$ with $l\geq 1$, $\delta=1$ if $l$ is even, and $\delta=2$ if $l$ is odd.  
\item For $B_l$ with $l\geq 2$, $\delta=2$ if $l$ is even, and $\delta=4$ if $l$ is odd.
\item For $C_l$ with $l\geq 3$, $\delta=2$.
\item For $D_l$ with $l\geq 4$, $\delta=1$ if $l\equiv 0,\, 1 \bmod 4$, and $\delta=2$ if $l\equiv 2,\, 3 \bmod 4$.
\item For $E_6$, $\delta=1$.
\item For $E_7$, $\delta=2$.
\item For $E_8$, $\delta=1$.
\item For $F_4$, $\delta=2$.
\item For $G_2$, $\delta=3$.
\item For $\mathfrak{osp}_{1|2l}$ with $l\geq 1$, $\delta=4$.
\end{enumerate}
\end{lemma}
\begin{proof}
The proof follows from a direct computation.    
\end{proof}

We now determine the solutions of \eqref{eq:Intro-Schellekens type}. There are two cases: 
\begin{align}
\label{eq:anti-C}\frac{\mathrm{sdim}\,\mathfrak{g}}{24} - 1 &= \frac{h_j^\vee}{k_j}=C, \quad \text{for any $1\leq j \leq s$,} \\
\label{eq:sym-C}\frac{\mathrm{sdim}\,\mathfrak{g}}{24}  &= \frac{h_j^\vee}{k_j}=C, \quad \text{with $k_j>1$,} \quad \text{for any $1\leq j \leq s$.}
\end{align}

\begin{proposition}\label{prop:solutions}
When the odd part of $\mathfrak{g}$ is trivial, Equation \eqref{eq:anti-C} has exactly $221$ solutions, and \eqref{eq:sym-C} has exactly $17$ solutions. When the odd part of $\mathfrak{g}$ is nontrivial, \eqref{eq:anti-C} has exactly $1106$ solutions, and \eqref{eq:sym-C} has exactly $28$ solutions. In total, Equation \eqref{eq:Intro-Schellekens type} has precisely $1372$ solutions. 
\end{proposition}

\begin{proof}
When the odd part of $\mathfrak{g}$ is trivial, Equation \eqref{eq:anti-C} was first derived by Schellekens \cite{Sch93} in 1993 for holomorphic vertex operator algebras of central charge $24$ (see \cite[Table 15.2]{SWW23} for a list of solutions); the solutions of Equation \eqref{eq:sym-C} were determined in our previous work \cite[Table 15.1]{SWW23}. 

When the odd part of $\mathfrak{g}$ is nontrivial, Equation \eqref{eq:anti-C} was first established by van Ekeren and Rodr\'iguez Morales \cite[Theorem 7.5]{ER23} in 2023 for $\ZZ$-graded holomorphic vertex operator superalgebras of central charge $24$. The count of solutions in \cite[Theorem 7.5]{ER23} requires a correction: the correct number is $1327$. The solutions of Equation \eqref{eq:sym-C} are recorded in Table \ref{table:25} below. 
\end{proof}

\begin{table}[ht]
\def\arraystretch{1.22}
\begin{minipage}[t]{0.32\textwidth}
\[
\begin{array}{r|r|l}
\rank & C & \mathfrak{g} \\\hline\hline
1 & {1}/{24} &  \color{blue}A_{1,36}^* \\ \hline
 2 & {1}/{12} &  A_{1,18}^{*2} \\
 2 & {1}/{6} &   \color{blue}A_{1,9}^* A_{1,12}\\
 2 & {1}/{4} &  \color{blue}C_{2,10}^* \\\hline
 3 & {1}/{8} &  A_{1,12}^{*3} \\
 3 & {3}/{8} &  A_{1,4}^*A_{2,8}  \\\hline
 4 & {1}/{6} &  A_{1,9}^{*4} \\
 4 & {1}/{4} & A_{1,6}^{*3} A_{1,8}  \\
 4 & {1}/{2} & C_{2,5}^{*2} \\ 
 4 &{1}/{2}  & A_{1,3}^{*2} B_{2,6} \\
\end{array}
\]
\end{minipage}
\begin{minipage}[t]{0.34\textwidth}
\[
\begin{array}{r|r|l}
\rank &C & \mathfrak{g} \\\hline\hline
 4 &{1}/{2}   & C_{2,5}^*A_{1,4}^2  \\
 4 &{1}/{2}   &\color{blue} A_{1,3}^* A_{1,4} A_{2,6}  \\ \hline
 6 &{1}/{4}  & A_{1,6}^{*6} \\
 6 &{1}/{2}   & A_{1,3}^{*3}A_{1,4}^3  \\
 6 &{1}/{2}   & A_{1,3}^{*4}A_{2,6}  \\
 6 &{1}/{2}   & A_{1,3}^{*3} C_{2,5}^*A_{1,4}  \\
 6 &{3}/{4}  & A_{1,2}^{*2}A_{2,4}^2  \\
 6 &{5}/{4}   & C_{2,2}^*A_{4,4}  \\
 6 &{3}/{2}  & C_{4,3}^* A_{2,2}  \\ \hline
\end{array}
\]
\end{minipage}
\begin{minipage}[t]{0.32\textwidth}
\[
\begin{array}{r|r|l}
\rank&C & \mathfrak{g} \\\hline\hline
 8 &{1}/{2}   & A_{1,3}^{*6}A_{1,4}^2  \\
 8 &{1}/{2}   & A_{1,3}^{*6} C_{2,5}^* \\ \hline
 9 &{3}/{8}  & A_{1,4}^{*9} \\ \hline
 10 &{1}/{2}   & A_{1,3}^{*9}A_{1,4}  \\
 10 &{3}/{4}   & A_{1,2}^{*8} B_{2,4} \\
 10 &{5}/{4}  & C_{2,2}^{*5} \\ \hline
 12 & {1}/{2}  & A_{1,3}^{*12} \\
 12 & {3}/{4} & A_{1,2}^{*10}A_{2,4} \\ \hline
 18 & {3}/{4} & A_{1,2}^{*18} \\ \hline 
\end{array}
\]
\end{minipage}
\bigskip
\caption{The 28 solutions of Equation \eqref{eq:sym-C} when the odd part of $\mathfrak{g}$ is nontrivial, where $A_{1,k}^*$ and $C_{r,k}^*$ label $\mathfrak{osp}_{1|2r}$ at level $k$ for $r=1$ and $r\geq 2$, respectively.}
\label{table:25} 
\end{table}

\begin{table}[ht]
\renewcommand\arraystretch{1.6}
\[
\begin{array}{|c|c|c|c|c|c|c|c|c|c|c|}
\hline 
\mathfrak{g} & A_n & B_n & C_n & D_n & E_6 & E_7 & E_8 & F_4 & G_2 & \mathfrak{osp}_{1|2n} \\ 
\hline 
\mathrm{sdim}\,\mathfrak{g} & n^2+2n & 2n^2+n & 2n^2+n & 2n^2-n & 78 & 133 & 248 & 52 & 14 & 2n^2-n \\ 
\hline
h^\vee_{\mathfrak{g}} & n+1 & 2n-1 & n+1 & 2(n-1) & 12 & 18 & 30 & 9 & 4 & n+\frac{1}{2} \\
\hline
Q^\vee_\mathfrak{g} & A_n & D_n & nA_1 & D_n & E_6 & E_7 & E_8 & D_4 & A_2 & nA_1 \\
\hline
P^\vee_\mathfrak{g} & A_n' & \ZZ^n & D_n'(2) & D_n' & E_6' & E_7' & E_8 & D_4 & A_2 & nA_1 \\
\hline
\end{array} 
\]
\smallskip
\caption{Data related to finite-dimensional simple Kac--Moody Lie superalgebras}
\label{tab:data}
\end{table}

For convenience, we record some useful data in Table \ref{tab:data}. Finite-dimensional simple Lie algebras are classified by irreducible root systems $A_n$ ($n\geq 1$), $B_n$ ($n\geq 2$), $C_n$ ($n\geq 3$), $D_n$ ($n\geq 4$), $E_6$, $E_7$, $E_8$, $F_4$, and $G_2$. We employ the same notation for the corresponding root lattices, endowed with the normalized bilinear form so that long roots have square norm $2$. Additionally, we adopt the notation $A_1^*$ for $\mathfrak{osp}_{1|2}$ and $C_n^*$ for $\mathfrak{osp}_{1|2n}$ ($n\geq 2$), since $A_1$ and $C_n$ describe the root systems of the even parts of these superalgebras, respectively. We also note $nA_1=\ZZ^n(2)$.

\section{Construction of singular automorphic products}\label{sec:additive-lift}
This section is devoted to the construction of the singular automorphic products classified in Theorem \ref{MTH:classification}. They fall into four classes. The first three classes were treated in our earlier paper \cite{SWW23}. We now  provide an automatic construction for the first two classes by Theorem \ref{MTH:reflectivity} and complete the construction for the last class. 

We begin with the first class. Let $V$ be the holomorphic vertex operator algebra of central charge $24$ with semi-simple structure $V_1=\mathfrak{g}$ (cf. \cite{SM19}). Then its full character $\chi_V$ is a weakly holomorphic Jacobi form of weight $0$ and index $\bQ$, whose Fourier coefficients are all non-negative integers. Therefore, $\Borch(\chi_V)$ is a singular automorphic product on $2U\oplus \bQ$ with Lie algebra structure $\mathfrak{g}$. By Theorem \ref{MTH:reflectivity}, there exists a unique even overlattice $L_\mathfrak{g}$ of $\bQ$ such that $\Borch(\chi_V)$ is reflective on $2U\oplus L_\mathfrak{g}$. We refer to \cite[\S 6]{SWW23} for a description of $L_\mathfrak{g}$. 

We next consider the second class. Let $\mathfrak{g}$ be one of the $\mathcal{N}=1$ superconformal structures of $F_{24}$ (cf. \cite[\S 7]{SWW23}). The denominator function $\vartheta_{\mathfrak{g}}$ of the affine Lie algebra $\hat{\mathfrak{g}}$ is a holomorphic Jacobi form of singular weight and index $\bP^{\mathrm{ev}}$, where $\bP^{\mathrm{ev}}$ denotes the maximal even sublattice of the coweight lattice $\bP$. This function has trivial character and vanishing order one at infinity. By the generalized theta block conjecture, which was proved in \cite[Theorem 1.3]{WW21}, the additive lift $\Grit(\vartheta_{\mathfrak{g}})$ (cf. Theorem \ref{th:additive}) is equal to a singular automorphic product $\Psi_{\mathfrak{g}}$ on $2U\oplus \bP^{\mathrm{ev}}$ with Lie algebra structure $\mathfrak{g}$. By Theorem \ref{MTH:reflectivity}, one can find an even overlattice $L_\mathfrak{g}$ of $\bP^\mathrm{ev}$ such that $\Psi_{\mathfrak{g}}$ is reflective on $2U\oplus L_\mathfrak{g}$. Then Theorem \ref{MTH:Lie structure} yields $\bQ<L_\mathfrak{g}<\bP$, so $L_\mathfrak{g}=\bP^{\mathrm{ev}}$. 

We refer the reader to \cite[\S 8]{SWW23} for the third class. These four singular automorphic products may be constructed as additive lifts of $\vartheta_\mathfrak{g}$ with characters.  

We now address the last class. For each of the four finite-dimensional semi-simple Kac–Moody Lie superalgebras labeled in Table \ref{tab:class-4}, that is,  $\mathfrak{osp}_{1|2}$ at level $36$, $\mathfrak{osp}_{1|4}$ at level $10$, $\mathfrak{osp}_{1|2}\oplus \mathfrak{sl}_2$ at levels $9$ and $12$, and $\mathfrak{osp}_{1|2}\oplus \mathfrak{sl}_2\oplus \mathfrak{sl}_3$ at levels $3$, $4$, and $6$, we construct the associated singular automorphic product as an additive lift. 

The inputs of these products have negative Fourier coefficients in their principal parts. Singular automorphic products of this kind are extremely rare. The only earlier example was discovered by Gritsenko and Nikulin in 1998 (cf. \cite[Theorem 1.11 (2)]{GN98}). We shall demonstrate that this example corresponds to $\mathfrak{osp}_{1|2}$ at level $36$, and further construct three novel examples of such products. Recall that $\zeta=e^{2\pi iz}$ for $z\in\CC$. Throughout this paper, we use the notation $\zeta^{\pm d}$ to denote $\zeta^d+\zeta^{-d}$ for $d=1,2$, and $\zeta_1^{\pm 1}\zeta_2^{\pm 1}$ to denote the four-term sum
$$
\zeta_1\zeta_2+\zeta_1\zeta_2^{-1}+\zeta_1^{-1}\zeta_2+\zeta_1^{-1}\zeta_2^{-1}.
$$

We now recall the Gritsenko–Nikulin construction. They observed \cite[Lemma 1.6]{GN98} that 
$$
\vartheta_{3/2}(\tau,z):=\,\eta(\tau)\frac{\vartheta(\tau,2z)}{\vartheta(\tau,z)}=\sum_{n\in\ZZ}  \left( \frac{12}{n} \right) q^{\frac{n^2}{24}} \zeta^{\frac{n}{2}}  
$$
is a holomorphic Jacobi form of weight $1/2$, character $v_\eta$, and index $\ZZ(3)$. Its trivial additive lift
$$
\Grit(\vartheta_{3/2})(\omega,z,\tau) = \sum_{m=1}^\infty  \left( \frac{12}{m} \right) \vartheta_{3/2}(\tau,mz)\cdot e^{2\pi i m^2\omega/24}
$$
is a modular form of weight $\frac{1}{2}$ for $\widetilde{\Orth}^+(2U\oplus A_1(36))$, as well as a Siegel paramodular form of genus $2$ and level $36$. Having guessed a priori that it is an automorphic product, they identified its input as the theta quotient
\begin{equation}
\begin{split}
\phi_{A_{1,36}^*}(\tau,z):=&-\frac{-\vartheta_{3/2}(\tau,5z)}{\vartheta_{3/2}(\tau,z)} = \frac{\vartheta(\tau,10z)\vartheta(\tau,z)}{\vartheta(\tau,5z)\vartheta(\tau,2z)}\\
=&\, (\zeta^2+\zeta^{-2})-(\zeta+\zeta^{-1})+1+O(q) \in J_{0,A_1(36)}^{\w},    
\end{split}    
\end{equation}
where $-\vartheta_{3/2}(\tau,5z)$ is the second nonzero Fourier–Jacobi coefficient of the lift. They then proved that $\Borch(\phi_{A_{1,36}^*})$ is a reflective automorphic product of singular weight and satisfies
$$
\Borch(\phi_{A_{1,36}^*}) = \Grit(\vartheta_{3/2}). 
$$
Theorem \ref{MTH:Lie structure} identifies $\Borch(\phi_{A_{1,36}^*})$ with $\mathfrak{osp}_{1|2}$ at level $36$, as $\vartheta_{3/2}$ is the superdenominator of $\widehat{\mathfrak{osp}}_{1|2}$. Note that $A_1(36)$ is canonical for this product. 

Inspired by this construction, we construct new singular automorphic products with negative Fourier coefficients in the principal parts of their inputs. We begin with the product for $\mathfrak{osp}_{1|4}$ at level $10$. Let $\vartheta_{C_2^*}$ denote the superdenominator of the affine Lie superalgebra $\widehat{\mathfrak{osp}}_{1|4}$ (see \eqref{eq:superdenominator-osp} or \eqref{eq:sd-osp_{1|4}} below).  The Hecke operators $T_{-}^{(Q)}(m)$ from Proposition \ref{prop:Hecke} will be used below. 

\begin{theorem}\label{th:C2-10}
Let $\mathfrak{z}=(z_1,z_2)\in \CC^2$ and $\zeta_j=e^{2\pi iz_j}$ for $j=1,2$.  Then
\begin{equation}
\begin{split}
\phi_{C_{2,10}^*}(\tau,\mathfrak{z}):=& - \frac{\big(\vartheta_{C_2^*}|_{1} T_{-}^{(4)}(5)\big)(\tau,\mathfrak{z})}{\vartheta_{C_2^*}(\tau,\mathfrak{z})} \\
=&\, \zeta_1^{\pm 1}\zeta_2^{\pm 1} + \zeta_1^{\pm 2} + \zeta_2^{\pm 2} - \zeta_1^{\pm 1} - \zeta_2^{\pm 1} + 2 +O(q)    
\end{split}    
\end{equation}
is a weak Jacobi form of weight $0$ and index $2A_1(10)$. Furthermore, $\Borch(\phi_{C_{2,10}^*})$ is a reflective singular automorphic product on $2U\oplus 2A_1(10)$, and $2A_1(10)$ is canonical. Moreover, 
\begin{align*}
\Borch(\phi_{C_{2,10}^*}) = \Grit(\vartheta_{C_2^*})
= \vartheta_{C_2^*}(\tau,\mathfrak{z})\cdot \xi^{\frac{1}{4}} + \big( \vartheta_{C_2^*}|_{1} T_{-}^{(4)}(5)\big)(\tau,\mathfrak{z})\cdot \xi^{\frac{5}{4}} + O(\xi^{\frac{9}{4}}).
\end{align*}
\end{theorem}
\begin{proof}
The zero structure of $\vartheta$ together with Proposition \ref{prop:Hecke} implies that $\phi_{C_{2,10}^*}$ has no poles and is a weak Jacobi form. Viewing $2A_1(10)$ as the rank-two lattice $\ZZ^2$ with Gram matrix $\mathrm{diag}(20,20)$, the principal part of the associated vector-valued modular form is
\begin{align*}
&2e_{(0, 0)} -q^{-1/40}\big( e_{(1/20, 0)}+e_{(9/20, 0)}+e_{(11/20, 0)} + e_{(19/20, 0)} + e_{(0, 1/20)} + e_{(1/4, 1/5)} \\
& + e_{(3/4, 1/5)} + e_{(1/5, 1/4)} + e_{(4/5, 1/4)} + e_{(0, 9/20)} + e_{(0, 11/20)} + e_{(1/5, 3/4)} + e_{(4/5, 3/4)} \\
&+ e_{(1/4, 4/5)} + e_{(3/4, 4/5)} + e_{(0, 19/20)} \big) + q^{-1/20} \big( e_{(1/20, 1/20)} + e_{(9/20, 1/20)} + e_{(11/20, 1/20)} \\
&+ e_{(19/20, 1/20)} + e_{(1/20, 9/20)} + e_{(9/20, 9/20)} + e_{(11/20, 9/20)} + e_{(19/20, 9/20)} + e_{(1/20, 11/20)} \\
&+ e_{(9/20, 11/20)} + e_{(11/20, 11/20)} + e_{(19/20, 11/20)} + e_{(1/20, 19/20)} + e_{(9/20, 19/20)} + e_{(11/20, 19/20)} \\
&+ e_{(19/20, 19/20)}\big) + q^{-1/10}\big( e_{(1/10, 0)} + e_{(9/10, 0)} + e_{(0, 1/10)} + e_{(1/2, 2/5)} + e_{(2/5, 1/2)} + e_{(3/5, 1/2)} \\
&+ e_{(1/2, 3/5)} + e_{(0, 9/10)}\big) + q^{-1/8}\big( e_{(1/2, 1/4)} + e_{(1/4, 1/2)} + e_{(3/4, 1/2)} + e_{(1/2, 3/4)} \big) \\
&+ q^{-1/5} \big( e_{(2/5, 2/5)} + e_{(3/5, 2/5)} + e_{(2/5, 3/5)} + e_{(3/5, 3/5)} \big).    
\end{align*}
We subsequently verify that the Borcherds theta lift of $\phi_{C_{2,10}^*}$ has no poles and is reflective, and that the underlying lattice $2A_1(10)$ is canonical. 

Observe that the first two non-zero Fourier--Jacobi coefficients of $\Borch(\phi_{C_{2,10}^*})$ and $\Grit(\vartheta_{C_2^*})$ agree. Taking $z_1=0$, the pullback of $\Grit(\vartheta_{C_2^*})$ along the embedding $2U\oplus A_1(10) \hookrightarrow 2U\oplus 2A_1(10)$ is equal to $2\Grit(f)$, while the similar pullback of $\Borch(\phi_{C_{2,10}^*})$ is given by $2\Borch(-f|_{1}T_{-}^{(4)}(5) / f)$, where $f(\tau,z_2):=\vartheta(\tau,z_2)\vartheta(\tau,2z_2)$. Taking into account the pullback of the identity $\Borch(\phi_{\mathfrak{g}})=\Grit(\vartheta_{\mathfrak{g}})$ for $\mathfrak{g}=A_{1,8}^2$ in \cite[Theorem 9.1]{SWW23}, we obtain $\Grit(f)=\Borch(-f|_{1}T_{-}^{(4)}(5) / f)$. Thus, the difference
$$
F:=\Borch(\phi_{C_{2,10}^*}) - \Grit(\vartheta_{C_2^*})
$$
vanishes along the divisor $z_1=0$. Suppose $F\neq 0$. Since both the automorphic product and the additive lift are even functions with respect to $z_1$, $F$ vanishes along $z_1=0$ with multiplicity at least 2. However, \cite[Theorem 2.1]{WW23} forces the multiplicity of $z_1=0$ in the divisor of the singular-weight modular form $F$ to be $1$. This leads to a contradiction, so $F$ is identically zero. 
\end{proof}

\begin{remark}
We find the identity
\begin{equation}\label{eq:sd-osp_{1|4}}
\begin{split}
\vartheta_{C_2^*}(\tau,\mathfrak{z})&=\frac{\vartheta(\tau,z_1+z_2)\vartheta(\tau,z_1-z_2)\vartheta(\tau,2z_1)\vartheta(\tau,2z_2)}{\vartheta(\tau,z_1)\vartheta(\tau,z_2)}\\
& =\vartheta(\tau,2z_1+z_2)\vartheta(\tau,z_1-2z_2) + \vartheta(\tau,2z_1-z_2)\vartheta(\tau,z_1+2z_2).    
\end{split}    
\end{equation}
Recall from \cite[Theorem 9.1]{SWW23} that $\Grit(\vartheta_{\mathfrak{g}})=\Borch(\phi_{\mathfrak{g}})$ for $\mathfrak{g}=A_{1,8}^2$. Therefore, $\Borch(\phi_{C_{2,10}^*})$ equals the sum of two automorphic products that are constructed as pullbacks of $\Borch(\phi_{\mathfrak{g}})$ for $\mathfrak{g}=A_{1,8}^2$.     
\end{remark}

We now construct the singular automorphic product for $\mathfrak{osp}_{1|2}\oplus \mathfrak{sl}_2$ at levels $9$ and $12$.

\begin{theorem}\label{th:A1-9}
Let $\mathfrak{z}=(z_1,z_2)\in \CC^2$ and $\zeta_j=e^{2\pi iz_j}$ for $j=1,2$. Then
\begin{equation}
\begin{split}
\phi_{A_{1,9}^* A_{1,12}}(\tau,\mathfrak{z}):=& - \frac{\big( (\vartheta_{3/2}\otimes\vartheta) |_{1} T_{-}^{(6)}(7)\big)(\tau,\mathfrak{z})}{\vartheta_{3/2}(\tau,z_1)\otimes\vartheta(\tau,z_2)}\\
=&\, \zeta_1^{\pm 2} - \zeta_1^{\pm 1} + \zeta_2^{\pm 1} + 2 +O(q)     
\end{split}    
\end{equation}
is a weak Jacobi form of weight $0$ and index $A_1(9)\oplus A_1(3)$. Furthermore, $\Borch(\phi_{A_{1,9}^* A_{1,12}})$ is a reflective singular automorphic product  on $2U\oplus A_1(9)\oplus A_1(3)$, and $A_1(9)\oplus A_1(3)$ is canonical.  Moreover, 
\begin{align*}
\Borch(\phi_{A_{1,9}^* A_{1,12}}) = \Grit(\vartheta_{3/2}\otimes\vartheta)
= (\vartheta_{3/2}\otimes\vartheta)\cdot \xi^{\frac{1}{6}} + \big( (\vartheta_{3/2}\otimes\vartheta)|_{1} T_{-}^{(6)}(7)\big)\cdot \xi^{\frac{7}{6}} + O(\xi^{\frac{13}{6}}).
\end{align*}
\end{theorem}
\begin{proof}
The proof is similar to that of Theorem \ref{th:C2-10}. We view $A_1(9)\oplus A_1(3)$ as the rank-two lattice $\ZZ^2$ with Gram matrix $\mathrm{diag}(18,6)$. Then the principal part of the vector-valued modular form corresponding to $\phi_{A_{1,9}^*A_{1,12}}$ is
\begin{align*}
&2e_{(0, 0)} -q^{-1/36}\big( e_{(1/18, 0)} + e_{(17/18, 0)} + e_{(5/18, 1/3)} + e_{(13/18, 1/3)} + e_{(5/18, 2/3)} + e_{(13/18, 2/3)}\big) \\
&+ q^{-1/12}\big( e_{(0, 1/6)} + e_{(1/3, 1/6)} + e_{(2/3, 1/6)} + e_{(0, 5/6)} + e_{(1/3, 5/6)} + e_{(2/3, 5/6)} \big) \\
&+ q^{-1/9}\big( e_{(1/9, 0)} + e_{(8/9, 0)} + e_{(4/9, 1/3)} + e_{(5/9, 1/3)} + e_{(4/9, 2/3)} + e_{(5/9, 2/3)}\big) + q^{-1/4}e_{(1/2, 0)}.
\end{align*}
We confirm that the automorphic product $\Borch(\phi_{A_{1,9}^* A_{1,12}})$ has no poles and is reflective, and that the underlying lattice is canonical. To prove the desired identity, we consider the pullback along the embedding $A_1(3)\hookrightarrow A_1(9)\oplus A_1(3)$ taking $z_1=0$. It is easy to verify $\Borch(-f|_{1}T_{-}^{(6)}(7) / f) = \Grit(f)$ for $f:=\eta(\tau)\vartheta(\tau,z_2)$ by \cite[Remark 3.4, Lemma 3.5]{GW20}, because $A_1(3)$ satisfies $\delta_L=3/2<2$ (see Remark \ref{rem:singular}). We then complete the proof in an argument similar to that of Theorem \ref{th:C2-10} (considering the zero divisor $z_1=0$). 
\end{proof}

We finally turn to the singular automorphic product for $\mathfrak{osp}_{1|2}\oplus \mathfrak{sl}_2\oplus \mathfrak{sl}_3$ at levels $3$, $4$, and $6$. 

\begin{theorem}\label{th:A1*-3,A1-4,A2-6}
Let $\mathfrak{z}=(z_1,z_2,z_3,z_4)\in \CC^4$ and $\zeta_j=e^{2\pi iz_j}$ for $1\leq j\leq 4$. Then  
\begin{equation}
\begin{split}
\phi_{A_{1,3}^* A_{1,4}A_{2,6}}(\tau,\mathfrak{z}):= &-\frac{\big( (\vartheta_{3/2}\otimes\vartheta\otimes\vartheta_{A_2}) |_{2} T_{-}^{(2)}(3)\big)(\tau,\mathfrak{z})}{\vartheta_{3/2}(\tau,z_1)\otimes\vartheta(\tau,z_2)\otimes\vartheta_{A_2}(\tau,z_3,z_4)}\\
=&\, \zeta_1^{\pm 2} - \zeta_1^{\pm 1} + \zeta_2^{\pm 1} +\zeta_3^{\pm 1} + \zeta_4^{\pm 1} + \zeta_3^{\pm 1}\zeta_4^{\mp 1}+ 4 +O(q)     
\end{split}    
\end{equation}
is a weak Jacobi form of weight $0$ and index $A_1(3)\oplus A_1\oplus A_2(2)$, where the denominator of the affine Lie algebra of type $A_2$ is given by 
$$
\vartheta_{A_2}(\tau,z_3,z_4)=\frac{\vartheta(\tau,z_3)\vartheta(\tau,z_4)\vartheta(\tau,z_3-z_4)}{\eta(\tau)} \in J_{1,A_2}(v_\eta^8).
$$
Furthermore, $\Borch(\phi_{A_{1,3}^* A_{1,4}A_{2,6}})$ is a reflective  singular automorphic product on $2U\oplus A_1(3)\oplus A_1\oplus A_2(2)$, and the underlying lattice is canonical. Moreover, 
\begin{align*}
\Borch(\phi_{A_{1,3}^* A_{1,4}A_{2,6}}) &= \Grit(\vartheta_{3/2}\otimes\vartheta\otimes\vartheta_{A_2}) \\
&= (\vartheta_{3/2}\otimes\vartheta\otimes\vartheta_{A_2})\cdot \xi^{\frac{1}{2}} + \big( (\vartheta_{3/2}\otimes\vartheta\otimes\vartheta_{A_2})|_{2} T_{-}^{(2)}(3)\big)\cdot \xi^{\frac{3}{2}} + O(\xi^{\frac{5}{2}}).
\end{align*}
\end{theorem}
\begin{proof}
The proof is similar to that of Theorem \ref{th:C2-10}.  We view $A_1(3)\oplus A_1\oplus A_2(2)$ as the rank-four lattice $\ZZ^4$ with Gram matrix $\begin{psmallmatrix}
6 & 0 \\ 0 & 2    
\end{psmallmatrix}\oplus \begin{psmallmatrix}
4 & -2 \\ -2 & 4    
\end{psmallmatrix}$. The principal part of the vector-valued modular form corresponding to $\phi_{A_{1,3}^* A_{1,4}A_{2,6}}$ is 
\begin{align*}
&4e_{(0, 0, 0, 0)}-q^{-1/12}\big( e_{(2/3, 1/2, 0, 1/2)} + e_{(1/3, 1/2, 0, 1/2)} +e_{(1/6, 0, 0, 0)} + e_{(5/6, 0, 0, 0)} \\
& +e_{(2/3, 1/2, 1/2, 0)} + e_{(1/3, 1/2, 1/2, 1/2)} + e_{(2/3, 1/2, 1/2, 1/2)} + e_{(1/3, 1/2, 1/2, 0)} \big) \\
& + q^{-1/6}\big( e_{(0, 0, 2/3, 5/6)} + e_{(0, 0, 1/3, 1/6)} + e_{(1/2, 1/2, 1/6, 5/6)} + e_{(1/2, 1/2, 5/6, 2/3)} \\
& + e_{(1/2, 1/2, 1/6, 1/3)} + e_{(1/2, 1/2, 5/6, 1/6)} + e_{(0, 0, 1/6, 1/3)} + e_{(0, 0, 5/6, 1/6)} \\
& + e_{(0, 0, 1/6, 5/6)} + e_{(0, 0, 5/6, 2/3)} + e_{(1/2, 1/2, 1/3, 1/6)} + e_{(1/2, 1/2, 2/3, 5/6)}\big) \\
& + q^{-1/4}\big( e_{(1/2, 0, 1/2, 0)} + e_{(1/2, 0, 1/2, 1/2)} + e_{(0, 1/2, 0, 0)} + e_{(1/2, 0, 0, 1/2)} \big) \\
&+ q^{-1/3}\big( e_{(1/3, 0, 0, 0)} + e_{(2/3, 0, 0, 0)}\big).
\end{align*}
It follows that $\Borch(\phi_{A_{1,3}^* A_{1,4}A_{2,6}})$ has no poles and is reflective and the underlying lattice is canonical. To prove the desired identity, we consider the pullback along $A_1\oplus A_2(2)\hookrightarrow A_1(3)\oplus A_1\oplus A_2(2)$ taking $z_1=0$. It is easy to verify $\Borch(-g|_{2}T_{-}^{(2)}(3) / g) = \Grit(g)$ for $g:=\eta(\tau)\vartheta(\tau,z_2)\vartheta_{A_2}(\tau,z_3,z_4)$ by \cite[Remark 3.4, Lemma 3.5]{GW20}, because $A_1\oplus A_2(2)$ satisfies $\delta_L=11/6<2$ (see Remark \ref{rem:singular}). We then complete the proof in an argument similar to that of Theorem \ref{th:C2-10}. 
\end{proof}

\begin{remark}
Each of the four singular automorphic products is modular for the full orthogonal group $\Orth^+(2U\oplus L_\mathfrak{g})$, and the canonical underlying lattice $L_\mathfrak{g}$ coincides with the coweight lattice $\bP$. For these singular products, we have $1/C\in\ZZ$, $\vec{B}\not\in L_\mathfrak{g}'$, and $2U\oplus L_\mathfrak{g}$ is not regular.     
\end{remark}

\section{Reflective automorphic products on \texorpdfstring{$2U\oplus A_1(t)\oplus A_1(s)$}{}}\label{sec:A1+A1}
We now classify reflective automorphic products on lattices of the form $2U\oplus A_1(t)\oplus A_1(s)$. These results will be applied to decide whether $A_{1,t}^*$ or $C_{n,t}^*$ can occur as an irreducible component of $\mathfrak{g}$ in solutions to \eqref{eq:Intro-Schellekens type}.   

By virtue of the classification of rank-$3$ hyperbolic reflective lattices, Gritsenko and Nikulin \cite[Theorem 2.2.3]{GN02} proved in 2002 that  $2U\oplus A_1(t)$ admits reflective automorphic products only when $t\leq 66$. Using the algorithm outlined at the end of Section \ref{subsec:reflective}, we compute all such products for $t\leq 66$ and derive the following result, where the weight formulas  are deduced from Lemma \ref{Lem:q^0-term}. 

If there exists a symmetric or anti-symmetric singular automorphic product on $2U\oplus L_\mathfrak{g}$ whose Lie superalgebra structure $\mathfrak{g}$ contains $A_{1,t}^*$ (or $C_{n,t}^*$) as an irreducible component, then $L_\mathfrak{g}=A_1(t)\oplus K$ (or $L_\mathfrak{g}=nA_1(t)\oplus K$) for an even positive-definite lattice $K$, and thus the quasi pullback (cf. \cite[\S 6]{GHS07} or \cite[\S 5]{Wan24}) yields a reflective automorphic product of the same type on $2U\oplus A_1(t)$ with parameter $a=-1$ (or $a=2n-3$) in the lemma below. 

\begin{lemma}\label{lem:A1}
The lattice $2U\oplus A_1(t)$ has reflective automorphic products if and only if $1 \leq t \leq 18$ or $t=20$, $21$, $22$, $24$, $25$, $26$, $28$, $30$, $33$, $34$, $36$, $39$, $42$, $45$. 
\begin{enumerate}
\item There is a reflective automorphic product on $2U\oplus A_1(t)$ whose Jacobi form input has leading Fourier expansion 
$$
\zeta^{\pm 2} + a\zeta^{\pm 1} + 2k +O(q), \quad a\in \ZZ,
$$
if and only if the pair $(t,a)$ satisfies
\begin{align*}
&t=2,3,4, \quad a\geq -1;\\
&t=5, \quad a=1, 6; \\
&t=6, \quad a=-1,0,1,2,3,4,5;\\
&t=7, \quad a=3;\\
&t=8, \quad a=0,2;\\
&t=9, \quad a=-1, 2;\\
&t=10, \quad a=1;\\
&t=12, \quad a=0,-1;\\
&t=16, \quad a=0;\\
&t=36, \quad a=-1.
\end{align*}
In this case, the weight of the reflective automorphic product is given by
$$
k=(6a+24)/t - (a+1). 
$$
\item There is a reflective automorphic product on $2U\oplus A_1(t)$ whose Jacobi form input has leading Fourier expansion  
$$
q^{-1} + \zeta^{\pm 2} + a\zeta^{\pm 1} + 2k +O(q), \quad a\in \ZZ,
$$
if and only if the pair $(t,a)$ satisfies
\begin{align*}
&t\leq 4, \quad a\geq -1;\\
&t=5, \quad a=1,6,11,16,21;\\
&t=6, \quad -1\leq a \leq 29;\\
&t=7, \quad a=10;\\
&t=8, \quad a=2a_0, \; 0\leq a_0 \leq 9; \\
&t=9, \quad a=-1,2,5,8,11;\\
&t=10, \quad a=1, 6, 11;\\
&t=12, \quad -1\leq a \leq 13;\\
&t=13, \quad a=9;\\
&t=14, \quad a=10;\\
&t=15, \quad a=6;\\
&t=16, \quad a=8;\\
&t=18, \quad a=5;\\
&t=20, \quad a=6;\\
&t=21, \quad a=3;\\
&t=24, \quad a=8;\\
&t=36, \quad a=-1, 2, 5, 8.
\end{align*}
In this case, the weight of the reflective automorphic product is given by
$$
k=(6a+24)/t + 11 - a. 
$$
\end{enumerate}
\end{lemma}

By the pullback trick, if $2U\oplus A_1(t)\oplus A_1(s)$ admits a reflective automorphic product, then so does $2U\oplus A_1(s)$. Combining the lemma above with the algorithm at the end of Section \ref{subsec:reflective}, one can easily establish the following three lemmas. We first determine all lattices of type $2U\oplus A_1(t)\oplus A_1(s)$ that admit reflective automorphic products. 

\begin{lemma}\label{lem:A1+A1}
Let $t$ and $s$ be positive integers with $t\leq s$. Then $2U\oplus A_1(t)\oplus A_1(s)$ has a reflective automorphic product if and only if the pair $(t,s)$ satisfies 
\begin{align*}
& t=1,\; 1\leq s \leq  14 \quad \textit{or} \quad s=20,\, 21, \, 30; \\
& t=2, \;  2\leq s \leq 10 \quad \textit{or} \quad s=12,\, 14, \, 15, \, 16, \, 18, \, 20, \, 24, \, 28; \\
& t=3,\; 3\leq s \leq  10 \quad \textit{or} \quad s=12,\, 14,\, 15,\, 18, \, 30; \\
& t=4,\; s=4,\, 5,\, 6, \, 8, \, 10, \, 12, \, 14; \\
& t=5,\; s=5,\, 6, \, 10; \\
& t=6,\; s=6, \, 8, \, 9, \, 10, \, 12, \, 15, \, 18; \\
& t=7,\; s=7;\\
& t=8, \; s=8;\\
& t=9, \; s=9;\\
& t=10, \; s=10;\\
& t=12, \; s=12;\\
& t=15, \; s=15.
\end{align*}
\end{lemma}

The following lemma serves to determine whether $C_{n,t}^*$ appears as an irreducible component of $\mathfrak{g}$ in solutions to \eqref{eq:Intro-Schellekens type}. Specifically, if $C_{n,t}^*$ is an irreducible component of $\mathfrak{g}$, then the quasi pullback of the anti-symmetric or symmetric singular automorphic product gives a reflective automorphic product of the same type on $2U\oplus 2A_1(t)$ with parameter $a=2n-5$. 

\begin{lemma}\label{lem:C_n-criterion}
Let $t$ be a positive integer.  
\begin{enumerate}
\item There is a reflective automorphic product on $2U\oplus 2A_1(t)$ whose Jacobi form input has leading Fourier expansion
$$
\zeta_1^{\pm 2} + \zeta_2^{\pm 2} + \zeta_1^{\pm 1}\zeta_2^{\pm 1} + a(\zeta_1^{\pm 1}+\zeta_2^{\pm 1}) + 2k +O(q)
$$
if and only if the pair $(t,a)$ satisfies
\begin{align*}
&t = 2, \; a\geq -1; \\
&t = 3, \; -1 \le a \le 7; \\
&t = 4, \; -1 \le a \le 3; \\
&t = 5, \; a = -1; \\
&t = 6, \; a = -1, 0; \\
&t = 10, \; a = -1.
\end{align*}
In this case, the weight of the reflective automorphic product is given by
$$
k=(36+6a)/t-(4+2a).
$$

\item There is a reflective automorphic product on $2U\oplus 2A_1(t)$ whose Jacobi form input has leading Fourier expansion
$$
q^{-1} + \zeta_1^{\pm 2} + \zeta_2^{\pm 2} + \zeta_1^{\pm 1}\zeta_2^{\pm 1} + a(\zeta_1^{\pm 1}+\zeta_2^{\pm 1}) + 2k +O(q)
$$
if and only if the pair $(t,a)$ satisfies
\begin{align*}
&t = 1, \; a\geq -1; \\
&t = 2, \; a\geq -1; \\
&t = 3, \; -1 \le a \le 19; \\
&t = 4, \; -1 \le a \le 11; \\
&t = 5, \; a = -1, 4; \\
&t = 6, \; -1 \le a \le 7; \\
&t = 8, \; a = 4; \\
&t = 10, \; a = -1, 4.
\end{align*}
In this case, the weight of the reflective automorphic product is given by
$$
k=(36+6a)/t+8-2a. 
$$
\end{enumerate}
\end{lemma}

The following lemma will be used to determine whether $\mathfrak{g}$, arising from solutions to \eqref{eq:Intro-Schellekens type}, has more than one irreducible component of type $\mathfrak{osp}_{1|2r}$. More precisely, if $\mathfrak{g}$ contains irreducible components $\mathfrak{osp}_{1|2l}$ and $\mathfrak{osp}_{1|2k}$ at levels $t$ and $s$, then the quasi pullback of the anti-symmetric or symmetric singular automorphic product produces a reflective automorphic product of the same type on $2U\oplus A_1(t)\oplus A_1(s)$ with parameters $a=2l-3$ and $b=2k-3$. 

\begin{lemma}\label{lem:AC-criterion}
Let $t$ and $s$ be positive integers with $t\leq s$. 
\begin{enumerate}
\item 
There is a reflective automorphic product on $2U\oplus A_1(t)\oplus A_1(s)$ whose Jacobi form input has leading Fourier expansion
$$
\zeta_1^{\pm 2} + \zeta_2^{\pm 2} + a\zeta_1^{\pm 1}+b\zeta_2^{\pm 1} + 2k +O(q)
$$
if and only if the quadruple $(t,s,a,b)$ satisfies 
\begin{align*}
&t=s=2, \; a=b \ge -1; \\
&t=2, \; s=3, \; (a, b) \in \{(0,2), (2,5), (4,8)\}; \\
&t = 2, \; s=4, \; (a, b) \in \{(-1, 2), (0, 4), (1, 6)\}; \\
&t = 2, \; s = 5, \; a = 0, \; b = 6; \\
&t=s=3, \; (a, b) \in \{(-1,-1),(0,0),(1,1),(2,2),(3,3),(4,4),(5,5)\}; \\
&t = 3, \; s = 4, \; a = 2, \; b = 4; \\
&t = 3, \; s = 6, \; a = -1, \; b = 2; \\
&t = 3, \; s = 7, \; a = -1, \; b = 3; \\
&t = s = 4, \; (a, b) \in \{(-1,-1),(0,0),(1,1),(2,2)\}; \\
&t = 4, \; s = 6, \; a = 0, \; b = 2; \\
&t = s = 5, \; a=b=1; \\
&t=s=6, \; (a, b) \in \{(-1,-1), (0,0)\}; \\
&t=s=8, \; a=b=0.
\end{align*}
In this case, $(4+a)s=(4+b)t$ and the weight of the reflective automorphic product is 
$$
k=(12+3a)/t+(12+3b)/s-(a+b+2). 
$$
\item There is a reflective automorphic product on $2U\oplus A_1(t)\oplus A_1(s)$ whose Jacobi form input has leading Fourier expansion
$$
q^{-1} + \zeta_1^{\pm 2} + \zeta_2^{\pm 2} + a\zeta_1^{\pm 1}+b\zeta_2^{\pm 1} + 2k +O(q)
$$
if and only if the quadruple $(t,s,a,b)$ satisfies 
\begin{align*}
&t=s=1, \; a=b \ge -1; \\
&t=1, \; s=2, \; b=2a+4, \; a \ge -1; \\
&t=1, \; s=3, \; b=3a+8, \; a \ge -1; \\
&t=1, \; s=4, \; (a, b) \in \{ (-1,8), (0, 12), (1, 16), (2, 20), (3,24)\}; \\
&t=1, \; s=5, \; (a,b)\in\{(-1,11), (0,16), (1,21)\}; \\
&t=1, \; s=6, \; a=-1,\, b=14;\\
&t=s=2, \; a=b \ge -1; \\
&t=2, \; s=3, \; (a,b) \in \{(0,2), (2,5), (4,8), (6,11), (8, 14), (10,17)\}; \\
&t=2, \; s = 4, \; -1 \le a \le 9, \; b = 4+2a; \\
&t=2, \; s=5, \; a=2, \; b = 11; \\
&t=2, \; s=6, \; (a, b) \in \{(0,8), (2,14)\}; \\
&t=2, \; s=7, \; a = 0, \; b = 10; \\
&t=2, \; s=8, \; (a, b) \in \{(-1, 8), (0, 12)\}; \\
&t=2, \; s=10, \; a=-1, \; b=11; \\
&t=s=3, \; a=b, \; -1 \le a \le 17; \\
&t=3,\; s=4, \; a=5, \; b=8; \\
&t=3, \; s=6, \; (a,b) \in \{(-1, 2), (2, 8)\}; \\
&t=3, \; s=7, \; a=2, \; b=10; \\
&t=3, \;s=9, \; (a,b) \in \{(-1, 5), (0, 8)\}; \\
&t=3, \;s=12, \; a=-1, \; b=8; \\
&t=s=4, \; a=b, \; -1 \le a \le 10; \\
&t=4, \; s=6, \;(a, b) \in \{(0, 2), (4, 8)\}; \\
&t=4, \; s=8, \; (a, b) \in \{(-1, 2), (0, 4), (1, 6), (2, 8), (3, 10)\}; \\
&t=s=5, \; (a, b) \in \{(1,1), (6, 6)\}; \\
&t=s=6, \; a=b, \; -1 \le a \le 7; \\
&t=6, \; s=12, \; a = 2, \; b = 8; \\
&t=s=8, \; a=b=4; \\
&t=s=10, \; a=b=1.
\end{align*}
In this case, $(4+a)s=(4+b)t$ and the weight of the reflective automorphic product is 
$$
k=(12+3a)/t+(12+3b)/s+10-(a+b). 
$$
\end{enumerate}
\end{lemma}

\section{Classification of singular automorphic products: symmetric case}\label{sec:sym-classification}

Comparing Proposition \ref{prop:solutions} with Theorem \ref{MTH:classification}, we find that the symmetric case of Equation \eqref{eq:Intro-Schellekens type} yields $5$ extraneous solutions with trivial odd parts and $24$ extraneous solutions with nontrivial odd parts. Our task is to eliminate these solutions. The former case has been dealt with in our earlier work \cite[\S 12]{SWW23}; now we proceed to the latter. The main theorem of this section is as follows. 

\begin{theorem}\label{th:classification-symmetric}
Let $F$ be a symmetric singular automorphic product on $2U\oplus L$ whose Jacobi form input has negative Fourier coefficients in its $q^0$-term. Assume $L$ is canonical. Then $L$ is isomorphic to $A_1(36)$, $2A_1(10)$, $A_1(9)\oplus A_1(3)$, or $A_1(3)\oplus A_1\oplus A_2(2)$, and in each case $F$ is unique. 
\end{theorem} 

\begin{proof}
In this setting, Equation \eqref{eq:sym-C} has exactly $28$ solutions, as stated in Proposition \ref{prop:solutions}; these are listed in Table \ref{table:25}. Lemma \ref{lem:sym-bound} excludes the $3$ solutions with $\rank(\mathfrak{g})>11$. We discuss the remaining cases individually. 
\begin{enumerate}
\item $A_{1,36}^*$: $L=A_1(36)$. The product $F$ was constructed by Gritsenko and Nikulin. 
\item $A_{1,18}^{*2}$: $L=2A_1(18)$. The pullback of $F$ along $A_1(18)\hookrightarrow 2A_1(18)$ would give a reflective automorphic product whose Jacobi form input has leading Fourier expansion
$$
\zeta^{\pm 2}-\zeta^{\pm 1} + 2 + O(q).
$$
By Lemma \ref{lem:A1}, such a form does not exist. This case can also be excluded by Lemma \ref{lem:A1+A1}. 
\item $C_{2,10}^*$: $L=2A_1(10)$. The product $F$ has been constructed in Theorem \ref{th:C2-10}.
\item $A_{1,9}^*A_{1,12}$: $N_\mathfrak{g}=72$, $L=A_1(9)\oplus L_1$ with $A_1(12)<L_1<A_1(3)$, so $L_1=A_1(3), A_1(12)$. Since the level of $L$ is $36$ or $72$, $L_1=A_1(3)$, and $F$ has been constructed in Theorem \ref{th:A1-9}. Note that Lemma \ref{lem:A1+A1} also implies the non-existence of $F$ for $L=A_1(9)\oplus A_1(12)$. 
\item $A_{1,12}^{*3}$: $L=3A_1(12)$. The pullback of $F$ along $2A_1(12)\hookrightarrow 3A_1(12)$ would give a reflective automorphic product whose Jacobi form input has leading Fourier expansion 
$$
\zeta_1^{\pm 2}+\zeta_2^{\pm 2}-\zeta_1^{\pm 1}-\zeta_2^{\pm 1} + 3 + O(q).
$$
By Lemma \ref{lem:AC-criterion}, such a form does not exist. 
\item $A_{1,4}^*A_{2,8}$: $N_\mathfrak{g}=16$, $L=A_1(4)\oplus L_1$ with $A_2(8)<L_1<A_2'(8)$. Thus, $L=A_1(4)\oplus A_2(8)$, leading to a contradiction, since the level of $L$ must be $16$ or $8$. 
\item $A_{1,9}^{*4}$: $L=4A_1(9)$. The pullback of $F$ along $2A_1(9)\hookrightarrow 4A_1(9)$ would give a reflective automorphic product on $2U\oplus 2A_1(9)$ whose Jacobi form input has Fourier expansion
$$
\zeta_1^{\pm 2} + \zeta_2^{\pm 2} - \zeta_1^{\pm 1} -\zeta_2^{\pm 1} + 4 + O(q).
$$
Such a form does not exist by Lemma \ref{lem:AC-criterion}.  
\item $A_{1,6}^{*3}A_{1,8}$: $N_\mathfrak{g}=48$, $L=3A_1(6)\oplus L_1$ with $A_1(8)<L_1<A_1(2)$. Since the level of $L$ is $24$ or $48$, we have $L_1=A_1(2)$. The quasi pullback of $F$ along $3A_1(6)\hookrightarrow 3A_1(6)\oplus A_1(2)$ would give a reflective automorphic product on $2U\oplus 3A_1(6)$ whose Jacobi form input has leading Fourier expansion 
$$
\zeta_1^{\pm 2} + \zeta_2^{\pm 2} + \zeta_3^{\pm 2} - \zeta_1^{\pm 1} -\zeta_2^{\pm 1} - \zeta_3^{\pm 1} + 2k + O(q).
$$
A direct computation shows that such a form does not exist, although $2U\oplus 3A_1(6)$ does have a reflective automorphic product.  
\item $C_{2,5}^{*2}$: $L=4A_1(5)$. There are no reflective automorphic products on $2U\oplus 3A_1(5)$ whose inputs have zero $q^{-1}e_0$ coefficient. By the pullback trick, we then rule out this case. 
\item $A_{1,3}^{*2}B_{2,6}$: $N_\mathfrak{g}=12$, $L=2A_1(3)\oplus L_1$ with $2A_1(6)<L_1<2A_1(3)$. Since the level of $L$ is $12$, we have $L_1=2A_1(3)$. This is impossible, because there are reflective automorphic products on $2U\oplus 4A_1(3)$, but none of them is of singular weight. 
\item $C_{2,5}^*A_{1,4}^2$: $L=2A_1(5)\oplus L_1$ with $2A_1(4)<L_1<2A_1$. By direct computation, there are no reflective automorphic products on $2U\oplus A_1\oplus 2A_1(5)$ whose input has zero $q^{-1}e_0$ coefficient. We then exclude this case by Lemma \ref{lem:sym-overlattice}. 
\item $A_{1,3}^*A_{1,4}A_{2,6}$: $C=1/2$, $N_\mathfrak{g}=24$, $L=A_1(3)\oplus L_1$ with $A_1(4)\oplus A_2(6)<L_1<A_1\oplus A_2(2)$. Since $L(1/2)$ is integral, $L_1(1/2)$ is also integral. It follows that 
$$
A_1(2)\oplus A_2(3)<L_1(1/2)<\ZZ\oplus A_2
$$
and thus $L_1(1/2)=A_1(2)\oplus A_2(3)$, $\ZZ\oplus A_2(3)$, $A_1(2)\oplus A_2$, or $\ZZ\oplus A_2$. Since the level of $L$ is $12$ or $24$, we exclude the first three cases. Therefore, $L_1(1/2)=\ZZ\oplus A_2$ and the singular product $F$ on $2U\oplus A_1(3)\oplus A_1\oplus A_2(2)$ was constructed in Theorem \ref{th:A1*-3,A1-4,A2-6}. 
\item $A_{1,6}^{*6}$: $L=6A_1(6)$. We eliminate it by considering the component $3A_1(6)$ as in Case (8).  
\item $A_{1,3}^{*3}A_{1,4}^3$: $L=3A_1(3)\oplus L_1$ with $3A_1(4)<L_1<3A_1$. We rule out this case by Lemma \ref{lem:sym-overlattice}, because there are no reflective automorphic products on $2U\oplus 3A_1(3)\oplus 3A_1$ whose inputs have zero $q^{-1}e_0$ coefficient. 
\item $A_{1,3}^{*4}A_{2,6}$: $N_\mathfrak{g}=12$, $L=4A_1(3)\oplus L_1$ with $A_2(6)<L_1<A_2(2)$. Since the level of $L$ is $12$, we have $L_1=A_2(2)$. Note that $A_1(3)\oplus A_2(2)<3A_1$, leading to $L<3A_1(3)\oplus 3A_1$. We then eliminate this case as in Case (14). 
\item $A_{1,3}^{*3}C_{2,5}^*A_{1,4}$: $L=3A_1(3)\oplus 2A_1(5)\oplus L_1$ with $A_1(4)<L_1<A_1$. As mentioned in Case (11), there are no reflective automorphic products on $2U\oplus A_1\oplus 2A_1(5)$ whose inputs have zero $q^{-1}e_0$ coefficient. We then exclude this case by Lemma \ref{lem:sym-overlattice}. 
\item $A_{1,2}^{*2}A_{2,4}^2$: $C=3/4$, $N_\mathfrak{g}=8$, $L=2A_1(2)\oplus L_1$ with $2A_2(4)<L_1<2A_2'(4)$. Recall that $L$ is even and $L(3/4)$ is integral. Thus, $L(1/4)$ is also integral. Since $2A_2<L_1(1/4)<2A_2'$, we have $L_1(1/4)=2A_2$ and thus $L=2A_1(2)\oplus 2A_2(4)$, leading to a contradiction, since the level of $L$ must be $8$ or $4$. 
\item $C_{2,2}^*A_{4,4}$: $N_\mathfrak{g}=8$, $L=2A_1(2)\oplus L_1$ with $A_4(4)<L_1<A_4'(4)$. Therefore, $L_1=A_4(4)$ and $L=2A_1(2)\oplus A_4(4)$, leading to a contradiction, since the level of $L$ must be $8$ or $4$.  
\item $C_{4,3}^*A_{2,2}$: $L=4A_1(3)\oplus L_1$ with $A_2(2)<L_1<A_2'(2)$, so $L_1=A_2(2)$ and $L=4A_1(3)\oplus A_2(2)$. Since $A_1(3)\oplus A_2(2)<3A_1$, we have $L<3A_1(3)\oplus 3A_1$. We then exclude it as in Case (14).
\item $A_{1,3}^{*6}A_{1,4}^2$: $L=6A_1(3)\oplus L_1$ with $2A_1(4)<L_1<2A_1$. Since $3A_1(3)<A_1\oplus A_2(2)$, we have 
$$
L<3A_1(3)\oplus 3A_1\oplus A_2(2).
$$
We then exclude it as in Case (14). 
\item $A_{1,3}^{*6}C_{2,5}^*$: $L=6A_1(3)\oplus 2A_1(5)$. Since $3A_1(3)<A_1\oplus A_2(2)$ and $2A_1(5)<2A_1$, we have 
$$
L<3A_1(3)\oplus 3A_1\oplus A_2(2) \quad \text{and} \quad L<2A_1(5)\oplus 2A_1\oplus 2A_2(2).
$$
We then exclude it as in Case (14) and Case (11). 
\item $A_{1,4}^{*9}$: $L=9A_1(4)$. Since $A_1(4)<A_1$ and $4A_1<D_4$, we have 
$$
L<2A_1(4)\oplus 3A_1\oplus D_4.
$$
We then exclude this case by Lemma \ref{lem:sym-overlattice}, because there are no reflective automorphic products on $2U\oplus 3A_1\oplus A_1(4)\oplus D_4$ (see forbidden component (38) in \cite[\S 11.1]{SWW23}).
\item $A_{1,3}^{*9}A_{1,4}$: $L=9A_1(3)\oplus L_1$ with $A_1(4)<L_1<A_1$. Since $3A_1(3)<A_1\oplus A_2(2)$, we have $L<3A_1(3)\oplus 3A_1\oplus 2A_2(2)$. We then exclude it as in Case (14). 
\item $A_{1,2}^{*8}B_{2,4}$: $L=8A_1(2)\oplus L_1$ with $2A_1(4)<L_1<2A_1(2)$. Since $A_1(2)\oplus 6A_1 < D_7$ and $2A_1(2)<2A_1$, we have 
$$
L<10A_1(2)<4A_1(2)\oplus 6A_1 < 3A_1(2)\oplus D_7.  
$$
We then exclude this case by Lemma \ref{lem:sym-overlattice}, because there are no reflective automorphic products on $2U\oplus 3A_1(2)\oplus D_7$ (see forbidden component (53) in \cite[\S 11.1]{SWW23}).
\item $C_{2,2}^{*5}$: $L=10A_1(2)$. We eliminate it as in Case (24). 
\end{enumerate}
The uniqueness of $F$ on the underlying lattices can easily be verified by obstruction theory. 
\end{proof}

\begin{remark}
 Applying Part (3) of Theorem \ref{MTH:Lie structure} and Lemma \ref{lem:values of delta}, the exclusion arguments of \cite[\S 12]{SWW23} can be significantly simplified. In the case where $\mathfrak{g}$ has trivial odd parts, the five extraneous solutions of Equation \eqref{eq:sym-C} are $A_{2,4}B_{2,4}$, $A_{2,2}D_{4,4}$, $A_{2,2}^2 B_{2,2}^2$, $A_{4,2}C_{4,2}$, and $A_{6,2}B_{4,2}$. These cases can now be ruled out by examining the level of $L$ through the bounds $\bQ < L < \bP$. Take $\mathfrak{g}=A_{2,2}D_{4,4}$ as an example.  Here $N_\mathfrak{g}=4$, and 
$$
A_2(2)\oplus D_4(4)<L<A_2'(2)\oplus D_4(2).
$$
The bounds require that the level of $L$ is divisible by $3$, while $N_\mathfrak{g}=4$ restricts the level to $2$ or $4$, which leads to a contradiction. 
\end{remark}

\section{Classification of singular automorphic products: anti-symmetric case}\label{sec:anti-sym-classification}

A comparison between Proposition \ref{prop:solutions} and Theorem \ref{MTH:classification} shows that, in the anti-symmetric case, Equation \eqref{eq:Intro-Schellekens type} possesses 
$152$ extraneous solutions with trivial odd parts and $1106$ extraneous solutions with nontrivial odd parts, all of which must be eliminated. Having treated the former in \cite[\S 11]{SWW23}, we now turn our attention to the latter, which is significantly more intricate. The main theorem of this section is as follows; its proof is divided into three subsections.  

\begin{theorem}\label{th:classification-anti-symmetric}
There are no anti-symmetric singular automorphic products on $2U\oplus L$ whose Jacobi form input contains negative Fourier coefficients in its $q^0$-term. 
\end{theorem}

\subsection{Forbidden components}\label{subsec:forbidden}

We first describe a general approach for excluding extraneous solutions. We begin with the following concept from \cite[\S 11.2]{SWW23}.  

\begin{definition}
An even positive-definite lattice $K$ is called a \emph{forbidden component} if there are no reflective automorphic products of any weight on $2U\oplus K$. 
\end{definition}

Symmetrization allows us to reduce the verification that $K$ is forbidden to checking that there are no reflective automorphic products of any weight for $\Orth^+(2U\oplus K)$. The algorithm at the end of Section \ref{subsec:reflective} provides a means to carry out this check. 

Suppose that $2U\oplus L$ has a reflective automorphic product with nonzero $q^{-1}e_0$ coefficient in its input. Lemma \ref{lem:overlattice} ensures that, for each even overlattice $L_1$ of $L$, there exists a reflective automorphic product of the same type on $2U\oplus L_1$. Furthermore, whenever a decomposition
$$
2U\oplus L_1 \cong 2U\oplus L_2\oplus L_3
$$
holds, the quasi pullback produces a reflective automorphic product of the same type on $2U\oplus L_2$. 

Hence, to rule out the existence of an anti-symmetric reflective singular automorphic product on $2U \oplus L$, it suffices to exhibit an even overlattice $L_1$ of $L$ with a direct sum decomposition $L_1 \cong K \oplus L_0$, where $K$ is forbidden. 

We will employ the above algorithm, along with the forbidden components from \cite[\S 11.1]{SWW23} (see Remark \ref{rem:correction} below for corrections), to rule out the extraneous solutions. The bounds $\bQ<L<\bP$ and the level restriction $N_\mathfrak{g}$ help to determine suitable overlattices $L_1$ (see Parts (1)-(3) of Theorem \ref{MTH:Lie structure} and Lemma \ref{lem:values of delta}). The following additional forbidden components are needed: 
\begin{align*}
&A_4(4),& &A_1(3)\oplus A_1(4)\oplus A_2(2),&  &2A_1\oplus A_1(5)\oplus A_2(2),&  &2A_1\oplus D_6\oplus D_4(2),& \\
&A_1\oplus A_6,& &3A_1(3)\oplus D_4,& &A_2\oplus A_2(8).&          
\end{align*}
We warn that $4A_1(3)$, $4A_1(4)$, $3A_1(6)$, and $4A_1(3)\oplus A_2(2)$ are not forbidden components. 

\begin{remark}\label{rem:correction}
We give some corrections to our previous work \cite[\S 11]{SWW23}. The lattices $A_1\oplus A_2(8)$, $A_1(2)\oplus A_2(8)$, $A_3'(24)$, $A_1\oplus 2A_1(2)\oplus A_2(4)$, and $2D_4(3)$ are in fact not forbidden components. Consequently, the previous elimination arguments fail for ten cases. Cases (6), (29), (34), (35), (53), and (54) can be fixed by the new forbidden component $A_2\oplus A_2(8)$; cases (25) and (42) can be fixed by our previous forbidden component $2A_1(6)\oplus A_2$; case (65) can be fixed by our previous forbidden component $A_1(3)\oplus A_2\oplus A_2(4)$; The last case (61) can be handled as follows: the bounds $\bQ<L<\bP$ together with $N_\mathfrak{g}=6$ imply that $L$ is in the genus of $2D_4(3)$ or $D_8'(6)$; reflective automorphic products do exist on the lattices, but they are not of singular weight (cf. \cite{Dit19}). In addition, in case (124), the maximal even overlattice of $5A_2(2)$ is in the genus of $A_2(2)\oplus E_8$, and this case can be ruled out by the forbidden component $A_1\oplus A_2(2)\oplus E_8$. In case (131), the maximal even sublattice of $\ZZ^{10}\oplus 2A_3'(4)$ should be of genus $\II_{16,0}(2_{\II}^{+2}4_{\II}^{+ 4})$, and this case can be excluded by the forbidden component labeled (71). In case (136), we omitted an additional genus of index-two sublattices of $2E_8\oplus 2A_1$, which has level $8$ and can be easily excluded by obstruction theory.    
\end{remark}

\subsection{Candidate reduction for solutions} 

In view of the large number of extraneous solutions, we employ a series of lemmas to narrow down the candidates. The proofs are based on the classification results in Section \ref{sec:A1+A1}.  Henceforth, we assume that $\mathfrak{g}$ has nontrivial odd parts and is among the $1106$ solutions of Equation \eqref{eq:Intro-Schellekens type} in the anti-symmetric case. We first impose some restrictions on the $\mathfrak{osp}_{1|2}$-type irreducible components of $\mathfrak{g}$. 

\begin{lemma}\label{lem:A-component}
If $A_{1,l}^*$ is an irreducible component of $\mathfrak{g}$, then the possible values of $l$ are $1$, $2$, $3$, $4$, $6$, $9$, $12$, $36$. Furthermore, its multiplicity $m$ in the decomposition of $\mathfrak{g}$ is bounded as follows: $m\leq 1$ for $l=9$, $12$, $36$; $m\leq 2$ for $l=6$; $m\leq 3$ for $l=4$; $m\leq 4$ for $l=3$;  $m\leq 6$ for $l=2$. 
\end{lemma} 

\begin{proof}
Let $F$ be a singular automorphic product on $2U\oplus L$ with $L$ canonical. Assume that the associated Lie superalgebra $\mathfrak{g}$ has an irreducible component of type $A_{1,l}^*$ with multiplicity $m$. Then $L=mA_1(l)\oplus K$ for an even positive-definite lattice $K$.   
The quasi pullback of $F$ along the embedding $A_1(l)\hookrightarrow mA_1(l)\oplus K$ gives a reflective automorphic product $F|_{A_{1,l}^*}$ of weight $k$ on $2U\oplus A_1(l)$ whose Jacobi form input has leading Fourier series
$$
q^{-1} + \zeta^{\pm 2} - \zeta^{\pm 1} + 2k +O(q).
$$
According to Lemma \ref{lem:A1}, the reflective product $F|_{A_{1,l}^*}$ exists only for $l=1,2,3,4,6,9,12,36$.     
We now determine the multiplicity of $A_{1,l}^*$ case by case. 
\begin{enumerate}
\item $A_{1,36}^*$: Suppose $m>1$. Since $2A_1(36)<2A_1(18)$, we have $L<2A_1(18)\oplus K_1$ for some $K_1$. This overlattice contains the forbidden component (2) in \cite[\S 11.1]{SWW23}, that is, $2A_1(18)$. 
\item $A_{1,12}^*$: Suppose $m>1$. Then the quasi pullback of $F$ gives a reflective automorphic product on $2U\oplus 2A_1(12)$ whose Jacobi form input has leading Fourier series
$$
q^{-1}+\zeta_1^{\pm 2}+\zeta_2^{\pm 2} - \zeta_1^{\pm 1} - \zeta_2^{\pm 1} + 2k +O(q). 
$$
By Lemma \ref{lem:AC-criterion}, such a product does not exist, a contradiction. 
\item $A_{1,9}^*$: The proof is similar to Case (2). 
\item $A_{1,6}^*$: Suppose $m>2$. Then the quasi pullback of $F$ yields a reflective automorphic product on $2U\oplus 3A_1(6)$ whose Jacobi form input has leading Fourier series
$$
q^{-1}+\zeta_1^{\pm 2}+\zeta_2^{\pm 2}+\zeta_3^{\pm 2} - \zeta_1^{\pm 1} - \zeta_2^{\pm 1}-\zeta_3^{\pm 1} + 2k +O(q). 
$$
A direct calculation shows that such a product does not exist, although $2U\oplus 3A_1(6)$ has reflective automorphic products, a contradiction. 
\item $A_{1,4}^*$: Suppose $m>3$. Then the quasi pullback of $F$ yields a reflective automorphic product on $2U\oplus 4A_1(4)$ whose Jacobi form input has leading Fourier series
$$
q^{-1}+\zeta_1^{\pm 2}+\zeta_2^{\pm 2}+\zeta_3^{\pm 2} + \zeta_4^{\pm 2} - \zeta_1^{\pm 1} - \zeta_2^{\pm 1}-\zeta_3^{\pm 1} -\zeta_4^{\pm 1} + 2k +O(q). 
$$
A direct calculation shows that such a product does not exist, although $2U\oplus 4A_1(4)$ has reflective automorphic products, a contradiction. 
\item $A_{1,3}^*$: Suppose $m>4$. Then the quasi pullback of $F$ yields a reflective automorphic product on $2U\oplus 5A_1(3)$ whose Jacobi form input has leading Fourier series
$$
q^{-1}+\sum_{j=1}^5 (\zeta_j^{\pm 2} - \zeta_j^{\pm 1}) + 2k +O(q). 
$$
A direct calculation shows that such a product does not exist, whereas a similar automorphic product does exist on $2U\oplus 4A_1(3)$.
\item $A_{1,2}^*$: Suppose $m>6$. Then the quasi pullback of $F$ yields a reflective automorphic product on $2U\oplus 7A_1(2)$ whose Jacobi form input has leading Fourier series
$$
q^{-1}+\sum_{j=1}^7 (\zeta_j^{\pm 2} - \zeta_j^{\pm 1}) + 2k +O(q). 
$$
A direct calculation shows that such a product does not exist, whereas a similar automorphic product does exist on $2U\oplus 6A_1(2)$. 
\end{enumerate}
The discussion above concludes the proof. 
\end{proof}

We then record some constraints on the irreducible components of $\mathfrak{g}$ of type $\mathfrak{osp}_{1|2n}$ with $n\geq 2$.  

\begin{lemma}\label{lem:C_n-component}
If $C_{n,l}^*$ occurs as an irreducible component of $\mathfrak{g}$, then the possible values of the pair $(n,l)$ are as follows:
\begin{align*}
&(2,10),& &(2,6),& &(2,5),& &(n,4), \; n \leq 3,& &(n,3), \; n\leq 6,& &(n,2), \; n \leq 9,& &(n,1).&     
\end{align*}
In addition, each of $C_{2,10}^*$, $C_{2,6}^*$, and $C_{2,5}^*$ has multiplicity at most $1$.  
\end{lemma}

\begin{proof}
Let $F$ be a singular automorphic product on $2U\oplus L$ with $L$ canonical. Assume that the associated Lie superalgebra $\mathfrak{g}$ has an irreducible component of type $C_{n,l}^*$ with multiplicity $m$. Then $L=nmA_1(l)\oplus K$ for some $K$. The quasi pullback of $F$ along $2A_1(l)\hookrightarrow L$ gives a reflective automorphic product on $2U\oplus 2A_1(l)$ whose Jacobi form input has leading Fourier series
$$
q^{-1} + \zeta_1^{\pm 2}+\zeta_2^{\pm 2} +\zeta_1^{\pm 1}\zeta_2^{\pm 1} + (2n-5)\cdot(\zeta_1^{\pm 1}+\zeta_2^{\pm 1}) + 2k +O(q).
$$
Regarding this product, the parameter $a$ in Part (2) of Lemma \ref{lem:C_n-criterion} is given by $a=2n-5$, which is an odd integer. Therefore, the possible values of $l$ are as follows:
\begin{enumerate}
\item $l=10$: $2n-5=-1$ and thus $n=2$. Therefore, only $C_{2,10}^*$ is possible.  
\item $l=6$: Suppose $n\geq 3$. Then the quasi pullback of $F$ yields a reflective automorphic product on $2U\oplus 3A_1(6)$ whose Jacobi form input has leading Fourier series
$$
q^{-1}+\sum_{i=1}^3\Big(\zeta_i^{\pm 2} + (2n-7)\zeta_i^{\pm 1}\Big) + \sum_{1\leq i<j\leq 3} \zeta_{i}^{\pm 1}\zeta_{j}^{\pm 1} + 2k + O(q). 
$$
By direct computation, such a product does not exist. Therefore, only $C_{2,6}^*$ is possible.  
\item $l=5$: $2n-5=-1$ and thus $n=2$. Therefore, only $C_{2,5}^*$ is possible.  
\item $l=4$: Suppose $n>3$. Then the quasi pullback of $F$ produces a reflective automorphic product on $2U\oplus 4A_1(4)$ whose Jacobi form has leading Fourier series
$$
q^{-1}+\sum_{i=1}^4\big(\zeta_i^{\pm 2} + a\zeta_i^{\pm 1}\big) + \sum_{1\leq i<j\leq 4} \zeta_{i}^{\pm 1}\zeta_{j}^{\pm 1} + 2k + O(q) 
$$
with $a=2n-9$. However, by direct computation, such a product exists if and only if $a=2$. Therefore, only the case $n\leq 3$ is possible. 
\item $l=3$: Suppose $n>3$. Then the quasi pullback of $F$ produces a reflective automorphic product on $2U\oplus 4A_1(3)$ whose Jacobi form has leading Fourier series
$$
q^{-1}+\sum_{i=1}^4\big(\zeta_i^{\pm 2} + a\zeta_i^{\pm 1}\big) + \sum_{1\leq i<j\leq 4} \zeta_{i}^{\pm 1}\zeta_{j}^{\pm 1} + 2k + O(q) 
$$
with $a=2n-9$. By direct computation, such a product exists if and only if $-1\leq a\leq 4$. Therefore, only the case $n\leq 6$ is possible. 
\item $l=2$: Suppose $n>9$. Then we have
$$
L=10A_1(2)\oplus K_1 < 3A_1(2)\oplus D_7 \oplus K_1, \quad \text{for some $K_1$.}
$$
This overlattice contains the forbidden component (53) in \cite[\S 11.1]{SWW23}, that is, $3A_1(2)\oplus D_7$.
\end{enumerate}

We now determine the multiplicity $m$ of $C_{2,l}^*$ for $l=10,6,5$. Suppose $m\geq 2$. For $C_{2,10}^*$, we have
$$
L=4A_1(10)\oplus K_1 < 2A_1(2)\oplus 2A_1(5) \oplus K_1, \quad \text{for some $K_1$}.
$$
This overlattice contains the forbidden component (7) in \cite[\S 11.1]{SWW23}, that is, $2A_1(2)\oplus A_1(5)$.

For $C_{2,6}^*$, the quasi pullback of $F$ along $3A_1(6)\hookrightarrow L=4A_1(6)\oplus K_1$ for some $K_1$ gives a reflective automorphic product on $2U\oplus 3A_1(6)$ whose Jacobi form input has leading Fourier series
$$
q^{-1}+(\zeta_{1}^{\pm 1}\zeta_{2}^{\pm 1}+\zeta_1^{\pm 2}+\zeta_2^{\pm 2} - \zeta_1^{\pm 1} - \zeta_2^{\pm 1})+(\zeta_3^{\pm 2} + \zeta_3^{\pm 1}) + 2k + O(q). 
$$
By direct computation, such a form does not exist.

For $C_{2,5}^*$, the quasi pullback of $F$ gives a reflective automorphic product on $2U\oplus 3A_1(5)$ whose Jacobi form input has leading Fourier series
$$
q^{-1}+(\zeta_{1}^{\pm 1}\zeta_{2}^{\pm 1}+\zeta_1^{\pm 2}+\zeta_2^{\pm 2} - \zeta_1^{\pm 1} - \zeta_2^{\pm 1})+(\zeta_3^{\pm 2} + \zeta_3^{\pm 1}) + 2k + O(q). 
$$
Obstruction theory leads to the non-existence of such a form. 
\end{proof}

We now determine when $A_{1,t}^*$ and $C_{n,s}^*$ can both occur as components of $\mathfrak{g}$.

\begin{lemma}\label{lem:AC-component}
Assume that $A_{1,t}^*$ is an irreducible component of $\mathfrak{g}$. 
If $t=2,3,4,6,9,12,36$, then $C_{n,s}^*$ is never an irreducible component of $\mathfrak{g}$ for $n\geq 2$ and $s\geq 1$. 
If $t=1$, then $C_{n,s}^*$ is an irreducible component of $\mathfrak{g}$ only for $n=4$ and $s=3$. 
\end{lemma}

\begin{proof}
Suppose that both $A_{1,t}^*$ and $C_{n,s}^*$ are irreducible components of $\mathfrak{g}$. Let $F$ denote the reflective automorphic product of singular weight on $2U\oplus L$. Then we have $L=A_1(t)\oplus nA_1(s) \oplus K$ for some $K$. The quasi pullback of $F$ along $A_1(t)\oplus A_1(s) < L$ yields a reflective automorphic product $\tilde{F}$ on $2U\oplus A_1(t)\oplus A_1(s)$ whose Jacobi form input has leading Fourier series
$$
q^{-1}+(\zeta_1^{\pm 2}-\zeta_1^{\pm 1}) + \big(\zeta_2^{\pm 2} + (2n-3)\zeta_2^{\pm 1}\big)+2k+O(q), \quad 3s=(2n+1)t. 
$$
Combining Lemma \ref{lem:A-component} with Lemma \ref{lem:AC-criterion}, we need to discuss the following cases. 
\begin{enumerate}
\item When $t=4,6,9,12,36$, $\tilde{F}$ does not exist, so $C_{n,s}^*$ is never an irreducible component of $\mathfrak{g}$.  
\item When $t=3$, we have $s=2n+1$, so $s=9$ and $n=4$. This contradicts Lemma \ref{lem:C_n-component}. 
\item When $t=2$, we have $3s=4n+2$, so $s=10$ and $n=7$. This contradicts Lemma \ref{lem:C_n-component}. 
\item When $t=1$, we have $3s=2n+1$, so $s=3$ and $n=4$, or $s=5$ and $n=7$. The first case is possible, while the second case contradicts Lemma \ref{lem:C_n-component}. 
\end{enumerate}
We have thus proved the desired lemma. 
\end{proof}

\subsection{Excluding the remaining extraneous solutions}
Using Theorem \ref{th:upper-bound-singular}, we first eliminate $179$ of the $1106$ extraneous solutions: those with nontrivial odd parts and $\rank(\mathfrak{g})>18$. Lemma \ref{lem:A-component} then removes $579$ of the remaining $927$ solutions with $\rank(\mathfrak{g})\leq 18$. Lemmas \ref{lem:C_n-component} and \ref{lem:AC-component} further reduce the list to $156$ candidates. The goal of this subsection is to eliminate these $156$ cases and then complete the proof of Theorem \ref{th:classification-anti-symmetric}. Our approach is uniform, following the strategy described in Section \ref{subsec:forbidden}, but requires a case-by-case analysis. 

\subsubsection{Rank 5} There are \textbf{4} extraneous solutions of rank $5$:
\begin{enumerate} 
\item $A_{1,36}^*\oplus A_{1,48}\oplus B_{3,120}$: $C=1/24$, $\bP=A_1(36)\oplus A_1(12)\oplus 3A_1(60)$. We have $L<\bP$, and $\bP$ contains the forbidden component $A_1(60)$; see Lemma \ref{lem:A1}. 
\item $A_{1,36}^*\oplus A_{1,48}\oplus C_{3,96}$: $C=1/24$, $\bP = A_1(36)\oplus A_1(12)\oplus A_3'(192)$. Since $A_3'(8) < A_1 \oplus 2A_1(2)$, we have $L<\bP<A_1(36)\oplus A_1(12)\oplus A_1(24)\oplus 2A_1(48)$. This overlattice contains the forbidden component $A_1(48)$; see Lemma \ref{lem:A1}. 
\item $A_{1,36}^*\oplus B_{2,72}\oplus G_{2,96}$: $C=1/24$, $\bP = 3A_1(36)\oplus A_2(96)$. Then $L<\bP<3A_1(36)\oplus A_2(24)$. This overlattice contains the forbidden component (5) in \cite[\S 11.1]{SWW23}, that is, $A_2(24)$. 
\item $A_{1,36}^*\oplus A_{4,120}$: $C=1/24$, $\bP = A_1(36)\oplus A_4'(120)$. Since $A_4'(5)<A_4$, we have $L<\bP < A_1\oplus K$, where $K$ is a maximal even overlattice of $A_4(6)$ and $\mathrm{disc}(K)=180$ (here, $K$ is unique up to genus). We check that $A_1\oplus K$ is a forbidden component. 
\end{enumerate}

\subsubsection{Rank 6} There are \textbf{16} extraneous solutions of rank $6$.
\begin{enumerate}[resume]
\item  $A_{1,9}^*\oplus 2A_{1,12}\oplus B_{3,30}$: $C=1/6$, $\bP = A_1(9) \oplus 2A_1(3)\oplus 3A_1(15)$. The overlattice $\bP$ of $L$ contains the forbidden component $A_1(9)\oplus A_1(15)$; see Lemma \ref{lem:A1+A1}.  
\item $A_{1,9}^*\oplus 2A_{1,12}\oplus C_{3,24}$: $C=1/6$,  $\bP = A_1(9)\oplus 2A_1(3)\oplus A_3'(48)$. Recall that $A_1\oplus A_1(3)<A_2$ and $A_3'(8)<3A_1$. Then $L<\bP<A_1\oplus 2A_1(3)\oplus 3A_1(6)<A_2\oplus A_1(3)\oplus 3A_1(6)$. This overlattice contains the forbidden component (16) in \cite[\S 11.1]{SWW23}, that is, $2A_1(6)\oplus A_2$. 
\item $A_{1,9}^*\oplus A_{1,12}\oplus B_{2,18}\oplus G_{2,24}$: $C=1/6$, $\bP = A_1(9)\oplus A_1(3)\oplus 2A_1(9)\oplus A_2(24)$. This overlattice contains the forbidden component (5) in \cite[\S 11.1]{SWW23}, that is, $A_2(24)$. 
\item $A_{1,9}^*\oplus A_{1,12}\oplus A_{4,30}$: $C=1/6$, $\bP = A_1(9) \oplus A_1(3) \oplus A_4'(30)$. Then $L<\bP<A_1\oplus A_1(3)\oplus A_4(6)<A_1\oplus A_1(3)\oplus K$, where $K$ is a maximal even overlattice of $A_4(6)$ and $\mathrm{disc}(K)=180$. This overlattice contains the forbidden component $A_1\oplus K$ as in Case (4).
\item $2A_{1,6}^*\oplus 2G_{2,16}$: $C=1/4$, $\bP = 2A_1(6) \oplus 2A_2(16)$.    This overlattice contains the forbidden component (4) in \cite[\S 11.1]{SWW23}, that is, $A_2(16)$.
\item $2A_{1,6}^*\oplus D_{4,24}$: $C=1/4$, $\bP = 2A_1(6) \oplus D_4(12)$. Recall that $D_4(3)<A_2\oplus A_2(2)$. Then $L<\bP<2A_1(6)\oplus A_2\oplus A_2(2)$. This overlattice contains the forbidden component (16) in \cite[\S 11.1]{SWW23}, that is, $2A_1(6)\oplus A_2$. 
\item $A_{1,6}^*\oplus A_{2,12} \oplus B_{3,20}$: $C=1/4$, $\bP = A_1(6)\oplus A_2(4)\oplus 3A_1(10)$. Note that $2A_1(5)<2A_1$ and $A_1(6)\oplus A_2(4)\oplus A_1(10)<A_4$. Then $L<\bP<A_4\oplus 2A_1(2)$. This overlattice contains the forbidden component (29) in \cite[\S 11.1]{SWW23}, that is, $2A_1(2)\oplus A_4$. 
\item $A_{1,6}^*\oplus A_{2,12}\oplus C_{3,16}$: $C=1/4$, $\bP = A_1(6)\oplus A_2(4) \oplus A_3'(32)$. Note that $L<\bP$ and $A_3'(32)<A_1(4)\oplus 2A_1(8)<A_1\oplus 2A_1(8)$. The last lattice is the forbidden component (6) in \cite[\S 11.1]{SWW23}. By Lemma \ref{lem:overlattice}, $A_3'(32)$ is an anti-symmetric forbidden component, which means that $2U\oplus A_3'(32)$ does not have a reflective automorphic product with nonzero $q^{-1}e_0$ coefficient in its input. 
\item $A_{1,6}^*\oplus G_{2,16}\oplus A_{3,16}$: $C=1/4$, $\bP = A_1(6)\oplus A_2(16)\oplus A_3'(16)$. The overlattice $\bP$ of $L$ contains the forbidden component (4) in \cite[\S 11.1]{SWW23}, that is, $A_2(16)$. 
\item $A_{1,8}\oplus C_{2,10}^*\oplus B_{3,20}$: $C=1/4$, $\bP = A_1(2)\oplus 5A_1(10)$. Recall that $2A_1(2)<2A_1$ and $2A_1(5)<2A_1$. Then $L<\bP< 3A_1(2)\oplus 2A_1(5)\oplus A_1(10)$. This overlattice contains the forbidden component (7) in \cite[\S 11.1]{SWW23}, that is, $2A_1(2)\oplus A_1(5)$.  
\item $A_{1,8}\oplus C_{2,10}^*\oplus C_{3,16}$: $C=1/4$, $\bP = A_1(2)\oplus 2A_1(10)\oplus A_3'(32)$. The overlattice $\bP$ of $L$ contains the anti-symmetric forbidden component $A_3'(32)$; see Case (12). 
\item $C_{2,10}^*\oplus B_{2,12}\oplus G_{2,16}$: $C=1/4$, $\bP=2A_1(10)\oplus 2A_1(6)\oplus A_2(16)$.  The overlattice $\bP$ of $L$ contains the forbidden component (4) in \cite[\S 11.1]{SWW23}, that is, $A_2(16)$. 
\item $C_{2,10}^*\oplus A_{4,20}$: $C=1/4$, $\bP = 2A_1(10)\oplus A_4'(20) < 2A_1(2)\oplus A_4$. The overlattice $2A_1(2)\oplus A_4$ is the forbidden component (29) in \cite[\S 11.1]{SWW23}. 
\item $A_{1,3}^*\oplus G_{2,8}\oplus B_{3,10}$: $C=1/2$,  $\bP = A_1(3)\oplus A_2(8)\oplus 3A_1(5)$. Since $2A_1(5)<2A_1$ and $A_1(3)\oplus A_1<A_2$, we have  $L<\bP<A_2\oplus A_2(8)\oplus A_1(5)\oplus A_1$. This overlattice contains the forbidden component $A_2\oplus A_2(8)$. 
\item $A_{1,3}^*\oplus G_{2,8}\oplus C_{3,8}$: $C=1/2$, $\bP = A_1(3)\oplus A_2(8)\oplus A_3'(16)$. Since $A_3'(8)<3A_1$, $2A_1(2)<2A_1$ and $A_1(3)\oplus A_1<A_2$, we have 
$$
L<\bP<3A_1(2)\oplus A_1(3)\oplus A_2(8)<A_2\oplus A_2(8)\oplus A_1\oplus A_1(2).
$$
This overlattice contains the forbidden component $A_2\oplus A_2(8)$. 
\item $A_{1,3}^*\oplus A_{5,12}$: $C=1/2$, $\bP = A_1(3)\oplus A_5'(12)$. Since $A_5'(12) < A_1\oplus A_2\oplus A_2(4)$, we have $L<\bP< A_1(3)\oplus A_1\oplus  A_2\oplus A_2(4)$. This overlattice contains the forbidden component (22) labeled in \cite[\S 11.1]{SWW23}, that is, $A_1(3)\oplus A_2\oplus A_2(4)$. 
\end{enumerate}

\subsubsection{Rank 7} There are \textbf{7} extraneous solutions of rank $7$.
\begin{enumerate}[resume]
\item $A_{1,36}^*\oplus 3A_{1,48}\oplus A_{3,96}$: $C=1/24$, $\bP= A_1(36)\oplus 3A_1(12)\oplus A_3'(96)$. By Lemma \ref{lem:A1+A1}, the overlattice $\bP$ of $L$ contains the forbidden component $A_1(12)\oplus A_1(36)$. 
\item $A_{1,36}^*\oplus 2A_{1,48}\oplus A_{2,72}\oplus B_{2,72}$: $C=1/24$, $\bP = A_1(36)\oplus 2A_1(12)\oplus A_2(24)\oplus 2A_1(36)$. The overlattice $\bP$ contains the forbidden component labeled (5) in \cite[\S 11.1]{SWW23}, that is,  $A_2(24)$. 
\item $A_{1,36}^*\oplus 3A_{2,72}$: $C=1/24$, $\bP = A_1(36)\oplus 3A_2(24)$.  The overlattice $\bP$ of $L$ contains the forbidden component labeled (5) in \cite[\S 11.1]{SWW23}, that is,  $A_2(24)$. 
\item $A_{1,12}^*\oplus 4A_{1,16}\oplus G_{2,32}$: $C=1/8$, $\bP = A_1(12)\oplus 4A_1(4)\oplus A_2(32)$. Since $A_2(4)<A_2$, $A_1(4)<A_1$, and $A_1(3)\oplus A_1<A_2$, we have $L<\bP<A_2\oplus 3A_1\oplus A_2(8)$. This overlattice contains the forbidden component $A_2\oplus A_2(8)$. 
\item $A_{1,12}^*\oplus 2A_{1,16}\oplus 2B_{2,24}$: $C=1/8$, $\bP = A_1(12)\oplus 2A_1(4)\oplus 4A_1(12)$. Since $3A_1(3) < A_1 \oplus A_2(2)$,  $A_1(4)<A_1$, and $A_1(3)\oplus A_1<A_2$, we have $L<\bP<A_1\oplus 2A_2\oplus A_2(8)$. This overlattice contains the forbidden component $A_2\oplus A_2(8)$. 
\item $A_{1,12}^*\oplus A_{1,16}\oplus A_{2,24}\oplus A_{3,32}$: $C=1/8$, $\bP = A_1(12)\oplus A_1(4)\oplus A_2(8)\oplus A_3'(32)$. We have $L<\bP < A_2\oplus A_2(8)\oplus A_3'(32)$.  This overlattice contains the forbidden component $A_2\oplus A_2(8)$. 
\item $A_{1,12}^*\oplus 2A_{2,24}\oplus B_{2,24}$: $C=1/8$, $\bP = A_1(12)\oplus 2A_2(8)\oplus 2A_1(12)$. Since $3A_1(3) < A_1 \oplus A_2(2)$ and $2A_2(2)<2A_2$, we have  $L<\bP< A_1(4)\oplus 3A_2(8)<A_1(4)\oplus 2A_2\oplus A_2(8)$. This overlattice contains the forbidden component $A_2\oplus A_2(8)$. 
\end{enumerate}

\subsubsection{Rank 8} There are \textbf{32} extraneous solutions of rank $8$.
\begin{enumerate}[resume]
\item $A_{1,9}^*\oplus 4A_{1,12}\oplus A_{3,24}$: $C=1/6$, $\bP= A_1(9)\oplus 4A_1(3)\oplus A_3'(24)$. Since $A_3'(8)<3A_1$ and $4A_1(3)<4A_1<D_4$, we have $L<\bP<A_1(9)\oplus 3A_1(3)\oplus D_4$. This overlattice contains the forbidden component $3A_1(3)\oplus D_4$.  
\item $A_{1,9}^*\oplus 3A_{1,12}\oplus A_{2,18}\oplus B_{2,18}$: $C=1/6$, $\bP = A_1(9)\oplus 3A_1(3)\oplus A_2(6)\oplus 2A_1(9)$. Since $A_1(9)<A_1$, we have $L<\bP<3A_1\oplus 3A_1(3)\oplus A_2(6)$.  This overlattice contains the forbidden component labeled (10) in \cite[\S 11.1]{SWW23}, that is, $A_1\oplus A_2(6)$. 
\item $A_{1,9}^*\oplus A_{1,12}\oplus 3A_{2,18}$: $C=1/6$, $\bP = A_1(9)\oplus A_1(3)\oplus 3A_2(6)$. Since $A_1(9)<A_1$, we have $L<\bP < A_1\oplus A_1(3)\oplus 3A_2(6)$.  This overlattice contains the forbidden component labeled (10) in \cite[\S 11.1]{SWW23}, that is, $A_1\oplus A_2(6)$. 
\item $2A_{1,6}^*\oplus2A_{1,8}\oplus A_{2,12}\oplus G_{2,16}$: $C=1/4$, $\bP = 2A_1(6)\oplus 2A_1(2)\oplus A_2(4)\oplus A_2(16)$. The overlattice $\bP$ contains the forbidden component labeled (4) in \cite[\S 11.1]{SWW23}, that is, $A_2(16)$. 
\item $2A_{1,6}^*\oplus A_{1,8}\oplus B_{2,12}\oplus A_{3,16}$: $C=1/4$, $\bP= 2A_1(6)\oplus A_1(2)\oplus 2A_1(6)\oplus A_3'(16)$. Since $4A_1(3)<D_4(3)<D_4<\ZZ^4$, we have $L<\bP < 4A_1\oplus A_1(2)\oplus A_3'(16)$. This overlattice contains the forbidden component labeled (18) in \cite[\S 11.1]{SWW23}, that is, $A_1\oplus A_3'(16)$. 
\item $2A_{1,6}^*\oplus A_{2,12}\oplus 2B_{2,12}$: $C=1/4$, $\bP = 2A_1(6)\oplus A_2(4)\oplus 4A_1(6)$. Since $4A_1(6)<4A_1$, $2A_1(2)<2A_1$, and $A_1\oplus A_1(3)<A_2$, we have $L<\bP < 4A_1\oplus 2A_1(3)\oplus A_2(4) < 3A_1\oplus A_1(3)\oplus A_2\oplus A_2(4)$. This overlattice contains the forbidden component labeled (22) in \cite[\S 11.1]{SWW23}, that is, $A_1(3)\oplus A_2\oplus A_2(4)$. 
\item $A_{1,6}^*\oplus 5A_{1,8}\oplus G_{2,16}$: $C=1/4$, $\bP = A_1(6)\oplus 5A_1(2)\oplus A_2(16)$.  The overlattice $\bP$ of $L$ contains the forbidden component labeled (4) in \cite[\S 11.1]{SWW23}, that is, $A_2(16)$. 
\item $A_{1,6}^*\oplus 3A_{1,8}\oplus 2B_{2,12}$: $C=1/4$, $\bP = A_1(6)\oplus 3A_1(2)\oplus 4A_1(6)$. Since $2A_1(2)<2A_1$ and $A_1\oplus A_1(3)<A_2$, we have $L<\bP< 3A_1(6)\oplus 2A_1(3)\oplus 2A_1\oplus A_1(2)<3A_1(6)\oplus 2A_2\oplus A_1(2)$. This overlattice contains the forbidden component labeled (16) in \cite[\S 11.1]{SWW23}, that is, $2A_1(6)\oplus A_2$. 
\item $A_{1,6}^*\oplus 2A_{1,8}\oplus A_{2,12}\oplus A_{3,16}$: $C=1/4$, $\bP = A_1(6)\oplus 2A_1(2)\oplus A_2(4)\oplus A_3'(16)$. Since $2A_1(2)<2A_1$, we have $L<\bP < A_1(6)\oplus 2A_1\oplus A_2(4)\oplus A_3'(16)$.  This overlattice contains the forbidden component labeled (18) in \cite[\S 11.1]{SWW23}, that is, $A_1\oplus A_3'(16)$. 
\item $A_{1,6}^*\oplus A_{1,8}\oplus 2A_{2,12}\oplus B_{2,12}$: $C=1/4$, $\bP= A_1(6)\oplus A_1(2)\oplus 2A_2(4)\oplus 2A_1(6)$. Since $A_2(4)<A_2$ and $2A_1(2)<2A_1$, we have $L<\bP < A_1(6)\oplus A_1(2)\oplus A_2\oplus A_2(4)\oplus 2A_1(3)$.  This overlattice contains the forbidden component labeled (22) in \cite[\S 11.1]{SWW23}, that is, $A_1(3)\oplus A_2\oplus A_2(4)$. 
\item $3A_{1,8}\oplus C_{2,10}^*\oplus A_{3,16}$: $C=1/4$, $\bP = 3A_1(2)\oplus 2A_1(10)\oplus A_3'(16)$. Since $2A_1(2)<2A_1$ and $2A_1(5)<2A_1$, we have $L<\bP < A_1(2)\oplus 4A_1\oplus A_3'(16)$.  This overlattice contains the forbidden component labeled (18) in \cite[\S 11.1]{SWW23}, that is, $A_1\oplus A_3'(16)$. 
\item $2A_{1,8}\oplus C_{2,10}^*\oplus A_{2,12}\oplus B_{2,12}$: $C=1/4$, $\bP= 2A_1(2)\oplus 2A_1(10)\oplus A_2(4)\oplus 2A_1(6)$. Since $2A_1(2)<2A_1$, we have $L<\bP<2A_1(2)\oplus 2A_1(5)\oplus A_2(4)\oplus 2A_1(6)$. This overlattice contains the forbidden component labeled (7) in \cite[\S 11.1]{SWW23}, that is, $2A_1(2)\oplus A_1(5)$. 
\item $C_{2,10}^*\oplus 3A_{2,12}$: $C=1/4$, $\bP = 2A_1(10)\oplus 3A_2(4)$. Since $A_2(4)<A_1\oplus A_1(3)<A_2$, we have $L<2A_1(10)\oplus A_1\oplus A_1(3)\oplus A_2\oplus A_2(4)$.  This overlattice contains the forbidden component labeled (22) in \cite[\S 11.1]{SWW23}, that is, $A_1(3)\oplus A_2\oplus A_2(4)$. 
\item $2A_{1,3}^*\oplus 2A_{1,4}\oplus 2G_{2,8}$: $C=1/2$, $\bP = 2A_1(3)\oplus 2A_1\oplus 2A_2(8)$. Since $A_1(3)\oplus A_1<A_2$, we have $L<\bP<2A_2\oplus 2A_2(8)$. This overlattice contains the forbidden component $A_2\oplus A_2(8)$. 
\item $2A_{1,3}^*\oplus 2A_{1,4}\oplus D_{4,12}$: $C=1/2$, $\bP = 2A_1(3)\oplus 2A_1\oplus D_4(6)$. Since $D_4(3)<A_2\oplus A_2(2)$ and $A_1\oplus A_1(3)<A_2$, we have $L<\bP < A_1\oplus A_1(3)\oplus A_2\oplus A_2(2)\oplus A_2(4)$. This overlattice contains the forbidden component labeled (22) in \cite[\S 11.1]{SWW23}, that is, $A_1(3)\oplus A_2\oplus A_2(4)$. 
\item $2A_{1,3}^*\oplus A_{1,4}\oplus B_{2,6}\oplus B_{3,10}$: $C=1/2$, $\bP= 2A_1(3)\oplus A_1\oplus 2A_1(3)\oplus 3A_1(5)$. Since $4A_1(3)<4A_1$, we have $L<\bP <5A_1\oplus 3A_1(5)$. This overlattice contains the forbidden component labeled (20) in \cite[\S 11.1]{SWW23}, that is, $3A_1\oplus 2A_1(5)$. 
\item $2A_{1,3}^*\oplus A_{1,4}\oplus B_{2,6}\oplus C_{3,8}$: $C=1/2$, $\bP= 2A_1(3)\oplus A_1\oplus 2A_1(3)\oplus A_3'(16)$. The overlattice $\bP$ of $L$ contains the forbidden component labeled (18) in \cite[\S 11.1]{SWW23}, that is, $A_1\oplus A_3'(16)$. 
\item $2A_{1,3}^*\oplus 2B_{2,6}\oplus G_{2,8}$: $C=1/2$, $\bP=6A_1(3)\oplus A_2(8)$. Since $4A_1(3)<4A_1$ and $A_1(3)\oplus A_1<A_2$, we have $L<\bP< 2A_2\oplus 2A_1\oplus A_2(8)$.  This overlattice contains the forbidden component $A_2\oplus A_2(8)$. 
\item $2A_{1,3}^*\oplus B_{2,6}\oplus A_{4,10}$: $C=1/2$, $\bP=4A_1(3)\oplus A_4'(10)$. Since $4A_1(3)<4A_1$ and $A_4'(5)<A_4$, we have $L<\bP< 4A_1\oplus A_4(2)$. This overlattice contains the forbidden component labeled (30) in \cite[\S 11.1]{SWW23}, that is, $2A_1\oplus A_4(2)$. 
\item $A_{1,3}^*\oplus 2A_{1,4}\oplus A_{2,6}\oplus B_{3,10}$: $C=1/2$,  $\bP = A_1(3)\oplus 2A_1\oplus A_2(2)\oplus 3A_1(5)$. Since $A_1(3)\oplus A_2(2)<3A_1$, we have $L<\bP<5A_1\oplus 3A_1(5)$.  This overlattice contains the forbidden component labeled (20) in \cite[\S 11.1]{SWW23}, that is, $3A_1\oplus 2A_1(5)$. 
\item $A_{1,3}^*\oplus 2A_{1,4}\oplus A_{2,6}\oplus C_{3,8}$: $C=1/2$, $\bP= A_1(3)\oplus 2A_1\oplus A_2(2)\oplus A_3'(16)$. This overlattice contains the forbidden component labeled (18) in \cite[\S 11.1]{SWW23}, that is, $A_1\oplus A_3'(16)$. 
\item $A_{1,3}^*\oplus 2A_{1,4}\oplus G_{2,8}\oplus A_{3,8}$: $C=1/2$, $\bP = A_1(3)\oplus 2A_1\oplus A_2(8)\oplus A_3'(8)$. Since $A_1(3)\oplus A_1<A_2$, we have $L<\bP<A_1\oplus A_2\oplus A_2(8)\oplus A_3'(8)$. This overlattice contains the forbidden component $A_2\oplus A_2(8)$. 
\item $A_{1,3}^*\oplus A_{1,4}\oplus A_{2,6}\oplus B_{2,6}\oplus G_{2,8}$: $C=1/2$,  $\bP=A_1(3)\oplus A_1\oplus A_2(2)\oplus 2A_1(3)\oplus A_2(8)$. Since $A_1(3)\oplus A_1<A_2$, we have $L<\bP<A_2\oplus A_2(2)\oplus 2A_1(3)\oplus A_2(8)$. This overlattice contains the forbidden component $A_2\oplus A_2(8)$. 
\item $A_{1,3}^*\oplus A_{1,4}\oplus A_{2,6}\oplus A_{4,10}$: $C=1/2$, $\bP=A_1(3)\oplus A_1\oplus A_2(2)\oplus A_4'(10)$. Since $A_1(3)\oplus A_2(2)<3A_1$ and $A_4'(5)<A_4$, we have $L<\bP < 4A_1\oplus A_4(2)$.  This overlattice contains the forbidden component labeled (30) in \cite[\S 11.1]{SWW23}, that is, $2A_1\oplus A_4(2)$. 
\item $A_{1,3}^*\oplus 2B_{2,6}\oplus A_{3,8}$: $C=1/2$,  $\bP = 5A_1(3)\oplus A_3'(8)$. Since $4A_1(3)<4A_1$ and $A_1\oplus A_1(3)<A_2$, we have $L<\bP< 3A_1\oplus A_2\oplus A_3'(8)$.  This overlattice contains the forbidden component labeled (27) in \cite[\S 11.1]{SWW23}, that is, $A_1\oplus A_2\oplus A_3'(8)$. 
\item $3A_{1,4}\oplus C_{2,5}^*\oplus B_{3,10}$: $C=1/2$,  $\bP = 3A_1\oplus 2A_1(5)\oplus 3A_1(5)$. The overlattice $\bP$ of $L$ contains the forbidden component labeled (20) in \cite[\S 11.1]{SWW23}, that is, $3A_1\oplus 2A_1(5)$. 
\item $3A_{1,4}\oplus C_{2,5}^*\oplus C_{3,8}$: $C=1/2$, $\bP = 3A_1\oplus 2A_1(5)\oplus A_3'(16)$. The overlattice $\bP$ of $L$ contains the forbidden component labeled (18) in \cite[\S 11.1]{SWW23}, that is, $A_1\oplus A_3'(16)$. 
\item $2A_{1,4}\oplus C_{2,5}^*\oplus B_{2,6}\oplus G_{2,8}$: $C=1/2$, $\bP = 2A_1\oplus 2A_1(5)\oplus 2A_1(3)\oplus A_2(8)$. Since $A_1(3)\oplus A_1<A_2$, we have $L<\bP<2A_2\oplus 2A_1(5)\oplus A_2(8)$. This overlattice contains the forbidden component $A_2\oplus A_2(8)$. 
\item $2A_{1,4}\oplus C_{2,5}^*\oplus A_{4,10}$: $C=1/2$, $\bP = 2A_1\oplus 2A_1(5)\oplus A_4'(10)$. Since $A_4'(5)<A_4$, we have $L<\bP<2A_1\oplus 2A_1(5)\oplus A_4(2)$.  This overlattice contains the forbidden component labeled (30) in \cite[\S 11.1]{SWW23}, that is, $2A_1\oplus A_4(2)$. 
\item $C_{2,5}^*\oplus 2A_{2,6}\oplus G_{2,8}$: $C=1/2$, $\bP = 2A_1(5)\oplus 2A_2(2)\oplus A_2(8)$. Since $2A_2(2)<2A_2$, we have $L<\bP < 2A_1(5)\oplus 2A_2\oplus A_2(8)$. This overlattice contains the forbidden component $A_2\oplus A_2(8)$.  
\item $C_{2,5}^*\oplus 3B_{2,6}$: $C=1/2$, $\bP = 2A_1(5)\oplus 6A_1(3)$. Since $4A_1(3)<4A_1$, we have $L<\bP<4A_1\oplus 2A_1(3)\oplus 2A_1(5)$.  This overlattice contains the forbidden component labeled (20) in \cite[\S 11.1]{SWW23}, that is, $3A_1\oplus 2A_1(5)$. 
\item $C_{2,5}^*\oplus 2A_{3,8}$: $C=1/2$, $\bP = 2A_1(5)\oplus 2A_3'(8)$. Since $A_3'(8)<3A_1$, we have $L<\bP< 6A_1\oplus 2A_1(5)$. This overlattice contains the forbidden component labeled (20) in \cite[\S 11.1]{SWW23}, that is, $3A_1\oplus 2A_1(5)$. 
\end{enumerate}

\subsubsection{Rank 9} There are \textbf{6} extraneous solutions of rank $9$.
\begin{enumerate}[resume]
\item $A_{1,36}^*\oplus 8A_{1,48}$: $C=1/24$, $\bP=A_1(36)\oplus 8A_1(12)$. The overlattice $\bP$ of $L$ contains the forbidden component $A_1(12)\oplus A_1(36)$ by Lemma \ref{lem:A1+A1}. 
\item $A_{1,12}^*\oplus 6A_{1,16}\oplus A_{2,24}$: $C=1/8$,  $\bP = A_1(12)\oplus 6A_1(4)\oplus A_2(8)$. Since $A_1(4)<A_1$ and $A_1(3)\oplus A_1<A_2$, we have $L<\bP<5A_1\oplus A_2\oplus A_2(8)$.  This overlattice contains the forbidden component $A_2\oplus A_2(8)$. 
\item $3A_{1,4}^*\oplus 3B_{2,8}$: $C=3/8$, $\bP = 9A_1(4)$. Since $A_1(4)<A_1$ and $4A_1<D_4$, we have $L<\bP < 4A_1\oplus A_1(4)\oplus D_4$.  This overlattice contains the forbidden component labeled (38) in \cite[\S 11.1]{SWW23}, that is, $3A_1\oplus A_1(4)\oplus D_4$. 
\item $A_{1,4}^*\oplus 4A_{2,8}$: $C=3/8$, $L=A_1(4)\oplus K$ with $4A_2(8)<K<4A_2'(8)$. Since $L(3/8)$ is integral, $L(1/8)$ and $K(1/8)$ are all integral. Then $4A_2<K(1/8)<4A_2'$ leads to $K(1/8)<E_8$. Since $E_8(4)<D_8<\ZZ^8$ and $4A_1<D_4$, we have $L<A_1(4)\oplus E_8(8)<A_1(4)\oplus 4A_1\oplus D_4$. This overlattice contains the forbidden component labeled (38) in \cite[\S 11.1]{SWW23}, that is, $3A_1\oplus A_1(4)\oplus D_4$. 
\item $3C_{2,4}^*\oplus B_{3,8}$: $C=5/8$, $\bP=9A_1(4)$. Thus, $L<\bP< 4A_1\oplus A_1(4)\oplus D_4$. This overlattice contains the forbidden component labeled (38) in \cite[\S 11.1]{SWW23}, that is, $3A_1\oplus A_1(4)\oplus D_4$. 
\item $3C_{3,4}^*$: $C=7/8$, $L=\bP=9A_1(4)<4A_1\oplus A_1(4)\oplus D_4$. This overlattice contains the forbidden component labeled (38) in \cite[\S 11.1]{SWW23}, that is, $3A_1\oplus A_1(4)\oplus D_4$. 
\end{enumerate}

\subsubsection{Rank 10} There are \textbf{40} extraneous solutions of rank $10$. 
\begin{enumerate}[resume]
\item $A_{1,9}^*\oplus 9A_{1,12}$: $C=1/6$, $\bP = A_1(9)\oplus 9A_1(3)$. Since $4A_1(3)<D_4$, we have $L<\bP<A_1(9)\oplus 5A_1(3)\oplus D_4$. This overlattice contains the forbidden component $3A_1(3)\oplus D_4$. 
\item $2A_{1,6}^*\oplus 6A_{1,8}\oplus B_{2,12}$: $C=1/4$, $\bP= 2A_1(6)\oplus 6A_1(2)\oplus 2A_1(6)$. Since $A_1(3)\oplus A_1 < A_2$ and $2A_2(2)<2A_2$, we have $L<\bP<4A_1(2)\oplus 2A_1(6)\oplus 2A_2$. This overlattice contains the forbidden component labeled (16) in \cite[\S 11.1]{SWW23}, that is, $2A_1(6)\oplus A_2$. 
\item $2A_{1,6}^*\oplus 4A_{1,8} \oplus 2A_{2,12}$: $C=1/4$, $\bP= 2A_1(6)\oplus 4A_1(2)\oplus 2A_2(4)$. Since $A_2(4)<A_2$, we have $L<\bP < 4A_1(2)\oplus 2A_1(6)\oplus 2A_2$.  This overlattice contains the forbidden component labeled (16) in \cite[\S 11.1]{SWW23}, that is, $2A_1(6)\oplus A_2$. 
\item $A_{1,6}^*\oplus 7A_{1,8}\oplus A_{2,12}$: $C=1/4$, $\bP=A_1(6)\oplus 7A_1(2)\oplus A_2(4)$. Since $2A_1(2)<2A_1$, we have $L<\bP<A_1(6)\oplus A_1(2)\oplus 6A_1\oplus A_2(4)<A_1(6)\oplus A_1(2)\oplus D_6\oplus A_2(4)$. This overlattice contains the forbidden component labeled (43) in \cite[\S 11.1]{SWW23}, that is,  $A_2(4)\oplus D_6$. 
\item $8A_{1,8}\oplus C_{2,10}^*$: $C=1/4$, $\bP = 8A_1(2)\oplus 2A_1(10)$. Since $2A_1(2)<2A_1$, we have $L<\bP<8A_1(2)\oplus 2A_1(5)$.  This overlattice contains the forbidden component labeled (7) in \cite[\S 11.1]{SWW23}, that is, $2A_1(2)\oplus A_1(5)$. 
\item $3A_{1,3}^*\oplus 4A_{1,4}\oplus B_{3,10}$: $C=1/2$, $\bP=3A_1(3)\oplus 4A_1\oplus 3A_1(5)$.  The overlattice $\bP$ of $L$ contains the forbidden component labeled (20) in \cite[\S 11.1]{SWW23}, that is, $3A_1\oplus 2A_1(5)$. 
\item $3A_{1,3}^*\oplus 4A_{1,4}\oplus C_{3,8}$: $C=1/2$, $\bP=3A_1(3)\oplus 4A_1\oplus A_3'(16)$.  The overlattice $\bP$ of $L$ contains the forbidden component labeled (18) in \cite[\S 11.1]{SWW23}, that is,  $A_1\oplus A_3'(16)$. 
\item $3A_{1,3}^*\oplus 3A_{1,4}\oplus B_{2,6}\oplus G_{2,8}$: $C=1/2$, $\bP= 3A_1(3)\oplus 3A_1\oplus 2A_1(3)\oplus A_2(8)$. Since $A_1(3)\oplus A_1<A_2$, we have $L<\bP<2A_1(3)\oplus 3A_2\oplus A_2(8)$. This overlattice contains the forbidden component $A_2\oplus A_2(8)$. 
\item $3A_{1,3}^*\oplus 3A_{1,4}\oplus A_{4,10}$: $C=1/2$, $\bP = 3A_1(3)\oplus 3A_1\oplus A_4'(10)$. Since $A_4'(5)<A_4$, we have $L<\bP<3A_1(3)\oplus 3A_1\oplus A_4(2)$. This overlattice contains the forbidden component labeled (30) in \cite[\S 11.1]{SWW23}, that is, $2A_1\oplus A_4(2)$. 
\item $3A_{1,3}^*\oplus A_{1,4}\oplus 2A_{2,6}\oplus G_{2,8}$: $C=1/2$, $\bP = 3A_1(3)\oplus A_1\oplus 2A_2(2)\oplus A_2(8)$. Since $2A_2(2)<2A_2$, we have $L<3A_1(3)\oplus A_1\oplus 2A_2\oplus A_2(8)$. This overlattice contains the forbidden component $A_2\oplus A_2(8)$. 
\item $3A_{1,3}^*\oplus A_{1,4}\oplus 3B_{2,6}$: $C=1/2$, $\bP=3A_1(3)\oplus A_1\oplus 6A_1(3)$. Since $4A_1(3)<D_4$, we have $L<\bP<A_1\oplus 5A_1(3)\oplus D_4$.  This overlattice contains the forbidden component $3A_1(3)\oplus D_4$. 
\item $3A_{1,3}^*\oplus A_{1,4}\oplus 2A_{3,8}$: $C=1/2$, $\bP = 3A_1(3)\oplus A_1\oplus 2A_3'(8)$. Since $A_3'(8)<3A_1$ and $A_1\oplus A_1(3)<A_2$, we have $L<\bP<A_1\oplus 3A_2\oplus A_3'(8)$.  This overlattice contains the forbidden component labeled (27) in \cite[\S 11.1]{SWW23}, that is, $A_1\oplus A_2\oplus A_3'(8)$. 
\item $3A_{1,3}^*\oplus A_{2,6}\oplus B_{2,6}\oplus A_{3,8}$: $C=1/2$, $\bP = 3A_1(3)\oplus A_2(2)\oplus 2A_1(3)\oplus A_3'(8)$. Since $A_1(3)\oplus A_2(2)<3A_1$ and $A_1(3)\oplus A_1<A_2$, we have $L<\bP<A_1\oplus 2A_1(3)\oplus 2A_2\oplus A_3'(8)$. This overlattice contains the forbidden component labeled (27) in \cite[\S 11.1]{SWW23}, that is, $A_1\oplus A_2\oplus A_3'(8)$. 
\item $2A_{1,3}^*\oplus 4A_{1,4}\oplus A_{2,6}\oplus G_{2,8}$: $C=1/2$, $\bP = 2A_1(3)\oplus 4A_1\oplus A_2(2)\oplus A_2(8)$. Since $A_1(3)\oplus A_1<A_2$, we have $L<\bP<2A_1\oplus 2A_2\oplus A_2(2)\oplus A_2(8)$. This overlattice contains the forbidden component $A_2\oplus A_2(8)$. 
\item $2A_{1,3}^*\oplus 3A_{1,4}\oplus B_{2,6}\oplus A_{3,8}$: $C=1/2$, $\bP = 2A_1(3)\oplus 3A_1\oplus 2A_1(3)\oplus A_3'(8)$. Since $A_1(3)\oplus A_1<A_2$, we have $L<\bP<A_1\oplus 2A_1(3)\oplus 2A_2\oplus A_3'(8)$. This overlattice contains the forbidden component labeled (27) in \cite[\S 11.1]{SWW23}, that is, $A_1\oplus A_2\oplus A_3'(8)$. 
\item $2A_{1,3}^*\oplus 2A_{1,4}\oplus A_{2,6}\oplus 2B_{2,6}$: $C=1/2$, $\bP = 2A_1(3)\oplus 2A_1\oplus A_2(2)\oplus 4A_1(3)$. Since $3A_1(3)<A_1\oplus A_2(2)$ and $4A_1<D_4$, we have $L<\bP<3A_2(2)\oplus D_4$.  This overlattice contains the forbidden component labeled (39) in \cite[\S 11.1]{SWW23}, that is, $2A_2(2)\oplus D_4$. 
\item $2A_{1,3}^*\oplus A_{1,4}\oplus 2A_{2,6}\oplus A_{3,8}$: $C=1/2$, $\bP = 2A_1(3)\oplus A_1\oplus 2A_2(2)\oplus A_3'(8)$. Since $2A_2(2)<2A_2$, we have $L<\bP<2A_1(3)\oplus A_1\oplus 2A_2\oplus A_3'(8)$.  This overlattice contains the forbidden component labeled (27) in \cite[\S 11.1]{SWW23}, that is, $A_1\oplus A_2\oplus A_3'(8)$. 
\item $2A_{1,3}^*\oplus 3A_{2,6}\oplus B_{2,6}$: $C=1/2$, $\bP = 2A_1(3)\oplus 3A_2(2)\oplus 2A_1(3)$. Since $4A_1(3)<D_4(3)<D_4$, we have $L<\bP<3A_2(2)\oplus D_4$.  This overlattice contains the forbidden component labeled (39) in \cite[\S 11.1]{SWW23}, that is, $2A_2(2)\oplus D_4$. 
\item $A_{1,3}^*\oplus 7A_{1,4}\oplus G_{2,8}$: $C=1/2$, $\bP = A_1(3)\oplus 7A_1\oplus A_2(8)$. Since $A_1(3)\oplus A_1<A_2$, we have $L<\bP<6A_1\oplus A_2\oplus A_2(8)$. This overlattice contains the forbidden component $A_2\oplus A_2(8)$. 
\item $A_{1,3}^*\oplus 5A_{1,4}\oplus 2B_{2,6}$: $C=1/2$, $\bP = A_1(3)\oplus 5A_1\oplus 4A_1(3)$. Since $4A_1<D_4$, $A_1\oplus A_1(3)<A_2$ and $3A_1(3)<A_1\oplus A_2(2)$, we have $L<\bP< 2A_2\oplus A_2(2)\oplus D_4$.  This overlattice is precisely the forbidden component labeled (50) in \cite[\S 11.1]{SWW23}, that is, $2A_2\oplus A_2(2)\oplus D_4$. 
\item $A_{1,3}^*\oplus 4A_{1,4}\oplus A_{2,6}\oplus A_{3,8}$: $C=1/2$, $\bP = A_1(3)\oplus 4A_1\oplus A_2(2)\oplus A_3'(8)$. Since $A_1\oplus A_1(3)<A_2$, we have $L<\bP<A_2\oplus 3A_1\oplus A_2(2)\oplus A_3'(8)$. This overlattice contains the forbidden component labeled (27) in \cite[\S 11.1]{SWW23}, that is, $A_1\oplus A_2\oplus A_3'(8)$. 
\item $A_{1,3}^*\oplus 3A_{1,4}\oplus 2A_{2,6}\oplus B_{2,6}$: $C=1/2$, $\bP = A_1(3)\oplus 3A_1\oplus 2A_2(2)\oplus 2A_1(3)$. Since $3A_1(3)<A_1\oplus A_2(2)$ and $4A_1<D_4$, we have $L<\bP<3A_2(2)\oplus D_4$. This overlattice contains the forbidden component labeled (39) in \cite[\S 11.1]{SWW23}, that is, $2A_2(2)\oplus D_4$. 
\item $A_{1,3}^*\oplus A_{1,4}\oplus 4A_{2,6}$: $C=1/2$, $\bP = A_1(3)\oplus A_1\oplus 4A_2(2)$. Since $A_1(3)\oplus A_2(2)<3A_1$ and $4A_1<D_4$, we have $L<\bP<3A_2(2)\oplus D_4$.  This overlattice contains the forbidden component labeled (39) in \cite[\S 11.1]{SWW23}, that is, $2A_2(2)\oplus D_4$. 
\item $5A_{1,4}\oplus C_{2,5}^*\oplus A_{3,8}$: $C=1/2$, $\bP = 5A_1\oplus 2A_1(5)\oplus A_3'(8)$.  The overlattice $\bP$ of $L$ contains the forbidden component labeled (20) in \cite[\S 11.1]{SWW23}, that is, $3A_1\oplus 2A_1(5)$. 
\item $4A_{1,4}\oplus C_{2,5}^*\oplus A_{2,6}\oplus B_{2,6}$: $C=1/2$, $\bP = 4A_1\oplus 2A_1(5)\oplus A_2(2)\oplus 2A_1(3)$.  The overlattice $\bP$ of $L$ contains the forbidden component labeled (20) in \cite[\S 11.1]{SWW23}, that is, $3A_1\oplus 2A_1(5)$. 
\item $2A_{1,4}\oplus C_{2,5}^*\oplus 3A_{2,6}$: $C=1/2$, $\bP = 2A_1\oplus 2A_1(5)\oplus 3A_2(2)$. The overlattice $\bP$ of $L$ contains the forbidden component $2A_1\oplus A_1(5)\oplus A_2(2)$. 
\item $4A_{1,2}^*\oplus B_{2,4}\oplus D_{4,8}$: $C=3/4$, $\bP = 6A_1(2)\oplus D_4(4)$. Since $3A_1(2)\oplus D_4(4)<D_7$, we have $L<\bP<3A_1(2)\oplus D_7$. This overlattice is the forbidden component labeled (53) in \cite[\S 11.1]{SWW23}, that is, $3A_1(2)\oplus D_7$. 
\item $2A_{1,2}^*\oplus 4B_{2,4}$: $C=3/4$, $\bP = 10A_1(2)$. Since $2A_1(2)<2A_1$ and $A_1(2)\oplus 6A_1<D_7$, we have $L<\bP< 3A_1(2)\oplus D_7$.  This overlattice is the forbidden component labeled (53) in \cite[\S 11.1]{SWW23}, that is, $3A_1(2)\oplus D_7$. 
\item $3C_{2,2}^*\oplus C_{4,4}$: $C=5/4$, $\bP = 6A_1(2)\oplus D_4(4)$. Thus, $L<\bP<10A_1(2)<3A_1(2)\oplus D_7$. This overlattice is the forbidden component labeled (53) in \cite[\S 11.1]{SWW23}, that is, $3A_1(2)\oplus D_7$. 
\item $2C_{2,2}^*\oplus 2B_{3,4}$: $C=5/4$, $\bP = 10A_1(2) < 3A_1(2)\oplus D_7$.  This overlattice is the forbidden component labeled (53) in \cite[\S 11.1]{SWW23}, that is, $3A_1(2)\oplus D_7$. 
\item $C_{2,2}^*\oplus 2A_{4,4}$: $C=5/4$, $L=2A_1(2)\oplus K$ with $2A_4(4)<K<2A_4'(4)$. Since both $K$ and $K(5/4)$ are integral, $K(1/4)$ is also integral. Note that $2A_4<K(1/4)<2A_4'$, so $K(1/4)$ is even and thus $K(1/4)<E_8$. It follows that $L<2A_1(2)\oplus E_8(4)$. As 2-elementary even lattices, $2A_1\oplus E_8(2)$ and $10A_1$ lie in the same genus. Therefore, $2U\oplus 2A_1(2)\oplus E_8(4)\cong 2U\oplus 10A_1(2)$. We rule out this case as in the previous case.  
\item $5A_{1,1}^*\oplus B_{5,6}$: $C=3/2$, $\bP = 5A_1\oplus 5A_1(3)$. Since $3A_1(3)<A_1\oplus A_2(2)$, $A_1\oplus A_1(3)<A_2$ and $4A_1<D_4$, we have $L<\bP<D_4\oplus 2A_2\oplus A_2(2)$.  This overlattice is the forbidden component labeled (50) in \cite[\S 11.1]{SWW23}, that is, $D_4\oplus 2A_2\oplus A_2(2)$. 
\item $5A_{1,1}^*\oplus C_{5,4}$: $C=3/2$, $\bP = 5A_1\oplus D_5'(8)$. Since $D_5'(8)<A_1\oplus D_4(4)<A_1\oplus 4A_1(2)$, we have $L<\bP<6A_1\oplus 4A_1(2)<D_7\oplus 3A_1(2)$. This overlattice is the forbidden component labeled (53) in \cite[\S 11.1]{SWW23}, that is, $3A_1(2)\oplus D_7$. 
\item $2C_{3,2}^*\oplus B_{4,4}$: $C=7/4$, $\bP = 10A_1(2)$.  We rule out this case as in Case (95).  
\item $2C_{5,2}^*$: $C=11/4$, $\bP = 10A_1(2)$. We rule out this case as in Case (95).  
\item $4A_{1,3}^*\oplus A_{1,4}\oplus A_{2,6}\oplus B_{3,10}$: $C=1/2$, $\bP=4A_1(3)\oplus A_1\oplus A_2(2)\oplus 3A_1(5)$. Since $A_1(3)\oplus A_1<A_2$ and $A_2\oplus A_2(2)<D_4$, we have $L<\bP<3A_1(3)\oplus D_4\oplus 3A_1(5)$. This overlattice contains the forbidden component $3A_1(3)\oplus D_4$. 
\item $4A_{1,3}^*\oplus A_{1,4}\oplus A_{2,6}\oplus C_{3,8}$: $C=1/2$, $\bP=4A_1(3)\oplus A_1\oplus A_2(2)\oplus A_{3}'(16)$. We rule out it as in Case (101) by the forbidden component $3A_1(3)\oplus D_4$.
\item $4A_{1,3}^*\oplus A_{1,4}\oplus A_{3,8}\oplus G_{2,8}$: $C=1/2$, $\bP=4A_1(3)\oplus A_1\oplus A_3'(8)\oplus A_2(8)$. Since $A_1\oplus A_1(3)<A_2$, we have $L<\bP<3A_1(3)\oplus A_2\oplus A_3'(8)\oplus A_2(8)$. This overlattice contains the forbidden component $A_2\oplus A_2(8)$. 
\item $4A_{1,3}^*\oplus A_{2,6}\oplus A_{4,10}$: $C=1/2$, $\bP=4A_1(3)\oplus A_2(2)\oplus A_4'(10)$. Since $A_4'(5)<A_4$ and $4A_1(3)<4A_1$, we have $L<\bP<4A_1\oplus A_2(2)\oplus A_4(2)$. This overlattice contains the forbidden component labeled (30) in \cite[\S 11.1]{SWW23}, that is, $2A_1\oplus A_4(2)$. 
\item $4A_{1,3}^*\oplus A_{2,6}\oplus B_{2,6}\oplus G_{2,8}$: $C=1/2$, $\bP=6A_1(3)\oplus A_2(2)\oplus A_{2}(8)$. Since $A_1(3)\oplus A_2(2)<3A_1$ and $A_1(3)\oplus A_1<A_2$, we have $L<\bP<2A_1(3)\oplus 3A_2\oplus A_2(8)$. This overlattice contains the forbidden component $A_2\oplus A_2(8)$. 
\end{enumerate}

\subsubsection{Rank 12} There are \textbf{28} extraneous solutions of rank $12$.
\begin{enumerate}[resume] 
\item $3A_{1,3}^*\oplus 6A_{1,4}\oplus A_{3,8}$: $C=1/2$, $\bP = 3A_1(3)\oplus 6A_1\oplus A_3'(8)$. Since $A_1\oplus A_1(3)<A_2$, we have $L<\bP < 3A_1\oplus 3A_2\oplus A_3'(8)$. This overlattice contains the forbidden component labeled (46) in \cite[\S 11.1]{SWW23}, that is, $3A_1\oplus 3A_2$. 
\item $3A_{1,3}^*\oplus 5A_{1,4}\oplus A_{2,6}\oplus B_{2,6}$: $C=1/2$, $\bP = 3A_1(3)\oplus 5A_1\oplus A_2(2)\oplus 2A_1(3)$. Since $3A_1(3)<A_1\oplus A_2(2)$ and $4A_1<D_4$, we have $L<\bP<2A_1\oplus 2A_1(3)\oplus 2A_2(2)\oplus D_4$.  This overlattice contains the forbidden component labeled (39) in \cite[\S 11.1]{SWW23}, that is, $2A_2(2)\oplus D_4$. 
\item $3A_{1,3}^*\oplus 3A_{1,4}\oplus 3A_{2,6}$: $C=1/2$, $\bP = 3A_1(3)\oplus 3A_1\oplus 3A_2(2)$. Since $2A_2(2)<2A_2$ and $3A_1(3)<A_1\oplus A_2(2)$, we have  $L<\bP<4A_1\oplus 4A_2$.  This overlattice contains the forbidden component labeled (46) in \cite[\S 11.1]{SWW23}, that is, $3A_1\oplus 3A_2$. 
\item $2A_{1,3}^*\oplus 8A_{1,4}\oplus B_{2,6}$: $C=1/2$, $\bP = 2A_1(3)\oplus 8A_1\oplus 2A_1(3)$. Since $A_1\oplus A_1(3)<A_2$, we have $L<\bP <4A_1\oplus 4A_2$. This overlattice contains the forbidden component labeled (46) in \cite[\S 11.1]{SWW23}, that is, $3A_1\oplus 3A_2$. 
\item $2A_{1,3}^*\oplus 6A_{1,4}\oplus 2A_{2,6}$: $C=1/2$, $\bP = 2A_1(3)\oplus 6A_1\oplus 2A_2(2)$. Since $A_1\oplus A_1(3)<A_2$ and $2A_2(2)<2A_2$, we have $L<\bP<4A_1\oplus 4A_2$.  This overlattice contains the forbidden component labeled (46) in \cite[\S 11.1]{SWW23}, that is, $3A_1\oplus 3A_2$. 
\item $A_{1,3}^*\oplus 9A_{1,4}\oplus A_{2,6}$: $C=1/2$, $\bP = A_1(3)\oplus 9A_1\oplus A_2(2)$. Since $8A_1<E_8$, we have $L<\bP<A_1(3)\oplus A_1\oplus E_8\oplus A_2(2)$.  This overlattice contains the forbidden component labeled (58) in \cite[\S 11.1]{SWW23}, that is, $A_1\oplus E_8\oplus A_2(2)$. 
\item $10A_{1,4}\oplus C_{2,5}^*$: $C=1/2$, $\bP = 10A_1\oplus 2A_1(5)$.  The overlattice $\bP$ of $L$ contains the forbidden component labeled (20) in \cite[\S 11.1]{SWW23}, that is, $3A_1\oplus 2A_1(5)$. 
\item $6A_{1,2}^*\oplus A_{2,4}\oplus D_{4,8}$: $C=3/4$, $L=6A_1(2)\oplus K$ with $A_2(4)\oplus D_4(8)<K<A_2'(4)\oplus D_4(4)$. Since both $K$ and $K(3/4)$ are integral, $K(1/4)$ is also integral. The bounds $A_2\oplus D_4(2)<K(1/4)<A_2'\oplus D_4$ yield that the level of $L$ is divisible by $3$. However, $N_\mathfrak{g}=8$, which forces the level of $L$ to be $8$, a contradiction. 
\item $4A_{1,2}^*\oplus A_{2,4}\oplus 3B_{2,4}$: $C=3/4$, $L=4A_1(2)\oplus K$ with $A_2(4)\oplus 6A_1(4)<K<A_2'(4)\oplus 6A_1(2)$. Similarly to Case (113), $K(1/4)$ is integral. The bounds $A_2\oplus 6A_1<K(1/4)<A_2'\oplus \ZZ^6$ yield that the level of $L$ is divisible by $3$. However, $N_\mathfrak{g}=8$, which forces the level of $L$ to be $8$, a contradiction. 
\item $2A_{1,2}^*\oplus 5A_{2,4}$: $C=3/4$, $L=2A_1(2)\oplus K$ with $5A_2(4)<K<5A_2'(4)$. Similarly to Case (113), $K(1/4)$ is integral. The bounds $5A_2<K(1/4)<5A_2'$ yield that the level of $L$ is divisible by $3$. However, $N_\mathfrak{g}=8$, which forces the level of $L$ to be $8$, a contradiction. 
\item $8A_{1,1}^*\oplus F_{4,6}$: $C=3/2$, $L=\bP=8A_1\oplus D_4(6)$. Since $D_4(3)<A_2\oplus A_2(2)$, we have $L<2A_1\oplus D_6\oplus A_2(2)\oplus A_2(4)$. This overlattice contains the forbidden component labeled (43) in \cite[\S 11.1]{SWW23}, that is, $A_2(4)\oplus D_6$. 
\item $4A_{1,1}^*\oplus 2D_{4,4}$: $C=3/2$, $\bP=4A_1\oplus 2D_4(2)$. Since $D_4(2)<4A_1$ and $6A_1<D_6$, we have $L<\bP<8A_1\oplus D_4(2)<2A_1\oplus D_6\oplus D_4(2)$. The overlattice $2A_1\oplus D_6\oplus D_4(2)$ is a forbidden component.
\item $4A_{1,1}^*\oplus D_{4,4}\oplus C_{4,3}^*$: $C=3/2$, $\bP=4A_1\oplus D_4(2)\oplus 4A_1(3)$. Since $D_4(2)<4A_1$ and $A_1\oplus A_1(3)<A_2$, we have $L<4A_1\oplus 4A_2$. This overlattice contains the forbidden component labeled (46) in \cite[\S 11.1]{SWW23}, that is, $3A_1\oplus 3A_2$. 
\item $4A_{1,1}^*\oplus 2C_{4,3}^*$: $C=3/2$, $\bP=4A_1\oplus 8A_1(3)$. Since $4A_1(3)<4A_1$ and $A_1(3)\oplus A_1<A_2$, we have $L<4A_1\oplus 4A_2$. This overlattice contains the forbidden component labeled (46) in \cite[\S 11.1]{SWW23}, that is, $3A_1\oplus 3A_2$. 
\item $2A_{1,1}^*\oplus 3B_{2,2}\oplus D_{4,4}$: $C=3/2$, $\bP=8A_1\oplus D_4(2)$.  Since $6A_1<D_6$, we have $L<\bP<2A_1\oplus D_6\oplus D_4(2)$. This overlattice is a forbidden component.
\item $2A_{1,1}^*\oplus 3B_{2,2}\oplus C_{4,3}^*$: $C=3/2$, $\bP=8A_1\oplus 4A_1(3)$. Since $A_1(3)\oplus A_1<A_2$, we have $L<4A_1\oplus 4A_2$. This overlattice contains the forbidden component labeled (46) in \cite[\S 11.1]{SWW23}, that is, $3A_1\oplus 3A_2$. 
\item $A_{1,1}^*\oplus 3A_{2,2}\oplus A_{5,4}$: $C=3/2$, $L=A_1\oplus K$ with $3A_2(2)\oplus A_5(4)<K<3A_2'(2)\oplus A_5'(4)$. Since both $K$ and $K(3/2)$ are integral, $K(1/2)$ is also integral. Note that 
$$
3A_2\oplus A_5(2)<K(1/2)<3A_2'\oplus A_5'(2).
$$
All maximal integral sublattices of $3A_2'\oplus A_5'(2)$ containing $3A_2\oplus A_5(2)$ lie in the genus of $2D_4\oplus \ZZ^3$. Therefore, we have 
\[
2U\oplus L<2U\oplus 4A_1\oplus 2D_4(2)<2U\oplus 8A_1\oplus D_4(2)<2U\oplus 2A_1\oplus D_6\oplus D_4(2).
\]
This overlattice is a forbidden component.
\item $4A_{2,2}\oplus C_{4,3}^*$: $C=3/2$, $L=4A_1(3)\oplus K$ with $4A_2(2)<K<4A_2'(2)$. Since both $K$ and $K(3/2)$ are integral, $K(1/2)$ is also integral. Note that $4A_2<K(1/2)<4A_2'$. Thus, $K(1/2)<E_8$ and then $L<4A_1(3)\oplus E_8(2)<A_1(3)\oplus A_1\oplus A_2(2)\oplus E_8$.  This overlattice contains the forbidden component labeled (58) in \cite[\S 11.1]{SWW23}, that is, $A_1\oplus A_2(2)\oplus E_8$.  
\item $3C_{2,1}^*\oplus D_{6,4}$: $C=5/2$, $\bP=6A_1\oplus D_6'(4)$. Since $2A_1\oplus D_4 < D_6$, we have $D_6'<2A_1'\oplus D_4'$ and thus $D_6'(4)<2A_1\oplus D_4(2)$. Therefore, $L<\bP<8A_1\oplus D_4(2)<2A_1\oplus D_6\oplus D_4(2)$. The overlattice $2A_1\oplus D_6\oplus D_4(2)$ is a forbidden component.
\item $2C_{2,1}^*\oplus 2C_{4,2}$: $C=5/2$, $\bP=4A_1\oplus 2D_4(2)$. Since $D_4(2)<4A_1$ and $6A_1<D_6$, we have $L<\bP<8A_1\oplus D_4(2)<2A_1\oplus D_6\oplus D_4(2)$. We rule out it, since the overlattice $2A_1\oplus D_6\oplus D_4(2)$ is a forbidden component. 
\item $C_{2,1}^*\oplus 2B_{3,2}\oplus C_{4,2}$: $C=5/2$, $\bP=8A_1\oplus D_4(2)$. We have $L<\bP<2A_1\oplus D_6\oplus D_4(2)$. This overlattice is a forbidden component. 
\item $2C_{3,1}^*\oplus C_{6,2}$: $C=7/2$, $\bP=6A_1\oplus D_6'(4)$. It is similar to Case (124). 
\item $C_{4,1}^*\oplus 2F_{4,2}$: $C=9/2$, $L=\bP=4A_1\oplus 2D_4(2)$. It is similar to Case (125). 
\item $4A_{1,3}^*\oplus 2A_{1,4}\oplus 2A_{2,6}\oplus B_{2,6}$: $C=1/2$, $\bP=6A_1(3)\oplus 2A_1\oplus 2A_2(2)$. Since $A_1(3)\oplus A_1<A_2$ and $A_2\oplus A_2(2)<D_4$, we have $L<\bP<4A_1(3)\oplus 2D_4$. The overlattice contains the forbidden component $3A_1(3)\oplus D_4$. 
\item $4A_{1,3}^*\oplus 3A_{1,4}\oplus A_{2,6}\oplus A_{3,8}$: $C=1/2$, $\bP=4A_1(3)\oplus 3A_1\oplus A_2(2)\oplus A_3'(8)$. As in Case (129), we have $L<\bP<3A_1(3)\oplus 2A_1\oplus D_4\oplus A_3'(8)$, so this case is excluded by the forbidden component $3A_1(3)\oplus D_4$. 
\item $4A_{1,3}^*\oplus 4A_{1,4}\oplus 2B_{2,6}$: $C=1/2$, $\bP=8A_1(3)\oplus 4A_1$. Since $4A_1<D_4$, we have $L<\bP<8A_1(3)\oplus D_4$. We eliminate it by the forbidden component $3A_1(3)\oplus D_4$. 
\item $4A_{1,3}^*\oplus 6A_{1,4}\oplus G_{2,8}$: $C=1/2$, $\bP=4A_1(3)\oplus 6A_1\oplus A_2(8)$. Since $A_1(3)\oplus A_1<A_2$, we have $L<\bP<2A_1\oplus 4A_2\oplus A_2(8)$. This overlattice contains the forbidden component $A_2\oplus A_2(8)$.  
\item $4A_{1,3}^*\oplus 4A_{2,6}$: $C=1/2$, $\bP=4A_1(3)\oplus 4A_2(2)$. Since $2A_2(2)<2A_2$ and $A_2\oplus A_2(2)<D_4$, we have $L<\bP<4A_1(3)\oplus 2D_4$. We exclude it by the forbidden component $3A_1(3)\oplus D_4$. 
\end{enumerate}

\subsubsection{Rank 14} There are \textbf{8} extraneous solutions of rank $14$.
\begin{enumerate}[resume]
\item $3A_{1,3}^*\oplus 11A_{1,4}$: $C=1/2$, $\bP=3A_1(3)\oplus 11A_1$. Since $A_1\oplus A_1(3)<A_2$, we have $L<\bP<3A_2\oplus 8A_1$. This overlattice contains the forbidden component labeled (46) in \cite[\S 11.1]{SWW23}, that is, $3A_1\oplus 3A_2$. 
\item $6A_{1,2}^*\oplus 2A_{2,4}\oplus 2B_{2,4}$: $C=3/4$, $L=6A_1(2)\oplus K$ with $2A_2(4)\oplus 4A_1(4)<K<2A_2'(4)\oplus 4A_1(2)$. Since both $K$ and $K(3/4)$ are integral, $K(1/4)$ is also integral. The bounds $2A_2\oplus 4A_1<K(1/4)<2A_2'\oplus \ZZ^4$ show that the level of $L$ is divisible by $3$. However, $N_\mathfrak{g}=8$, forcing the level of $L$ to be $8$, a contradiction. 
\item $5C_{2,2}^*\oplus A_{4,4}$: $C=5/4$, $L=10A_1(2)\oplus K$ with $A_4(4)<K<A_4'(4)$. Thus, $K=A_4(4)$, and the level of $L$ is divisible by $5$. However, $N_\mathfrak{g}=8$, forcing the level of $L$ to be $8$, a contradiction. 
\item $5A_{1,1}^*\oplus 2B_{2,2}\oplus A_{5,4}$: $C=3/2$, $L=5A_1\oplus K$ with $4A_1(2)\oplus A_5(4)<K<4A_1\oplus A_5'(4)$.  Since both $K$ and $K(3/2)$ are integral, $K(1/2)$ is also integral. The bounds $4A_1\oplus A_5(2)<K(1/2)<\ZZ^4\oplus A_5'(2)$ show that the level of $L$ is divisible by $3$.  However, $N_\mathfrak{g}=8$, forcing the level of $L$ to be $4$ or $8$, a contradiction. 
\item $4A_{1,1}^*\oplus A_{2,2}\oplus 2B_{2,2}\oplus D_{4,4}$: $C=3/2$, $L=4A_1\oplus K$ with $A_2(2)\oplus 4A_1(2)\oplus D_4(4)<K<A_2'(2)\oplus 4A_1\oplus D_4(2)$. Thus, the level of $L$ is divisible by $3$.  However, $N_\mathfrak{g}=4$, forcing the level of $L$ to be $4$, a contradiction. 
\item $4A_{1,1}^*\oplus A_{2,2}\oplus 2B_{2,2}\oplus C_{4,3}^*$: $C=3/2$, $L=4A_1\oplus 4A_1(3)\oplus K$ with $A_2(2)\oplus 4A_1(2)<K<A_2'(2)\oplus 4A_1$. Since both $K$ and $K(3/2)$ are integral, $K(1/2)$ is also integral. Note that $A_2\oplus 4A_1<K(1/2)<A_2'\oplus \ZZ^4$. The maximal integral sublattice of $A_2'\oplus \ZZ^4$ is $A_2\oplus \ZZ^4$, so $K<A_2(2)\oplus 4A_1$. Thus, $L<8A_1\oplus 4A_1(3)\oplus A_2(2)<4A_1\oplus 4A_2\oplus A_2(2)$. This overlattice contains the forbidden component labeled (46) in \cite[\S 11.1]{SWW23}, that is, $3A_1\oplus 3A_2$. 
\item $2A_{1,1}^*\oplus A_{2,2}\oplus 5B_{2,2}$: $C=3/2$, $L=2A_1\oplus K$ with $A_2(2)\oplus 10A_1(2)<K<A_2'(2)\oplus 10A_1$. The bounds show that the level of $L$ is divisible by $3$.  However, $N_\mathfrak{g}=4$, forcing the level of $L$ to be $4$, a contradiction. 
\item $4A_{1,3}^*\oplus 8A_{1,4}\oplus A_{2,6}$: $C=1/2$, $\bP=4A_1(3)\oplus 8A_1\oplus A_2(2)$. Since $4A_1<D_4$, we have $L<\bP<4A_1(3)\oplus 2D_4\oplus A_2(2)$. This overlattice contains the forbidden component $3A_1(3)\oplus D_4$. 
\end{enumerate}

\subsubsection{Rank 16} There are \textbf{8} extraneous solutions of rank $16$.
\begin{enumerate}[resume]
\item $7A_{1,1}^*\oplus A_{2,2}\oplus B_{2,2}\oplus A_{5,4}$: $C=3/2$, $L=7A_1\oplus K$ with 
$$
A_2(2)\oplus 2A_1(2)\oplus A_5(4)<K<A_2'(2)\oplus 2A_1\oplus A_5'(4).
$$
Since both $K$ and $K(3/2)$ are integral, $K(1/2)$ is also integral. Note that 
$$
A_2\oplus 2A_1\oplus A_5(2)<K(1/2)<A_2'\oplus \ZZ^2\oplus A_5'(2).
$$
The maximal even overlattices of $A_2\oplus 2A_1\oplus A_5(2)$ belong to the genus of $D_5\oplus 2A_2$, so the maximal integral overlattices are in the genus of $\ZZ^5 \oplus 2A_2$. Therefore, 
$$
2U\oplus L<2U\oplus 12A_1\oplus 2A_2(2)<2U\oplus 4A_1\oplus E_8\oplus 2A_2(2).
$$
This overlattice contains the forbidden component (58) in \cite[\S 11.1]{SWW23}, that is, $A_1\oplus A_2(2)\oplus E_8$. 
\item $6A_{1,1}^*\oplus 2A_{2,2}\oplus B_{2,2}\oplus D_{4,4}$: $C=3/2$, $L=6A_1\oplus K$ with 
$$
2A_2(2)\oplus 2A_1(2)\oplus D_4(4)<K<2A_2'(2)\oplus 2A_1\oplus D_4(2).
$$
Since both $K$ and $K(3/2)$ are integral, $K(1/2)$ is also integral. Note that $2A_2\oplus 2A_1\oplus D_4(2)<K(1/2)<2A_2'\oplus \ZZ^2\oplus D_4$. The maximal integral sublattice of $2A_2'\oplus \ZZ^2\oplus D_4$ is $2A_2\oplus \ZZ^2\oplus D_4$, so $K<2A_2(2)\oplus 2A_1\oplus D_4(2)$. Thus, 
$$
L<8A_1\oplus 2A_2(2)\oplus D_4(2)<4A_1\oplus E_8\oplus 2A_2(2),
$$
which contains the forbidden component (58) in \cite[\S 11.1]{SWW23}, that is, $A_1\oplus A_2(2)\oplus E_8$.  
\item $6A_{1,1}^*\oplus 2A_{2,2}\oplus B_{2,2}\oplus C_{4,3}^*$: $C=3/2$, $L=6A_1\oplus 4A_1(3)\oplus K$ with 
$$
2A_2(2)\oplus 2A_1(2)<K<2A_2'(2)\oplus 2A_1.
$$
Since $A_1\oplus A_1(3)<A_2$, we have $L<3A_1\oplus 3A_2\oplus A_1(3)\oplus K$. This overlattice contains the forbidden component labeled (46) in \cite[\S 11.1]{SWW23}, that is, $3A_1\oplus 3A_2$. 
\item $4A_{1,1}^*\oplus 2A_{2,2}\oplus 4B_{2,2}$: $C=3/2$, $L=4A_1\oplus K$ with 
$$
2A_2(2)\oplus 8A_1(2)<K<2A_2'(2)\oplus 8A_1.
$$
The integrality of both $K$ and $K(3/2)$ shows that $K(1/2)$ is integral. The maximal integral sublattice of $2A_2'\oplus \ZZ^8$ is $2A_2\oplus \ZZ^8$, so the bounds $2A_2\oplus 8A_1<K(1/2)<2A_2'\oplus \ZZ^8$ yield $K<2A_2(2)\oplus 8A_1$. Therefore, $L<12A_1\oplus 2A_2(2)<4A_1\oplus E_8\oplus 2A_2(2)$. This overlattice contains the forbidden component labeled (58) in \cite[\S 11.1]{SWW23}, that is, $A_1\oplus A_2(2)\oplus E_8$.  
\item $2A_{1,1}^*\oplus 6A_{2,2}\oplus B_{2,2}$:  $C=3/2$, $L=2A_1\oplus K$ with 
$$
6A_2(2)\oplus 2A_1(2)<K<6A_2'(2)\oplus 2A_1.
$$
Since $K$ and $K(3/2)$ are integral, $K(1/2)$ is also integral. The maximal integral overlattice of $6A_2\oplus 2A_1$ is in the genus of $E_8\oplus 2A_2\oplus \ZZ^2$. Since $6A_2\oplus 2A_1<K(1/2)<6A_2'\oplus \ZZ^2$, we conclude that $L$ has an even overlattice that lies in the genus of $4A_1\oplus E_8(2)\oplus 2A_2(2)$ and
$$
2U\oplus L < 2U\oplus 4A_1\oplus E_8(2)\oplus 2A_2(2)< 2U\oplus 4A_1\oplus E_8\oplus 2A_2(2),
$$
which contains the forbidden component labeled (58) in \cite[\S 11.1]{SWW23}, that is, $A_1\oplus A_2(2)\oplus E_8$. 
\item $4C_{2,1}^*\oplus A_{4,2}\oplus C_{4,2}$: $C=5/2$, $L=8A_1\oplus K$ with $A_4(2)\oplus 4A_1(2)<K<A_4'(2)\oplus D_4(2)$. The bounds show that the level of $L$ is divisible by $5$. However, $N_\mathfrak{g}=4$, forcing the level of $L$ to be $4$, a contradiction. 
\item $3C_{2,1}^*\oplus 2B_{3,2}\oplus A_{4,2}$: $C=5/2$, $L=6A_1\oplus K$ with $A_4(2)\oplus 6A_1(2)<K<A_4'(2)\oplus 6A_1$. The bounds show that the level of $L$ is divisible by $5$. However, $N_\mathfrak{g}=8$, forcing the level of $L$ to be $4$ or $8$, a contradiction. 
\item $2C_{2,1}^*\oplus 3A_{4,2}$: $C=5/2$, $L=4A_1\oplus K$ with $3A_4(2)<K<3A_4'(2)$. Thus, the level of $L$ is divisible by $5$. However, $N_\mathfrak{g}=4$, forcing the level of $L$ to be $4$, a contradiction. 
\end{enumerate}

\subsubsection{Rank 18} There are \textbf{7} extraneous solutions of rank $18$.
\begin{enumerate}[resume]
\item $9C_{2,2}^*$: $C=5/4$, $L=\bP=18A_1(2)$. Since $2A_1(2)<2A_1$ and $6A_1\oplus A_1(2)<D_7$, we have $L<11A_1(2)\oplus D_7$.  This overlattice contains the forbidden component labeled (53) in \cite[\S 11.1]{SWW23}, that is, $3A_1(2)\oplus D_7$. 
\item $9A_{1,1}^*\oplus 2A_{2,2}\oplus A_{5,4}$: $C=3/2$, $L=9A_1\oplus K$ with $2A_2(2)\oplus A_5(4)<K<2A_2'(2)\oplus A_5'(4)$. The bounds imply that the level of $L$ is divisible by $3$. However, $N_\mathfrak{g}=8$, forcing the level of $L$ to be $4$ or $8$, a contradiction. 
\item $8A_{1,1}^*\oplus 3A_{2,2}\oplus D_{4,4}$: $C=3/2$, $L=8A_1\oplus K$ with $3A_2(2)\oplus D_4(4)<K<3A_2'(2)\oplus D_4(2)$. The bounds imply that the level of $L$ is divisible by $3$. However, $N_\mathfrak{g}=4$, forcing the level of $L$ to be $4$, a contradiction. 
\item $8A_{1,1}^*\oplus 3A_{2,2}\oplus C_{4,3}^*$: $C=3/2$, $L=8A_1\oplus 4A_1(3)\oplus K$ with $3A_2(2)<K<3A_2'(2)$. Since $A_1\oplus A_1(3)<A_2$, we have $L<4A_1\oplus 4A_2\oplus K$. This overlattice contains the forbidden component labeled (46) in \cite[\S 11.1]{SWW23}, that is, $3A_1\oplus 3A_2$. 
\item $6A_{1,1}^*\oplus 3A_{2,2}\oplus 3B_{2,2}$: $C=3/2$, $L=6A_1\oplus K$ with $3A_2(2)\oplus 6A_1(2)<K<3A_2'(2)\oplus 6A_1$. The bounds show that the level of $L$ is divisible by $3$. However, $N_\mathfrak{g}=4$, forcing the level of $L$ to be $4$, a contradiction. 
\item $4A_{1,1}^*\oplus 7A_{2,2}$: $C=3/2$, $L=4A_1\oplus K$ with $7A_2(2)<K<7A_2'(2)$. Thus, the level of $L$ is divisible by $3$. However, $N_\mathfrak{g}=4$, forcing the level of $L$ to be $4$, a contradiction. 
\item $4C_{3,1}^*\oplus A_{6,2}$: $C=7/2$, $L=12A_1\oplus K$ with $A_6(2)<K<A_6'(2)$. Thus, the level of $L$ is divisible by $7$. However, $N_\mathfrak{g}=4$, forcing the level of $L$ to be $4$, a contradiction.  
\end{enumerate}

\section{Uniqueness of singular automorphic products}\label{sec:uniqueness}
This section is devoted to the proof of Theorem \ref{MTH:Uniqueness}. We will employ Part (3) of Theorem \ref{MTH:Lie structure} and Lemma \ref{lem:values of delta} to determine the level of the underlying canonical lattice $L$. 

We first treat the symmetric case, in which singular automorphic products give rise to $16$ Lie superalgebras $\mathfrak{g}$; see Tables \ref{tab:class-2}-\ref{tab:class-4}. The following is our main result. 

\begin{theorem}\label{th:unique-symmetric}
In the symmetric case, each $\mathfrak{g}$ determines a unique canonical lattice $L$ such that $2U\oplus L$ admits a singular automorphic product with Lie superalgebra structure $\mathfrak{g}$, and this product is unique. Moreover, these products are all modular for the full orthogonal groups. 
\end{theorem} 

\begin{proof} 
When $\mathfrak{g}$ has nontrivial odd parts, the uniqueness has been established in Theorem \ref{th:classification-symmetric}. It remains to consider the case where $\mathfrak{g}$ has trivial odd parts. There are $12$ cases. We first consider the $8$ cases associated with the holomorphic VOSA $F_{24}$, in which $C=1$. 
\begin{enumerate}
\item $A_{4,5}$: $N_\mathfrak{g}=5$, $\bQ=A_4(5)$, $\bP=A_4'(5)$. We derive from $N_\mathfrak{g}=5$ that the level of $L$ is $5$. Thus, the bounds $\bQ<L<\bP$ restrict to  $L=A_4'(5)$. 
\item $A_{1,2}\oplus B_{3,5}$: $N_\mathfrak{g}=20$, $\bQ=A_1(2)\oplus A_3(5)$, $\bP=\ZZ\oplus \ZZ^3(5)$. The maximal even sublattice $\bP^{\mathrm{ev}}$ of $\bP$ has genus $\II_{4,0}(2_{\II}^{+2}5^{+3})$. Note that $\bQ$ has discriminant $2^4\cdot 5^3$ and level $40$. Since the level of $L$ is $10$ or $20$, the bounds $\bQ<L<\bP^{\mathrm{ev}}$ lead to $L=\bP^{\mathrm{ev}}$. 
\item $A_{1,2}\oplus C_{3,4}$: $N_\mathfrak{g}=8$, $\bQ=A_1(2)\oplus 3A_1(4)$, $\bP=\ZZ\oplus A_3'(8)$. The maximal even sublattice $\bP^{\mathrm{ev}}$ of $\bP$ has genus $\II_{4,0}(2_3^{-1}4_1^{+1}8_{\II}^{-2})$. Note that $\bQ$ has discriminant $2^{11}$ and level $16$. Since the level of $L$ is $8$, the bounds $\bQ<L<\bP^{\mathrm{ev}}$ restrict to $L=\bP^{\mathrm{ev}}$. 
\item $B_{2,3}\oplus G_{2,4}$: $N_\mathfrak{g}=12$, $\bQ=2A_1(3)\oplus A_2(4)$, $\bP=\ZZ^2(3)\oplus A_2(4)$. The maximal even sublattice of $\bP$ is $\bP^{\mathrm{ev}}=2A_1(3)\oplus A_2(4)$. Therefore, $L=\bP^{\mathrm{ev}}=\bQ$.  
\item $3A_{2,3}$: $N_\mathfrak{g}=3$, $\bQ=3A_2(3)$, $\bP=3A_2$. Since the level of $L$ is $3$, the bounds $\bQ<L<\bP$ yield $L=E_6'(3)$ or $3A_2$. By \cite[Theorem 6.27]{Sch17}, the reflective automorphic products of singular weight on $2U\oplus 3A_2$ and $2U\oplus E_6'(3)$ are the same function. However, $E_6'(3)$ is not canonical for these products (see Remark \ref{rem:canonical-check}). Therefore, $L=3A_2$. 
\item $3A_{1,2}\oplus A_{3,4}$: $N_\mathfrak{g}=8$, $\bQ=3A_1(2)\oplus A_3(4)$, $\bP=\ZZ^3\oplus A_3'(4)$. The maximal even sublattice $\bP^{\mathrm{ev}}$ of $\bP$ has genus $\II_{6,0}(2_6^{+2}4_{\II}^{+2})$. The level $N$ of $L$ is $4$ or $8$. Observe that the Weyl vector $\rho$ lies in the dual of $\bP^{\mathrm{ev}}$, leading to $\rho\in U\oplus L'$. By Part (5) of Corollary \ref{Cor:relation to regular}, we have
$$
U\oplus L\cong U(N)\oplus K \quad  \text{for some $K$}. 
$$
The bounds $\bQ<L<\bP^{\mathrm{ev}}$, together with the level and splitting conditions, reduce the possibilities for $L$ to just three cases: 
\begin{enumerate}
\item  $L=\bP^{\mathrm{ev}}$ with $N=4$ and $K=D_6$;
\item $L=D_4(2)\oplus 2A_1$ with $N=4$ and $K=D_4\oplus 2A_1$;
\item $L=D_6'(4)$ with $N=4$ and $K=\bP^{\mathrm{ev}}$. 
\end{enumerate} 
In particular, the level of $L$ must be $4$. 
Using the algorithm at the end of Section \ref{subsec:reflective}, we find that $2U\oplus D_4(2)\oplus 2A_1$ has precisely $12$ reflective singular automorphic products, but $D_4(2)\oplus 2A_1$ is not canonical for these products (see Remark \ref{rem:canonical-check}). Similarly, we conclude that $2U\oplus D_6'(4)$ has precisely $20$ reflective singular automorphic products, but $D_6'(4)$ is not canonical for these products. Therefore, $L=\bP^{\mathrm{ev}}$. 
\item $2A_{1,2}\oplus A_{2,3}\oplus B_{2,3}$: $N_\mathfrak{g}=12$, $\bQ=2A_1(2)\oplus A_2(3)\oplus 2A_1(3)$, $\bP=\ZZ^2\oplus A_2\oplus \ZZ^2(3)$. The maximal even sublattice $\bP^{\mathrm{ev}}$ of $\bP$ has genus $\II_{6,0}(2_{\II}^{+2}3^{-3})$. The level $N$ of $L$ is $6$ or $12$. We verify that $\vec{B}$ is in the dual of $\bP^{\mathrm{ev}}$, so $\vec{B}\in L'$. By Part (5) of Corollary \ref{Cor:relation to regular}, we have 
$$
2U\oplus L\cong U\oplus U(N)\oplus K \quad \text{for some $K$}. 
$$
The bounds $\bQ<L<\bP^{\mathrm{ev}}$ together with the splitting condition imply that $L=\bP^{\mathrm{ev}}$.
\item $8A_{1,2}$: $N_\mathfrak{g}=4$, $\bQ=8A_1(2)$, $\bP=\ZZ^8$. The maximal even sublattice of $\bP$ is $\bP^{\mathrm{ev}}=D_8$. The level of $L$ is $2$ or $4$. When $L$ has level $2$, we conclude from \cite[Theorem 6.27]{Sch17} that $L=D_8$. We now suppose that the level of $L$ is $4$. Observe that the Weyl vector $\rho$ lies in the dual of $\bP^{\mathrm{ev}}=D_8$, leading to $\rho\in U\oplus L'$. By Part (5) of Corollary \ref{Cor:relation to regular}, we have
$$
U\oplus L\cong U(4)\oplus K \quad \text{for some $K$}.
$$
In addition, there is a real simple root $\alpha$ with $(\alpha,\alpha)=\frac{1}{2}$ and $(\rho,\alpha)=1/4$, otherwise the level of $L$ would be $2$ by the proof of Part (1) of Corollary \ref{Cor:relation to regular}. Thus, $(4\rho,\alpha)=1$, and the above $U(4)$ can be taken to be generated by $4\rho$ and $4(\rho-\alpha)$. From $8A_1(2)<L<D_8$, we conclude that $U(4)\oplus E_8$ is an even overlattice of $U(4)\oplus K\cong U\oplus L$. Then Lemma \ref{lem:sym-overlattice} implies that $U\oplus U(4)\oplus E_8$ would have a reflective automorphic product whose input has zero $q^{-1}e_0$-coefficient. By obstruction theory, we confirm that such products do not exist, although this lattice has a reflective automorphic product with nonzero $q^{-1}e_0$-coefficient in its input. Therefore, $L=D_8$. 
\end{enumerate}

We have thus proved the uniqueness of the underlying canonical lattices for the eight $\mathfrak{g}$ related to $F_{24}$. It remains to show the uniqueness of the corresponding singular automorphic products. For cases (1), (2), (5), (7), (8), the level of $L$ is squarefree. By \cite[Theorem 6.1]{DW21}, the singular products on $2U\oplus L_\mathfrak{g}$ with $L_\mathfrak{g}=\bP^{\mathrm{ev}}$ are unique, as they vanish along the mirrors of all possible reflections with the exception of $2$-reflections. In cases (3), (4), and (6), the uniqueness of the singular automorphic products follows directly from obstruction theory, since the lattice discriminants are not prohibitively large. We now turn to the $4$ exceptional cases.
\begin{enumerate}
\item $A_{1,16}$: $N_\mathfrak{g}=32$, $C=1/8$, $\bQ=A_1(16)$, $\bP=A_1(4)$. Since the level of $L$ is $16$ or $32$, the bounds $\bQ<L<\bP$ restrict to $L=\bP=A_1(4)$.  
\item $A_{2,9}$: $N_\mathfrak{g}=9$, $C=1/3$, $\bQ=A_2(9)$, $\bP=A_2(3)$. Since the level of $L$ is $9$, the bounds $\bQ<L<\bP$ lead to $L=\bP=A_2(3)$.
\item $2A_{1,8}$: $N_\mathfrak{g}=16$, $C=1/4$, $\bQ=2A_1(8)$, $\bP=2A_1(2)$. The level of $L$ is $8$ or $16$. Note that $L(1/4)$ is integral and bounded by $2A_1(2)<L(1/4)<\ZZ^2$. Since the level of $L(1/4)$ is $2$ or $4$, we have $L(1/4)=2A_1$ or $\ZZ^2$, and thus $L=2A_1(4)$ or $2A_1(2)$. By direct computation, there are reflective singular automorphic products on $2U\oplus 2A_1(4)$, but this underlying lattice is not canonical. Therefore, $L=\bP=2A_1(2)$.  
\item $4A_{1,4}$: $N_\mathfrak{g}=8$, $C=1/2$, $\bQ=4A_1(4)$, $\bP=4A_1$. The level of $L$ is $4$ or $8$. Note that $L(1/2)$ is integral and bounded by $4A_1(2)<L(1/2)<\ZZ^4$. We determine all possible $L(1/2)$ of level $2$ or $4$; there are exactly $6$ candidate $L$ up to genus. Although reflective singular automorphic products do occur on these lattices, they are apparently induced by a singular automorphic product on $2U\oplus \bP$, as the corresponding proper sublattices of $\bP$ are not canonical (see Remark \ref{rem:canonical-check}). Therefore, $L=\bP=4A_1$.  
\end{enumerate}

The uniqueness of the underlying lattices has thus been established. We complete the proof by verifying the uniqueness of the singular automorphic products via obstruction theory.
\end{proof}

\begin{remark}
For $\mathfrak{g}=8A_{1,2}$, the singular automorphic product on $2U\oplus D_8$ remains reflective on both the level-$2$ sublattice $2U\oplus 2D_4$ and the level-$4$ sublattice $U(2)\oplus A_1\oplus A_1(-1)\oplus E_8$, which makes the proof of this case particularly delicate. 
\end{remark}

\begin{remark}
In the symmetric case, the unique underlying canonical lattice $L_\mathfrak{g}=\bP^{\mathrm{ev}}$ associated with each $\mathfrak{g}$ is the single class in its genus, where $\bP^{\mathrm{ev}}$ denotes the maximal even sublattice of $\bP$.   \end{remark}

Combining Theorem \ref{MTH:classification} and Theorem \ref{th:unique-symmetric}, we have the following corollary.

\begin{corollary}
Up to isomorphism, there are exactly $16$ even positive-definite lattices $L$ for which $2U\oplus L$ admits a singular automorphic product of symmetric type and $L$ is canonical. Moreover, for each such lattice, the product is unique. In other words, up to conjugation, there are exactly $16$ symmetric singular automorphic products on lattices of the form $2U\oplus L$, and precisely $4$ of them involve negative Fourier coefficients in the principal parts of their inputs. 
\end{corollary}

We now proceed to the anti-symmetric case. In this setting, singular automorphic products give rise to $69$ semi-simple Lie algebras $\mathfrak{g}$ (all with trivial odd parts), together with the abelian Lie algebra of dimension $24$ (see Tables \ref{table:class-1-8}-\ref{table:class-1-3}). These $\mathfrak{g}$ are all of even rank. Recall from Lemma \ref{lem:canonical-lattice-unique} that if a singular automorphic product is reflective on $2U\oplus L$, then $L$ is canonical. 

\begin{theorem}\label{th:unique-anti-symmetric}
In the anti-symmetric case, each $\mathfrak{g}$ determines a unique canonical lattice $L$ such that $2U\oplus L$ admits a singular automorphic product with Lie algebra structure $\mathfrak{g}$, and the product is unique up to $\Orth(L)$. Consequently, up to isomorphism, there are exactly $11$ lattices of type $2U\oplus L$ that have anti-symmetric reflective automorphic products of singular weight. Moreover, for each such lattice $M$, the product is unique up to $\Orth^+(M)$. In other words, up to conjugation, there are exactly $11$ anti-symmetric singular automorphic products on lattices of the form $2U\oplus L$, and each of them has non-negative principal parts in its input. 
\end{theorem} 

\begin{proof}
Let $F$ be an anti-symmetric singular automorphic product on $2U\oplus L$ for which $L$ is canonical and the associated Lie algebra is $\mathfrak{g}$. Recall that the Weyl vector of $F$ is $\rho=(-A, \vec{B}, -C)$ with $A=C+1$ (see Part (3) of Theorem \ref{MTH:Lie structure}).  For each of the $57$ Lie algebras $\mathfrak{g}$ with $C\in\ZZ$ (see Table \ref{table:class-1-8}), Corollary \ref{Cor:relation to regular} shows that $2U\oplus L$ is regular and the input of $F$ has non-negative principal parts. Moreover, $U\oplus L\cong U(N)\oplus K$ for some $K$, where $N$ denotes the level of $L$. We now consider the remaining $12$ Lie algebras $\mathfrak{g}$ with $C\in 1/2+\ZZ$ (see Table \ref{table:class-1-3}). We discuss each case. 
\begin{enumerate}
\item $C_{4,10}$: $N_{\mathfrak{g}}=20$, $C=1/2$, $\bQ=4A_1(10)$, $\bP=D_4(10)$. Since the level of $L$ is $20$, we have $L=D_4(10)$, which has genus $\II_{4,0}(2_{\II}^{-2}4_{\II}^{-2}5^{+4})$. It is easy to check $\vec{B}\in L'$.
\item $F_{4,6}\oplus A_{2,2}$: $N_\mathfrak{g}=12$, $C=3/2$, $\bQ=D_4(6)\oplus A_2(2)$, $\bP=D_4(6)\oplus A_2'(2)$. The bounds $\bQ<L<\bP$ restrict to $L=D_4(6)\oplus A_2(2)$. We check $\vec{B}\in L'$. 
\item $D_{4,12}\oplus A_{2,6}$: $N_\mathfrak{g}=12$, $C=1/2$, $\bQ=D_4(12)\oplus A_2(6)$, $\bP=D_4(6)\oplus A_2(2)$. The level of $L$ is $12$. Note that $L(1/2)$ is integral and lies between
$$
D_4(6)\oplus A_2(3) < L(1/2) < D_4(3)\oplus A_2.
$$
Since $L(1/2)$ has level $6$, we have $L(1/2)=D_4(3)\oplus A_2$. Therefore, $L=D_4(6)\oplus A_2(2)$, which has genus $\II_{6,0}(2_{\II}^{+4}4_{\II}^{-2}3^{+5})$. We find $\vec{B}\in L'$. 
\item $A_{8,2}\oplus F_{4,2}$: $N_\mathfrak{g}=4$, $C=9/2$, $\bQ=A_8(2)\oplus D_4(2)$, and $\bP=A_8'(2)\oplus D_4(2)$. The level of $L$ is $4$. From the integrality of both $L$ and $L(9/2)$, we derive that $L(1/2)$ is integral. Since $L(1/2)$ has level $2$ and lies between
$$
A_8\oplus D_4 < L(1/2) < A_8'\oplus D_4,
$$
we deduce $L(1/2)=E_8\oplus D_4$, and thus $L=E_8(2)\oplus D_4(2)$. Note that $\vec{B}\in L'$. 
\item  $C_{4,2}\oplus 2A_{4,2}$: $N_\mathfrak{g}=4$, $C=5/2$, $\bQ=4A_1(2)\oplus 2A_4(2)$, and $\bP=D_4(2)\oplus 2A_4'(2)$. The level of $L$ is $4$. Similarly, $L(1/2)$ is integral. Since $L(1/2)$ has level $2$ and lies between
$$
4A_1\oplus 2A_4 < L(1/2) < D_4\oplus 2A_4',
$$
we have $L(1/2)=E_8\oplus D_4$, and thus $L=E_8(2)\oplus D_4(2)$. Note that $\vec{B}\in L'$. 
\item  $D_{4,4}\oplus 4A_{2,2}$: $N_\mathfrak{g}=4$, $C=3/2$, $\bQ=D_4(4)\oplus 4A_2(2)$, and $\bP=D_4(2)\oplus 4A_2'(2)$. The level of $L$ is $4$. Note that $L(1/2)$ is integral. Since $L(1/2)$ has level $2$ and lies between
$$
D_4(2)\oplus 4A_2 < L(1/2) < D_4\oplus 4A_2',
$$
we deduce $L(1/2)=E_8\oplus D_4$, and thus $L=E_8(2)\oplus D_4(2)$. Note that $\vec{B}\in L'$. 
\item Let $\mathfrak{g}$ be $B_{12,2}$, $2B_{6,2}$, $3B_{4,2}$, $4B_{3,2}$, $6B_{2,2}$, or $12A_{1,4}$. In these cases, $C\in 1/2+\ZZ$, $N_\mathfrak{g}=4$ or $8$, and $\bP=12A_1$. Thus, the level of $L$ is $4$ or $8$. In addition, $L(1/2)$ is integral and has the upper bound $L(1/2)<\ZZ^{12}$. From the lower bound $\bQ<L$, we conclude $\mathrm{disc}(L)=2^a$ for some positive integer $a$. Since $L$ does not have $2$-roots, $L(1/2)$ is contained in an index-two sublattice of $\ZZ^{12}$ denoted by $K$. Note that $\mathrm{disc}(K)=4$. There are two cases.

(i) When $K$ is even, $K=D_{12}$ and $L<D_{12}(2)$. We check $\vec{B}\in D_{12}'(1/2)$, so $\vec{B} \in L'$. 

(ii) Now we suppose that $K$ is odd. Then $K$ lies in the genus of $\ZZ^{11}\oplus A_1(2)$, $\ZZ^9\oplus A_3$, $\ZZ^{10}\oplus 2A_1$, or $\ZZ^{8}\oplus D_4$. The first two cases are impossible because $K(2)$ would have level $16$, which contradicts the level of $L$. In the third case, we have
$$
2U\oplus L<2U\oplus K(2) < 2U\oplus E_8\oplus 2A_1\oplus 2A_1(2);
$$
while we find that $2U\oplus E_8\oplus 2A_1\oplus 2A_1(2)$ does not have reflective automorphic products, a contradiction. In the last case, we have
$$
2U\oplus L<2U\oplus K(2) <  2U\oplus E_7\oplus A_1\oplus D_4(2);
$$
while we find that $2U\oplus E_7\oplus A_1\oplus D_4(2)$ does not have reflective automorphic products, a contradiction. Therefore, we exclude the case where $K$ is odd. 
\end{enumerate}

In conclusion, for the $12$ Lie algebras $\mathfrak{g}$ with $C\in 1/2+\ZZ$, we have confirmed $\vec{B}\in L'$. Therefore, Corollary \ref{Cor:relation to regular} shows that $2U\oplus L$ is regular and the input of $F$ has non-negative principal parts. Recall that Driscoll-Spittler, Scheithauer, and Wilhelm \cite[Theorems 5.15, 6.13]{DSW23} have classified anti-symmetric reflective automorphic products on even-rank regular lattices of type $2U\oplus L$ whose inputs have non-negative principal parts. Their results, together with Theorem \ref{MTH:classification} and the above discussions, imply the desired theorem.        
\end{proof}

\begin{remark}
In the anti-symmetric case, the $69$ Lie algebras correspond to $11$ isomorphism classes of lattices $M=2U\oplus L$. For each of the $8$ lattices $M$ in Table \ref{table:class-1-8}, the singular automorphic product on $M$ is unique and is modular for $\Orth^+(M)$. In addition, the Lie algebras $\mathfrak{g}$ attached to this singular product are in bijection with the isomorphism classes in the genus of $L$. 

For each of the $3$ lattices $M$ in Table \ref{table:class-1-3}, the singular automorphic products on $M$ are not unique, though they are conjugate up to the action of $\Orth^+(M)$. For example, for $\mathfrak{g}=C_{4,10}$, the lattice $2U\oplus D_4(10)$ has exactly $3$ reflective automorphic products of singular weight, which are conjugate for the three cosets of $W(F_4)/W(C_4)$, where $W(F_4)$ and $W(C_4)$ denote the Weyl group of $F_4$ and $C_4$, respectively. In this case, $D_4(10)$ is the single class in its genus, corresponding to a unique $\mathfrak{g}$. For the remaining $2$ genera, when the associated cycle shape is $2^3 6^3$, the genus of $L$ has a single class, corresponding to $2$ distinct $\mathfrak{g}$; when the associated cycle shape is $2^{12}$, the genus of $L$ has precisely $2$ classes, corresponding to $6$ and $3$ distinct $\mathfrak{g}$, respectively. 
\end{remark}

In the remainder of this section, using an approach analogous to the symmetric case, we give an alternative proof of Theorem \ref{th:unique-anti-symmetric} that does not rely on \cite[Theorems 3.6 and 5.4]{DSW23}. 

\vspace{2mm}

We start with Case (7) in the proof of Theorem \ref{th:unique-anti-symmetric}. According to the previous discussion, the singular automorphic product $F$ has non-negative principal parts in its input; $L$ is regular and has level $4$ or $8$; and $L(1/2)<D_{12}$, so $L=K(2)$ for some even lattice $K<D_{12}$. After enumerating all regular lattices $L=K(2)$ of level $4$ (there are $5$ candidates $L$ up to genus), by obstruction theory, the existence of $F$ yields $L=D_{12}(2)$ as in \cite{DSW23}. Suppose that the level of $L$ is $8$. Then $K$ is contained in a level-$4$ sublattice of $D_{12}$. After enumerating all rank-$12$ and level-$4$ sublattices of $D_{12}$, we find either $K<E_8\oplus 4A_1$ or $K<\tilde{K}$ with $U\oplus \tilde{K}\cong U(4)\oplus D_4\oplus E_8$. 
In either case, we have
$$
2U\oplus L< 2U\oplus K(2)<2U\oplus E_8\oplus 2A_1\oplus 2A_1(2).
$$
This overlattice does not have any reflective automorphic products. Therefore, the level of $L$ cannot be $8$. Combining this case with Cases (1)-(7) in the proof of Theorem \ref{th:unique-anti-symmetric}, we have thus established the uniqueness of $L$ for the $12$ Lie algebras $\mathfrak{g}$ with $C\not\in\ZZ$. 

\vspace{2mm} 

We now turn to the $57$ Lie algebras $\mathfrak{g}$ with $C\in\ZZ$ in Table \ref{table:class-1-8}. Let $N$ denote the level of the underlying lattice $L$. Parts (3)-(5) of Corollary \ref{Cor:relation to regular} imply that $U\oplus L\cong U(N)\oplus K$ for some $K$ and the singular automorphic products have non-negative principal parts in their inputs. 

For the $23$ Lie algebras $\mathfrak{g}$ of rank $24$, the lattice $L$ must be the associated Niemeier lattice. 

For the $17$ Lie algebras $\mathfrak{g}$ of rank $16$, we have $N_\mathfrak{g}=2$ or $4$, and thus the level $N$ of $L$ is $2$ or $4$. When $N=2$, by \cite[Theorem 6.27]{Sch17}, $L$ has genus $\II_{16,0}(2_{\II}^{+10})$. We rule out the case $N=4$ by the following lemma. 

\begin{lemma}
Let $L$ be an even positive-definite lattice of rank $16$ and level $4$. Then $2U\oplus L$ has no reflective automorphic products.      
\end{lemma}

\begin{proof}
Suppose that $2U\oplus L$ has a reflective automorphic product denoted by $F$. Then the input of $F$ has nonzero $q^{-1}e_0$ coefficient by Lemma \ref{lem:sym-bound}, and there exist reflective automorphic products for any even overlattice $M$ of $2U\oplus L$. If $L$ is $2$-elementary, then $2U\oplus E_8\oplus E_7\oplus A_1$ is an even overlattice of $2U\oplus L$, which has no reflective automorphic products by obstruction theory, a contradiction. We now assume that the exponent of $L$ is $4$. By \cite[Lemma 1.7]{Ma18} or \cite[Lemma 2.3]{Wan23b}, $2U\oplus L$ has an even overlattice $M$ of exponent $4$ and length $\leq 5$. Recall that the length of $M$ is the minimal number of generators of $M'/M$. After enumerating all such lattices $M$, obstruction theory shows that none of them admits reflective automorphic products. 
\end{proof} 

For the $6$ Lie algebras $\mathfrak{g}$ corresponding to cycle shape $1^6 3^6$, we have $N_\mathfrak{g}=3$ or $6$, so the level $N$ of $L$ is $3$ or $6$. When $N=3$, by \cite[Theorem 6.27]{Sch17}, $L$ has genus $\II_{12,0}(3^{-8})$. The case $N=6$ is excluded by the following lemma, whose proof is analogous to that of the previous lemma.

\begin{lemma}
Let $L$ be an even positive-definite lattice of rank $12$ and level $6$. Then $2U\oplus L$ has no reflective automorphic products.      
\end{lemma}

We are left with $11$ cases, which we now address individually.

\begin{enumerate}
\item $C_{7,2}\oplus A_{3,1}$: $N_\mathfrak{g}=4$, $\bQ=7A_1(2)\oplus A_3$, and $\bP=D_7'(4)\oplus A_3'$. Notice that $\bQ$ has discriminant $4^8$ and level $8$. Since the level of $L$ is $4$, we conclude from $\bQ<L<\bP$ that $L$ is the maximal even sublattice of $\bP$ and has genus $\II_{10,0}(2_2^{+2} 4_{\II}^{+6})$. 

\item $E_{6,4}\oplus B_{2,1}\oplus A_{2,1}$: $N_\mathfrak{g}=4$, $\bQ=E_6(4)\oplus 2A_1\oplus A_2$, and $\bP=E_6'(4)\oplus \ZZ^2\oplus A_2'$. Notice that $\bQ$ has discriminant $3^2\cdot4^7$ and level $12$. Since the level of $L$ is $4$, the bounds $\bQ < L < \bP$ restrict to $\mathrm{disc}(L)=4^7$, and thus $L$ has genus $\II_{10,0}(2_2^{+2} 4_{\II}^{+6})$. 

\item $A_{7,4}\oplus 3A_{1,1}$: $N_\mathfrak{g}=8$, $\bQ=A_7(4)\oplus 3A_1$, and $\bP=A_7'(4)\oplus \ZZ^3(1/2)$. Notice that $\bQ$ has level $64$ and discriminant $4^{10}$.
By Part (3) of Theorem \ref{MTH:Lie structure}, the level of $L$ is $N=4$ or $8$. Recall that $C\in\ZZ$. According to Parts (3)-(5) of Corollary \ref{Cor:relation to regular}, the singular automorphic products have non-negative principal parts and 
$$
U\oplus L \cong U(N)\oplus K, \quad \text{for some $K$.}
$$
Let $\bQ^\mathrm{max}$ denote the maximal even overlattice of $\bQ$. Recall that maximal even overlattices of a given even lattice always belong to the same genus. Therefore,
$$
2U\oplus L \cong U\oplus U(N)\oplus K < U\oplus U(N)\oplus \bQ^{\mathrm{max}}. 
$$
It is easy to see that $\bQ^{\mathrm{max}}$ is in the genus of $E_8\oplus 2A_1$. By obstruction theory, we check that $U\oplus U(8)\oplus E_8\oplus 2A_1$ does not have reflective automorphic products, leading to $N\neq 8$. 

We now consider the case $N=4$. We first enumerate all $L$ satisfying $\bQ<L<\bP$ and $U\oplus L\cong U(4)\oplus K$ for some $K$. For each such $L$, we then determine whether $2U\oplus L$ admits anti-symmetric reflective singular automorphic products with non-negative principal parts, as in \cite[Theorem 5.4]{DSW23}. The result is that $L$ must belong to the genus $\II_{10,0}(2_2^{+2} 4_{\II}^{+6})$. 

\item $D_{5,4}\oplus C_{3,2}\oplus 2A_{1,1}$: $N_\mathfrak{g}=4$, $\bQ=D_5(4)\oplus 3A_1(2)\oplus 2A_1$, and $\bP=D_5'(4)\oplus A_3'(4)\oplus \ZZ^2(1/2)$. The level $N$ of $L$ is $4$. Similarly to Case (3), after enumerating all $L$ satisfying $\bQ<L<\bP$ and $U\oplus L\cong U(4)\oplus K$ for some $K$, we check each candidate using obstruction theory, leading to the result that $L$ has genus $\II_{10,0}(2_2^{+2} 4_{\II}^{+6})$. 

\item $3A_{3,4}\oplus A_{1,2}$: $N_\mathfrak{g}=8$, $\bQ=3A_3(4)\oplus A_1(2)$, and $\bP=3A_3'(4)\oplus \ZZ$. The level $N$ of $L$ is $4$ or $8$. We adopt the same strategy as in Case (3). In this case, $\bQ^{\max}$ is also in the genus of $E_8\oplus 2A_1$, and thus a similar analysis shows that $L$ has genus $\II_{10,0}(2_2^{+2} 4_{\II}^{+6})$.  

\item $D_{6,5}\oplus 2A_{1,1}$: $N_\mathfrak{g}=10$, $\bQ=D_6(5)\oplus 2A_1$, and $\bP=D_6'(5)\oplus \ZZ^2(1/2)$. Notice that $\bQ$ has level $20$ and discriminant $2^4\cdot 5^6$. The level $N$ of $L$ is $5$ or $10$. If $N=5$, then $\mathrm{disc}(L)=5^6$, and $L$ has genus $\II_{8,0}(5^{+6})$. If $N=10$, then $\bQ^{\mathrm{max}}=E_8$ and as in Case (3) we have
$$
2U\oplus L\cong U\oplus U(10)\oplus K < U\oplus U(10)\oplus E_8.
$$
This overlattice has no reflective automorphic products, which contradicts the existence of singular automorphic products on $2U\oplus L$. Therefore, $N=5$ and $L$ has genus $\II_{8,0}(5^{+6})$. 

\item $2A_{4,5}$: $N_\mathfrak{g}=5$, $\bQ=2A_4(5)$, and $\bP=2A_4'(5)$. Note that $\bQ$ has level $25$ and discriminant $5^{10}$. Since $L$ is of level $5$, from $\bQ<L<\bP$ we deduce $L=E_8(5)$ or $2A_4'(5)$. However, $2U\oplus E_8(5)$ has no reflective automorphic product of weight $4$. Therefore, $L=2A_4'(5)$. 

\item $C_{5,3}\oplus G_{2,2}\oplus A_{1,1}$: $N_\mathfrak{g}=6$, $\bQ=5A_1(3)\oplus A_2(2)\oplus A_1$, and $\bP=D_5'(6)\oplus A_2(2)\oplus \ZZ(1/2)$. Notice that $\bQ$ has level $12$ and discriminant $2^8\cdot 3^6$. Since $L$ is of level $6$, from $\bQ<L<\bP$ we conclude that $L$ is a maximal even sublattice of $\bP$ and has genus $\II_{8,0}(2_{\II}^{+6}3^{-6})$.  

\item $A_{5,6}\oplus B_{2,3}\oplus A_{1,2}$: $N_\mathfrak{g}=12$,  $\bQ=A_5(6)\oplus 2A_1(3)\oplus A_1(2)$, and $\bP=A_5'(6)\oplus \ZZ^2(3)\oplus \ZZ$. The maximal even sublattice $\bP^{\mathrm{ev}}$ of $\bP$ is of genus $\II_{8,0}(2_{\II}^{+6}3^{-6})$. The level $N$ of $L$ is $6$ or $12$. We adopt the same strategy as in Case (3). In this case, $\bQ^{\mathrm{max}}=E_8$. If $N=12$, then we have
$$
2U\oplus L \cong U\oplus U(12)\oplus K < U\oplus U(12)\oplus E_8.
$$
This overlattice does not have reflective automorphic products. It follows that $N=6$. After enumerating all $L$ satisfying $\bQ<L<\bP$ and $U\oplus L\cong U(6)\oplus K$ for some $K$, we check each candidate using obstruction theory, leading to $L=\bP^{\mathrm{ev}}$. 

\item $A_{6,7}$: $N_\mathfrak{g}=7$, $\bQ=A_6(7)$, and $\bP=A_6'(7)$. Notice that $\mathrm{disc}(\bQ)=7^7$ and $\mathrm{disc}(\bP)=7^5$. Since the level of $L$ is $7$, the bounds $\bQ<L<\bP$ restrict to $L=A_6'(7)$. 

\item  $D_{5,8}\oplus A_{1,2}$: $N_\mathfrak{g}=8$, $\bQ=D_5(8)\oplus A_1(2)$, and $\bP=D_5'(8)\oplus \ZZ$. The level of $L$ is $8$. From 
$$
D_5(8)\oplus A_1(2) <  L < D_5'(8)\oplus A_1(2),
$$
we obtain $L=D_5'(8)\oplus A_1(2)$ with genus $\II_{6,0}(2_{5}^{+1}4_{1}^{+1}8_{\II}^{+4})$. 
\end{enumerate}

The uniqueness of the underlying lattice $L$ in the anti-symmetric case has thus been proved for all $\mathfrak{g}$. It seems that obstruction theory provides the only means of establishing the uniqueness of the associated singular automorphic products, especially for those $\mathfrak{g}$ with $C\not\in\ZZ$, as in \cite{DSW23}.

\section{Hyperbolization of affine Lie superalgebras}\label{sec:hyperbolization}
Every finite-dimensional semi-simple Kac--Moody Lie superalgebra $\mathfrak{g}$ extends to an affine Lie superalgebra $\hat{\mathfrak{g}}$, which is infinite-dimensional but of polynomial growth. This $\hat{\mathfrak{g}}$ has no imaginary simple roots, and its superdenominator function gives a holomorphic Jacobi form of singular weight and index $\bQ$ via theta quotients. Moving beyond the affine case, the hyperbolic Lie algebras are the next simplest. In 1983, Feingold and Frenkel \cite{FF83} investigated the hyperbolic Kac–Moody algebra with Cartan matrix
$$
\left( \begin{array}{ccc}
2 & -2 & 0 \\
-2 & 2 & -1 \\
0 & -1 & 2
\end{array} \right),
$$
as a hyperbolic extension of affine $A_1$. They showed that its characters and denominators are not genuinely modular, yet are related to Siegel modular forms of genus $2$. In 1996, Gritsenko and Nikulin \cite{GN96a} gave an automorphic correction of this algebra by extending it to a hyperbolic BKM algebra $\mathcal{A}_1$ with infinitely many imaginary simple roots, so that the denominator becomes the Igusa cusp form of weight $35$ on $\Sp_2(\ZZ)$ \cite{Igu64}. The resulting BKM algebra $\mathcal{A}_1$ has the same real roots as the algebra of Feingold and Frenkel, but its imaginary simple roots are far more complex, and no natural construction is known beyond the presentation by generators and relations. Further such automorphic constructions can be found in \cite{GN96b, GN98a, GN98, GN02, GN18}. 

Motivated by Theorem \ref{MTH:Lie structure}, we now pursue a different type of extension. Affine Lie superalgebras are extended to hyperbolic BKM superalgebras with singular automorphic products as their superdenominators, which completes our previous work \cite{SWW23}. This extension introduces both imaginary and additional real simple roots. However, a key feature is that imaginary simple roots are easy to describe and the resulting BKM superalgebra has natural ties to vertex algebras. Consequently, they may admit natural constructions with symmetry groups inherited from vertex algebras, in the spirit of the monster and fake monster algebras. 

Let $\Psi_\mathfrak{g}$ be a singular automorphic product on $2U\oplus L_\mathfrak{g}$ for which $L_\mathfrak{g}$ is canonical and $\mathfrak{g}$ is the associated Lie superalgebra introduced in Theorem \ref{MTH:Lie structure}. Denote the Jacobi form input of $\Psi_{\mathfrak{g}}$ by 
$$
\phi_\mathfrak{g} = \sum_{n\in\ZZ,\; \ell\in L_\mathfrak{g}'} f(n,\ell)q^n \zeta^\ell \in J_{0,L_\mathfrak{g}}^!.
$$
By Theorem \ref{MTH:relation to BKM}, there is a corresponding BKM superalgebra $\mathcal{G}_\mathfrak{g}$ whose simple roots span the hyperbolic lattice $U\oplus L_\mathfrak{g}'$ and whose superdenominator is precisely the product expansion of $\Psi_\mathfrak{g}$: 
$$
e^\rho \prod_{\alpha>0}\left(1-e^{-\alpha}\right)^{f(nm,\ell)}, \quad \alpha=(n,\ell,m)\in U\oplus L_\mathfrak{g}',\quad \rho=(-A,\vec{B}, -C).
$$
A root $\alpha$ is positive if and only if either $m>0$, or $m=0$ and $n>0$, or $m=n=0$ and $\ell>0$. The coefficient $f(nm,\ell)$ gives the supermultiplicity of a candidate root $\alpha=(n,\ell,m)$: 
$$
s\text{-}\mathrm{mult}(\alpha):=\mathrm{mult}_{\bar{0}}(\alpha)-\mathrm{mult}_{\bar{1}}(\alpha)=f(nm,\ell).
$$
In particular, $\alpha$ is an even root if $f(nm,\ell)>0$ and an odd root if $f(nm,\ell)<0$. A root may be both even and odd simultaneously. Consequently, even when $f(nm,\ell)=0$, $\alpha$ may still be a root.  Recall that every real root is either even or odd and has multiplicity $1$. If $\alpha^2=(\ell,\ell)-2nm>0$ and $f(nm,\ell)\neq 0$, then $f(nm,\ell)=\pm 1$, and $f(4nm,2\ell)=1$ whenever $f(nm,\ell)=-1$. The sets of even and odd real roots are respectively given by 
\begin{align*}
\Delta_{\bar{0}}^{\mathrm{re}}&=\{ \alpha=(n,\ell,m) \in U\oplus L_\mathfrak{g}' : \; (\ell,\ell)>2nm, \; f(nm,\ell)=1 \},  \\
\Delta_{\bar{1}}^{\mathrm{re}}&=\{ \alpha=(n,\ell,m) \in U\oplus L_\mathfrak{g}' : \; (\ell,\ell)>2nm, \; f(nm,\ell)=-1 \}. 
\end{align*}
If $\alpha$ is an odd real root, then $2\alpha$ is an even real root. If $\alpha$ is an even real root, then a scalar $k\alpha$ is an even root if and only if $k=\pm 1$. 

Denote $\Delta^{\mathrm{re}}:=\Delta_{\bar{0}}^{\mathrm{re}} \cup \Delta_{\bar{1}}^{\mathrm{re}}$. By Theorem \ref{MTH:Lie structure}, the finite subset of $\Delta^{\mathrm{re}}$ defined by
$$
\mathcal{R}_+:=\{ (0,\ell,0): \; 0<\ell\in L_\mathfrak{g}', \; f(0,\ell)=\pm 1 \}
$$
agrees with the set of positive roots of $\mathfrak{g}$. It follows that the set of positive roots of the affine Lie superalgebra $\hat{\mathfrak{g}}$ coincides with
$$
\widehat{\mathcal{R}}_+:=\{ (n,0,0): n\in \ZZ_{>0} \}\cup \mathcal{R}_+ \cup \{ (n,\ell,0):\; n\in\ZZ_{>0},\; 0\neq \ell\in L_\mathfrak{g}', \; f(0,\ell)=\pm 1 \},
$$
where we identify $K=\delta=(1,0,0)$ and $d=\hat{w}_0=(0,0,-1)$ (cf. Section \ref{subsec:osp}). Observe that $\rho$ serves simultaneously as the Weyl vector of $\Psi_{\mathfrak{g}}$, $\mathcal{G}_\mathfrak{g}$, and $\hat{\mathfrak{g}}$. A real root $\alpha$ of $\mathcal{G}_\mathfrak{g}$ is simple if and only if
$$
(\rho,\alpha) = (\alpha,\alpha)/2. 
$$
Consequently, the real simple roots of $\mathcal{G}_\mathfrak{g}$ include those of $\hat{\mathfrak{g}}$. Since $C>0$, each imaginary simple root $-nt\rho$ has strictly positive last coordinate. Let $\cH$ denote the generalized Cartan subalgebra of $\mathcal{G}_\mathfrak{g}$. By Chevalley--Serre construction,  $\hat{\mathfrak{g}}$ embeds into $\mathcal{G}_\mathfrak{g}$ as a Lie sub-superalgebra with the common Cartan subalgebra $\cH$. Let $\mathcal{G}_\alpha$ denote the corresponding root space for each root $\alpha$. Then we have
$$
\mathcal{G}_\mathfrak{g}=\cH\oplus\bigoplus_{\alpha\in U\oplus L_\mathfrak{g}'\backslash\{0\}} \mathcal{G}_\alpha \quad \text{and} \quad \mathcal{G}_\mathfrak{g}[0]:=\bigoplus_{\alpha\in \widehat{\mathcal{R}}_+} \mathcal{G}_{\alpha} \oplus \mathcal{H} \oplus \bigoplus_{\alpha\in \widehat{\mathcal{R}}_+} \mathcal{G}_{-\alpha}\cong \hat{\mathfrak{g}}. 
$$
For any nonzero integer $m$, we define $\mathcal{G}_\mathfrak{g}[m]$ as the direct sum of root spaces for roots of the form $(*,*,m)$. Since $\mathcal{G}_\mathfrak{g}$ is graded by its root lattice $U\oplus L_\mathfrak{g}'$, we derive the following result.

\begin{theorem}\label{th:hyperbolization}
Under the above assumptions, the decomposition 
$$
\mathcal{G}_\mathfrak{g}=\bigoplus_{m\in\ZZ} \mathcal{G}_\mathfrak{g}[m]
$$
provides $\mathcal{G}_\mathfrak{g}$ with a $\ZZ$-graded module structure over $\hat{\mathfrak{g}}$. In particular, each $\mathcal{G}_\mathfrak{g}[m]$ is a module of $\hat{\mathfrak{g}}$. 
\end{theorem}

We therefore naturally call $\mathcal{G}_\mathfrak{g}$ a \textit{hyperbolization} of $\hat{\mathfrak{g}}$. At the standard $1$-dimensional cusp of type $2U$, the leading Fourier--Jacobi coefficient of $\Psi_{\mathfrak{g}}$ is a holomorphic Jacobi form of singular weight and index $L_\mathfrak{g}(C)$, given by the following theta quotient:
$$
\eta(\tau)^{\rank(L)} \prod_{\substack{0<\ell\in L_\mathfrak{g}'\\ f(0,\ell)=-1}} \frac{\vartheta(\tau, 2(\ell, \mathfrak{z}))}{\vartheta(\tau, (\ell, \mathfrak{z}))}\prod_{\substack{0<\ell\in L_\mathfrak{g}'\\ f(0,\ell)=1 \\ f(0,\ell/2)=0}} \frac{\vartheta(\tau, (\ell, \mathfrak{z}))}{\eta(\tau)}, \quad \mathfrak{z}\in L\otimes\CC.
$$
This function is identical to the superdenominator $\vartheta_\mathfrak{g}$ of $\hat{\mathfrak{g}}$. Thus, the singular automorphic product $\Psi_{\mathfrak{g}}$ can also be seen as a hyperbolization of the singular theta quotient $\vartheta_\mathfrak{g}$. 

Equivalently, Theorem \ref{th:unique-symmetric} and Theorem \ref{th:unique-anti-symmetric} assert that precisely $85$ affine Lie superalgebras $\hat{\mathfrak{g}}$ admit a hyperbolization, and exactly $27$ BKM superalgebras occur as such hyperbolizations (up to ordinary multiplicities of imaginary simple roots; see Section \ref{subsec:relation to BKM}). The corresponding finite-dimensional Lie superalgebras $\mathfrak{g}$ are summarized in Tables \ref{table:class-1-8}-\ref{tab:class-4}. Among these, $4$ of them have nontrivial odd parts, which appear only in the symmetric case (see Table \ref{tab:class-4}). 

Following Theorem \ref{th:hyperbolization}, we consider the following $\hat{\mathfrak{g}}$-module: 
$$
\mathcal{G}_\mathfrak{g}[-1]=\bigoplus_{n,\ell}\mathcal{G}_{(n,\ell,-1)}. 
$$
The supermultiplicity of $\mathcal{G}_{(-n,\ell,-1)}$ is given by the Fourier coefficient $f(n,-\ell)=f(n,\ell)$. Therefore, the supercharacter of $\mathcal{G}_\mathfrak{g}[-1]$ is precisely the Jacobi form input  $\phi_\mathfrak{g}$. Theorem \ref{MTH:classification} shows that $\phi_{\mathfrak{g}}$ is a $\ZZ$-linear combination of the supercharacters of the affine VOSA generated by $\hat{\mathfrak{g}}$. This suggests that $\mathcal{G}_\mathfrak{g}[-1]$ might support a hidden vertex algebra structure. We end this section with two remarks. 

\begin{remark}
Different affine Lie superalgebras can have the same hyperbolization. In other words, a BKM superalgebra may contain several affine Lie superalgebras as Lie sub-superalgebras. This is because a given automorphic product in $M\cong 2U\oplus L$ can admit different Fourier–Jacobi expansions at different $1$-dimensional cusps, corresponding to distinct classes in the genus of $L$ and automorphisms from $\Orth(M'/M)$ modulo its subgroup preserving the vector-valued form input.     
\end{remark}

\begin{remark}
Let $\mathcal{G}$ be a BKM superalgebra with root lattice $U\oplus L'$ whose superdenominator is a reflective automorphic product $F$ of non-singular weight on $2U\oplus L$. Then $\mathcal{G}$ may still contain an affine Lie superalgebra $\hat{\mathfrak{g}}$, determined by the $q^0$-term in the Fourier expansion of the Jacobi form input. However, in this case $\hat{\mathfrak{g}}$ is not isomorphic to $\mathcal{G}[0]$, since the supermultiplicity of the imaginary root $(n,0,0)$ is given by $f(0,0)>\rank(\mathfrak{g})=\rank(L)$. For example, the aforementioned automorphic correction $\mathcal{A}_1$ to the Feingold--Frenkel algebra contains the affine Lie algebra of type $A_1$; however, the imaginary root $(n,0,0)$ has multiplicity $70$, since its denominator function has the leading Fourier-Jacobi coefficient $\eta(\tau)^{69}\vartheta(\tau,z)$.  
\end{remark}

\section{Modularity of the denominators of BKM superalgebras}\label{sec:denominator-modularity}
Let $\mathfrak{g}=\oplus_{j=1}^s\mathfrak{g}_{j,k_j}$ be a finite-dimensional semi-simple Kac--Moody Lie superalgebra with affine extension $\hat{\mathfrak{g}}$ from Theorem \ref{MTH:classification}. As before, we denote the associated singular automorphic product, canonical lattice, and BKM superalgebra by $\Psi_{\mathfrak{g}}$, $L_\mathfrak{g}$, and $\mathcal{G}_\mathfrak{g}$, respectively. Let $\phi_\mathfrak{g}\in J_{0,L_{\mathfrak{g}}}^!$ be the Jacobi form input of $\Psi_{\mathfrak{g}}$ with Fourier expansion 
$$
\phi_{\mathfrak{g}}(\tau,\mathfrak{z})=\sum_{n\in\ZZ,\, \ell\in L_\mathfrak{g}'} f(n,\ell)q^n\zeta^\ell. 
$$
The superdenominator function of $\mathcal{G}_{\mathfrak{g}}$ agrees with the following product expansion of $\Psi_\mathfrak{g}$:
\begin{equation}
\Psi_{\mathfrak{g}}(Z)=q^{A}\zeta^{\vec{B}}\xi^C \prod_{\substack{n,m\in\ZZ, \, \ell\in L_\mathfrak{g}' \\ (n,\ell,m)>0}} \Big( 1-q^n\zeta^{-\ell}\xi^m \Big)^{f(nm,\ell)}.
\end{equation}
The Weyl vector $\rho=(-A,\vec{B},-C)$ is given by
\begin{align*}
 A&=\frac{1}{24}\sum_{\ell\in L_\mathfrak{g}'} f(0,\ell) = \frac{\mathrm{sdim}(\mathfrak{g})}{24} = C +f(-1,0),\\
 \vec{B}&=\frac{1}{2}\sum_{\ell>0}f(0,\ell)\ell = \sum_{j=1}^s \frac{\rho_{j}}{k_j},\\
 C&=\frac{1}{\rank(\mathfrak{g})}\sum_{\ell>0}f(0,\ell)(\ell,\ell) = \frac{h_j^\vee}{k_j}, \quad \text{for $1\leq j\leq s$,}
\end{align*}
where $\rho_j$ and $h_j^\vee$ denote the Weyl vector and the dual Coxeter number of $\mathfrak{g}_j$, respectively. 

In the anti-symmetric case, i.e., $f(-1,0)=1$, we can choose $\mathcal{G}_\mathfrak{g}$ to have trivial odd parts, in which case its superdenominator coincides with its denominator. However, in the symmetric case, $\mathcal{G}_\mathfrak{g}$ has nontrivial odd parts, so its denominator is distinct from its superdenominator. This section aims to study the modularity of these denominator functions. In what follows, we consider the symmetric case, where there are precisely $16$ Lie superalgebras $\mathfrak{g}$, corresponding to $16$ distinct singular automorphic products $\Psi_\mathfrak{g}$. 

Recall from \eqref{eq:BKM-d} that the denominator of $\mathcal{G}_\mathfrak{g}$ takes the form
$$
\widetilde{\Psi}_{\mathfrak{g}}:=e^{\rho} \cdot \prod_{\alpha>0}\big(1-e^{-\alpha}\big)^{\mathrm{mult}_{\bar{0}}(\alpha)} \cdot \prod_{\alpha>0}\big(1+e^{-\alpha}\big)^{-\mathrm{mult}_{\bar{1}}(\alpha)}, \quad \alpha=(n,\ell,m)\in U\oplus L_\mathfrak{g}',
$$
where $\mathrm{mult}_{\bar{0}}(\alpha)$ and $\mathrm{mult}_{\bar{1}}(\alpha)$ are the dimensions of the even and odd parts of the root space $\mathcal{G}_\alpha$, respectively. Let $\delta_{\mathfrak{g}}$ denote the number of positive odd roots of $\mathfrak{g}$. The formal leading Fourier--Jacobi coefficient of $\widetilde{\Psi}_{\mathfrak{g}}$ (contributed by the $m=0$ part of the product expansion) is given by
\begin{equation}
\eta(\tau)^{\rank(\mathfrak{g})-\delta_{\mathfrak{g}}} \cdot \eta(2\tau)^{\delta_{\mathfrak{g}}} \cdot \prod_{\substack{\ell>0\\ f(0,\ell)=1}} \frac{\vartheta(\tau,(\ell,\mathfrak{z}))}{\eta(\tau)} \cdot \prod_{\substack{\ell>0\\ f(0,\ell)=-1}} \frac{\vartheta(\tau,(\ell,\mathfrak{z}))}{\vartheta(2\tau,2(\ell,\mathfrak{z}))}, 
\end{equation}
which coincides with the denominator function of the affine Lie superalgebra $\hat{\mathfrak{g}}$; see \eqref{eq:denominator-osp}. 

The following theorem establishes the modularity of $\widetilde{\Psi}_{\mathfrak{g}}$ for the $8$ Lie algebras $\mathfrak{g}$ attached to the holomorphic VOSA $F_{24}$ of central charge $12$.  

\begin{theorem}\label{th:modularity-d-F24}
For each $\mathfrak{g}$ in Table \ref{tab:class-2}, there exists a BKM superalgebra $\mathcal{G}_\mathfrak{g}$ with root lattice $U\oplus L_\mathfrak{g}'$ and superdenominator $\Psi_\mathfrak{g}$ such that its denominator $\widetilde{\Psi}_{\mathfrak{g}}$ agrees with the product expansion of a singular automorphic product on $U(2)\oplus U\oplus L_\mathfrak{g}$ at the $0$-dimensional cusp determined by $U(2)$.  With notation in \cite[\S 7.1]{SWW23} and \cite[Theorem 4.9]{WW25}, the input of $\widetilde{\Psi}_{\mathfrak{g}}$ can be realized as a pair of Jacobi forms on congruence subgroups:
\begin{align}
\widetilde{\phi}_{\mathfrak{g}}(\tau,\mathfrak{z})&=(\chi_{\mathrm{NS}} - \chi_{\widetilde{\mathrm{NS}}} + \chi_{\mathrm{R}})/2 = \sum_{n,\,\ell} \tilde{f}(n,\ell)q^n\zeta^\ell \in J_{0,L_\mathfrak{g}}^!(\Gamma_0(2)),\\
\phi_{\mathfrak{g}}(\tau,\mathfrak{z})&=(\chi_{\mathrm{NS}} - \chi_{\widetilde{\mathrm{NS}}} - \chi_{\mathrm{R}})/2 = \sum_{n,\,\ell} f(n,\ell)q^n\zeta^\ell \in J_{0,L_\mathfrak{g}}^!(\SL_2(\ZZ)), 
\end{align} 
where $\chi_{-}$ are characters of $F_{24}$ with respect to the $\mathcal{N}=1$ structure $\mathfrak{g}$. More precisely, $\mathrm{NS}$ and $\widetilde{\mathrm{NS}}$ stand for the character and supercharacter of the irreducible Neveu--Schwarz module, respectively; while $\mathrm{R}$ stands for the character of the irreducible Ramond module. 

Furthermore, the denominator function $\widetilde{\Psi}_{\mathfrak{g}}$ can be interpreted as follows:
\begin{equation}\label{eq:d-F24}
\begin{split}
\widetilde{\Psi}_{\mathfrak{g}}(Z)&=q^{A}\zeta^{\vec{B}}\xi^C \prod_{(n,\ell,m)>0} \Big( 1-q^n\zeta^{-\ell}\xi^m \Big)^{f_{\bar{0}}(nm,\ell)}\cdot\prod_{(n,\ell,m)>0} \Big( 1+q^n\zeta^{-\ell}\xi^m \Big)^{-f_{\bar{1}}(nm,\ell)}\\
&=q^{A}\zeta^{\vec{B}}\xi^C \prod_{(n,\ell,m)>0} \Big( 1-q^n\zeta^{-\ell}\xi^m \Big)^{\tilde{f}(nm,\ell)}\cdot\prod_{(n,\ell,m)>0} \Big( 1-\big(q^n\zeta^{-\ell}\xi^m\big)^2 \Big)^{-f_{\bar{1}}(nm,\ell)},
\end{split}
\end{equation}
where the root multiplicities are given by
\begin{align*}
f_{\bar{0}}(nm,\ell)&:= \frac{1}{2}\big( \tilde{f}(nm,\ell)+f(nm,\ell)\big)=\mathrm{mult}_{\bar{0}}\big((n,\ell,m)\big),\\
f_{\bar{1}}(nm,\ell)&:=\frac{1}{2}\big( \tilde{f}(nm,\ell)-f(nm,\ell)\big)=\mathrm{mult}_{\bar{1}}\big((n,\ell,m)\big). 
\end{align*}
\end{theorem}

\begin{proof}
The proof is obtained by refining the proof of Theorem \ref{MTH:relation to BKM}. As in Section \ref{sec:BKM from singular}, we define the real simple roots to be the vectors $\alpha=(n,\ell,m)\in U\oplus L_{\mathfrak{g}}'$ that satisfy
$$
f(nm,\ell)\neq 0 \quad \text{and} \quad  (\rho,\alpha)=(\alpha,\alpha)/2=(\ell,\ell)/2-nm>0.
$$
We then regard $\alpha$ as even when $f(nm,\ell)=1$ and odd when $f(nm,\ell)=-1$. Note that $\vec{B}\in L_\mathfrak{g}'$ and $A=C=1$ in this case. As a refinement, for any integer $n>0$, if $\tilde{f}(n^2,n\vec{B})\neq 0$, we define $-n\rho$ as an imaginary simple root that can be both even and odd, with even and odd multiplicities equal to $f_{\bar{0}}(n^2, n\vec{B})$ and $f_{\bar{1}}(n^2, n\vec{B})$, respectively; otherwise $-n\rho$ is not a root. The supermultiplicity of $-n\rho$ is still $f(n^2,n\vec{B})$. From these data, by Chevalley--Serre construction, we obtain a BKM superalgebra $\mathcal{G}_\mathfrak{g}$ with root lattice $U\oplus L_\mathfrak{g}'$ and superdenominator $\Psi_\mathfrak{g}$. By \eqref{eq:BKM-d}, the sum expansion of the denominator is completely determined by simple roots and their even and odd multiplicities. By \cite[Theorem 4.9]{WW25} and direct computation, the Borcherds theta lift of the pair $(\widetilde{\phi}_{\mathfrak{g}},\phi_{\mathfrak{g}})$ is a holomorphic modular form of singular weight on $U(2)\oplus U\oplus L_\mathfrak{g}$, denoted $\widetilde{\Psi}_{\mathfrak{g}}$. Using a similar argument as in Section \ref{subsec:Fourier-expansion}, we determine the Fourier expansion of $\widetilde{\Psi}_{\mathfrak{g}}$ at the $0$-dimensional cusp of type $U(2)$, and find that it coincides with the sum expansion of the denominator of $\mathcal{G}_\mathfrak{g}$.  
\end{proof}

\begin{remark}
The singular automorphic product $\widetilde{\Psi}_{\mathfrak{g}}$ in Theorem \ref{th:modularity-d-F24} is not always reflective on $U(2)\oplus U\oplus L_\mathfrak{g}$. In fact, it is not reflective if and only if $\mathfrak{g}=A_{1,2}B_{3,5}$, $B_{2,3}G_{2,4}$, or $A_{1,2}^2 A_{2,3}B_{2,3}$. For the three cases, the reflection $\sigma_\lambda$ still preserves $U\oplus L_\mathfrak{g}'$ and the equality $\sigma_\lambda(\widetilde{\Psi}_{\mathfrak{g}})=-\widetilde{\Psi}_{\mathfrak{g}}$ holds whenever $\widetilde{\Psi}_{\mathfrak{g}}$ vanishes along $\lambda^\perp$ for $\lambda\in U\oplus L_\mathfrak{g}'$. Therefore, we can use an argument similar to that in Section \ref{subsec:Fourier-expansion} to compute their Fourier expansions at $U(2)$.  In addition, $\widetilde{\Psi}_{\mathfrak{g}}$ is modular for $\Orth^+(U(2)\oplus U\oplus L_\mathfrak{g})$ if and only if $\mathfrak{g}=A_{4,5}$ or $A_{2,3}^3$. We make a few further observations. 
\begin{enumerate}
\item When $\mathfrak{g}=A_{4,5}$, we have $L_\mathfrak{g}=A_4'(5)$ and $U(2)\oplus A_4'(5)\cong U(10)\oplus A_4\cong U\oplus K$ for some $K$. Therefore, the product expansion of $\widetilde{\Psi}_{\mathfrak{g}}$ at $U$ is identical to the superdenominator of the BKM superalgebra associated with $A_{1,2}B_{3,5}$. 
\item When $\mathfrak{g}=A_{2,3}^3$, we have $L_\mathfrak{g}=3A_2$ and $U(2)\oplus 3A_2\cong U(6)\oplus E_6\cong U\oplus K$ for some $K$. Therefore, the product expansion of $\widetilde{\Psi}_{\mathfrak{g}}$ at $U$ is identical to the superdenominator of the BKM superalgebra associated with $A_{1,2}^2 A_{2,3}B_{2,3}$. 
\item When $\mathfrak{g}=A_{1,2}^3 A_{3,4}$, the canonical lattice $L_\mathfrak{g}$ has genus $\II_{6,0}(2_6^{+2}4_{\II}^{+2})$, and we have
$$
U(2)\oplus U\oplus L_\mathfrak{g}\cong U(2)\oplus U(4)\oplus D_6\cong 2U\oplus 2A_1\oplus D_4(2). 
$$
Therefore, the denominator $\widetilde{\Psi}_{\mathfrak{g}}$ agrees with the superdenominator $\Psi_{\mathfrak{g}}$ of the same BKM superalgebra as functions on symmetric domains. In fact, $\widetilde{\Psi}_{\mathfrak{g}}$ is a pullback of $\Psi_{\mathfrak{g}}$. 
\item When $\mathfrak{g}=A_{1,2}^2 A_{2,3}B_{2,3}$, the canonical lattice $L_\mathfrak{g}$ has genus $\II_{6,0}(2_{\II}^{+2}3^{-3})$, and we have
$$
U(2)\oplus U\oplus L_\mathfrak{g}\cong 2U(2)\oplus 3A_2\cong 2U\oplus A_2\oplus 2A_2(2).
$$
Therefore, the denominator $\widetilde{\Psi}_{\mathfrak{g}}$ agrees with the superdenominator $\Psi_{\mathfrak{g}}$ of the same BKM superalgebra as functions on symmetric domains. In fact, $\widetilde{\Psi}_{\mathfrak{g}}$ is a pullback of $\Psi_{\mathfrak{g}}$. 
\item When $\mathfrak{g}=A_{1,2}^8$, we have $L_\mathfrak{g}=D_8$ and $U(2)\oplus D_8\cong U\oplus 2D_4$. 
Therefore, the denominator $\widetilde{\Psi}_{\mathfrak{g}}$ agrees with the superdenominator $\Psi_{\mathfrak{g}}$ of the same BKM superalgebra as functions on symmetric domains. Indeed, $\widetilde{\Psi}_{\mathfrak{g}}$ is a pullback of $\Psi_{\mathfrak{g}}$.  
\end{enumerate}

For the remaining three $\mathfrak{g}$, we do not have $U(2)\oplus U\oplus L_\mathfrak{g}\cong 2U\oplus K$ for any lattice $K$. 
\end{remark}

\begin{remark}
For each $\mathfrak{g}$ in Theorem \ref{th:modularity-d-F24}, the superdenominator $\Psi_{\mathfrak{g}}$ is equal to the additive lift of the denominator $\vartheta_\mathfrak{g}$ of the affine Lie algebra $\hat{\mathfrak{g}}$ (cf. \cite[Theorem 5.1]{DW21} and \cite[Theorem 9.1]{SWW23}). Following \cite[Theorem 14.3]{Bor98}, there are also additive lifts on lattices of type $U(N)\oplus U\oplus L$. With notation in \cite[Remark 4.15]{WW25}, the denominator $\widetilde{\Psi}_{\mathfrak{g}}$ can be realized on $U(2)\oplus U\oplus L_\mathfrak{g}$ as an additive lift of the pair $(\vartheta_\mathfrak{g}, -\vartheta_\mathfrak{g})$ of Jacobi forms.   
\end{remark} 

\begin{remark}
The second expansion in \eqref{eq:d-F24} likewise occurs as the superdenominator of a BKM superalgebra $\mathcal{G}_\mathfrak{g}'$, potentially different from $\mathcal{G}_\mathfrak{g}$. The algebras  $\mathcal{G}_\mathfrak{g}'$ and $\mathcal{G}_\mathfrak{g}$ share the same real simple roots, while the supermultiplicities of imaginary simple roots may be distinct. Consequently, $\mathcal{G}_\mathfrak{g}'$ can also be viewed as a hyperbolization of $\hat{\mathfrak{g}}$, though now with underlying lattice $U(2)\oplus U\oplus L_\mathfrak{g}$ in place of $2U\oplus L_\mathfrak{g}$. 
\end{remark}

\begin{remark}\label{rem:express-character-d-F24}
For each $\mathfrak{g}$ in Theorem \ref{th:modularity-d-F24}, every character $\chi_{-}$ of $F_{24}$ can be written as a $\ZZ$-linear combination of the full characters of the affine VOA generated by $\hat{\mathfrak{g}}$ (see \cite[\S 7.3]{SWW23}). Thus, $\widetilde{\phi}_\mathfrak{g}$ has a similar expression. By \cite[Lemma 7.4]{SWW23}, the Jacobi forms  $\widetilde{\phi}_{\mathfrak{g}}$ and $\phi_{\mathfrak{g}}$ have the same singular Fourier coefficients, that is, $\tilde{f}(n,\ell)=f(n,\ell)$ for any $n\in\ZZ$ and $\ell\in L_\mathfrak{g}'$ with $2n<(\ell,\ell)$. This explains why the Weyl groups of $\widetilde{\Psi}_{\mathfrak{g}}$ at $U(2)$ and $\Psi_{\mathfrak{g}}$ at $U$ are the same. 
\end{remark} 

\begin{remark}
By \cite{CDR18}, there are two additional holomorphic VOSAs of central charge $12$ besides $F_{24}$ (composed of $24$ chiral fermions): the Conway VOSA $V^{f\natural}$ (with trivial weight-$1/2$ subspace) and $V^{fE_8}$ (composed of $8$ chiral bosons and $8$ chiral fermions). We now turn to the corresponding BKM superalgebras $\mathcal{G}$ via the BRST cohomology and make some corrections to \cite[Remarks 7.7 and 7.8]{SWW23}. Similarly to $F_{24}$, the superdenominator of $\mathcal{G}$ is given by the expansion of an automorphic product $\Psi$ on $2U\oplus L$ at cusp $U$ with Jacobi form input
$$
\phi=(\chi_{\mathrm{NS}} - \chi_{\widetilde{\mathrm{NS}}} - \chi_{\mathrm{R}}+\chi_{\widetilde{\mathrm{R}}})/2 \in J_{0,L}^!(\SL_2(\ZZ));
$$
the denominator of $\mathcal{G}$ is given by the expansion of an automorphic product $\widetilde{\Psi}$ on $U(2)\oplus U \oplus L$ at cusp $U(2)$ with Jacobi form pair input $(\widetilde{\phi}, \phi)$, where
$$
\widetilde{\phi}=(\chi_{\mathrm{NS}} - \chi_{\widetilde{\mathrm{NS}}} + \chi_{\mathrm{R}} - \chi_{\widetilde{\mathrm{R}}})/2 \in J_{0,L}^!(\Gamma_0(2)). 
$$
Note that $\chi_{\widetilde{\mathrm{R}}}\equiv 0$ for $F_{24}$ and $V^{fE_8}$, while $\chi_{\widetilde{\mathrm{R}}}\equiv -24$ for $V^{f\natural}$. 

(i) For $V^{fE_8}$, we have $L=E_8$, $\phi\equiv0$, $\Psi\equiv 1$ (so its weight is $0$), and $\widetilde{\Psi}$ has singular weight $4$. 

(ii) For $V^{f\natural}$, we have $L=0$, $\phi\equiv -24$, $\widetilde{\Psi}$ has weight $0$ (although it is not constant), and $\Psi$ is given by $\Delta(\tau)^{-1}\cdot\Delta(\omega)^{-1}$ with weight $-12$. 
\end{remark}

Complementing the proof of Theorem \ref{th:modularity-d-F24}, we compute the Fourier expansions of $\Psi_\mathfrak{g}$ and $\widetilde{\Psi}_\mathfrak{g}$, identify the associated cycle shapes, and record the even and odd multiplicities of the imaginary simple roots. Recall that $A=C=1$ and $\rho=(-A,\vec{B},-C) \in U\oplus L_\mathfrak{g}'$ in these cases. Let $N^*_\mathfrak{g}$ denote the order of $\rho$ in the discriminant group (with care that $N^*_\mathfrak{g}$ may differ from $N_\mathfrak{g}$ defined in Part (3) of Theorem \ref{MTH:Lie structure}), which is equal to the level of $L_\mathfrak{g}$. Equation \eqref{eq:Singular-Fourier} gives the following Fourier expansion for the singular automorphic product $\Psi_\mathfrak{g}$: 
$$
\Psi_\mathfrak{g}(Z) = \sum_{\sigma\in W_\mathfrak{g}} \varepsilon'(\sigma) \cdot \sigma\Big( \eta_g\big((\rho,Z)\big) \Big).
$$
Recall that $\eta_g$ is determined by the Fourier coefficients $f(n^2,n\vec{B})$ for positive integers $n$, which depend only on the class of $n$ modulo $N^*_\mathfrak{g}$. 
A straightforward computation shows that $\eta_g$ is an eta quotient attached to a cycle shape $g=\prod_{k}k^{b_k}$, meaning that 
$$
\eta_g(\tau)=\prod_{k}\eta(k\tau)^{b_k}.
$$
We apply the same argument as in Section \ref{subsec:Fourier-expansion} to conclude that $\widetilde{\Psi}_\mathfrak{g}$ has the Fourier expansion 
$$
\widetilde{\Psi}_\mathfrak{g}(Z) = \sum_{\sigma\in W_\mathfrak{g}} \varepsilon(\sigma) \cdot \sigma\Big( \eta_{\tilde{g}}\big((\rho,Z)\big) \Big),
$$
where $\tilde{g}$ is a cycle shape different from $g$, which is determined by both $\tilde{f}(n^2,n\vec{B})$ and $f(n^2,n\vec{B})$ for positive $n\in\ZZ$ through product expansion \eqref{eq:d-F24}. Note that $\varepsilon'=\varepsilon:=\det$ in these cases, since $\mathcal{G}$ has no odd real roots. The cycle shapes $g$ and $\tilde{g}$ are listed in Table \ref{tab:class-2}. The integers 
$$
\mathrm{m}_{\bar{0}}(n):=\frac{1}{2}\big(\tilde{f}(n^2,n\vec{B}) +f(n^2,n\vec{B})\big) \quad \text{and} \quad \mathrm{m}_{\bar{1}}(n):=\frac{1}{2}\big(\tilde{f}(n^2,n\vec{B}) - f(n^2,n\vec{B})\big)
$$
give the even and odd multiplicities of the imaginary simple root $-n\rho$ for any positive $n\in\ZZ$, respectively. Their values depend only on $n \; (\bmod\, N^*_\mathfrak{g})$ and are formulated as follows. Observe that $\mathcal{G}_\mathfrak{g}$ has imaginary simple roots that are both even and odd if and only if $\mathfrak{g}=A_{4,5}$ or $A_{2,3}^3$. 

\begin{table}[ht]
\begin{minipage}{0.4\textwidth}
\centering
\[
\begin{array}{|c|c|c|c|c|c|} \hline
n & 1 & 2 & 3 & 4 & 5   \\ \hline
\mathrm{m}_{\bar{0}}(n)& 1 & 1 & 1 & 1 & 4 \\ \hline
\mathrm{m}_{\bar{1}}(n)& 2 & 2 & 2 &2 & 0 \\ \hline
\end{array}
\]
\caption{$A_{4,5}$}
\end{minipage}
\begin{minipage}{0.5\textwidth}
\centering 
\[
\begin{array}{|c|c|c|c|c|c|c|c|c|c|c|} \hline
n & 1 & 2 & 3 & 4 & 5 & 6 & 7 & 8 & 9 & 10  \\ \hline
\mathrm{m}_{\bar{0}}(n)& 0 & 1 & 0 & 1 & 0 & 1 & 0 & 1 & 0 & 4 \\ \hline
\mathrm{m}_{\bar{1}}(n)& 2 & 0 & 2 & 0 & 0 & 0 & 2 & 0 & 2 & 0 \\ \hline
\end{array}
\]
\caption{$A_{1,2}B_{3,5}$}
\end{minipage}
\end{table} 

\begin{table}[ht]
\begin{minipage}{0.35\textwidth}
\centering 
\[
\begin{array}{|c|c|c|} \hline
n & 1 & 2    \\ \hline
\mathrm{m}_{\bar{0}}(n)& 0 & 8 \\ \hline
\mathrm{m}_{\bar{1}}(n)& 8 & 0 \\ \hline
\end{array}
\]
\caption{$A_{1,2}^8$}
\end{minipage}
\begin{minipage}{0.55\textwidth}
\centering 
\[
\begin{array}{|c|c|c|c|c|c|c|c|c|c|c|c|c|} \hline
n & 1 & 2 & 3 & 4 & 5 & 6 & 7 & 8 & 9 & 10 & 11 & 12 \\ \hline
\mathrm{m}_{\bar{0}}(n)& 0&0&0&1&0&2&0&1&0&0&0&4 \\ \hline
\mathrm{m}_{\bar{1}}(n)& 2&0&0&0&2&0&2&0&0&0&2&0 \\ \hline
\end{array}
\]
\caption{$B_{2,3}G_{2,4}$}
\end{minipage}
\end{table}

\begin{table}[ht]
\begin{minipage}{0.4\textwidth}
\centering
\[
\begin{array}{|c|c|c|c|} \hline
n & 1 & 2 & 3   \\ \hline
\mathrm{m}_{\bar{0}}(n)& 1 & 1 & 6 \\ \hline
\mathrm{m}_{\bar{1}}(n)& 4 & 4 & 0 \\ \hline
\end{array}
\]
\caption{$A_{2,3}^3$}
\end{minipage}
\begin{minipage}{0.4\textwidth}
\centering 
\[
\begin{array}{|c|c|c|c|c|} \hline
n & 1 & 2 & 3 & 4   \\ \hline
\mathrm{m}_{\bar{0}}(n)& 0 & 2 & 0 & 6 \\ \hline
\mathrm{m}_{\bar{1}}(n)& 4&0&4&0 \\ \hline
\end{array}
\]
\caption{$A_{1,2}^3A_{3,4}$}
\end{minipage}
\end{table} 

\begin{table}[ht]
\begin{minipage}{0.4\textwidth}
\centering
\[
\begin{array}{|c|c|c|c|c|c|c|} \hline
n & 1 & 2 & 3 & 4 & 5 & 6  \\ \hline
\mathrm{m}_{\bar{0}}(n)& 0 & 1 & 0 & 1 & 0 & 6 \\ \hline
\mathrm{m}_{\bar{1}}(n)& 4 & 0 & 0 & 0 & 4 & 0 \\ \hline
\end{array}
\]
\caption{$A_{1,2}^2A_{2,3}B_{2,3}$}
\end{minipage}
\begin{minipage}{0.4\textwidth}
\centering
\[
\begin{array}{|c|c|c|c|c|c|c|c|c|} \hline
n & 1 & 2 & 3 & 4 & 5 & 6 & 7 & 8  \\ \hline
\mathrm{m}_{\bar{0}}(n)& 0 & 1 & 0 & 2 & 0 & 1 & 0 & 4 \\ \hline
\mathrm{m}_{\bar{1}}(n)& 2 & 0 & 2 & 0 & 2 & 0 & 2 & 0 \\ \hline
\end{array}
\]
\caption{$A_{1,2}C_{3,4}$}
\end{minipage}
\end{table}

For ease of reference, we formulate the cycle shapes attached to the superdenominator $\Psi_\mathfrak{g}$ for the $8$ exceptional $\mathfrak{g}$ from Tables \ref{tab:class-3} and \ref{tab:class-4}. In these cases, $L_\mathfrak{g}=\bP$; $A=C$; $1/C\in\ZZ$; $\vec{B}\in L_\mathfrak{g}'$ if and only if $\mathfrak{g}=A_{2,9}$; and the smallest positive integer $t$ with $t\rho\in U\oplus L_\mathfrak{g}'$ is given by $t=1/C$. For $n\in\ZZ_{>0}$, the supermultiplicity of the imaginary simple root $-nt\rho$ is equal to the Fourier coefficient $f(n^2t^2AC,nt\vec{B})$ of the Jacobi form input $\phi_\mathfrak{g}$. Recall that $N^*_\mathfrak{g}\rho\in U\oplus L_\mathfrak{g}$, where $N^*_\mathfrak{g}$ is the level of $L_\mathfrak{g}$.  The value of 
$$
\mathrm{sm}(n):=f(n^2,nt\vec{B})=f(n^2t^2AC,nt\vec{B})=\text{$s$-}\mathrm{mult}(-nt\rho)
$$
depends only on the class of $n$ modulo $N^*_\mathfrak{g}/t$. We formulate these values as follows. 
\begin{enumerate}
\item $A_{1,16}$: $N^*_\mathfrak{g}=16$, $t=8$, and we have
$$
\mathrm{sm}(1)=-1, \quad \mathrm{sm}(2)=1.
$$
\item $A_{2,9}$: $N^*_\mathfrak{g}=9$, $t=3$, and we have
$$
\mathrm{sm}(1)=-1, \quad \mathrm{sm}(2)=-1, \quad \mathrm{sm}(3)=2.
$$
\item $A_{1,8}^2$: $N^*_\mathfrak{g}=8$, $t=4$, and we have
$$
\mathrm{sm}(1)=-2, \quad \mathrm{sm}(2)=2.
$$
\item $A_{1,4}^4$: $N^*_\mathfrak{g}=4$, $t=2$, and we have
$$
\mathrm{sm}(1)=-4, \quad \mathrm{sm}(2)=4.
$$
\item $A_{1,36}^*$: $N^*_\mathfrak{g}=144$, $t=24$, and we have
$$
\mathrm{sm}(n)=-1,\, 0,\, 1,\,0,\,-1,\,1,\quad \text{for $1\leq n\leq 6$}.
$$
\item $A_{1,9}^*A_{1,12}$: $N^*_\mathfrak{g}=36$, $t=6$, and we have
$$
\mathrm{sm}(n)=-2,\, 1,\, 0,\,1,\,-2,\,2,\quad \text{for $1\leq n\leq 6$}.
$$
\item $C_{2,10}^*$: $N^*_\mathfrak{g}=40$, $t=4$, and we have
$$
\mathrm{sm}(n)=-1,\, 0,\, -1,\,0,\,2,\,0,\,-1,\,0,\,-1,\,2,  \quad \text{for $1\leq n\leq 10$}.
$$
\item $A_{1,3}^*A_{1,4}A_{2,6}$: $N^*_\mathfrak{g}=12$, $t=2$, and we have
$$
\mathrm{sm}(n)=-3,\, 0,\,2,\,0,\,-3,\,4,\quad \text{for $1\leq n\leq 6$}.
$$
\end{enumerate}

\begin{remark}
We are unable to extend Theorem \ref{th:modularity-d-F24} to the $8$ exceptional $\mathfrak{g}$. It remains an open question whether the denominators of these BKM superalgebras give rise to singular automorphic products. 
Contrary to Theorem \ref{th:modularity-d-F24} and Remark \ref{rem:express-character-d-F24}, no nontrivial $\CC$-linear combination of the supercharacters of the affine VOSA generated by $\hat{\mathfrak{g}}$ defines a Jacobi form on $\Gamma_0(2)$ (but not on $\SL_2(\ZZ)$). Pursuing this line of thought, we find the following construction. 
\begin{enumerate}
\item For $\mathfrak{g}=A_{1,8}^2$, we have $C=1/4$, $L_\mathfrak{g}=2A_1(2)$, and $U(5)\oplus U\oplus L_{\mathfrak{g}}$ has a reflective automorphic product of singular weight, whose input is given by the Jacobi form pair:
\begin{small}
\begin{align*}
\widetilde{\phi}_{A_{1,8}^2}&=\big(\chi^{A_{1,8}}_{0,0}+\chi^{A_{1,8}}_{8,2}\big)\otimes\big(\chi^{A_{1,8}}_{2,\frac15}+\chi^{A_{1,8}}_{6,\frac65}\big)+\big(\chi^{A_{1,8}}_{2,\frac15}+\chi^{A_{1,8}}_{6,\frac65}\big)\otimes \big(\chi^{A_{1,8}}_{0,0}+\chi^{A_{1,8}}_{8,2}\big) +3\chi^{A_{1,8}}_{4,\frac35}\otimes\chi^{A_{1,8}}_{4,\frac35},\\
\phi_{A_{1,8}^2}&=\big(\chi^{A_{1,8}}_{0,0}+\chi^{A_{1,8}}_{8,2}\big)\otimes\big(\chi^{A_{1,8}}_{2,\frac15}+\chi^{A_{1,8}}_{6,\frac65}\big)+\big(\chi^{A_{1,8}}_{2,\frac15}+\chi^{A_{1,8}}_{6,\frac65}\big)\otimes \big(\chi^{A_{1,8}}_{0,0}+\chi^{A_{1,8}}_{8,2}\big)-2\chi^{A_{1,8}}_{4,\frac35}\otimes\chi^{A_{1,8}}_{4,\frac35}.
\end{align*}
\end{small}
This product shares the same Weyl group and Weyl vector as $\Psi_\mathfrak{g}=\Borch(\phi_{A_{1,8}^2})$. The attached cycle shape $\tilde{g}$ is $4^3 8^{-1} 20^{-1}40^{1}$.
\item For $\mathfrak{g}=A_{1,4}^4$, we have $C=1/2$, $L_\mathfrak{g}=4A_1$, and $U(3)\oplus U\oplus L_{\mathfrak{g}}$ has a reflective automorphic product of singular weight, whose input is given by the Jacobi form pair:
\begin{small}
\begin{align*}
\widetilde{\phi}_{A_{1,4}^4}&= \sum_{\text{cyclic}}\big(\chi^{A_{1,4}}_{0,0}+\chi^{A_{1,4}}_{4,1}\big)\otimes \big(\chi^{A_{1,4}}_{0,0}+\chi^{A_{1,4}}_{4,1}\big)\otimes \big(\chi^{A_{1,4}}_{0,0}+\chi^{A_{1,4}}_{4,1}\big)\otimes\chi^{A_{1,4}}_{2,\frac13}+5\bigotimes\chi^{A_{1,4}}_{2,\frac13},\\
\phi_{A_{1,4}^4}&= \sum_{\text{cyclic}}\big(\chi^{A_{1,4}}_{0,0}+\chi^{A_{1,4}}_{4,1}\big)\otimes \big(\chi^{A_{1,4}}_{0,0}+\chi^{A_{1,4}}_{4,1}\big)\otimes \big(\chi^{A_{1,4}}_{0,0}+\chi^{A_{1,4}}_{4,1}\big)\otimes\chi^{A_{1,4}}_{2,\frac13} -4\bigotimes\chi^{A_{1,4}}_{2,\frac13},
\end{align*}
\end{small}
where the sum contains $4$ terms by permutation. This product shares the same Weyl group and Weyl vector as $\Psi_\mathfrak{g}=\Borch(\phi_{A_{1,4}^4})$. The attached cycle shape $\tilde{g}$ is $2^5 4^{-1} 6^{-3} 12^3$. 
\end{enumerate}

The underlying lattices take the general form $U(1+t)\oplus U\oplus L_\mathfrak{g}$, where $t=1/C$. Recall that $t=1$ for $\mathfrak{g}$ related to $F_{24}$. For the remaining $6$ exceptional $\mathfrak{g}$, we expect analogous singular automorphic products on such lattices, though they would likely be non-reflective. 
\end{remark}

At the end of this section, guided by these cycle shapes, we examine the relationship between singular automorphic products $\Psi_\mathfrak{g}$ and the twisted denominators of the fake monster algebra. 

Let $g$ be a conjugacy class of Conway's group $\mathrm{Co_0}$ (the integral orthogonal group of the Leech lattice $\Lambda$), with  associated fixed-point sublattice $\Lambda^g$. Let $n_g$ denote the order of $g$, and let $N_g$ denote the minimal level of the associated eta quotient $\eta_g$ as a modular form on congruence subgroups. 

The last two named authors \cite[\S 6]{WW25} proved that the $g$-twisted denominator of the fake monster algebra gives a singular automorphic product $\Phi_g$ on $U(N_g)\oplus U\oplus \Lambda^g$, and $\Phi_g$ is reflective if $n_g=N_g$. This confirms a conjecture of Borcherds \cite{Bor92} and \cite[\S 15, Example 3]{Bor95}. Recall that Borcherds also realized the $g$-twisted denominator as the ordinary superdenominator of a BKM superalgebra $\mathbb{G}_g$ with root lattice $U\oplus \Lambda^g$ and Weyl vector $\rho_g=(-1,0,0)$. Moreover, the imaginary simple roots of $\mathbb{G}_g$ are precisely $-m\rho_g$ with supermultiplicity 
$\sum_{k\mid (m,n_g)} b_k$ for all positive integers $m$, where $\prod_{k|n_g}k^{b_k}$ denotes the cycle shape of $g$. Note that the level $N_g$ is equal to the smallest multiple of $n_g$ such that $N_g\sum_{k\mid n_g}\frac{b_k}{k}$ is divisible by $24$.  

For each of the singular automorphic products $\Psi_{\mathfrak{g}}$ in Tables \ref{table:class-1-8}-\ref{tab:class-2}, the attached cycle shape $g$ arises from a certain conjugacy class of $\mathrm{Co}_0$. With notation in Remark \ref{rem:U(N)-BKM}, there exists an even positive definite lattice $K_\mathfrak{g}\cong \rho^\perp/\rho$ such that $U(N^*_\mathfrak{g})\oplus K_\mathfrak{g}$ is an index-$t$ sublattice of $U\oplus L_\mathfrak{g}$. We find 
$$
N^*_\mathfrak{g}=N_g,\quad  t=N_g/n_g, \quad K_\mathfrak{g}\cong {\Lambda^g}'(N_g),
$$
where ${\Lambda^g}'$ denotes the dual of $\Lambda^g$. In particular, the two root lattices $U(1/N^*_\mathfrak{g})\oplus K_\mathfrak{g}'$ and $U\oplus \Lambda^g$ are identical up to scaling by $N_g$. Taking $\mathfrak{g}=B_{12,2}$ as an example, the cycle shape $g$ is $2^{12}$, and we have $n_g=2$, $N^*_\mathfrak{g}=N_g=4$, $t=2$, $L_\mathfrak{g}=D_{12}(2)$, $\Lambda^g=D_{12}^+(2)\cong {\Lambda^g}'(4)$. 
\begin{itemize}
\item[(i)] When $C\in\ZZ$, we have $t=1$; $\Psi_\mathfrak{g}$ is identical to $\Phi_g$; and the BKM superalgebras $\mathcal{G}_\mathfrak{g}$ and $\mathbb{G}_g$ are isomorphic (here we choose $\mathcal{G}_\mathfrak{g}$ to have the same ordinary multiplicities of imaginary simple roots as $\mathbb{G}_g$).
\item[(ii)] When $C\not\in\ZZ$, we have $t=2$. Although $\Psi_\mathfrak{g}$ remains identical to $\Phi_g$, the two algebras are no longer isomorphic: $\mathcal{G}_\mathfrak{g}\not\cong \mathbb{G}_g$. They share the same real roots and superdenominator, but differ in that $-\rho$ is an imaginary simple root of $\mathbb{G}_g$ (both even and odd with equal multiplicity), while it is not a root of $\mathcal{G}_\mathfrak{g}$ (since $\rho \not\in U\oplus L_\mathfrak{g}'$).    
\end{itemize}
 
For each of the $8$ exceptional $\mathfrak{g}$ in Tables \ref{tab:class-3} and \ref{tab:class-4}, the associated cycle shape is realized by a conjugacy class of $\mathrm{Co}_0$ if and only if $\mathfrak{g}=A_{2,9}$, $A_{1,8}^2$, $A_{1,4}^4$. For these cases, the corresponding $\Lambda^g$ are $A_2(3)$, $2A_1(2)$, and $D_4(2)$, respectively. The relation between $\mathcal{G}_\mathfrak{g}$ and $\mathbb{G}_g$ is unclear. It is a challenging open problem to find a natural construction of the resulting BKM superalgebra. The exceptional modular invariants to be discussed in the next two sections may provide some clues.

\section{Modular invariants for affine   \texorpdfstring{$\mathfrak{osp}_{1|2}$}{}}\label{sec:modular-invariants}
In this section, we establish a theory of modular invariants for affine Lie superalgebras of type $\mathfrak{osp}_{1|2}$. The results obtained here will be used in the next section to construct exceptional modular invariants related to the singular automorphic products of Table \ref{tab:class-4}.  

\subsection{Supercharacters of affine \texorpdfstring{$\mathfrak{osp}_{1|2}$}{}}\label{subsec:osp(1|2)}
We derive explicit formulas from Section \ref{subsec:osp} for computing the supercharacters of $\widehat{\mathfrak{osp}}_{1|2}$ at positive integral levels. Let $k$ be a positive integer. The affine vertex operator superalgebra $L_{\widehat{\mathfrak{osp}}_{1|2}}(k\hat{w}_0)$ has central charge $c_k=2k/(2k+3)$, and its irreducible modules are indexed by dominant integral weights $k\hat{w}_0+j\epsilon_1$ for $0\leq j\leq k$. For simplicity, we denote the supercharacter of $L_{\widehat{\mathfrak{osp}}_{1|2}}(k,j\epsilon_1)$ by $\chi^{\mathfrak{osp}_{1|2,k}}_j$, and its modification by
$$
\widetilde{\chi}^{\mathfrak{osp}_{1|2,k}}_j:=q^{h_j-c_k/24}\cdot \chi^{\mathfrak{osp}_{1|2,k}}_j, \quad h_j=\frac{j(j+1)}{2(2k+3)},
$$
where $h_j$ is the conformal weight. When no confusion arises, we abbreviate these as $\chi_j$ and $\widetilde{\chi}_j$. Recall that $\vartheta_{3/2}(\tau,z):=\eta(\tau)\vartheta(\tau,2z)/\vartheta(\tau,z)$ is the superdenominator function of $\widehat{\mathfrak{osp}}_{1|2}$. By direct computation, for any $0\leq j \leq k$, we find 
\begin{equation}
\widetilde{\chi}^{\mathfrak{osp}_{1|2,k}}_j(\tau,z)=\frac{\Theta_{2k+3, 2j+1}(\tau,z)}{\vartheta_{3/2}(\tau,z)} \quad \text{and} \quad \widetilde{\chi}^{\mathfrak{osp}_{1|2,k}}_j(\tau,0)=\frac{\theta_{2k+3,2j+1}(\tau)}{\eta(\tau)},
\end{equation}
where for any positive odd integer $N$ and any odd integer $a$,  
\begin{equation}
\Theta_{N,a}(\tau,z) =  \sum_{t\in\ZZ} (-1)^t q^{(2Nt+a)^2/(8N)}\big( \zeta^{(2Nt+a)/2} + \zeta^{-(2Nt+a)/2} \big),
\end{equation}
and $\theta_{N,a}(\tau)=\frac{1}{2}\Theta_{N,a}(\tau,0)$. Obviously, one has
\begin{equation}\label{eq:sign-Theta}
\Theta_{N,-a}=\Theta_{N,a}, \quad \Theta_{N,2N+a}=-\Theta_{N,a}, \quad \Theta_{N,N}=0.     
\end{equation}
These supercharacters can also be expressed as follows:
\begin{equation}\label{eq:Andrews--Gordon}
\chi^{\mathfrak{osp}_{1|2,k}}_j(\tau,0)=\prod_{\substack{n\neq 0,\,\pm (k-j+1) \\ \bmod (2k+3) }} \frac{1}{1-q^n} = \sum_{n_1,\ldots,n_k\geq 0} \frac{q^{N_1^2+\cdots+N_k^2+N_{k-j+1}+\cdots+N_k}}{(q)_{n_1}\cdots(q)_{n_k}},   
\end{equation}
where $N_{k+1}=0$ and $N_i:=n_i+\cdots+n_k$ for $1\leq i \leq k$, and $(q)_n:=\prod_{m=1}^n(1-q^m)$. The first equality follows from the Jacobi triple product formula, and the last equality follows from the celebrated Andrews--Gordon identity \cite{And74, Gor61}. In particular, $\chi^{\mathfrak{osp}_{1|2,k}}_j(\tau,0)$ has non-negative Fourier coefficients. 

Recall that $\widetilde{\chi}^{\mathfrak{osp}_{1|2,k}}_j(\tau,z)$ for $0\leq j\leq k$ form a vector-valued weakly holomorphic Jacobi form of weight $0$ and index $A_1(k)$ for a certain unitary representation $\varrho_k$ of $\SL_2(\ZZ)$. Note that $\varrho_k(T)$ is diagonal and $\dim \varrho_k=k+1$. Since the conformal weights are distinct, the supercharacters completely determine $\varrho_k$. The transformation matrix $\varrho_k(S)$ is given by
\begin{equation}\label{eq:S-osp(1|2)}
\varrho_k(S)_{jj'}=\frac{2}{\sqrt{2k+3}}\cos\Big(\frac{2(j+1/2)(j'+1/2)}{2k+3}\pi\Big).    
\end{equation}

\subsection{Macdonald-type identity for affine \texorpdfstring{$\mathfrak{osp}_{1|2}$}{}}
We present a complete classification of the one-dimensional subrepresentations of $\varrho_k$ and derive the corresponding Macdonald-type identity. 

For simple affine Lie algebras, Macdonald-type identities state that certain integral linear combinations of affine characters with special levels become constant when setting the Lie algebra fugacities to zero (see \cite[Page 612]{CFT}). For example, for the affine Lie algebra of type $A_1$ at level $2n^2-2$ with $n\in\ZZ_{>1}$, there exists the following refined Macdonald-type identity (see \eqref{eq:sl_2 characters})
\begin{equation}\label{eq:Macdonald-A1}
\sum_{j=0}^{n-1} (-1)^j \cdot \widetilde{\chi}^{\mathfrak{sl}_{2,\, 2n^2-2}}_{n-1+2nj}(\tau,z)  = \frac{\vartheta(\tau,nz)}{\vartheta(\tau,z)}. 
\end{equation}

We find an analog of the above identity for the affine Lie superalgebra of type $\mathfrak{osp}_{1|2}$. 

\begin{theorem}\label{th:Macdonald-osp}
The unitary representation $\varrho_k$ has a one-dimensional subrepresentation if and only if $k=6n(n+1)$ for a positive integer $n$. When $k=6n(n+1)$, the one-dimensional subrepresentation is unique and trivial. Moreover, it is characterized by the Macdonald-type identity
\begin{equation}\label{eq:Macdonald-osp}
\sum_{j=0}^{3n} \mathrm{sgn}(j)\cdot   \widetilde{\chi}^{\mathfrak{osp}_{1|2,\, 6n(n+1)}}_{n+(2n+1)j}(\tau,z)  = \frac{\vartheta(\tau,z)\vartheta\big(\tau,2(2n+1)z\big)}{\vartheta(\tau,2z)\vartheta\big(\tau,(2n+1)z\big)},
\end{equation}
where $\mathrm{sgn}(j)=1,0,-1,-1,0,1$ for $j=0,1,2,3,4,5$ and $\mathrm{sgn}(t+6)=\mathrm{sgn}(t)$ for any $t\in\NN$. 
\end{theorem}

\begin{proof}
Suppose $\varrho_k$ has a one-dimensional subrepresentation. The  $k+1$ modified supercharacters have distinct leading $q$-orders and thus are linearly independent over $\CC$, so there exists a $\CC$-linear combination of supercharacters that defines a weakly holomorphic Jacobi form of weight $0$ and index $A_1(k)$ on $\SL_2(\ZZ)$. Denote this Jacobi form by $\phi(\tau,z)$. Linear independence yields $\phi(\tau,0)\neq 0$. By the supercharacter formula, $\hat{\phi}(\tau,z):=\vartheta_{3/2}(\tau,z)\cdot \phi(\tau,z)$ is a holomorphic Jacobi form of weight $1/2$ and index $\ZZ(2k+3)$ on $\SL_2(\ZZ)$ with a character. By \cite[Theorem 12.1]{GSZ19}, $\hat{\phi}(\tau,z)$ can only be $\vartheta(\tau,mz)$ or $\vartheta_{3/2}(\tau,mz)$ for a positive integer $m$. Since $\hat{\phi}(\tau,0)=2\eta(\tau)\phi(\tau,0)\neq 0$, we obtain $\hat{\phi}(\tau,z)=c\cdot \vartheta_{3/2}(\tau,mz)$ for a nonzero constant $c$. It follows that $2k+3=3m^2$ and thus $m$ is an odd integer denoted $2n+1$, leading to $k=6n(n+1)$. Therefore, $\phi$ is a Jacobi form of trivial character and the one-dimensional subrepresentation is trivial. The uniqueness follows from the linear independence. 

To prove the theorem, it remains to verify that the left hand side of \eqref{eq:Macdonald-osp} is invariant under the action of two generators $T$ and $S$. The supercharacters with $\mathrm{sgn}(j)\neq 0$ in \eqref{eq:Macdonald-osp} are precisely those with integral leading $q$-orders $h_j-c_k/24$. This leads to the invariance under the action of $T$. By \eqref{eq:S-osp(1|2)}, the invariance under the action of $S$ follows from the identity
\begin{align*}
&\sum_{j=0}^{3n}\mathrm{sgn}(j)\cdot \cos\Big( \frac{(2j+1)(2j'+1)\pi}{6(2n+1)}\Big) \\
=&\left\{\begin{array}{ll}
\frac{(2n+1)\sqrt{3}}{2}\cdot\mathrm{sgn}(t), &\text{if $j'=n+(2n+1)t$ for $t\in\NN$},\\
0, &\text{otherwise.}
\end{array}\right.    
\end{align*}
To prove this identity, we observe $\mathrm{sgn}(j)=\frac{2}{\sqrt{3}}\cos\big( (2j+1)\pi/6\big)$ and calculate
\begin{align*}
f:=&\frac{2}{\sqrt{3}}\sum_{j=0}^{3n}\cos\Big( \frac{(2j+1)\pi}{6} \Big) \cdot \cos\Big( \frac{2j'+1}{2n+1} \cdot \frac{(2j+1)\pi}{6} \Big) \\
=&\frac{1}{\sqrt{3}}\sum_{j=0}^{3n} \left(\cos\Big( \frac{2j'+2n+2}{2n+1} \cdot \frac{(2j+1)\pi}{6} \Big) + \cos\Big( \frac{2j'-2n}{2n+1} \cdot \frac{(2j+1)\pi}{6} \Big) \right).
\end{align*}
For any $x\in \RR$ with $\sin(x)\neq 0$, one has
\begin{equation}\label{eq:sin}
\sum_{j=0}^M \cos\big( (2j+1)x \big)= \frac{\sin\big( 2(M+1)x \big)}{2\sin(x)}. \end{equation}
We set $x:=\frac{j'+n+1}{3(2n+1)}\pi$ and $x':=\frac{j'-n}{3(2n+1)}\pi$. When $j'\neq n+(2n+1)t$ for any $t\in\NN$, $\sin(x)\neq 0$ and $\sin(x')\neq 0$, and thus \eqref{eq:sin} gives
\begin{align*}
f&=\frac{1}{2\sqrt{3}}\left( \frac{\sin\big((j'+n+1)\pi - x\big)}{\sin(x)} + \frac{\sin\big((j'-n)\pi - x'\big)}{\sin(x')} \right)\\
&=\frac{-1}{2\sqrt{3}}\big( (-1)^{j'+n+1}+(-1)^{j'-n} \big)=0.     
\end{align*}
When $j'=n+(2n+1)t$ for some $t\in\NN$, $x=\frac{t+1}{3}\pi$ and $x'=\frac{t}{3}\pi$, and thus 
$$
f=\frac{1}{\sqrt{3}}\sum_{j=0}^{3n}\left( \cos\Big( \frac{(t+1)(2j+1)}{3}\pi \Big) +\cos\Big( \frac{t(2j+1)}{3}\pi \Big) \right).
$$
Taking into account $t$ modulo $3$, we prove $f=\sqrt{3}(n+1/2)\cdot\mathrm{sgn}(t)$.
\end{proof}

The above theorem has an immediate consequence. 

\begin{corollary}
For any positive integer $n$, 
\begin{equation}\label{eq:1=osp}
\sum_{j=0}^{3n} \mathrm{sgn}(j)\cdot   \widetilde{\chi}^{\mathfrak{osp}_{1|2,\, 6n(n+1)}}_{n+(2n+1)j}(\tau,0)=1.
\end{equation}  
\end{corollary}

By \eqref{eq:Andrews--Gordon}, Equality \eqref{eq:1=osp} can be reformulated as follows, establishing a link with partitions. 
\begin{corollary}
For any positive integer $n$, 
\begin{equation}
1=\sum_{j=0}^{3n}\mathrm{sgn}(j)\cdot q^{\frac{j(j+1)}{6}} \cdot \prod_{\substack{m\neq 0,\, \pm(2n+1)(3n-j+1)\\ \bmod 3(2n+1)^2}}\frac{1}{1-q^m}.
\end{equation}  
Let $A_{k,i}(n)$ denote the number of the partitions of $n\in\NN$ whose parts are not congruent to $0$, $i$, or $-i$ modulo $k$, and put $A_{k,i}(n)=0$ for $n\not\in\NN$. Then for any positive integers $m$ and $n$, we have 
\begin{equation}
\sum_{j=0}^{3n} \mathrm{sgn}(j) \cdot A_{3(2n+1)^2,\, (2n+1)(3n-j+1)}\big(m-j(j+1)/6\big)=0.    
\end{equation}
\end{corollary}

In particular, the specific cases of $n=1, 2$ yield that for any positive integer $m$, 
\begin{align*}
&A_{27,3}(m-2)+A_{27,6}(m-1) = A_{27,12}(m),\\
&A_{75,5}(m-7) + A_{75,10}(m-5)+A_{75,35}(m)=A_{75,20}(m-2)+A_{75,25}(m-1).
\end{align*}

\subsection{Modular invariants of affine \texorpdfstring{$\mathfrak{osp}_{1|2}$}{}}
We continue to classify the subrepresentations of $\varrho_k$, which lead to non-diagonal modular invariants of affine $\mathfrak{osp}_{1|2}$. We first briefly recall the theory of modular invariants following \cite[Chapter 17]{CFT}. A two-dimensional rational CFT consists of holomorphic and anti-holomorphic sectors, whose combination gives rise to modular invariants, that is, $\SL_2(\mathbb{Z})$-invariant sesquilinear combinations of characters:
$$
Z=\sum M_{ab} \chi_a\overline{\chi}_b,\qquad M_{ab} \in\mathbb{N}.
$$
Clearly, the torus partition function $\sum_i|\chi_i|^2$, i.e., the sum of the norm squares of the characters of all conformal primary fields, is $\SL_2(\mathbb{Z})$-invariant. It is called a \emph{diagonal modular invariant} as the matrix $M_{ab}$ is diagonal. Sometimes, a rational CFT may possess \emph{non-diagonal modular invariants}, which are often related to VOA extensions. There are two principal approaches to constructing non-diagonal modular invariants: \emph{conformal embedding} and \emph{simple current extension}. 
Modular invariants that cannot be obtained from simple current extensions are called \emph{exceptional}. These are rare and typically arise from special conformal embeddings or from nontrivial automorphisms of fusion algebras. 
The classification of modular invariants for affine Lie algebras was a rich and active subject, especially in the 1980s and 1990s. Complete classifications are known for $\widehat{\mathfrak{sl}}_2$ \cite{Cappelli,Kato} and $\widehat{\mathfrak{sl}}_3$ \cite{Gannon}. For more recent developments, we refer the reader to \cite{gannon2023exotic,edie2024}.

In this subsection, we propose constructions of non-diagonal modular invariants for the affine Lie superalgebra $\widehat{\mathfrak{osp}}_{1|2}$. Additional exceptional modular invariants related to $\widehat{\mathfrak{osp}}_{1|2r}$ will be constructed in the next section.  In our construction, we replace the characters of a 2D rational CFT by the supercharacters of $\widehat{\mathfrak{osp}}_{1|2r}$, and we do not require the coefficients $M_{ab}$ to be non-negative. 
In previous work \cite{SWW23}, we discovered a surprising connection between exceptional modular invariants and singular automorphic products. The present paper aims to explore this connection further. 

The modular invariants constructed later are closely related to the Virasoro minimal models, so we first review their theory. Minimal models are 2D rational CFTs with only Virasoro symmetry. They are classified by sets $\{p,p'\}$ of coprime integers greater than $1$, and we denote them by $\mathrm{Vir}(p,p')$ or $\mathrm{Vir}(p',p)$. This theory has central charge $1-\frac{6(p-p')^2}{pp'}$ and has exactly $(p-1)(p'-1)/2$ conformal primaries, which are labeled by integer pairs $(r,s)$ located inside the rectangle
$$
E_{p,p'}:=\{ (r,s)\in\ZZ^2: \; 1\leq r < p', \; 1\leq s < p \}
$$
modulo the equivalence relation $(r,s)\sim (p'-r,p-s)$.
The conformal weights of $\mathrm{Vir}(p,p')$ are 
$$
\frac{(pr-p's)^2-(p-p')^2}{4pp'}.
$$
The characters of $\mathrm{Vir}(p,p')$ (modified by the central charge and conformal weights) are given by 
\begin{equation}
\widetilde{\chi}_{r,s}^{\mathrm{Vir}(p,p')} = \frac{1}{\eta(\tau)} \sum_{n\in\ZZ}\Big(q^{\frac{(2pp'\cdot n+pr-p's)^2}{4pp'}} - q^{\frac{(2pp'\cdot n+pr+p's)^2}{4pp'}}  \Big),  
\end{equation}
and satisfy the symmetry relation
\begin{equation}
\widetilde{\chi}_{r,\, s}^{\mathrm{Vir}(p,p')} = \widetilde{\chi}_{p'-r,\, p-s}^{\mathrm{Vir}(p,p')}. 
\end{equation}
We choose the conformal primary whose character has the smallest leading $q$-order as the vacuum. In other words, the vacuum is labeled by the unique pair $(r,s)\in E_{p,p'}$ with $pr-p's=1$. This choice yields the effective theory $\mathrm{Vir}_{\mathrm{eff}}(p,p')$, whose effective central charge and effective conformal weights are given as follows: 
$$
c_{\mathrm{eff}}=1-\frac{6}{pp'}, \qquad h_{r,s}^{\mathrm{eff}}=\frac{(pr-p's)^2-1}{4pp'}. 
$$
Upon setting $z=0$, we find that the supercharacters of the affine VOSA $L_{\widehat{\mathfrak{osp}}_{1|2}}(k\hat{w}_0)$ coincide precisely with the characters of the effective Virasoro minimal model  $\mathrm{Vir}_{\mathrm{eff}}(2k+3,2)$, commonly referred to as the Lee–Yang model of level $k$ in the physics literature, see, e.g., \cite{DLS22}. 

The modular invariants of minimal models have been completely classified; a full description is given in \cite[Chapter 10.7]{CFT}. Each modular invariant is labeled by a pair of ADE Lie algebras. In the special case where $p'=2(2m+1)$ with $m\in \ZZ_{>0}$, the minimal model $\mathrm{Vir}(p,p')$ admits a non-diagonal modular invariant
\begin{equation}
Z_{D_{p'/2+1},\, A_{p-1}}=\frac{1}{4}\sum_{\substack{(r,s)\in E_{p,p'}\\ \text{$r$ is odd}}}\big|\widetilde{\chi}_{r,s}^{\mathrm{Vir}(p,p')} + \widetilde{\chi}_{p'-r,s}^{\mathrm{Vir}(p,p')} \big|^2. 
\end{equation}

Inspired by this classification, we propose the following non-diagonal modular invariants of $\widehat{\mathfrak{osp}}_{1|2}$. 

\begin{theorem}\label{conj:modular invariants-osp}
Let $p$ and $p'$ be positive odd integers, and let $k$ be a positive integer. 
\vspace{2mm}

\noindent
\textbf{Part (I).} Suppose $2k+3=p^2p'$ with $p'\geq 3$. For any integer $0\leq l\leq \frac{p'-3}{2}$, we define
\begin{equation}
\phi_l:=\sum_{j=0}^{(p-1)/2}(-1)^j \cdot \widetilde{\chi}^{\mathfrak{osp}_{1|2,k}}_{pp'j+pl+\frac{p-1}{2}} +  \sum_{j=1}^{(p-1)/2}(-1)^j\cdot \widetilde{\chi}^{\mathfrak{osp}_{1|2,k}}_{pp'j-pl-\frac{p+1}{2}}.   
\end{equation}
Then we have 
\begin{equation}\label{eq:osp=osp}
\phi_l(\tau,z)=\frac{\vartheta_{3/2}(\tau,pz)}{\vartheta_{3/2}(\tau,z)}\cdot \widetilde{\chi}^{\mathfrak{osp}_{1|2,\, (p'-3)/2}}_{l}(\tau,pz).     
\end{equation}
In particular, 
\begin{equation}
\phi_l(\tau,0) = \widetilde{\chi}^{\mathfrak{osp}_{1|2,\, (p'-3)/2}}_{l}(\tau,0), 
\end{equation}
that is, at $z=0$, these $\phi_l$ coincide with supercharacters of $\widehat{\mathfrak{osp}}_{1|2}$ at level $\frac{p'-3}{2}$ or equivalently, characters of the effective minimal model $\mathrm{Vir}_{\mathrm{eff}}(p', 2)$. Consequently, $\varrho_k$ has a subrepresentation of dimension $\frac{p'-1}{2}$ denoted by $\hat{\varrho}_k$, which is isomorphic to $\varrho_{(p'-3)/2}$ and characterized by the $\SL_2(\ZZ)$-invariant (up to a factor that degenerates to $1$ when $z=0$, here and throughout) partition function
\begin{equation}\label{eq:Macdonald-osp-extension}
\sum_{l=0}^{(p'-3)/2} |\phi_l(\tau,z)|^2.     
\end{equation}

\vspace{2mm}

\noindent
\textbf{Part (II).} Suppose $p,\, p'>1$, $\mathrm{gcd}(p,p')=1$, and $2k+3=pp'$. Define
\begin{align*}
J_p &:= \left\{ j\in \ZZ : \; 0\leq j\leq k, \; j\equiv (p-1)/2 \; \bmod p \right\},\\
J_{p'} &:= \left\{ j\in \ZZ : \; 0\leq j\leq k, \; j\equiv (p'-1)/2 \; \bmod p' \right\},\\
J &:= \{0,1,...,k\} \backslash (J_p\cup J_{p'}).
\end{align*}
Note that $0,k\in J$, $J_p\cap J_{p'}=\emptyset$, and
$$
|J_p|=\frac{p'-1}{2}, \quad |J_{p'}|=\frac{p-1}{2}, \quad |J|=\frac{(p-1)(p'-1)}{2}.
$$
For any $j\in J$, there is a unique $j'\in J$ such that $j'\neq j$ and either
\begin{equation}\label{eq:condition-1}
\delta^{(p)}_j:=\frac{j-j'}{p}\in \ZZ \quad \text{and} \quad \delta^{(p')}_j:=\frac{j+j'+1}{p'}\in \ZZ    
\end{equation}
or 
\begin{equation}\label{eq:condition-2}
\delta^{(p)}_j:=\frac{j+j'+1}{p}\in \ZZ \quad \text{and} \quad \delta^{(p')}_j:=\frac{j-j'}{p'}\in \ZZ. 
\end{equation}
In either case $\delta^{(p)}_j+\delta^{(p')}_j$ is odd. 
For any $j\in J$, we define
\begin{align}
\phi_j^{(p)}:=&\ \widetilde{\chi}^{\mathfrak{osp}_{1|2,k}}_j+(-1)^{\delta^{(p)}_j}\cdot \widetilde{\chi}^{\mathfrak{osp}_{1|2,k}}_{j'}, \\ 
\phi^{(p')}_j:=&\ \widetilde{\chi}^{\mathfrak{osp}_{1|2,k}}_j+(-1)^{\delta^{(p')}_j}\cdot \widetilde{\chi}^{\mathfrak{osp}_{1|2,k}}_{j'}.      
\end{align}
Then partition functions
\begin{align*}
Z^{(p)}_{\mathrm{D}} &:=\frac{1}{2}\sum_{j\in J}\Big|\phi^{(p)}_j\Big|^2 + 2\sum_{j\in J_{p'}}\Big|\widetilde{\chi}^{\mathfrak{osp}_{1|2,k}}_j\Big|^2,\\    
Z^{(p')}_{\mathrm{D}} &:=\frac{1}{2}\sum_{j\in J}\Big|\phi^{(p')}_j\Big|^2 + 2\sum_{j\in J_p}\Big|\widetilde{\chi}^{\mathfrak{osp}_{1|2,k}}_j\Big|^2
\end{align*}
are $\SL_2(\ZZ)$-invariant and satisfy
\begin{equation}
Z^{(p)}_{\mathrm{D}} + Z^{(p')}_{\mathrm{D}} = 2\sum_{j=0}^k \Big|\widetilde{\chi}^{\mathfrak{osp}_{1|2,k}}_j\Big|^2. 
\end{equation}
At $z=0$, $Z^{(p)}_{\mathrm{D}}$ and $Z^{(p')}_{\mathrm{D}}$ agree with the $D$-type modular invariants of $\mathrm{Vir}_{\mathrm{eff}}(p,2p')$ and $\mathrm{Vir}_{\mathrm{eff}}(p',2p)$, respectively. More precisely,
for any $j\in J$, up to an overall sign, each of $\phi^{(p)}_j(\tau,0)$ and $\phi^{(p')}_j(\tau,0)$ can be  expressed as a sum of two distinct minimal model characters, and thus its Fourier coefficients are non-negative; for any $j\in J_p\cup J_{p'}$,  the supercharacter $\widetilde{\chi}^{\mathfrak{osp}_{1|2,k}}_j$ equals a single minimal model character. Let $\varrho^{(p)}_k$ and $\varrho^{(p')}_k$ denote the subrepresentations of $\varrho_k$ characterized by $Z^{(p)}_{\mathrm{D}}$ and $Z^{(p')}_{\mathrm{D}}$, respectively. Then we have $\varrho_k=\varrho^{(p)}_k\oplus \varrho^{(p')}_k$, and
$$
\dim \varrho^{(p)}_k = \frac{(p-1)(p'+1)}{4}, \quad
\dim \varrho^{(p')}_k = \frac{(p+1)(p'-1)}{4}. 
$$
In addition, all components of $Z^{(p)}_{\mathrm{D}}$ at $z=0$ satisfy a monic MLDE of order $\dim \varrho^{(p)}_k$, and those of $Z^{(p')}_{\mathrm{D}}$ satisfy a monic MLDE of order $\dim \varrho^{(p')}_k$.
\end{theorem}

Recall that a monic MLDE of order $m$ is a modular linear differential equation of type 
$$
\bigg(\frac{d^m}{d\tau^m} + \sum_{j=0}^{m-1} f_j\frac{d^{j}}{d\tau^{j}}\bigg)\chi=0, 
$$
where $f_j \in \CC[E_2,E_4,E_6]$ is a quasi-modular form of weight $2(m-j)$ for any $0\leq j \leq m-1$. We refer to \cite[Section 2.2]{DLS22} and references therein for details. All characters of $\mathrm{Vir}(2k+3,2)$ satisfy a monic MLDE of order $k+1$. The MLDEs provide an effective approach to classifying admissible characters of 2d rational CFTs \cite{Eguchi87,Mathur88}.

\begin{proof}[Proof of Theorem \ref{conj:modular invariants-osp}]
We first prove Part (I). Since $p$ is odd, we have 
$$
\Theta_{p',2l+1}(\tau,pz)=\sum_{r=-(p-1)/2}^{(p-1)/2} (-1)^r \cdot \Theta_{p^2p', p(2l+1)+2pp'r}(\tau,z). 
$$
Dividing by $\vartheta_{3/2}(\tau,z)$, this equality leads to \eqref{eq:osp=osp}. The rest is clear. 

We then prove Part (II). Let us choose $\omega$ modulo $4pp'$ with $\omega \equiv -1 \bmod 4p$ and $\omega\equiv 1 \bmod p'$. It satisfies $\omega^2\equiv 1 \bmod 8pp'$. Thus multiplication by $\omega$ is an orthogonal involution of 
$$
D=\ZZ/(4pp'\ZZ) \quad \text{with} \quad Q(x)=x^2/(8pp') \; \bmod \ZZ. 
$$
In $\CC[D]$, we consider the vectors
$$
v_a=e_a+e_{-a}-e_{a+2pp'}-e_{-a+2pp'}, \quad \text{$a$ odd}. 
$$
These vectors satisfy
$$
v_{-a}=v_a, \quad v_{a+2pp'}=-v_a, \quad v_{pp'}=0. 
$$
Their span $W$ is invariant with respect to the dual Weil representation of $\mathrm{Mp}_2(\ZZ)$ associated with $D$. From the decomposition of $\Theta_{pp',a}$ into ordinary Jacobi theta functions associated with the lattice $A_1(2pp')$, we see that $v_a$ corresponds to $\Theta_{pp',a}$. It follows that $v_1$, $v_3$, ..., $v_{pp'-2}$ form a basis of $W$. Multiplication by $\omega$ induces $M_\omega(e_a):=e_{\omega a}$. Then $M_\omega(v_a)=v_{\omega a}$, and thus $M_\omega$ preserves $W$.  

For any $a=2j+1\in\{1,3,...,pp'-2\}$, there is an odd integer $r$ satisfying $-pp'<r\leq pp'$ and $\omega a\equiv r \bmod 2pp'$. Note that $r\neq pp'$. Then $a':=|r|\in\{1,3,...,pp'-2\}$, and we write $a'=2j'+1$. Put $\sigma_j:=(-1)^{(\omega a - r)/(2pp')}$. Then $v_{\omega a}=\sigma_j\cdot v_{a'}$. Note that $\omega \equiv \omega^{-1} \bmod 4pp'$ and $\sigma_j=\sigma_{j'}$. 

Furthermore, $M_\omega$ induces a signed permutation matrix $P_p$ as follows: for any $0\leq j\leq (pp'-3)/2$,  
$$
P_p:\; v_{2j+1} \mapsto \sigma_j\cdot v_{2j'+1}, \quad (P_p)_{j',j}=\sigma_j. 
$$
Moreover, $P_p$ defines a signed permutation on the supercharacter labels that commutes with the representation $\varrho_k$ associated with the affine $\mathfrak{osp}_{1|2}$ at level $k$, that is, $P_p:\; \widetilde{\chi}_j \mapsto \sigma_j\cdot \widetilde{\chi}_{j'}$. By direct computation, and in accordance with the notation of the theorem, we obtain
\begin{equation}\label{eq:P}
 P_p(\widetilde{\chi}_{j})=\begin{cases}
  \widetilde{\chi}_{j},&j\in J_{p'},\\
  -\widetilde{\chi}_{j},&j\in J_p,\\
  (-1)^{\delta_j^{(p)}}\cdot \widetilde{\chi}_{j'},&j\in J. 
 \end{cases}
\end{equation}
Define $\widetilde{\chi}:=(\widetilde{\chi}_0,\widetilde{\chi}_1,...,\widetilde{\chi}_k)^{\mathrm{T}}$. We check that
$$
Z_D^{(p)}=\widetilde{\chi}^*(I+P_p)\widetilde{\chi} \quad \text{and} \quad Z_D^{(p')}=\widetilde{\chi}^*(I-P_p)\widetilde{\chi}.
$$
Since $\varrho_k$ is unitary and the Hermitian matrices $I\pm P_p$ commute with $\varrho_k$, the above equalities yield the modular invariance of the  partition functions. 

We now identify $Z_D^{(p)}$ with the $D$-type modular invariant of $\mathrm{Vir}_{\mathrm{eff}}(p,2p')$. Take $z=0$ in this part. Let $1\leq r < 2p'$ and $1\leq s < p$. For odd $r$, we verify 
\begin{equation}\label{eq:identification}
\widetilde{\chi}_{r,\, s}^{\mathrm{Vir}(p,2p')}+\widetilde{\chi}_{2p'-r,\, s}^{\mathrm{Vir}(p,2p')} = \frac{\theta_{pp', pr-2p's}-\theta_{pp', pr+2p's}}{\eta}.     
\end{equation}
In particular, for $r=p'$, the two characters on the left coincide and the identity reduces to
$$
\widetilde{\chi}_{p',\, s}^{\mathrm{Vir}(p,2p')} = \frac{\theta_{pp',p'(p-2s)}}{\eta}. 
$$
The above two identities are enough to identify $Z_D^{(p)}$ with the following $D$-type modular invariant:
\begin{equation}
\frac{1}{4}\sum_{\substack{1\leq r<2p',\;  \text{$r$ odd}\\ 1\leq s<p}}\big|\widetilde{\chi}_{r,s}^{\mathrm{Vir}(p,2p')} + \widetilde{\chi}_{2p'-r,s}^{\mathrm{Vir}(p,2p')} \big|^2. 
\end{equation}
Similarly, $Z_D^{(p')}$ may be dealt with by exchanging $p$ and $p'$. 

To prove that all components of $Z^{(p)}_{\mathrm{D}}$ at $z=0$ satisfy a monic MLDE of order $d_p:=\dim \varrho^{(p)}_k$, by the valence formula and modular Wronskian, it suffices to verify the equality
\begin{equation}\label{eq:monic}
d_p(d_p-1)=12\sum_{j=1}^{d_p} \alpha_j,    
\end{equation}
where these $\alpha_j$ denote the leading $q$-orders of components (at $z=0$). In fact, these components are given by \eqref{eq:identification} for $r=1,3,...,p'$ and $1\leq s\leq (p-1)/2$, and their leading $q$-orders are 
$$
\alpha_{r,s}=\frac{(pr-2p's)^2}{8pp'}-\frac{1}{24}.
$$
These exponents are pairwise distinct and their sum agrees with $d_p(d_p-1)/12$. We then complete the proof. The case of $Z^{(p')}_{\mathrm{D}}$ is similar. 
\end{proof}

\begin{remark}
We expect that \emph{all} non-diagonal modular invariants of $\widehat{\mathfrak{osp}}_{1|2}$ can be obtained from Theorem \ref{conj:modular invariants-osp} with Parts (I) and (II) combined together. We have verified this up to level $300$. Given that the highest level for the appearance of exceptional modular invariants is $28$ for $\widehat{\mathfrak{sl}}_2$ and $21$ for $\widehat{\mathfrak{sl}}_3$, we believe this is strong evidence for completeness. 
\end{remark}

\begin{remark} Let us take $p=2n+1$ and $p'=3$ in Part (I) of Theorem \ref{conj:modular invariants-osp}. In this case, $k=6n(n+1)$, $\mathrm{Vir}(3, 2)$ is trivial, and $\phi_0(\tau,z)=\vartheta_{3/2}(\tau,pz)/\vartheta_{3/2}(\tau,z)$. Therefore, Part (I) contains the existence and explicit identity in Theorem \ref{th:Macdonald-osp}, and gives an extension of the Macdonald-type identity for $\widehat{\mathfrak{osp}}_{1|2}$.  
\end{remark}

\begin{remark}
Theorem \ref{conj:modular invariants-osp} shows that $\varrho_k$ is not irreducible if $2k+3\neq p$, $p^2$ for any odd prime $p$. We note that $\varrho_k$ is irreducible if $2k+3=p$ or $p^2$ for an odd prime $p$.     
\end{remark}

\begin{remark}
The partition functions $Z^{(p)}_{\mathrm{D}}$ and $Z^{(p')}_{\mathrm{D}}$ define modular invariants that are analogous to the $D$-type modular invariants of the affine Lie algebra $\widehat{\mathfrak{sl}}_2$ at levels $4n$. They may come from simple current extensions. The difference between the diagonal modular invariant and \eqref{eq:Macdonald-osp-extension} induces modular invariants that are analogous to the $D$-type modular invariants of $\widehat{\mathfrak{sl}}_2$ at levels $4n-2$. These modular invariants of $\widehat{\mathfrak{sl}}_2$ are due to the fermionization of $\widehat{\mathfrak{sl}}_2$ at levels $4n-2$, which has an $\SL_2(\ZZ)$-invariant sector $\widetilde{R}$. A similar type of modular invariant for $\widehat{\mathfrak{osp}}_{1|2}$ may be related to the $\ZZ_p$-parafermionic structure of $\widehat{\mathfrak{osp}}_{1|2}$ at levels $k$ with $2k+3=p^2p'$. 
\end{remark}

We specialize Theorem \ref{conj:modular invariants-osp} in several cases, beginning with the case $p=3$ in Part (I). 

\begin{corollary}\label{th:9n+3}
Let $n$ be a positive integer and $k=9n+3$. Then $\varrho_k$ has a subrepresentation $\hat{\varrho}_k$ of dimension $n$ that is characterized by the $\SL_2(\ZZ)$-invariant function
\begin{equation}\label{eq:modular level 9n+3}
\sum_{j=0}^{n-1}|\widetilde{\chi}_{3j+1}-\widetilde{\chi}_{6n+1-3j}-\widetilde{\chi}_{6n+4+3j}|^2. 
\end{equation}
At $z=0$, for any $0\leq j\leq n-1$, we have
\begin{equation}\label{eq:minimal model-9n}
\chi_{3j+1}^{\mathfrak{osp}_{1|2,9n+3}} - q^{n-j}\cdot\chi_{6n+1-3j}^{\mathfrak{osp}_{1|2,9n+3}}- q^{n+j+1}\cdot\chi_{6n+4+3j}^{\mathfrak{osp}_{1|2,9n+3}} = \chi_{j}^{\mathfrak{osp}_{1|2,n-1}}.  
\end{equation}
Therefore, $\hat{\varrho}_k$ corresponds to the effective Virasoro minimal model $\mathrm{Vir}_{\mathrm{eff}}(2n+1,2)$. 
\end{corollary} 

We then consider the case $p=3$ and $p'=6n-1$ in Part (II).

\begin{corollary}\label{th:9n-3}
Let $n$ be a positive integer and $k=9n-3$. Then $\varrho_k$ has a  subrepresentation $\varrho^{(3)}_k$ of dimension $3n$ that is characterized by the $\SL_2(\ZZ)$-invariant partition function
\begin{equation}\label{eq:modular level 9n-3}
\begin{split}
Z_{\mathrm{D}}^{(3)}:= &\sum_{\substack{0\leq j< 3n-1\\ j\equiv0\bmod 3}}|\widetilde{\chi}_j+\widetilde{\chi}_{6n-1+j}|^2 + \sum_{\substack{3n-2<j\leq k\\ j\equiv0\bmod 6}}|\widetilde{\chi}_j-\widetilde{\chi}_{12n-3-j}|^2\\
&+\sum_{\substack{0\leq j< 3n-1\\ j\equiv2\bmod 3}}|\widetilde{\chi}_j+\widetilde{\chi}_{6n-2-j}|^2 
+2|\widetilde{\chi}_{3n-1}|^2. 
\end{split}
\end{equation}
The modular invariant $Z_{\mathrm{D}}^{(3)}$ is identical to the $D$-type modular invariant of $\mathrm{Vir}_{\mathrm{eff}}(12n-2,3)$. 
\end{corollary}

\begin{remark}
There are precisely $6n-1$ supercharacters of $\widehat{\mathfrak{osp}}_{1|2}$ appearing in \eqref{eq:modular level 9n-3}. The sum involves $\widetilde{\chi}_0+\widetilde{\chi}_{6n-1}$, and the conformal weight of $\chi_{6n-1}$ is $n$. The theory $\varrho^{(3)}_6$ is also identical to $\mathrm{Vir}_{\mathrm{eff}}(5,2)\otimes \mathrm{Vir}_{\mathrm{eff}}(5,2)$.  
\end{remark}

We also consider the case $p=3$ and $p'=6n+1$ in Part (II).

\begin{corollary}\label{th:9n}
Let $n$ be a positive integer and $k=9n$. Then $\varrho_k$ has a subrepresentation $\varrho^{(3)}_k$ of dimension $3n+1$ that is characterized by the $\SL_2(\ZZ)$-invariant function
\begin{equation}\label{eq:modular level 9n}
\begin{split}
Z_{\mathrm{D}}^{(3)}:=&\sum_{\substack{0\leq j< 3n\\ j\equiv0\bmod 3}}|\widetilde{\chi}_j+\widetilde{\chi}_{6n-j}|^2 + \sum_{\substack{3n<j\leq k\\ j\equiv2\bmod 6}}|\widetilde{\chi}_j-\widetilde{\chi}_{12n+1-j}|^2 \\
&+\sum_{\substack{0\leq j<3n\\ j\equiv2\bmod 3}}|\widetilde{\chi}_j+\widetilde{\chi}_{6n+1+j}|^2 
+2|\widetilde{\chi}_{3n}|^2.
\end{split}
\end{equation}
The modular invariant $Z_{\mathrm{D}}^{(3)}$ is identical to the $D$-type modular invariant of $\mathrm{Vir}_{\mathrm{eff}}(12n+2,3)$. 
\end{corollary}

\begin{remark}\label{rk:osp-9n}
There are exactly $6n+1$ supercharacters of $\widehat{\mathfrak{osp}}_{1|2}$ appearing in \eqref{eq:modular level 9n}. The sum involves $\widetilde{\chi}_0+\widetilde{\chi}_{6n}$, and $\chi_{6n}$ has conformal weight $n$.  Modular invariants related to $\varrho^{(3)}_9$ and $\varrho^{(3)}_{36}$ will be used later to construct exceptional modular invariants related to singular automorphic products. 
\end{remark}

\begin{remark}
The affine VOSA $L_{\widehat{\mathfrak{osp}}_{1|2r}}(\hat{w}_0)$ at level $1$ has central charge $\tilde{c}_r=2r(2r-1)/(2r+3)$, and has exactly $r+1$ irreducible modules that are indexed by dominant integral weights
$\hat{w}_0+\sum_{t=1}^{j}\epsilon_t$ with conformal weight 
$$
\tilde{h}_j=\frac{j(2r+1)-j^2}{2(2r+3)}, \qquad j=0,1,\ldots,r. 
$$
We observe 
$$
\frac{r(r+1)}{12}=\sum_{j=0}^r \Big( \tilde{h}_j -\frac{\tilde{c}_r}{24}\Big).
$$
Therefore, the supercharacters of $L_{\widehat{\mathfrak{osp}}_{1|2r}}(\hat{w}_0)$ satisfy a monic modular linear differential equation of order $r+1$ (see \eqref{eq:monic} and \cite[Section 2.2]{DLS22}). 
\end{remark}

\section{Singular automorphic products and exceptional modular invariants}\label{sec:singular-exceptional} 
Let $\mathfrak{g}$ be $\mathfrak{osp}_{1|2}$ at level $36$, $\mathfrak{osp}_{1|4}$ at level $10$, $\mathfrak{osp}_{1|2}\oplus \mathfrak{sl}_2$ at levels $9$ and $12$, or $\mathfrak{osp}_{1|2}\oplus \mathfrak{sl}_2\oplus \mathfrak{sl}_3$ at levels $3$, $4$, and $6$. Let $\phi_{\mathfrak{g}}$ denote the Jacobi form input of the singular automorphic product constructed in Section \ref{sec:additive-lift}. In this section, we express each $\phi_{\mathfrak{g}}$ as a $\ZZ$-linear combination of the supercharacters of the affine VOSA generated by the affine Lie superalgebra $\hat{\mathfrak{g}}$, and we discuss the relation between $\phi_{\mathfrak{g}}$ and the exceptional modular invariants of $\hat{\mathfrak{g}}$. This completes the proof of the last part of Theorem \ref{MTH:classification}. It is our hope that these exceptional modular invariants will inspire a natural construction of the associated BKM superalgebra as the BRST cohomology of an appropriate vertex algebra (or a suitable analog thereof). 

\subsection{Case of \texorpdfstring{$\mathfrak{osp}_{1|2}$}{} at level 36}
The particular case $n=2$ of the Macdonald-type identity \eqref{eq:Macdonald-osp} in Theorem \ref{th:Macdonald-osp} leads to the following result.

\begin{proposition}
The Jacobi form input of the singular automorphic product $\Borch(\phi_{A_{1,36}^*})$ can be expressed in the supercharacters of $L_{\widehat{\mathfrak{osp}}_{1|2}}(36\hat{w}_0)$ as follows:
\begin{equation}\label{eq:expression-A1*}
\phi_{A_{1,36}^*} = \widetilde{\chi}^{\mathfrak{osp}_{1|2,36}}_2 - \widetilde{\chi}^{\mathfrak{osp}_{1|2,36}}_{12} - \widetilde{\chi}^{\mathfrak{osp}_{1|2,36}}_{17} + \widetilde{\chi}^{\mathfrak{osp}_{1|2,36}}_{27} + \widetilde{\chi}^{\mathfrak{osp}_{1|2,36}}_{32}.    
\end{equation}
\end{proposition}

The particular case $n=4$ of Corollary \ref{th:9n} leads to the following modular invariant: 
\begin{equation}
Z^{\mathfrak{osp}_{1|2,36}}_{\rm D}=|\phi_0|^2+|\phi_1|^2+|\phi_2|^2+2|\phi_3|^2+|\phi_4|^2+\dots+|\phi_{12}|^2,
\end{equation}
where the $13$ components are defined as
\begin{align*}
&\phi_0=\widetilde{\chi}^{\mathfrak{osp}_{1|2,36}}_{0} + \widetilde{\chi}^{\mathfrak{osp}_{1|2,36}}_{24},&
&\phi_{1} =\widetilde{\chi}^{\mathfrak{osp}_{1|2,36}}_{2} + \widetilde{\chi}^{\mathfrak{osp}_{1|2,36}}_{27},& \\
&\phi_{2} =\widetilde{\chi}^{\mathfrak{osp}_{1|2,36}}_{17} - \widetilde{\chi}^{\mathfrak{osp}_{1|2,36}}_{32},& 
&\phi_{3}  =\widetilde{\chi}^{\mathfrak{osp}_{1|2,36}}_{12},& \\
&\phi_{4} =\widetilde{\chi}^{\mathfrak{osp}_{1|2,36}}_{3} + \widetilde{\chi}^{\mathfrak{osp}_{1|2,36}}_{21},& 
&\phi_{5} =\widetilde{\chi}^{\mathfrak{osp}_{1|2,36}}_{5} + \widetilde{\chi}^{\mathfrak{osp}_{1|2,36}}_{30},& \\
&\phi_{6} =\widetilde{\chi}^{\mathfrak{osp}_{1|2,36}}_{6} + \widetilde{\chi}^{\mathfrak{osp}_{1|2,36}}_{18},&
&\phi_{7} = \widetilde{\chi}^{\mathfrak{osp}_{1|2,36}}_{8} + \widetilde{\chi}^{\mathfrak{osp}_{1|2,36}}_{33},& \\
&\phi_{8} = \widetilde{\chi}^{\mathfrak{osp}_{1|2,36}}_{9} + \widetilde{\chi}^{\mathfrak{osp}_{1|2,36}}_{15},&
&\phi_{9} = \widetilde{\chi}^{\mathfrak{osp}_{1|2,36}}_{11} + \widetilde{\chi}^{\mathfrak{osp}_{1|2,36}}_{36},& \\ 
&\phi_{10} =\widetilde{\chi}^{\mathfrak{osp}_{1|2,36}}_{14} -\widetilde{\chi}^{\mathfrak{osp}_{1|2,36}}_{35},& 
&\phi_{11} =\widetilde{\chi}^{\mathfrak{osp}_{1|2,36}}_{20} - \widetilde{\chi}^{\mathfrak{osp}_{1|2,36}}_{29},& \\
&\phi_{12} = \widetilde{\chi}^{\mathfrak{osp}_{1|2,36}}_{23} - \widetilde{\chi}^{\mathfrak{osp}_{1|2,36}}_{26}.&
\end{align*}
These $\phi_j$ form a vector-valued weakly holomorphic modular form of weight $0$ for a representation $\rho$ of $\SL_2(\ZZ)$, and satisfy a monic MLDE of order $13$. The transformation matrix $\rho(T)$ is diagonal and unitary, with entries determined by the central charge $c=\frac{24}{25}$ and the conformal weights
$$
0,\frac{1}{25},\frac{51}{25},\frac{26}{25},\frac{2}{25},\frac{1}{5},\frac{7}{25},\frac{12}{25},\frac{3}{5},\frac{22}{25},\frac{7}{5},\frac{14}{5},\frac{92}{25}.
$$
The transformation matrix $\rho(S)$ is non-symmetric and real-valued. Equivalently, we can take $\phi_{13}$ as a copy of $\phi_3$ so that
\begin{equation}
Z^{\mathfrak{osp}_{1|2,36}}_{\rm D}=\sum_{i=0}^{13}|\phi_i|^2
\end{equation}
induces a representation $\rho'$ of dimension $14$ for which $\rho'(S)$ is symmetric and $\rho'(S)^2=\mathrm{Id}$. 

Recall that $\rho$ is not irreducible and contains the trivial subrepresentation characterized by 
\begin{equation}\label{eq:osp-Z_3}
\phi_{A_{1,36}^*} = \phi_1 - \phi_2 - \phi_3.     
\end{equation}
Therefore,
\begin{equation}
\begin{split}
&Z^{\mathfrak{osp}_{1|2,36}}_{\rm exc}:=  Z^{\mathfrak{osp}_{1|2,36}}_{\rm D} - |\phi_{A_{1,36}^*}|^2\\
=&\,|\phi_0|^2+ (\phi_1\bar{\phi}_2+\phi_2\bar{\phi}_1)+ (\phi_1\bar{\phi}_3+\phi_3\bar{\phi}_1)- (\phi_2\bar{\phi}_3+\phi_3\bar{\phi}_2)+ |\phi_4|^2+\dots+|\phi_{13}|^2
\end{split}    
\end{equation}
is $\SL_2(\ZZ)$-invariant. It appears that this exceptional modular invariant is not associated with a nontrivial automorphism of $S$-matrix.

\subsection{Case of \texorpdfstring{$\mathfrak{osp}_{1|4}$}{} at level 10} 
We first give some explicit formulas for the supercharacters of $\widehat{\mathfrak{osp}}_{1|4}$. The affine vertex operator superalgebra $L_{\widehat{\mathfrak{osp}}_{1|4}}(k\hat{w}_0)$ has central charge $12k/(2k+5)$, and its irreducible modules are indexed by dominant integral weights
$$
k\hat{w}_0+k_1w_1+2k_2w_2=k\hat{w}_0+(k_1+k_2)\epsilon_1+k_2\epsilon_2,
$$
where $k_1$ and $k_2$ are non-negative integers satisfying $k_1+k_2\leq k$. We denote the supercharacter of $L_{\widehat{\mathfrak{osp}}_{1|4}}(k,k_1w_1+2k_2w_2)$ by $\chi^{\mathfrak{osp}_{1|4,k}}_{k_1,k_2}$ and its modification $q^{h_{k_1,k_2}-c_k/24}\cdot \chi^{\mathfrak{osp}_{1|4,k}}_{k_1,k_2}$ by $\widetilde{\chi}^{\mathfrak{osp}_{1|4,k}}_{k_1,k_2}$, where 
$$
h_{k_1,k_2}=\frac{(k_1+k_2)(k_1+k_2+3)+k_2(k_2+1)}{2(2k+5)}
$$
is the conformal weight. 
By direct calculation, we obtain
\begin{equation}
\chi^{\mathfrak{osp}_{1|4,k}}_{k_1,k_2}(\tau,z_1,z_2)=\frac{\Theta^{\mathfrak{osp}_{1|4,k}}_{k_1,k_2}(\tau,z_1,z_2)}{D^{\mathfrak{osp}_{1|4}}(\tau,z_1,z_2)},
\end{equation}
where the denominator function is
\begin{equation}
\begin{split}
D^{\mathfrak{osp}_{1|4}}(\tau,z_1,z_2)=&\ (1+\zeta_1^{-1})(1+\zeta_2^{-1})(1-\zeta_1^{-1}\zeta_2)(1-\zeta_1^{-1}\zeta_2^{-1})\\
&\cdot \frac{\prod_{m=1}^\infty (1-q^m)^2(1-q^m\zeta_1^{\pm2})(1-q^m\zeta_2^{\pm2})(1-q^m\zeta_1^{\pm1}\zeta_2^{\pm1})}{\prod_{m=1}^\infty(1-q^m\zeta_1^{\pm1})(1-q^m\zeta_2^{\pm1})}    
\end{split}
\end{equation}
and the numerator function is
\begin{equation}
\begin{split}
&\Theta^{\mathfrak{osp}_{1|4,k}}_{k_1,k_2}(\tau,z_1,z_2)\\
=& \sum_{n,m\in \mathbb{Z}}(-1)^{n+m}q^{(k+\frac{5}{2})(n^2+m^2)+(k_1+k_2+\frac{3}{2}) n+(k_2+\frac{1}{2}) m}\Big( \zeta_1^{(2k+5)n+k_1+k_2} \zeta_2^{(2k+5)m+k_2} \\
&- \zeta_1^{(2k+5)m+k_2-1}\zeta_2^{(2k+5)n+k_1+k_2+1}+\zeta_1^{(2k+5)n+k_1+k_2}\zeta_2^{-((2k+5)m+k_2+1)}\\
&-\zeta_1^{-((2k+5)m+k_2+2)}\zeta_2^{(2k+5)n+k_1+k_2+1}-\zeta_1^{(2k+5)m+k_2-1}\zeta_2^{-((2k+5)n+k_1+k_2+2)}\\
&+\zeta_1^{-((2k+5)n+k_1+k_2+3)}\zeta_2^{(2k+5)m+k_2}
-\zeta_1^{-((2k+5)m+k_2+2)}\zeta_2^{-((2k+5)n+k_1+k_2+2)} \\
&+\zeta_1^{-((2k+5)n+k_1+k_2+3)}\zeta_2^{-((2k+5)m+k_2+1)} \Big),
\end{split}
\end{equation}
where $\zeta_j=e^{2\pi iz_j}$ for $j=1,2$. The degenerate case $k=0$ leads to identity
$$
D^{\mathfrak{osp}_{1|4}}(\tau,z_1,z_2)=\Theta^{\mathfrak{osp}_{1|4,0}}_{0,0}(\tau,z_1,z_2).
$$

The affine VOSA $L_{\widehat{\mathfrak{osp}}_{1|4}}(k\hat{w}_0)$ gives rise to a unitary representation $\varrho_{\mathfrak{osp}_{1|4},k}$ such that all modified supercharacters of level $k$ form a vector-valued Jacobi form of weight $0$ and index $2A_1(k)$ for $\varrho_{\mathfrak{osp}_{1|4},k}$. Note that  $\dim\varrho_{\mathfrak{osp}_{1|4},k}=(k^2+3k+2)/2$. The transformation matrix $\varrho_{\mathfrak{osp}_{1|4},k}(S)$ is given by
\begin{equation}\label{eq:Smatrixosp14k}
\begin{split}
&\varrho_{\mathfrak{osp}_{1|4},k}(S)_{(k_1,k_2),\, (k_1',k'_2)}
=\frac{4}{2k+5} \times \\
&\quad\Big[ \cos\Big( \frac{2\pi}{2k+5}\big(k_1+k_2+3/2\big)\big(k_2'+1/2\big)\Big)\cdot \cos\Big( \frac{2\pi}{2k+5}\big(k_2+1/2\big)\big(k_1'+k_2'+3/2\big)\Big)\\
&\quad -\cos\Big( \frac{2\pi}{2k+5}\big(k_1+k_2+3/2\big)\big(k_1'+k_2'+3/2\big)\Big)\cdot \cos\Big( \frac{2\pi}{2k+5}\big(k_2+1/2\big)\big(k'_2+1/2\big)\Big) \Big].      
\end{split} 
\end{equation}

Using the above formulas, we can easily obtain the following expression.

\begin{proposition}
The Jacobi form input of the singular automorphic product $\Borch(\phi_{C_{2,10}^*})$ can be expressed in the supercharacters of $L_{\widehat{\mathfrak{osp}}_{1|4}}(10\hat{w}_0)$ as follows:
\begin{equation}\label{eq:expression-C2*}
\phi_{C_{2,10}^*} = \widetilde{\chi}^{\mathfrak{osp}_{1|4,10}}_{2,0} - \widetilde{\chi}^{\mathfrak{osp}_{1|4,10}}_{4,2} - \widetilde{\chi}^{\mathfrak{osp}_{1|4,10}}_{1,4}+\widetilde{\chi}^{\mathfrak{osp}_{1|4,10}}_{8,1}+\widetilde{\chi}^{\mathfrak{osp}_{1|4,10}}_{0,8}+\widetilde{\chi}^{\mathfrak{osp}_{1|4,10}}_{5,5}.
\end{equation}
\end{proposition}

We now construct a non-diagonal modular invariant for $\widehat{\mathfrak{osp}}_{1|4}$ at level $10$. Recall that the affine VOSA $L_{\widehat{\mathfrak{osp}}_{1|4}}(10\hat{w}_0)$ has central charge $c=\frac{24}{5}$. Let us define
\begin{align*}
\phi_0= &\,\widetilde{\chi}^{\mathfrak{osp}_{1|4,10}}_{0,0}-\widetilde{\chi}^{\mathfrak{osp}_{1|4,10}}_{2,5}+\widetilde{\chi}^{\mathfrak{osp}_{1|4,10}}_{3,6}+\widetilde{\chi}^{\mathfrak{osp}_{1|4,10}}_{5,3}- \widetilde{\chi}^{\mathfrak{osp}_{1|4,10}}_{6,4} ,\\
\phi_{\frac15}= &\,\widetilde{\chi}^{\mathfrak{osp}_{1|4,10}}_{2,0}-\widetilde{\chi}^{\mathfrak{osp}_{1|4,10}}_{1,4}+\widetilde{\chi}^{\mathfrak{osp}_{1|4,10}}_{8,1}+\widetilde{\chi}^{\mathfrak{osp}_{1|4,10}}_{0,8}+\widetilde{\chi}^{\mathfrak{osp}_{1|4,10}}_{5,5},\\
\phi_{\frac65}  = & \, \widetilde{\chi}^{\mathfrak{osp}_{1|4,10}}_{4,2} ,\\
\phi_{\frac25}=&\, \widetilde{\chi}^{\mathfrak{osp}_{1|4,10}}_{0,5}+\widetilde{\chi}^{\mathfrak{osp}_{1|4,10}}_{1,9}+\widetilde{\chi}^{\mathfrak{osp}_{1|4,10}}_{2,1}+\widetilde{\chi}^{\mathfrak{osp}_{1|4,10}}_{6,3}+\widetilde{\chi}^{\mathfrak{osp}_{1|4,10}}_{7,0},\\
\phi_{\frac35}=&\, \widetilde{\chi}^{\mathfrak{osp}_{1|4,10}}_{3,1}-\widetilde{\chi}^{\mathfrak{osp}_{1|4,10}}_{0,3}-\widetilde{\chi}^{\mathfrak{osp}_{1|4,10}}_{1,8}+\widetilde{\chi}^{\mathfrak{osp}_{1|4,10}}_{2,6}+\widetilde{\chi}^{\mathfrak{osp}_{1|4,10}}_{10,0},\\
\phi_{\frac45}=&\, \widetilde{\chi}^{\mathfrak{osp}_{1|4,10}}_{1,3}+\widetilde{\chi}^{\mathfrak{osp}_{1|4,10}}_{3,4}+\widetilde{\chi}^{\mathfrak{osp}_{1|4,10}}_{5,0}+\widetilde{\chi}^{\mathfrak{osp}_{1|4,10}}_{7,1}-\widetilde{\chi}^{\mathfrak{osp}_{1|4,10}}_{0,10}.   \end{align*}
Here, the subscript in $\phi_{*}$ indicates its conformal weight. It is easy to verify from \eqref{eq:Smatrixosp14k} that
\begin{equation}
Z^{\mathfrak{osp}_{1|4,10}}_{\mathrm{D}}=|\phi_0|^2+|\phi_{\frac15}|^2+5|\phi_{\frac65}|^2+|\phi_{\frac25}|^2+|\phi_{\frac35}|^2+|\phi_{\frac45}|^2
\end{equation}
is $\SL_2(\ZZ)$-invariant and thus a modular invariant. The components of $Z^{\mathfrak{osp}_{1|4,10}}_{\mathrm{D}}$ form a vector-valued modular form of weight $0$ for a representation $\rho$ of $\SL_2(\ZZ)$, and satisfy a MLDE of order $6$ and index $3$. The transformation matrix $\rho(S)$ is given by
\begin{equation}
\left(
\begin{array}{cccccc}
 \frac{1}{10} \left(3-\sqrt{5}\right) & \frac{1}{5} & 1 & \frac{1}{10} \left(\sqrt{5}+3\right) & \frac{1}{5} \left(\sqrt{5}-1\right) & \frac{1}{5} \left(\sqrt{5}+1\right) \\[+1mm]
 \frac{1}{5} & \frac{4}{5} & -1 & \frac{1}{5} & \frac{1}{5} & -\frac{1}{5} \\[+1mm]
 \frac{1}{5} & -\frac{1}{5} & 0 & \frac{1}{5} & \frac{1}{5} & -\frac{1}{5} \\[+1mm]
 \frac{1}{10} \left(\sqrt{5}+3\right) & \frac{1}{5} & 1 & \frac{1}{10} \left(3-\sqrt{5}\right) & \frac{1}{5} \left(-\sqrt{5}-1\right) & \frac{1}{5} \left(1-\sqrt{5}\right) \\[+1mm]
 \frac{1}{5} \left(\sqrt{5}-1\right) & \frac{1}{5} & 1 & \frac{1}{5} \left(-\sqrt{5}-1\right) & \frac{1}{10} \left(\sqrt{5}+3\right) & \frac{1}{10} \left(\sqrt{5}-3\right) \\[+1mm]
 \frac{1}{5} \left(\sqrt{5}+1\right) & -\frac{1}{5} & -1 & \frac{1}{5} \left(1-\sqrt{5}\right) & \frac{1}{10} \left(\sqrt{5}-3\right) & \frac{1}{10} \left(\sqrt{5}+3\right) \\
\end{array}
\right).
\end{equation}
Obviously, $\rho(S)$ is not symmetric. If we take $5$ copies of $\phi_{\frac{6}{5}}$, then $\rho$ can be regarded as a representation $\rho'$ of dimension $10$ so that $\rho'(S)$ is symmetric and $\rho'(S)^2=\mathrm{Id}$. In this way, we find that there exists a nontrivial automorphism $\phi_\frac{1}{5}\leftrightarrow \phi_\frac{6}{5}$ for the extended $S$-matrix $\rho'(S)$, leading to the exceptional modular invariant
\begin{equation}
Z^{\mathfrak{osp}_{1|4,10}}_{\mathrm{exc}}=|\phi_0|^2+(\phi_{\frac15}\bar{\phi}_{\frac65}+\phi_{\frac65}\bar{\phi}_{\frac15})+4|\phi_{\frac65}|^2+|\phi_{\frac25}|^2+|\phi_{\frac35}|^2+|\phi_{\frac45}|^2.
\end{equation} 
Observe that $\rho$ has a trivial subrepresentation characterized by
\begin{equation}
\phi_{C_{2,10}^*} = \phi_{\frac{1}{5}} - \phi_{\frac{6}{5}}   
\end{equation}
and we have the following equality:
\begin{equation}
 Z^{\mathfrak{osp}_{1|4,10}}_{\mathrm{D}}-Z^{\mathfrak{osp}_{1|4,10}}_{\mathrm{exc}}=|\phi_{C_{2,10}^*}|^2. 
\end{equation}
This resembles the four exotic cases, such as $A_{1,16}$ and $A_{2,9}$ in our earlier work \cite[Section 8]{SWW23}.

\subsection{Case of \texorpdfstring{$\mathfrak{osp}_{1|2}\oplus\mathfrak{sl}_2$}{} at levels 9 and 12} 
Let $k$ be a positive integer. The affine VOA $L_{\widehat{\mathfrak{sl}}_2}(k\hat{w}_0)$ generated by the affine Lie algebra $\widehat{\mathfrak{sl}}_2$ at level $k$ has central charge $c_k=3k/(k+2)$, and has $k+1$ irreducible modules with conformal weights $h_j$ and character $\chi^{\mathfrak{sl}_{2,k}}_j$:
\begin{equation}\label{eq:sl_2 characters}
\chi^{\mathfrak{sl}_{2,k}}_j(\tau,z)=\frac{\Theta^{\mathfrak{sl}_{2,k}}_j(\tau,z)}{\prod_{n=1}^\infty(1-q^n)(1-q^n\zeta)(1-q^{n-1}\zeta^{-1})}, \quad h_j=\frac{j(j+2)}{4(k+2)}, \quad 0\leq j\leq k,
\end{equation}
where the numerator function is
$$
\Theta^{\mathfrak{sl}_{2,k}}_j(\tau,z) = \sum_{n\in\ZZ} q^{(k+2)n^2+(j+1)n}\Big(\zeta^{(k+2)n+j/2} - \zeta^{-(k+2)n-j/2-1} \Big). 
$$
The associated transformation matrix $\rho(S)$ for modified characters  is given by
$$
\rho(S)_{jj'}= \sqrt{\frac{2}{k+2}} \sin\left( \frac{(j+1)(j'+1)}{k+2}\pi \right), \quad \text{for any} \; 0\leq j,j'\leq k. 
$$

Using the above formulas and the formulas in Section \ref{subsec:osp(1|2)}, we find the following expression. 

\begin{proposition}
The Jacobi form input of the singular automorphic product $\Borch(\phi_{A_{1,9}^*A_{1,12}})$ can be expressed in the supercharacters of $L_{\widehat{\mathfrak{osp}}_{1|2}}(9\hat{w}_0)\otimes L_{\widehat{\mathfrak{sl}}_{2}}(12\hat{w}_0)$ as follows:
\begin{equation}\label{eq:expression-A1*A1}
\begin{split}
\phi_{A_{1,9}^*A_{1,12}} =&\, \big( \widetilde{\chi}^{\mathfrak{osp}_{1|2,9}}_{0}+\widetilde{\chi}^{\mathfrak{osp}_{1|2,9}}_{6}\big) \big(\widetilde{\chi}^{\mathfrak{sl}_{2,12}}_{2}+\widetilde{\chi}^{\mathfrak{sl}_{2,12}}_{10}\big) + \big(\widetilde{\chi}^{\mathfrak{osp}_{1|2,9}}_{2} +\widetilde{\chi}^{\mathfrak{osp}_{1|2,9}}_{9}\big)\big(\widetilde{\chi}^{\mathfrak{sl}_{2,12}}_{0}+\widetilde{\chi}^{\mathfrak{sl}_{2,12}}_{12}\big)\\
&\quad +\big(\widetilde{\chi}^{\mathfrak{osp}_{1|2,9}}_{8} -\widetilde{\chi}^{\mathfrak{osp}_{1|2,9}}_{5}\big) \big(\widetilde{\chi}^{\mathfrak{sl}_{2,12}}_{4}+\widetilde{\chi}^{\mathfrak{sl}_{2,12}}_{8}\big)- 2\widetilde{\chi}^{\mathfrak{osp}_{1|2,9}}_{3} \widetilde{\chi}^{\mathfrak{sl}_{2,12}}_{6}.
\end{split}
\end{equation}
\end{proposition}

The particular case $n=1$ of Corollary\ref{th:9n} leads to the following modular invariant:
\begin{equation}
\begin{split}
Z^{\mathfrak{osp}_{1|2,9}}_{\mathrm{D}}=&\, |\psi_0|^2+|\psi_{\frac17}|^2+2|\psi_{\frac27}|^2+|\psi_{\frac57}|^2\\
=&\,|\psi_0|^2+|\psi_{\frac17}|^2+|\psi_{\frac27}|^2+|\psi'_{\frac27}|^2+|\psi_{\frac57}|^2,
\end{split}
\end{equation}
where $\psi_{\frac{2}{7}}'$ is a copy of $\psi_{\frac{2}{7}}$, and the components are defined as
\begin{align*}
\psi_0 &=\widetilde{\chi}^{\mathfrak{osp}_{1|2,9}}_0+\widetilde{\chi}^{\mathfrak{osp}_{1|2,9}}_6,\\
\psi_{\frac17} &=\widetilde{\chi}^{\mathfrak{osp}_{1|2,9}}_{2}+\widetilde{\chi}^{\mathfrak{osp}_{1|2,9}}_{9},\\
\psi_{\frac57} &=\widetilde{\chi}^{\mathfrak{osp}_{1|2,9}}_{5}-\widetilde{\chi}^{\mathfrak{osp}_{1|2,9}}_{8},\\
\psi_{\frac27} &=\widetilde{\chi}^{\mathfrak{osp}_{1|2,9}}_{3}.
\end{align*}
This non-diagonal modular invariant defines a theory of central charge $\frac{6}{7}$, and the conformal weights can be found in the subscripts of these $\psi_{-}$.

The affine Lie algebra $\widehat{\mathfrak{sl}}_2$ at level $12$ has the following non-diagonal modular invariant
\begin{equation}
\begin{split}
Z^{\mathfrak{sl}_{2,12}}_{\mathrm{D}}&= |\phi_0|^2+|\phi_{\frac17}|^2+|\phi_{\frac37}|^2+2|\phi_{\frac67}|^2\\
&=|\phi_0|^2+|\phi_{\frac17}|^2+|\phi_{\frac37}|^2+|\phi_{\frac67}|^2+|\phi'_{\frac67}|^2,
\end{split}
\end{equation}
where $\phi_{\frac{6}{7}}'$ is a copy of $\phi_{\frac{6}{7}}$, and the components are defined as
\begin{align*}
\phi_0 &=\widetilde{\chi}^{\mathfrak{sl}_{2,12}}_0+\widetilde{\chi}^{\mathfrak{sl}_{2,12}}_{12},\\
\phi_{\frac17} &=\widetilde{\chi}^{\mathfrak{sl}_{2,12}}_{2}+\widetilde{\chi}^{\mathfrak{sl}_{2,12}}_{10},\\
\phi_{\frac37} &=\widetilde{\chi}^{\mathfrak{sl}_{2,12}}_{4}+\widetilde{\chi}^{\mathfrak{sl}_{2,12}}_{8},\\
\phi_{\frac67} &=\widetilde{\chi}^{\mathfrak{sl}_{2,12}}_{6}.
\end{align*}
This non-diagonal modular invariant defines a theory of central charge $\frac{18}{7}$, whose conformal weights appear in the subscripts of these $\phi_{-}$.

The tensor product gives a non-diagonal modular invariant for $\widehat{\mathfrak{osp}}_{1|2}\oplus \widehat{\mathfrak{sl}}_2$ at levels $9$ and $12$ as 
\begin{equation}
Z_{\mathrm{D}}:=Z^{\mathfrak{osp}_{1|2,9}}_{\mathrm{D}} \otimes Z^{\mathfrak{sl}_{2,12}}_{\mathrm{D}}. 
\end{equation}
This modular invariant gives a theory of central charge $\frac{24}{7}$ and has a further non-diagonal modular invariant defined as
\begin{equation}
\begin{split}
Z_{\mathrm{D}}':=& \,|\varphi_0|^2+|\varphi_{\frac17}|^2+7|\varphi_{\frac87}|^2+|\varphi_{\frac27}|^2+|\varphi_{\frac37}|^2+|\varphi_{\frac47}|^2+|\varphi_{\frac57}|^2+|\varphi_{\frac67}|^2\\
=& \,|\varphi_0|^2+|\varphi_{\frac17}|^2+\sum_{i=1}^7|\varphi_{\frac87,i}|^2+|\varphi_{\frac27}|^2+|\varphi_{\frac37}|^2+|\varphi_{\frac47}|^2+|\varphi_{\frac57}|^2+|\varphi_{\frac67}|^2,  
\end{split}    
\end{equation}
where these components are expressed as
\begin{align*}
\varphi_0&=\psi_0\phi_0+\psi_{\frac{1}{7}}\phi_{\frac{6}{7}},\\
\varphi_{\frac17}&=\psi_0\phi_{\frac{1}{7}}+\psi_{\frac{1}{7}}\phi_0-\psi_{\frac{5}{7}}\phi_{\frac{3}{7}}- \psi_{\frac{2}{7}}\phi_{\frac{6}{7}},\\
\varphi_{\frac87}&= \psi_{\frac{2}{7}}\phi_{\frac{6}{7}}  ,\\
\varphi_{\frac27}&=\psi_{\frac{1}{7}}\phi_{\frac{1}{7}} -\psi_{\frac{2}{7}}\phi_{0} ,\\
\varphi_{\frac37}&=\psi_{0}\phi_{\frac{3}{7}} +\psi_{\frac{2}{7}}\phi_{\frac{1}{7}} ,\\
\varphi_{\frac47}&=\psi_{\frac{1}{7}}\phi_{\frac{3}{7}} +\psi_{\frac{5}{7}}\phi_{\frac{6}{7}} ,\\
\varphi_{\frac57}&=\psi_{\frac{2}{7}}\phi_{\frac{3}{7}} -\psi_{\frac{5}{7}}\phi_{0} ,\\
\varphi_{\frac67}&=\psi_{0}\phi_{\frac{6}{7}} +\psi_{\frac{5}{7}}\phi_{\frac{1}{7}},
\end{align*}
and $\varphi_{\frac{8}{7},i}$ are copies of $\varphi_{\frac{8}{7}}$ for $1\leq i\leq 7$. The components $\varphi_j$ satisfy an MLDE of order $8$ and index $10$. When we take $7$ copies of $\varphi_{\frac{8}{7}}$, the $14$ components lead to a representation $\rho$ of dimension $14$ for which $\rho(S)$ is symmetric and $\rho(S)^2=\mathrm{Id}$. Observe that there is a nontrivial automorphism $\varphi_{\frac17}\leftrightarrow \varphi_{\frac87,i}$ for the extended $S$-matrix $\rho(S)$. This yields an exceptional modular invariant
\begin{equation}
Z_{\mathrm{exc}} = |\varphi_0|^2+(\varphi_{\frac17}\bar{\varphi}_{\frac87}+\varphi_{\frac87}\bar{\varphi}_{\frac17})+6|\varphi_{\frac87}|^2+|\varphi_{\frac27}|^2+|\varphi_{\frac37}|^2+|\varphi_{\frac47}|^2+|\varphi_{\frac57}|^2+|\varphi_{\frac67}|^2. 
\end{equation}
We find that $\rho$ has a trivial subrepresentation characterized by
\begin{equation}
\phi_{A_{1,9}^*A_{1,12}}=\varphi_{\frac17} - \varphi_{\frac87}, 
\end{equation}
and we have the following equality:
\begin{equation}
Z_{\mathrm{D}}' - Z_{\mathrm{exc}} = |\phi_{A_{1,9}^*A_{1,12}}|^2. 
\end{equation}

\subsection{Case of \texorpdfstring{$\mathfrak{osp}_{1|2}\oplus \mathfrak{sl}_2\oplus \mathfrak{sl}_3$}{} at levels 3, 4, and 6}
The affine VOSA $L_{\widehat{\mathfrak{osp}}_{1|2}}(3\hat{w}_0)$ has central charge $\frac{2}{3}$, and its supercharacters $\widetilde{\chi}^{\mathfrak{osp}_{1|2,3}}_{j}$ ($0\leq j\leq 3$) have conformal weights $0$, $\frac{1}{9}$, $\frac{1}{3}$, $\frac{2}{3}$, respectively. For convenience, we label them as $\chi_h$ with subscript indicating the conformal weight. The affine VOA $L_{\widehat{\mathfrak{sl}}_2}(4\hat{w}_0)$ has central charge $2$, and its characters $\widetilde{\chi}^{\mathfrak{sl}_{2,4}}_{j}$ ($0\leq j\leq 4$) have conformal weights $0$, $\frac{1}{8}$, $\frac{1}{3}$, $\frac{5}{8}$, $1$, respectively. Moreover, it has a $D$-type modular invariant (cf. \cite{CFT}) that can also be constructed from the conformal embedding $(A_1)_4\subset (A_2)_1$:
\begin{equation}
Z^{\mathfrak{sl}_{2,4}}_\mathrm{D}=|\phi_0|^2+2|\phi_{\frac13}|^2, \quad \phi_0:=\widetilde{\chi}^{\mathfrak{sl}_{2,4}}_{0}+\widetilde{\chi}^{\mathfrak{sl}_{2,4}}_{4}, \quad \phi_{\frac{1}{3}}:=\widetilde{\chi}^{\mathfrak{sl}_{2,4}}_{2}. 
\end{equation}
The affine VOA $L_{\widehat{\mathfrak{sl}}_3}(6\hat{w}_0)$ has central charge $\frac{16}{3}$, and has a $D$-type modular invariant (cf. \cite{CFT}):
\begin{equation}
Z^{\mathfrak{sl}_{3,6}}_{\mathrm{D}}=|\varphi_0|^2+|\varphi_{\frac13}|^2+|\varphi_{\frac23}|^2+3|\varphi_{\frac89}|^2,  
\end{equation}
where these components are expressed as 
\begin{align*}
\varphi_0 &= \widetilde{\chi}^{\mathfrak{sl}_{3,6}}_{00,0} +  \widetilde{\chi}^{\mathfrak{sl}_{3,6}}_{60,2} + \widetilde{\chi}^{\mathfrak{sl}_{3,6}}_{06,2},\\
\varphi_{\frac{1}{3}} &= \widetilde{\chi}^{\mathfrak{sl}_{3,6}}_{11,\frac13} +  \widetilde{\chi}^{\mathfrak{sl}_{3,6}}_{14,\frac43} + \widetilde{\chi}^{\mathfrak{sl}_{3,6}}_{41,\frac43},\\
\varphi_{\frac{2}{3}} &= \widetilde{\chi}^{\mathfrak{sl}_{3,6}}_{03,\frac23} +  \widetilde{\chi}^{\mathfrak{sl}_{3,6}}_{30,\frac23} + \widetilde{\chi}^{\mathfrak{sl}_{3,6}}_{33,\frac53},\\
\varphi_{\frac{8}{9}} &= \widetilde{\chi}^{\mathfrak{sl}_{3,6}}_{22,\frac89}.
\end{align*} 
Here, the subscripts of $\widetilde{\chi}^{\mathfrak{sl}_{3,6}}_{nm,h}$ indicate that the corresponding highest weight is $(6-n-m)\hat{w}_0+n\hat{w}_1+m\hat{w}_2$, with conformal weight $h$. 

From these data, we derive the following expression.

\begin{proposition}
The Jacobi form input of the singular automorphic product $\Borch(\phi_{A_{1,3}^*A_{1,4}A_{2,6}})$ can be expressed in the supercharacters of $L_{\widehat{\mathfrak{osp}}_{1|2}}(3\hat{w}_0)\otimes L_{\widehat{\mathfrak{sl}}_{2}}(4\hat{w}_0)\otimes L_{\widehat{\mathfrak{sl}}_{3}}(6\hat{w}_0)$ as follows:
\begin{equation}\label{eq:expression-A1*A1*A2}
\begin{split}
\phi_{A_{1,3}^*A_{1,4}A_{2,6}} =&\, \chi_{\frac13} \phi_0\varphi_0+\chi_0 \phi_{\frac13}\varphi_0+\chi_0 \phi_0\varphi_{\frac13}+\chi_{\frac23} \phi_0\varphi_{\frac23} \\
& +\chi_{\frac23} \phi_{\frac13}\varphi_{\frac13}-\chi_{\frac13} \phi_{\frac13}\varphi_{\frac23}-3\chi_{\frac19} \phi_{\frac13}\varphi_{\frac89}.
\end{split}
\end{equation}
\end{proposition}

We find that the above expression actually splits into two parts:
\begin{align*}
&\phi_0(\chi_{\frac13} \varphi_0+\chi_0 \varphi_{\frac13}+\chi_{\frac23} \varphi_{\frac23})-3\chi_{\frac19} \phi_{\frac13}\varphi_{\frac89},\\
&\phi_{\frac13} ( \chi_0 \varphi_0-\chi_{\frac13} \varphi_{\frac23} +\chi_{\frac23} \varphi_{\frac13}),
\end{align*}
for which both are $\SL_2(\ZZ)$-invariant at $z=0$ (but not for general $z$). In fact, taking $\mathfrak{z}=0$, the two parts degenerate into the constants $9$ and $3$, respectively. This motivates us to construct the following non-diagonal modular invariant for $L_{\widehat{\mathfrak{osp}}_{1|2}}(3\hat{w}_0)\otimes L_{\widehat{\mathfrak{sl}}_{3}}(6\hat{w}_0)$ with central charge $6$:
\begin{align}
Z^{\mathfrak{osp}_{1|2,3}\oplus \mathfrak{sl}_{3,6}}:=|\psi_0|^2+|\psi_{\frac13}|^2+|\psi_{\frac23}|^2+|\psi_{1}|^2,
\end{align}
where these components are defined as
\begin{align*}
\psi_0 & =  \chi_0 \varphi_0-\chi_{\frac13} \varphi_{\frac23} +\chi_{\frac23} \varphi_{\frac13}  ,\\
\psi_{\frac13} & = \chi_{0} \varphi_{\frac13} + \chi_{\frac13} \varphi_0+\chi_{\frac23} \varphi_{\frac23} ,\\
\psi_{\frac23} & =\chi_{0} \varphi_{\frac23} +\chi_{\frac13} \varphi_{\frac13}  -\chi_{\frac23} \varphi_0 ,\\
\psi_{1} & =  3\chi_{\frac19} \varphi_{\frac89} .
\end{align*}
We find that the above components satisfy a monic MLDE of order $4$, and give rise to a representation $\rho$ of dimension $4$ for $\SL_2(\ZZ)$ with 
\begin{align}
\rho(S)=\frac{1}{\sqrt{3}}\left(
\begin{array}{cccc}
 0 & 1 & 1 & 1 \\
 1 & 1 & 0 & -1 \\
 1 & 0 & -1 & 1 \\
 1 & -1 & 1 & 0 \\
\end{array}
\right), \quad \rho(S)^2=\mathrm{Id}.
\end{align}
It appears that all $\psi_{\frac{1}{3}}(\tau,0)$, $\psi_{\frac{2}{3}}(\tau,0)$, and $\psi_{1}(\tau,0)$ have non-negative Fourier coefficients. However, $\psi_0$ involves negative Fourier coefficients and
$$
\psi_0(\tau,0)= q^{-\frac14}\big(1-3 q+8 q^3-9 q^4+17 q^6-27 q^7+46 q^9-57 q^{10}+O\left(q^{11}\right)\big). 
$$

Using these symbols, we have the expression
\begin{equation}
\phi_{A_{1,3}^*A_{1,4}A_{2,6}} = \phi_0 \psi_{\frac13}-\phi_{\frac13} \psi_1 + \phi_{\frac13}\psi_0.  
\end{equation}
The Macdonald-type identities are
\begin{align}
    9&=(\phi_0 \psi_{\frac13}-\phi_{\frac13} \psi_1)(\tau,0),\\
    3&=(\phi_{\frac13}\psi_0)(\tau,0).
\end{align}
We also find an equality related to the modular function $j$ on $\SL_2(\ZZ)$ as follows
\begin{align}
j(\tau)^{\frac13}=E_4(\tau)/\eta(\tau)^8=\big(\phi_0(\psi_0+\psi_1)+2\phi_{\frac13}\psi_{\frac23}\big)(\tau,0).
\end{align}
Obviously, the product
\begin{equation}
\begin{split}
Z:=&\,Z^{\mathfrak{osp}_{1|2,3}\oplus \mathfrak{sl}_{3,6}}\cdot Z^{\mathfrak{sl}_{2,4}}_{\mathrm{D}} \\
=&\,(|\psi_0|^2+|\psi_{\frac13}|^2+|\psi_{\frac23}|^2+|\psi_{1}|^2)(|\phi_0|^2+2|\phi_{\frac13}|^2)    
\end{split}   
\end{equation}
gives a non-diagonal modular invariant for $\mathfrak{osp}_{1|2}\oplus \mathfrak{sl}_2\oplus \mathfrak{sl}_3$. Furthermore, since both $\phi_0 \psi_{\frac13}-\phi_{\frac13} \psi_1$ and $\phi_{\frac13}\psi_0$ are invariant under $\SL_2(\ZZ)$ at $z=0$, the difference
\begin{equation}
 Z - |\phi_0 \psi_{\frac13}-\phi_{\frac13} \psi_1|^2-|\phi_{\frac13}\psi_0|^2    
\end{equation}
is still invariant under $\SL_2(\ZZ)$ at $z=0$. It seems that this exceptional modular invariant cannot be obtained from a nontrivial automorphism. This is like the case of $\mathfrak{osp}_{1|2}$ at level $36$.

\section{Appendix: an integrality criterion for Jacobi forms of weight 0}\label{sec:appendix}
Recall that the inputs of the Borcherds theta lift on an even lattice $M$ of signature $(l,2)$ are weakly holomorphic modular forms of weight $1-l/2$ with integral principal part for the Weil representation of $\mathrm{Mp}_2(\ZZ)$ associated with $M'/M$. In this appendix, we show that all Fourier coefficients of each input are integral whenever $M\cong 2U\oplus L$. This result is required to construct the BKM superalgebra in Theorem \ref{MTH:relation to BKM}. 

Let $L$ be an even positive-definite lattice of rank $\rank(L)$ with quadratic form $Q(x)=(x,x)/2$ and dual $L'$.  As in Section \ref{subsec:Jacobi}, let $J_{k,L}^!$ denote the space of weakly holomorphic Jacobi forms of weight $k$, index $L$, and trivial character for $\SL_2(\ZZ)$. A Fourier coefficient $f(n,\ell)$ of 
\begin{equation}
\phi(\tau,\mathfrak{z})=\sum_{n\in\ZZ,\, \ell\in L'} f(n,\ell) q^n\zeta^\ell \in J_{k,L}^!, \quad q=e^{2\pi i\tau},\;  \zeta^\ell=e^{2\pi i(\ell,\mathfrak{z})}
\end{equation}
is called \textit{singular} if $n-Q(\ell)<0$. We now state the main result. 

\begin{theorem}\label{th:integrality}
If all singular Fourier coefficients of $\phi\in J_{0,L}^!$ are integral, then $\phi$ has integral Fourier expansion. 
\end{theorem}

The main ingredient of the proof is the following saturation statement. 
\begin{lemma}\label{lem:saturation}
Let $p$ be any prime. Suppose $\phi\in J_{0,L}^!$ has integral Fourier expansion. If $p\mid f(n,\ell)$ whenever $n<Q(\ell)$, then $p$ divides every Fourier coefficient of $\phi$.     
\end{lemma}    

\begin{proof}
Let $N\geq 3$ be an integer with $p\nmid N$. For any $a,b\in N^{-1}L$, set the torsion specialization
$$
\phi_{a,b}(\tau):=q^{Q(a)}\cdot\phi(\tau, a\tau+b).
$$
Similarly to \cite[Theorem 1.3]{EZ85}, we show that $\phi_{a,b}$ is a weakly holomorphic modular form of weight $0$ and trivial character for $\Gamma(N^2)$. Moreover, its transforms under $\SL_2(\ZZ)$ are roots of unity in $\ZZ[\zeta_{N^2}]$ times functions of the same type, where $\zeta_{N^2}=e^{2\pi i/N^2}$. At infinity its expansion is
\begin{equation}\label{eq:torsion-Fourier}
\phi_{a,b}(\tau)=\sum_{n,\ell}f(n,\ell) e^{2\pi i (\ell,b)} q^{E_a(n,\ell)}, \quad E_a(n,\ell)=n-Q(\ell)+Q(\ell+a). 
\end{equation}
These are Laurent series in $q^{1/N^2}$ with integral cyclotomic coefficients at every cusp. Indeed, there is a constant $B$ such that $n-Q(\ell)\geq -B$ if $f(n,\ell)\neq 0$, so $E_a(n,\ell)\leq X$ implies $Q(\ell+a)\leq X+B$; hence, only finitely many pairs contribute below any fixed $q$-exponent. 

Choose a prime $\mathfrak{p}$ above $p$ in $\QQ(\zeta_{N^2})$, and use bars for the reduction modulo $\mathfrak{p}$. By assumption, every term after reduction that survives satisfies $n-Q(\ell)\geq 0$, so \eqref{eq:torsion-Fourier} shows that the reduced Fourier expansions of $\phi_{a,b}$ have only non-negative $q$-exponents at every cusp. Consequently, $\overline{\phi}_{a,b}$ is a global regular function on the geometrically integral compact modular curve of full level $N^2$ with a fixed Weil pairing over the residue field and is therefore constant. Here, we use the standard good-reduction and $q$-expansion results of \cite[\S 1.4-1.6]{Kaz73}. In fact, we first multiply $\phi_{a,b}$ by a sufficiently large power $\Delta^t$ to obtain an integral section, reduce it, and then divide it by $\overline{\Delta}^t$. There are no interior zeros of $\overline{\Delta}$, and the absence of negative $q$-exponents in the reduced Fourier expansions makes the quotient regular at the cusps. This argument also applies when $p=2$ or $3$, since $p\nmid N$. 

If $a\not\in L'$, then the constant term of the above quotient is also zero. Indeed, a surviving term of exponent zero for the reduction would satisfy 
$$
0=E_a(n,\ell)=n-Q(\ell)+Q(\ell+a),
$$
where both summands $n-Q(\ell)$ and $Q(\ell+a)$ are non-negative. Since $L$ is positive-definite, we have $\ell=-a$, contrary to $\ell\in L'$. Therefore,
\begin{equation}\label{eq:nonzero}
\overline{\phi}_{a,b}=0, \quad \text{for} \; a,b\in N^{-1} L, \; a\not\in L',\; p\nmid N.    
\end{equation}

Suppose now that $\overline{\phi}\neq 0$. Since $p\mid f(n,\ell)$ if $n<Q(\ell)$, the support of $\overline{\phi}$ satisfies $n\geq Q(\ell)\geq 0$. Let $n_0$ be the smallest $q$-exponent of $\overline{\phi}$ with a nonzero coefficient, and put
$$
S_0=\{\ell\in L': \; f(n_0,\ell)\not\equiv 0 \bmod p \}.
$$
This set is finite and nonempty. We choose a nonzero vector $v\in L$ for which $(\ell,v)$, $\ell\in S_0$, are pairwise distinct, and let $\ell_0\in S_0$ uniquely minimize $(\ell, v)$ among them. Take $a=v/N$ and $b=0$, with $N$ sufficiently large and $p\nmid N$. Then $a\not\in L'$ for large $N$. Among the terms of $q$-exponent $n_0$, the pair $(n_0,\ell_0)$ has the uniquely smallest specialized $q$-exponent $E_a(n_0,\ell_0)=n_0+(\ell_0,a)+Q(a)$. For every surviving term with $q$-exponent $n\geq n_0+1$, by $Q(\ell)\leq n$ and Cauchy--Schwarz, we have
$$
E_a(n,\ell)\geq \big( \sqrt{n}-\sqrt{Q(a)} \big)^2. 
$$
For sufficiently small $a$ this lower bound is uniformly greater than $E_a(n_0,\ell_0)$ for all $n\geq n_0+1$. Consequently, the leading coefficient of $\overline{\phi}_{a,0}$ is $\overline{f(n_0,\ell_0)}\neq 0$, contradicting \eqref{eq:nonzero}. Hence $\overline{\phi}=0$.  
\end{proof}

\begin{proof}[Proof of Theorem \ref{th:integrality}]
Under theta decomposition, $\phi$ corresponds to a weakly holomorphic modular form $f$ of weight $-\rank(L)/2<0$ for the Weil representation associated with $L$, and the singular Fourier coefficients of $\phi$ correspond to the principal part of $f$. McGraw's integral-basis theorem \cite[Theorem 5.6]{McG03} implies that $f$ has rational Fourier coefficients with bounded denominators. Thus, there is a least positive integer $D$ such that $Df$ and $D\phi$ have integral Fourier expansions. If $p\mid D$, then the singular Fourier coefficients of $D\phi$ are divisible by $p$, so the above lemma implies that $(D/p)\phi$ still has integral Fourier coefficients. Therefore, minimality gives $D=1$.  
\end{proof}

\begin{remark}
Theorem \ref{th:integrality} still holds when the weight of $\phi$ is negative. However, it does not hold in general when the weight of $\phi$ is positive. For example, a certain lift (cf. \cite{Sch09}) of 
$$
\big(2E_2(2\tau)-E_2(\tau)\big)\eta(\tau)^8/\eta(2\tau)^{16}=q^{-1}+16-132q+\cdots \in q^{-1}\ZZ[[q]]
$$
gives a weakly holomorphic modular form of weight $-2$ for the Weil representation associated with $3E_8(2)$ with principal part $q^{-1}e_0$ and zero-component constant coefficient $c_0(0)=\frac{255}{16}$. This form corresponds to a weakly holomorphic Jacobi form of weight $10$, index $3E_8(2)$, and trivial character for $\SL_2(\ZZ)$ with integral singular Fourier coefficients and non-integral constant term.     
\end{remark}

\vspace{5mm}

\noindent
\textbf{Acknowledgments.} 
KS thanks Drazen Adamovi\'c, Chongying Dong, Jethro van Ekeren, and Ching Hung Lam for useful discussions. KS is supported by China NSF grant No.1250010384. HW thanks Richard Borcherds, Valery Gritsenko, Jialin Li, and Xingjun Lin for helpful discussions. BW acknowledges the support of the Deutsche Forschungsgemeinschaft (DFG) through the Collaborative Research Centre TRR 326 ``Geometry and Arithmetic of Uniformized Structures", project number 444845124. 

\vspace{5mm}

\noindent
\textbf{Declaration of AI Use.} 
The authors did not use AI during the conduct of the work or in the preparation of the original version of the manuscript. After the original manuscript was completed, the authors used ChatGPT (OpenAI) to assist with language polishing and proofreading, as well as to check mathematical arguments. Theorem \ref{conj:modular invariants-osp} was only a conjecture in the original manuscript (it is not necessary for the main results of this paper). Its current proof and identity \eqref{eq:osp=osp} were generated by GPT-6 Astra. The original manuscript also omitted the integrality of $f(n^2t^2AC,nt\vec{B})$ needed for the proof of Theorem \ref{MTH:relation to BKM}. GPT identified this gap and proposed a proof of Lemma \ref{lem:saturation} to fix it. Both proofs were subsequently verified and reorganized by the authors. The proofs in Sections 8–11 depend on a large number of computer computations, which were performed by the authors using SageMath in the original manuscript. Based on the authors' algorithms and code (https://github.com/btw-47/weilrep), GPT-6 Astra checked these computations and identified some minor errors. The authors have revised them and take full responsibility for the mathematical arguments and computational claims.

\bibliographystyle{plainnat}
\bibliofont
\bibliography{refs}
\end{document}